\documentclass[a4paper,10pt,reqno]{amsart} %
\usepackage{latexsym,amssymb}
\usepackage{hyperref}
\usepackage{amsmath,amsthm,amsxtra}
\usepackage{amsfonts}
\usepackage[american]{babel}
\usepackage{mathrsfs}
\usepackage{psfrag}
\usepackage{epsfig,inputenc,eucal}
\usepackage{graphpap,latexsym,epsf}
\usepackage{color}
\usepackage{upgreek}
\usepackage{todonotes}
\usepackage[normalem]{ulem}
\usepackage{amssymb,eucal,paralist,color,enumitem}
\usepackage{graphicx}

\usepackage{mathtools}
\mathtoolsset{showonlyrefs}

\newcommand{\ep}{\varepsilon}
\newcommand{\R}{\mathbb{R}}
\newcommand{\N}{\mathbb{N}}

\mathchardef\emptyset="001F

\newtheorem{maintheorem}{Theorem}
\newtheorem{theorem}{Theorem}[section]
\newtheorem{maindefinition}[theorem]{Definition}

\newtheorem{lemma}[theorem]{Lemma}
\newtheorem{remark}[theorem]{Remark}

\newtheorem{definition}[theorem]{Definition}
\newtheorem{proposition}[theorem]{Proposition}
\newtheorem{example}[theorem]{Example}
\newtheorem{notation}[theorem]{Notation}
\newtheorem{corollary}[theorem]{Corollary}

\numberwithin{equation}{section}

\newcommand{\up}{\uparrow}

\newcommand{\eps}{\varepsilon}

\newcommand{\weakto}{\rightharpoonup} 

\newcommand{\DSolution}{Dissipative Viscosity solution}
\newcommand{\DSolutions}{Dissipative Viscosity solutions}
\newcommand{\BSolution}{Balanced Viscosity solution}
\newcommand{\BSolutions}{Balanced Viscosity solutions}

\newcommand{\down}{\downarrow}
\newcommand{\weaksto}{\rightharpoonup^*}

\newcommand{\AC}{\mathrm{AC}}

\newcommand{\Ene}{\mathcal{E}}

\newcommand{\Hilbert}{\mathsf{H}}

 \def\calE{{\mathcal E}} \def\calF{{\mathcal F}}

 \def\rme{{\mathrm e}}

\def\rmA{{\mathrm A}}  \def\rmC{{\mathrm C}}
\def\rmD{{\mathrm D}}  
  \def\rmI{{\mathrm I}}
\def\rmJ{{\mathrm J}}

  \def\rmU{{\mathrm U}}
\def\rmV{{\mathrm V}} \def\rmW{{\mathrm W}}

\def\dd{\;\!\mathrm{d}} 

\newcommand{\pairing}[4]{ \sideset{_{ #1 }}{_{ #2 }}  {\mathop{\langle #3 , #4
\rangle}}}

\newcommand{\teta}{\vartheta}

\newcommand{\nchi}{{\raise.2ex\hbox{$\chi$}}}

\definecolor{ddcyan}{rgb}{0,0.1,0.9}
\definecolor{ddmagenta}{rgb}{0.8,0,0.8}
\definecolor{orange}{rgb}{0.6,0.2,0}

\newcommand{\piecewiseConstant}[2]{\overline{#1}_{\kern-1pt#2}}

\newcommand{\foraa}{\text{for a.a.}}

\newcommand{\cE}{\mathcal{E}}

\newcommand{\ene}[2]{\mathcal{E}_{#1}(#2)}

\newcommand{\eneb}[2]{\widetilde{\mathcal{E}}_{#1}(#2)}
\newcommand{\pt}[2]{\mathcal{P}_{#1}(#2)}

\newcommand{\admis}[3]{\mathscr{A}(#3;#1,#2)}
\newcommand{\la}{\langle}
\newcommand{\ra}{\rangle}
\newcommand{\cost}[3]{\mathsf{c}_{#1}(#2;#3)}
\newcommand{\costname}[1]{\mathsf{c}_{#1}}

\newcommand{\limen}{\mathscr{E}}
\newcommand{\limp}{\mathscr{P}}
\newcommand{\BV}{\mathrm{BV}}

 \def\trait #1 #2 #3 {\vrule width #1pt height #2pt depth #3pt}

 \def\fin{\hfill
         \trait .3 5 0
         \trait 5 .3 0
         \kern-5pt
         \trait 5 5 -4.7
         \trait 0.3 5 0
 \medskip}
 
\newcommand{\QED}{\mbox{}\hfill\rule{5pt}{5pt}\medskip\par}

\newcommand{\minsub}[2]{\partial\cE^\circ_{#1}(#2)}

\newcommand{\mdt}[3]{\|{#1}'_{#2}(#3)\|}
\newcommand{\mdtq}[4]{\|{#1}'_{#2}(#3)\|^{#4}}

\newcommand{\minpartial}[3]{\|\partial{#1}^{\circ}_{#2}(#3)\|}

\newcommand{\minpartialq}[3]{\|\partial{#1}^{\circ}_{#2}(#3)\|^2}
\newcommand{\argminpartial}[3]{\partial{#1}^{\circ}_{#2}(#3)}

\newcommand{\karrow}{\stackrel{\mathrm{K}}{\longrightarrow}}
\newcommand{\kliminf}{\mathop{\mathrm{Li}}}
\newcommand{\klimsup}{\mathop{\mathrm{Ls}}}
\newcommand{\pij}[2]{\critset(#2,#1)}

\newcommand{\ulim}{\mathrm{U}}
\newcommand{\vlim}{\mathrm{V}}

\newcommand{\gder}[4]{\mathrm{D}^{#1} {#2}_{#3}(#4)}

\newcommand{\opx}{O}

\newcommand{\NN}{\mathrm{N}}
\newcommand{\RR}{\mathrm{R}}
\newcommand{\codim}{\mathrm{codim}}
\newcommand{\ind}{\mathrm{ind}}
\newcommand{\parcur}{\gamma}
\newcommand{\curv}[1]{\mathfrak{#1}}

\newcommand{\Hausdorff}{{\mathrm H}}

\newenvironment{rcomm}{\color{red}}{\color{black}}

\newenvironment{rnew}{\color{ddmagenta}}{\color{black}}

\newcommand{\berin}{\begin{rnew}}
\newcommand{\erin}{\end{rnew}}

\newcommand{\beroc}{\begin{rcomm}}
\newcommand{\eroc}{\end{rcomm}}

\newenvironment{newricky}{\color{ddcyan}}{\color{black}}
\newcommand{\bnr}{\begin{newricky}}
\newcommand{\enr}{\end{newricky}}

\newcommand{\RRR}{\color{black}}

\newcommand{\GG}{\color{red}}

\newcommand{\nc}{\normalcolor}

\newcommand{\EEE}{\color{black}}

\definecolor{vgreen}{rgb}{0.1,0.5,0.2}

\definecolor{dcyan}{rgb}{0,0.3,0.8}

\newcommand{\bE}{{\boldsymbol{\mathsf E}}}
\newcommand{\domainenergy}{{\mathsf E}}
\newcommand{\EK}{\domainenergy}
\newcommand{\sublevel}[1]{\domainenergy[#1]}
\newcommand{\subl}[1]{\domainenergy[#1]}
\newcommand{\Sublevel}[1]{\bE[#1]}
\newcommand{\Subl}[1]{\bE[#1]}

\newcommand{\sfd}{\mathsf d}

\newcommand{\comp}[2]{\critset(#1,#2)}
\newcommand{\cU}{U}
\newcommand{\crit}{\mathbf C}
\newcommand{\sing}{\mathbf S}
\newcommand{\critset}{\mathrm C}
\newcommand{\singset}{\mathrm S}
\newcommand{\necrit}{\mathbf G}
\newcommand{\necritset}[1]{\mathrm{G}(#1)}
\newcommand{\necri}[2]{G_{#1}[#2]}
\newcommand{\totaldisc}{{\mathrm{TD}_0}}
\newcommand{\Honezero}{{\mathrm{TD}_1}}
\newcommand{\noHonezero}{{\mathrm{TD}_1^c}}
\newcommand{\nototaldisc}{{\mathrm{TD}_0^c}}
\newcommand{\res}{\mathop{\hbox{\vrule height 7pt width .5pt depth 0pt
\vrule height .5pt width 6pt depth 0pt}}\nolimits}
\newcommand{\sft}{\mathsf t}
\newcommand{\sfs}{\mathsf s}
\newcommand{\sfr}{\mathsf r}

\newcommand{\sfv}{\mathsf v}
\newcommand{\disc}[1]{\operatorname{Disc}(#1)}
\newcommand{\SV}[1]{\mathsf{SV}(#1)}
\newcommand{\jump}[1]{\operatorname{J}(#1)}
\newcommand{\bfTheta}{\boldsymbol{\Theta}}
\newcommand{\bfXi}{\boldsymbol{\Xi}}
\newcommand{\sfe}{e}
\renewcommand{\d}{\mathrm d}

\newcommand{\unn}{u_{\eps_n}}

\newcommand{\Gunn}{\mathbf{U}_{\eps_n}}
\newcommand{\Gu}{\mathbf{U}}
\newcommand{\Gv}{\mathbf{V}}

\newcommand{\Gvf}{\mathrm{V}}

\newcommand{\newclass}[1]{\mathfrak{E}_{#1}(\Hilbert)} 
\newcommand{\AS}[2]{\mathfrak{S}_{#1}(#2)}
\newcommand{\AG}[3]{\mathfrak{A}(#1;#2;#3)}
\newcommand{\slo}{S}
\newcommand{\Conn}{\Gamma}
\newcommand{\gmetr}{\boldsymbol{\mathsf{X}}}
\newcommand{\gdist}{\mathsf{d}}
\newcommand{\ti}{{\times}}
\newcommand{\jumpname}{\mathrm{J}}

\newcommand{\imsetG}{\mathsf{G}}

\newcommand{\Xsp}{\mathsf{X}}
\newcommand{\Ysp}{\mathsf{Y}}

\title[Singularly Perturbed Gradient Flows and Evolutions of Critical Points]{Singularly Perturbed Gradient Flows and\\  Evolution  of Critical Points in Infinite Dimensions}
\author{Virginia Agostiniani}
\address{V.\ Agostiniani, Dipartimento di Matematica, Universit\`a degli studi di Trento, via Sommarive 14, 38123 Povo (Trento) - Italy   }
\email{\ttfamily virginia.agostiniani\,@\,unitn.it}

\author{Riccarda Rossi}
\address{R.\ Rossi, Dipartimento di Ingegneria Meccanica e Industriale, Universit\`a degli studi di Brescia, via Valotti 9, 25133 Brescia - Italy}
\email{riccarda.rossi\,@\,unibs.it}

\author{Giuseppe Savar\'e}
\address{G.\ Savar\'e, Dipartimento di Scienze delle Decisioni  and BIDSA, Universit\`a Bocconi, via  R\"ontgen 1, 20136 Milano - Italy}
\email{giuseppe.savare\,@\,unibocconi.it}

\thanks{
The authors acknowledge
support from the 
Gruppo Nazionale per l’Analisi Matematica, la Probabilit\`a e le loro Applicazioni (GNAMPA-INDAM).
V.A.\
has been partially supported by the PRIN project ``Contemporary perspectives on geometry and gravity''.
R.R.\ and G.S.\  acknowledge
  support from IMATI (CNR), Pavia.
G.S.\ has been partially supported by funding from the European
Research Council (ERC) under the European Union’s Horizon Europe
research and innovation programme (grant agreement No. 101200514,
project acronym OPTiMiSE).
 Views and opinions expressed are however those of the author(s) only and do not necessarily reflect those of the European Union or the European Research Council Executive Agency. Neither the European Union nor the granting authority can be held responsible for them. \\
 During the final revision of this work, the authors used Claude (Anthropic) in a limited, auxiliary capacity: to improve the language and readability of the authors' own text, to help word a few brief routine passages from material supplied by the authors, and to generate TikZ code for the figures whose content and structure were specified by the authors. All mathematical ideas, results, and proofs are the authors' own; all AI-assisted material was reviewed, revised where necessary, and independently verified by the authors, who take full responsibility for the content of the manuscript.
 }

\begin{document}

\begin{abstract}

 We consider singularly perturbed gradient flows in Hilbert spaces,
  driven by a time-dependent, nonconvex, and nonsmooth   energy, and 
address the convergence of their solutions to curves of critical points of the driving energy functional. 
The degenerating nature of the estimates along the gradient-flow curves calls for novel compactness arguments, which we carefully develop by combining 
tools from the variational approach to  Hilbert and metric gradient   flows \cite{RossiSavare06,AGS08}, with fine  requirements on the   set of critical points of the energy. 
\par
This leads us to prove that  subsequential limits of singularly perturbed gradient flows are \emph{Dissipative Viscosity}  solutions of the limiting problem, i.e., 
 curves of critical points satisfying a suitable  balance between the energy and a defect measure, encoding dissipation.
 This energy-dissipation balance
 encompasses information on the dynamics of the process at  jump  
 times, recording, in particular, the re-emergence of viscous behavior.
 \par
 Under  a suitable  rectifiability condition on the critical set, we show that \DSolutions\ improve to \emph{Balanced Viscosity} solutions, which have the key property that the 
 dissipation measure is purely atomic. 
 \par
In the second part of the paper we show that, for smooth energies
whose second differential is a Fredholm operator, the condition that
the kernel of the Hessian has dimension at most one at every critical
point already implies our measure-theoretic assumptions. We further
relate them to the transversality conditions from bifurcation theory
and show that they have a \emph{generic} character
 with respect  to a wide family of linear perturbations of the energy.

The results are new also in finite dimension. Applications to various
PDE models are presented.
\end{abstract}

\noindent

\maketitle

\centerline{\today}

\tableofcontents

\section{Introduction}

Several evolution processes in  science are driven by an energy
landscape that  is modulated by external factors on a time scale much
slower than the intrinsic relaxation time of the system. When the
evolution is observed on the external slow time scale, one is
led to the singularly perturbed gradient flow
 \begin{equation}
\label{eq:GFmodel-intro}
-\eps  u'(t)\in 
 \partial \ene t{u(t)} ,\quad t \in (0,T),
\end{equation}
where the parameter $\eps>0$ represents the ratio between the
relaxation time of the system and the characteristic time of the
variation of the time-dependent energy $\ene t\cdot$, defined in a Hilbert space $\Hilbert$.
Here and throughout, 
the map $u\mapsto \ene tu$ 
is
allowed to be \emph{nonsmooth} and \emph{nonconvex}, and $\partial\calE_t$
stands for its \emph{Fr\'echet subdifferential}, a natural localization to the
nonconvex setting of the subdifferential of convex analysis, whose definition
is recalled in the next section; when $\calE_t$ is differentiable,
$\partial\calE_t$ reduces to the single-valued differential $\rmD\calE_t$, and
\eqref{eq:GFmodel-intro} to the classical gradient flow
$\eps u'(t)=-\rmD\ene t{u(t)}$. 
The limit $\eps \down 0$ describes the \emph{quasistatic} regime: at
almost every time the state of the system is at equilibrium, i.e.\ it
is a critical point of the energy $\mathcal E_t(\cdot)$ at the current
time, a solution of the limit equation
\begin{equation}
\label{eq:critical-intro}
\partial \ene t{u(t)} \ni 0, \quad t \in (0,T);
\end{equation}
still, the slow modulation of the landscape may destabilize the
equilibrium,
thereby activating the fast internal dynamics, which drives the system
towards a new equilibrium.

On the slow time scale such fast transients appear as \emph{jumps}, at
some $t\in (0,T)$, between two values $u(t\pm)$,
and the energy they dissipate is not seen by the formal limit equation
\eqref{eq:critical-intro}: it has to be recorded by an additional
\emph{defect measure}, which captures the trajectory followed along
the jump and the corresponding energy dissipation.
In the simplest
cases, this is a
heteroclinic orbit for the ``frozen energy'' $\ene t\cdot$ 
\begin{equation}
\label{eq:heteroclinic-intro}
 \teta'(s)\in -
 \partial \ene t{\teta(s)} ,\quad s \in (-\infty,+\infty),\quad
 \teta(-\infty)=u(t-),\quad
 \teta(+\infty)=u(t+).
\end{equation}
Under cyclic driving this mechanism dissipates energy: over a closed loading cycle the dissipated energy equals the work expended by the loading along the cycle. For a scalar loading this is the area of the classical hysteresis loop (as in the stress–strain or magnetization–field diagrams).
The emergence of
rate-independent hysteresis from viscous dynamics under slow loading
is a classical theme in continuum mechanics, from ferromagnetism to
elasto-plasticity, phase transformations, damage, and fracture. The
mathematical derivation of such rate-independent hysteresis as a
\emph{vanishing-viscosity} limit of viscous gradient dynamics, the
viewpoint closest to the present paper,
goes back to \cite{EfeMie06RILS,MieThe04RIHM}, cf.\ also
\cite{MRS2016,MiRo23} and the monographs \cite{Visintin94,BrokateSprekels96,MieRouBOOK}.

\par
The two-scale structure of \eqref{eq:GFmodel-intro}
also arises
naturally in nonconvex optimization. For small learning rates, the 
gradient-descent training of a model is a discretization of the
gradient flow of the loss functional; in several algorithmic
frameworks of current interest the loss itself varies slowly during
the training process: this is the case for continuation and homotopy
strategies, for the annealing of regularization or penalization
parameters, for curriculum-learning protocols, and for online learning
under a drifting data distribution (see \cite{Simonetto20} for a survey of such \emph{time-varying optimization} problems, mostly in the convex setting). 
In all of  these situations, 
the parameter $\eps$
measures the ratio between the response time of the optimizer and the
time scale on which the objective is modulated, and the singular limit
$\eps\down0$ captures the effective dynamics of the training process:
the trainable parameters follow a branch of critical points of the
evolving loss and, when the branch degenerates, they undergo a fast
transition towards a new equilibrium, a picture close to the \emph{saddle-to-saddle} training dynamics of deep linear networks \cite{JacotSaddle21}. We refer to \cite{CiFoSca25} for
a description of deep learning architectures in terms of a closely
related (mean-field) control and gradient-flow formalism, and to
\cite{AlForKleSca25} for a constructive approach to quasistatic
evolutions of critical points with an emphasis on computational
implementability. Let us also mention that evolutions of critical
points of time-dependent energies play a role in data science also as
\emph{forward models} of inverse problems: in \cite{AlFoHu20} the
driving energy is \emph{learned} from the observation of evolutions of
its critical points originating from random initial equilibria.
A
precise description of the class of limit evolutions, which is the
object of the present paper, thus provides a natural and powerful framework
for the analysis of such identification problems.

\par
The rigorous analysis of
the limit $\eps\down0$
faces two intertwined
difficulties. On the one hand, the natural a priori estimates along
\eqref{eq:GFmodel-intro}
degenerate as $\eps \down 0$, so that
compactness of the trajectories fails in every classical sense. On the
other hand, beyond the convex, or non-degenerate, regime, the set of
equilibria of $\calE_t$ may exhibit a complicated structure.
For instance, due to overparametrization and to the intrinsic symmetries of the
parametrization (e.g., invariance under permutations or rescalings of
the parameters), the critical points of machine-learning loss
functionals are typically \emph{non-isolated} and \emph{degenerate}:
they occur in manifolds, or in even more complicated continua (cf.\
the two guiding examples and Remark \ref{cex:rmk-symmetry} below).
Approaches requiring isolated or non-degenerate critical points are
therefore ruled out from the outset, whereas the measure-theoretic
conditions at the core of our analysis remain compatible with this
degeneracy, to an extent that we shall make precise.

\par

The aim of this paper is to provide a rigorous description of the
limit evolution, and corresponding solution concepts for the limit
problem \eqref{eq:critical-intro}, under assumptions on the
\emph{critical set}
\begin{equation}
  \label{eq:47}
\crit:=\big\{(t,u)\in [0,T]\ti \Hilbert:\, \partial\ene tu\ni 0\big\},
\qquad
\critset(t):=\big\{u\in \Hilbert:\,(t,u)\in \crit\big\}
\end{equation}
(the fibers $\critset(t)$ collect the solutions of
\eqref{eq:critical-intro} at time $t$)
that are measure-theoretic in nature, robust, and generic, and
that in particular dispense with any non-degeneracy or isolatedness
requirement on critical points. Our results have a graded structure,
in which finer descriptions of the limit are matched by stronger
conditions on the critical set. The weakest condition --- the
Sard-type property $\mathrm{(C)}$ below, requiring the
$\mathscr L^1$-negligibility of the sets $\ene t{\critset(t)}$ of critical values of the
energy --- allows for continua of critical points in every fiber
$\critset(t)$, and already supports a basic solution concept (that of
\emph{Dissipative Viscosity} solutions), for which we prove existence
and consistency results. The finer properties of the limit evolution
--- the sequential compactness of the families of gradient flows with
respect to pointwise convergence on $[0,T]$, and the purely
atomic structure of the dissipation --- rely on stronger conditions. 
 The validity of such properties is still compatible with the presence of 
continua in the fibers, but only at exceptional
(typically at
most countably many) times. Such stronger conditions, while not
automatic, are \emph{generic}: for $\rmC^3$ energies with Fredholm
differential (see Sections \ref{sec:6} and \ref{s:8} for the precise
setting), they can be enforced by arbitrarily small \emph{linear}
perturbations of the energy, forming a residual set (of full
Lebesgue measure in finite dimension). It is precisely the weakness of
our measure-theoretic requirements that leads to such a strong
genericity result: since they follow from the \emph{first-order}
transversality conditions alone, merely linear perturbations (rather
than the quadratic corrections needed for the full set of
transversality conditions of bifurcation theory) and a moderate
smoothness of the energy suffice.

\subsection*{Two guiding examples}
Before describing our results, we pick two elementary
one-dimensional examples to illustrate  the phenomena that any notion of limit
evolution for \eqref{eq:GFmodel-intro} has to deal with, and
the role that our assumptions will play.
Despite their simplicity,
they capture the main features of the more complex and involved infinite-dimensional applications of the theory, cf.\ Theorem \ref{cex:th-AC-short} below and Section \ref{ss:9.2}. Singularly perturbed gradient flows in finite dimension have been analyzed in detail in \cite{Zanini,agostiniani2012,AgoRos17,ScillaSolombrino18-1}, to which we refer for further examples.

\begin{example}[Dissipative jumps at fold points]
\label{cex:ex-fold}
\upshape
Consider the double-well energy $W$ in $\Hilbert=\R$ perturbed by a time-dependent loading
\begin{equation}
\label{cex:fold-energy}
\ene tu:= W(u)-\ell(t)\,u,\quad
W(u):=\frac14\,(u^2-1)^2 , \quad \ell:[0,T]\to \R\text{ increasing,}\quad
\ell(0)<-\ell^*, \ \ell(T)>\ell^*:=\tfrac2{3\sqrt3}.
\end{equation}
The fibers of the critical set are $\critset(t)=\{u \in \R:\,
u^3-u=\ell(t)\}$: for $|\ell(t)|<\ell^*$ they consist of three points
(two local minimizers $m_\pm(t)$, separated by a local maximizer),
which merge in pairs at the \emph{fold} times $t^*_\pm$ when
$\ell(t^*_\pm)=\pm\ell^*$, cf.\ Figure \ref{cex:fig-fold}.
As $\eps\down0$, the solutions of \eqref{eq:GFmodel-intro} starting
from the negative well converge pointwise to a limit $u$, which tracks
the branch $m_-$
well beyond the time $t_0$ at which $\ell(t_0)=0$ and $u(t_0)$ 
ceases to be the global minimizer,
and in fact up to the fold time
$t_+^*$, consistently with Theorem~\ref{rmk:localmin-nojump} below.
There, $m_-$ collides with the unstable branch at the degenerate critical
point $u^-=-1/\sqrt3$.
Metastability is thus an essential
feature of the vanishing-viscosity limit, distinguishing it from the
so-called \emph{energetic} evolutions built on \emph{global} minimization, which would jump already at $t_0$, where the two wells exchange global stability.
\par
The existence of a Dissipative Viscosity solution on the whole interval
$[0,T]$ is guaranteed by Theorem \ref{mainth:1}. Now,
since the loading crosses the threshold $\ell^*$, for $t>t_+^*$ the critical fiber reduces to $\critset(t)=\{m_+(t)\}$.
No such time can belong to the jump set $\mathrm J$ (see Definition \ref{def:sols-notions-1}), since a jump requires two
distinct critical components with positive transition cost. 
Thus $u$ is the curve of critical points
\[
u(t)=m_-(t) \ \text{ for } t<t^*_+,
\qquad
u(t)=m_+(t)\ \text{ for } t> t^*_+,
\]
whose jump at $t^*_+$, from $u^-$ to $u^+=2/\sqrt3$, is realized by (a reparametrization of) the heteroclinic solution of the gradient flow of the \emph{frozen-time} energy, $\teta'=-\rmD \ene{t^*_+}{\teta}$, connecting $u^-$ to $u^+$: the internal viscous time scale, neglected by the formal limit, re-emerges and governs the jump regime. The energy dissipated in the transition is recorded by  a  defect measure $\mu$:   with the terminology introduced below, $u$ is a \BSolution, with
\[
\mu=\cost{t^*_+}{u^-}{u^+}\,\delta_{t^*_+}\nc,
\qquad
\cost{t^*_+}{u^-}{u^+}\nc=\int_{u^-}^{u^+} |W'(\teta)-\ell^*|\,\d\teta=\ene{t^*_+}{u^-}-\ene{t^*_+}{u^+}>0 .
\]
Under a cyclic loading that first increases beyond the positive fold
threshold $\ell^*$ before reversing direction, and then decreases
beyond the negative threshold $-\ell^*$,
the
same mechanism produces two jumps per cycle, and the trajectory
encloses a hysteresis loop whose area is the total dissipated energy,
regardless of the speed of the cycle: the limit evolution is
\emph{rate-independent}
and provides the elementary mechanism of magnetic hysteresis, see e.g.~\cite{Bertotti98,Visintin94,Brokate00}, which frames the switching as a fold (saddle--node) catastrophe of the free energy; the slow passage through such a fold, and the ensuing hysteresis loop, are analyzed in dynamical-systems terms in \cite{Berglund01}.
\end{example}

\begin{figure}[ht]
\centering
\begin{tikzpicture}[xscale=1.5,yscale=2.5,baseline=(current bounding box.center)]
  \draw[->,gray] (-1.65,0) -- (1.8,0) node[below,black] {$u$};
  \draw[->,gray] (0,-0.06) -- (0,0.5) node[left,black] {\small $W$};
  \draw[very thick] plot[domain=-1.45:1.45,samples=90,variable=\u] ({\u},{0.25*(\u*\u-1)*(\u*\u-1)});
  \node[circle,fill,inner sep=1.2pt] at (-1,0) {};
  \node[circle,fill,inner sep=1.2pt] at (1,0) {};
  \node[circle,draw=red,inner sep=1.1pt] at (0,0.25) {};
  \node[below] at (-1,-0.015) {\small $-1$};
  \node[below] at (1,-0.015) {\small $1$};
  \node[above right] at (0.02,0.25) {\small $0$};
\end{tikzpicture}
\hfil
\begin{tikzpicture}[xscale=3.6,yscale=1.0,baseline=(current bounding box.center)]
  \draw[->,gray] (-0.85,0) -- (0.85,0) node[below,black] {$\ell$};
  \draw[->,gray] (0,-1.75) -- (0,1.85) node[left,black] {$u$};
  \draw[very thick] plot[domain=0.5774:1.26,samples=80,variable=\u] ({\u*\u*\u-\u},{\u});
  \draw[very thick] plot[domain=-1.26:-0.5774,samples=80,variable=\u] ({\u*\u*\u-\u},{\u});
  \draw[thick,densely dashed] plot[domain=-0.5774:0.5774,samples=60,variable=\u] ({\u*\u*\u-\u},{\u});
  \node[circle,fill,inner sep=1.2pt] at (0.3849,-0.5774) {};
  \node[circle,fill,inner sep=1.2pt] at (-0.3849,0.5774) {};
  \draw[->,thick] (0.3849,-0.50) -- (0.3849,1.06);
  \node[circle,draw,inner sep=1.1pt] at (0.3849,1.1547) {};
  \draw[->,thick] (-0.3849,0.50) -- (-0.3849,-1.06);
  \node[circle,draw,inner sep=1.1pt] at (-0.3849,-1.1547) {};
  \node[below right] at (0.39,-0.02) {$\ell^*$};
  \node[above left] at (-0.40,0.02) {$-\ell^*$};
  \node[below] at (-0.10,-1.05) {$m_-$};
  \node[above] at (0.10,1.15) {$m_+$};
  \node[right] at (0.41,0.30) {\small jump at $t^*_+$};
  \node[left] at (-0.41,-0.30) {\small jump at $t^*_-$};
  \node[below right] at (0.36,-0.62) {$u^-$};
  \node[above] at (0.40,1.15) {$u^+$};
\end{tikzpicture}
{\caption{\emph{Left}: the double well $W(u)=\tfrac14(u^2-1)^2$ (the frozen energy at $\ell=0$), with its two minima at $\pm1$ --- the two phases --- and the barrier at $0$. \emph{Right}: the critical set of Example \ref{cex:ex-fold} in the $(\ell,u)$-plane. Solid: branches of local minimizers $m_\pm$; dashed: local maximizers; the fold points are marked. Under a loading cycle the limit evolution tracks $m_-$ and jumps up to $m_+$ at $t^*_+$ (where $\ell=\ell^*$), then it tracks $m_+$ and jumps back down at $t^*_-$ (where $\ell=-\ell^*$), each jump along a heteroclinic of the frozen-time gradient flow; the two jumps enclose the hysteresis loop.}\label{cex:fig-fold}}
\end{figure}
\begin{example}[Jumps through a continuum of critical points]
\label{cex:ex-plateau}
\upshape
Let $\Hilbert=\R$, let $\ell$ be a continuous loading which increases
from $\ell_0=\ell(0)<-1/2$ to a positive maximum and then
decreases to $\ell(T)=\ell_0$, and consider
\begin{equation}
\label{cex:plateau-energy}
\ene tu:= F(u)-\ell(t)\,u,
\qquad
F(u):=
\begin{cases}
\tfrac12\,u^2 & \text{if } u \le 0,\\
-\tfrac12\,u^2 & \text{if } 0\le u\le \tfrac12,\\
\tfrac12\,(u-1)^2 -\tfrac14 & \text{if } \tfrac12 \le u\le 1,\\
-\tfrac14 & \text{if } 1\le u \le 2,\\
\tfrac12\,(u-2)^2-\tfrac14 & \text{if } u \ge 2,
\end{cases}
\end{equation}
a piecewise quadratic $\rmC^{1,1}$ potential, cf.\ Figure
\ref{cex:fig-plateau}. The
critical fibers are computed by inspection: for $\ell(t)\in(-\tfrac12,0)$ the fiber
$\critset(t)$ consists of the three points
$\ell(t)<-\ell(t)<1+\ell(t)$, two stable critical points separated by
an unstable one, which ranges in the concavity interval
$(0,\tfrac12)$ of $F$; for $\ell(t)>0$ it reduces to the single point
$2+\ell(t)$, and for $\ell(t)<-\tfrac12$ to the single point
$\ell(t)$; at the times where $\ell(t)=0$, one has
\[
\critset(t)=\{0\}\cup[1,2],
\]
the union of a degenerate critical point and of a \emph{connected
  component which is a continuum}.
\par
As $\eps\down0$, the solutions starting from $u_0=\ell_0$ converge
pointwise to the limit
\[
u(t)=
\begin{cases}
\ell(t) & \text{for } t<t_1,\\
2+\ell(t) & \text{for } t_1<t<t_2,\\
1+\ell(t) & \text{for } t_2<t<t_3,\\
\ell(t) & \text{for } t>t_3,
\end{cases}
\]
where $t_1<t_2$ are the two times at which $\ell$ vanishes and
$t_3\in (t_2,T)$ is the (last) time at which $\ell(t_3)=-\tfrac12$.
The three
discontinuities of $u$ have different natures.
\par
At $t_1$ the evolution jumps from the destabilized fold point
$u(t_1-)=0$ to $u(t_1+)=2$, \emph{passing through the continuum}, and concatenates
the heteroclinic descent
from $0$ to the left endpoint of the plateau with a free slide along
$[1,2]$, crossed at zero cost.
Transitions of this kind are the reason why, in Section
\ref{ss:2.2-added}, the cost $\mathsf c_t$ is defined between
\emph{connected components} of the critical set, and not merely
between critical points.
\par
At $t_2$, on the return, the branch $2+\ell(t)$ terminates
\emph{into} the right endpoint of the plateau, and the limit slides
back across it, jumping from $u(t_2-)=2$ to $u(t_2+)=1$ at
\emph{zero cost}: no energy is dissipated, and this discontinuity is
invisible to the energy balance.
Moreover, the trajectories $u_\eps$
cross the plateau within a vanishing time window around $t_2$: their
graphs converge in the Kuratowski sense to a limit set
whose fiber at $t_2$ is the \emph{whole segment} $[1,2]$, and a
solution can only be recovered from the limit graph through a
measurable \emph{selection} --- non-uniquely, since any value in
$[1,2]$ may be selected at $t_2$.
\par
Finally, at $t_3$ the branch $1+\ell(t)$ is annihilated against
the unstable branch at the second fold $(-\tfrac12,\tfrac12)$, and
the limit jumps from $u(t_3-)=\tfrac12$ to
$u(t_3+)=-\tfrac12$ along a heteroclinic of the frozen energy,
as in Example \ref{cex:ex-fold}, with
$\cost{t_3}{\tfrac12}{-\tfrac12}=
\ene{t_3}{\frac12}-\ene{t_3}{-\frac12}
=
\tfrac14$: the state is back on
the initial branch, and the trajectory encloses a hysteresis loop.

Summarizing,
$u$ is a \BSolution\ with defect measure
$\mu=\tfrac14\,\delta_{t_1}+\tfrac14\,\delta_{t_3}$, and  its discontinuity set 
$\disc u=\{t_1,t_2,t_3\}$ \emph{strictly} contains the set
$\jumpname=\{t_1,t_3\}$ of atoms of $\mu$ (whence the careful
distinction between the two sets in Definition
\ref{def:sols-notions-1} below).
\par
None of the features displayed by  Example  \ref{cex:ex-plateau} is compatible with the
assumptions of \cite{Zanini,AgoRos17},
where the fibers $\critset(t)$ consist of isolated points. On the
other hand, in both of the above guiding examples the conditions of Theorems
\ref{mainth:1} and \ref{mainth:2} hold by inspection: the critical
sets are unions of finitely many 
Lipschitz arcs,
hence
countably $\mathscr H^1$-rectifiable. Thus, condition
$\mathrm{(C)}$ and the countability of the set of times at which
$\critset(t)$ fails to be totally disconnected follow, e.g.,
from Lemma \ref{le:criterion} ahead. Note, finally, that $\calE$ is merely
$\rmC^{1,1}$: this is immaterial for our approach, which never
differentiates the energy twice, while techniques based on
bifurcation theory require $\rmC^3$-smoothness.
\end{example}
\nc

\GG
\begin{figure}[ht]
\centering
\begin{tikzpicture}[xscale=1.25,yscale=2.3,baseline=(current bounding box.center)]
  \draw[->,gray] (-1.15,0) -- (3.35,0) node[below,black] {$u$};
  \draw[very thick] plot[domain=-1.05:0,samples=40,variable=\u] ({\u},{0.5*\u*\u});
  \draw[very thick,densely dashed] plot[domain=0:0.5,samples=25,variable=\u] ({\u},{-0.5*\u*\u});
  \draw[very thick] plot[domain=0.5:1,samples=25,variable=\u] ({\u},{0.5*(\u-1)*(\u-1)-0.25});
  \draw[line width=2.4pt] (1,-0.25) -- (2,-0.25);
  \draw[very thick] plot[domain=2:3.15,samples=40,variable=\u] ({\u},{0.5*(\u-2)*(\u-2)-0.25});
  \node[circle,fill,inner sep=1.2pt] at (0,0) {};
  \node[circle,fill,inner sep=1.2pt] at (0.5,-0.125) {};
  \draw[densely dotted] (0.5,0) -- (0.5,-0.125);
  \draw[densely dotted] (1,0) -- (1,-0.25);
  \draw[densely dotted] (2,0) -- (2,-0.25);
  \node[above] at (0.5,0.015) {\small $\tfrac12$};
  \node[above] at (1,0.015) {\small $1$};
  \node[above] at (2,0.015) {\small $2$};
  \node[above left] at (-0.02,0.02) {\small $0$};
  \node[above right] at (-1.02,0.44) {\small $F$};
  \node[below] at (1.5,-0.27) {\small $-\tfrac14$};
\end{tikzpicture}
\hfil
\begin{tikzpicture}[xscale=2,yscale=1.5,baseline=(current bounding box.center)]
  \draw[->,gray] (-0.85,0) -- (0.75,0) node[below,black] {$\ell$};
  \draw[->,gray] (0,-0.85) -- (0,2.75) node[left,black] {$u$};
  \draw[very thick] (-0.75,-0.75) -- (0,0);
  \draw[very thick] (-0.5,0.5) -- (0,1);
  \draw[very thick] (0,2) -- (0.55,2.55);
  \draw[thick,densely dashed] (-0.5,0.5) -- (0,0);
  \draw[line width=2.4pt] (0,1) -- (0,2);
  \node[circle,fill,inner sep=1.2pt] at (0,0) {};
  \node[circle,fill,inner sep=1.2pt] at (-0.5,0.5) {};
  \draw[->,thick] (0.05,0.06) -- (0.05,1.94);
  \draw[->,thick,dotted] (-0.05,1.94) -- (-0.05,1.06);
  \draw[->,thick] (-0.5,0.44) -- (-0.5,-0.44);
  \node[circle,draw,inner sep=1.1pt] at (-0.5,-0.5) {};
  \draw[densely dotted] (-0.5,0.5) -- (0,0.5);
  \node[left] at (-0.0,0.5) {\small $\tfrac12$};
  \node[right] at (-0.0,-0.5) {\small $-\tfrac12$};
  \node[below right] at (-0.4,-0.30) {\small $u=\ell$};
  \node[above left] at (-0.2,0.7) {\small $u=1+\ell$};
  \node[above left] at (0.50,2.4) {\small $u=2+\ell$};
  \node[left] at (-0.02,1) {\small $1$};
  \node[left] at (-0.02,2) {\small $2$};
  \node[right] at (0.08,1.55) {\small jump at $t_1$};
  \node[left] at (-0.08,1.50) {\small zero cost at $t_2$};
  \node[left] at (-0.48,0.20) {\small jump at $t_3$};
\end{tikzpicture}
\caption{Left: the potential $F$ of Example \ref{cex:ex-plateau}, i.e.\ the frozen energy at the times where $\ell=0$: its critical set $\{0\}\cup[1,2]$ is the union of the degenerate point $0$ and of the flat piece (thick); the dashed arc is the concavity region $(0,\tfrac12)$, which hosts the unstable critical points. Right: the critical set in the $(\ell,u)$-plane (solid: stable branches; dashed: unstable branch; thick: the continuum component $[1,2]$ of the fibers at $\ell=0$) and the three discontinuities of the limit evolution: the jump at $t_1$ from the fold point $0$ to $2$, \emph{through the continuum} (upward arrow, cost $\tfrac14$); the return slide across the component, at zero cost, at $t_2$ (dotted arrow); the return jump at $t_3$ from the second fold $(-\tfrac12,\tfrac12)$ (downward arrow, cost $\tfrac14$). The trajectory encloses a hysteresis loop of area $\tfrac12$, the total dissipation of the cycle.}
\label{cex:fig-plateau}
\end{figure}
\nc

\begin{remark}[Symmetries, 
continua in the critical fibers,
and the Sard property]
\label{cex:rmk-symmetry}
\slshape
Continua of critical points within a fixed fiber $\critset(t)$,
as in Example \ref{cex:ex-plateau}, are by no means exotic: they arise systematically in the presence of continuous symmetries. If $\calE_t$ is invariant under the action of a continuous group of isometries, critical points occur in orbits of positive dimension, and every point of such an orbit is a \emph{degenerate} critical point. The simplest instance is the ``mexican hat'' energy in $\Hilbert=\R^2$,
\[
\ene tu = \frac 14\big(|u|^2-1\big)^2-\la \ell(t),u\ra,
\]
at the times where the loading $\ell$ vanishes: there, the fiber of the critical set contains the whole unit circle. Loss functionals of overparametrized machine-learning models provide high-dimensional examples of the same mechanism: rescaling and permutation invariances of the parametrization produce manifolds of critical points. Observe that in all these situations the energy is constant along each orbit, so that 
such symmetry-induced continua in the fibers
are perfectly compatible with condition $\mathrm{(C)}$. In this respect, condition $\mathrm{(C)}$ has the nature of a \emph{Sard-type} property: for energies of class $\rmC^d$ on $\R^d$ it is automatically satisfied by the Morse--Sard theorem \cite{Sard42}, whereas in infinite dimensions it is a genuine structural assumption, since Sard's theorem may fail even for infinitely differentiable functionals \cite{Kupka65}. Sufficient conditions of a functional-analytic nature (Fredholmness of the differential, together with $\dim \NN(\rmD^2\calE_t(u))\leq 1$ at critical points), as well as the generic character of $\mathrm{(C)}$, are discussed in Sections \ref{sec:6} and \ref{s:8}.
\end{remark}

\subsection*{The results in brief}
Throughout the paper we consider time-dependent energies
$\calE\colon[0,T]\ti\Hilbert \to (-\infty,+\infty]$ on a separable
Hilbert space $\Hilbert$, possibly \emph{nonsmooth} and
\emph{nonconvex}, satisfying three structural conditions:
$\lambda$-convexity of $u\mapsto \ene tu$, compactness of the energy
sublevels, and a power control on $\partial_t \calE$ (which forces
$\calE_t$ to have constant finiteness domain).
Our standing assumption on the critical set $\critset(t)$ is
measure-theoretic:
\begin{itemize}
\item[$\mathrm{(C)}$] for every $t\in[0,T]$, the set of critical values $\ene t{\critset(t)}$ is Lebesgue-negligible.
\end{itemize}
A distinguished role will be played by the set $\totaldisc$ of
times $t$ at which the fiber $\critset(t)$ is \emph{totally
  disconnected} (i.e., all its connected components are singletons):
its complement $\nototaldisc:=[0,T]\setminus \totaldisc$ collects the
\emph{topologically singular} times, at which $\critset(t)$ contains a
nontrivial continuum. In the two guiding examples,
$\nototaldisc=\emptyset$ and $\nototaldisc=\{t_1,t_2\}$,
respectively.
Our main results can then be summarized as follows.
\begin{itemize}
\item[-] \emph{Existence and consistency} (Theorem \ref{mainth:1}, claims (1)--(3), together with
Proposition~\ref{prop:improved-C}): under $\mathrm{(C)}$ alone, for every initial datum there exists a \DSolution\ of the limit problem \eqref{eq:critical-intro}, namely
a curve $u$, critical outside an at most countable set $\rmJ$,
continuous at every time of $\totaldisc\setminus \rmJ$, and satisfying
the energy-dissipation balance
  \begin{equation}
    \label{EDbal-DISS-intro}
\mu([s,t]) + \ene t{u(t)} = \ene s{u(s)} + \int_s^t \partial_t \ene
r{u(r)}\dd r ,
\quad \text{for all }
s,t\in [0,T]
\setminus \rmJ,\quad \text{ with } s \leq t,
\end{equation}
where the \emph{defect measure} $\mu$, whose set of atoms is
precisely $\rmJ$, records the dissipated energy and satisfies, at each of its atoms, the jump relation
\[
\mu(\{t\}) \,=\, \ene t{u(t-)}-\ene t{u(t+)} \,=\, \cost t{u(t-)}{u(t+)} .
\]
The \emph{energy-dissipation cost} $\mathsf c$ is obtained by minimizing energy-dissipation integrals over transitions connecting the two states in the frozen landscape of $\calE_t$; optimal transitions are concatenations of (reparametrized) gradient flows of $\calE_t$: at the jumps, the internal viscous scale neglected by the limit re-emerges and resolves the discontinuity. Moreover, \emph{every} pointwise limit of gradient flows extends to a \DSolution.
\item[-] \emph{Compactness} (Theorem \ref{mainth:1}, claims (4)--(5)):
if, in addition, the set $\nototaldisc$ of topologically singular
times is at most countable, then every family of gradient flows admits
a subsequence converging pointwise on $[0,T]$ to a \DSolution.
\item[-] \emph{Balanced Viscosity solutions} (Theorem \ref{mainth:2}): if the critical set $\crit$ is countably $\mathscr H^1$-rectifiable, then the defect measure of any \DSolution\ is \emph{purely atomic}: the energy balance involves only the (at most countably many) jump costs.
\item[-] \emph{Verifiable conditions, and genericity} (Theorems \ref{thm:dim-kernel-ass1}, \ref{prop:7.1}, and \ref{thm:weak-perturb}): for smooth energies whose differential is a Fredholm map, the sole bound $\dim \NN(\rmD^2\calE_t(u))\leq 1$ at critical points implies $\mathrm{(C)}$; 
some first-order
transversality conditions imply the rectifiability of $\crit$; finally,  the transversality conditions are \emph{generic}: they can be enforced by a residual set of linear perturbations,
and the set of such perturbations has  full Lebesgue measure in finite dimension
(see Remark \ref{better-in-finite-dimension}).
\end{itemize}
This hierarchy is strict, and no level of it is redundant: as we
will show in Section \ref{ss:9.2}, the gradient flow of the (nonsmooth,
nonconvex) double-obstacle energy realizes its different levels as the
length of the underlying space domain varies. In particular, in the
large-domain regime, continua of critical points occur at \emph{every}
time and only condition $\mathrm{(C)}$ survives: the existence and
consistency of \DSolutions\ persist, whereas the finer conclusions ---
compactness, and the atomic structure of the dissipation --- are no
longer guaranteed.

\paragraph{\bf The finite-dimensional case.}
Although our analysis is developed in an infinite-dimensional Hilbertian setting, our results appear to be new even in finite dimension. Indeed, for $\Hilbert=\R^d$ the available results on the singular perturbation \eqref{eq:GFmodel-intro} \cite{Zanini,agostiniani2012,AgoRos17,ScillaSolombrino18-1} require the smoothness of the energy, as well as structural conditions on its critical points (finitely many degenerate critical points and the transversality conditions in \cite{Zanini}; fibers $\critset(t)$ consisting of isolated points in \cite{AgoRos17}). Both restrictions are removed here: our results cover nonsmooth, $\lambda$-convex energies --- such as those arising in constrained, or $\ell^1$-regularized, nonconvex minimization --- as well as energies whose critical points occur in continua, as it systematically happens in the presence of symmetries. Moreover, for smooth energies in $\R^d$ condition $\mathrm{(C)}$ holds automatically, by the Morse--Sard theorem \cite{Sard42}: the existence and consistency of \DSolutions\ (claims (1)--(3) of Theorem \ref{mainth:1}) are thus guaranteed \emph{without any assumption on the critical set}. Finally, since in finite dimension the perturbations of Theorem \ref{thm:weak-perturb} can be taken in a set of full Lebesgue measure (cf.\ Remark \ref{better-in-finite-dimension}), for almost every linear perturbation of the energy the limit evolution improves to a \BSolution.

\begin{corollary}[The smooth finite-dimensional case]
\label{cex:cor-findim}
Let $\Hilbert=\R^d$ and let $\calE$ comply with the power control
\eqref{power-control-intro}, have compact sublevels, and be
$\lambda$-convex.
Then:
\begin{itemize}
\item[(1)] If $\calE \in \rmC^d([0,T]\ti\R^d) $
  then $\mathrm{(C)}$ holds and therefore, for every initial datum $u_0$, \DSolutions\ starting from $u_0$ exist, and every pointwise limit of the gradient flows \eqref{eq:GFmodel-intro} extends to a \DSolution\ (Theorem \ref{mainth:1}, claims (1)--(2)), \emph{without any further assumption on the critical set};
\item[(2)]
  If $\calE \in \rmC^3([0,T]\ti\R^d) $
  then 
  there exists a Lebesgue negligible set $N\subset \R^d$
  such that for every $y \in \R^d\setminus N$ the perturbed energy
  $\calE^y_t(u):=\ene tu + \la y,u\ra$ satisfies the first order
  transversality conditions, 
  its critical set is countably $\mathscr H^1$-rectifiable and
  \emph{all} the conclusions of Theorem
  \ref{mainth:2}  hold: every sequence of $\eps_n$-gradient flows of
  $\calE^y$, $\eps_n\downarrow0$,
  admits a subsequence converging pointwise on $[0,T]$ to a limit that  is a \BSolution.
\end{itemize}
\end{corollary}

\paragraph{\bf A nonsmooth finite-dimensional example in hysteresis.}
In the smooth finite-dimensional
setting, Corollary \ref{cex:cor-findim} shows that the structural assumptions underlying
our main results are satisfied generically. No analogous genericity
statement is available for nonsmooth energies; nevertheless, the required
structural properties can be checked directly, as the following
example shows.

\begin{example}[The nonconvex double obstacle in $\R$: relay hysteresis]
\label{cex:ex-relay}
\upshape
The simplest energy displaying the whole structure of the theory is the scalar, nonsmooth, nonconvex
\begin{equation}
\label{cex:relay-energy}
\ene tu:= \rmI_{[-1,1]}(u)-\frac{u^2}2 - \ell(t)\,u, \qquad u \in \R,
\end{equation}
where $\rmI_{[-1,1]}$ denotes the indicator function of the interval $[-1,1]$.
It is $\lambda$-convex ($\lambda=1$) and complies with  our standing conditions on the driving energy, cf.\ $\mathrm{(E)}$
in Section \ref{ss:2.1} ahead. 
 For $|\ell(t)|<1$ the fiber is $\critset(t)=\{-1,\,{-\ell(t)},\,1\}$, i.e., it consists of  two local minimizers (the endpoints, i.e.\ the two ``phases'') separated by the interior maximizer; the endpoint $u=\mp1$ ceases to be critical as soon as $\pm\ell(t)>1$ (a \emph{nonsmooth fold}: the interior maximizer collides with the constrained minimizer). The critical set $\crit$ is the union of three Lipschitz curves, so that \eqref{rectifiable-intro} holds and Theorems \ref{mainth:1} and \ref{mainth:2} apply: for a loading oscillating beyond the thresholds $\pm1$, the \BSolution\ switches between the two phases at the thresholds, dissipating the cost $\cost{t^*}{\mp1}{\pm1}=2$ at each switch. The limit evolution is thus precisely the \emph{non-ideal relay}, so  called because it switches at \emph{two} distinct thresholds $\ell=\pm1$, enclosing a hysteresis loop, unlike the single-threshold, 
 memoryless \emph{ideal} relay; see Figure \ref{cex:fig-relay}. 
 \par
 We point out  the relay operator is the  elementary building block of  Preisach-type models of hysteresis, cf.\ e.g.\ \cite{Visintin94}: the present theory derives it rigorously as the singular limit of a nonsmooth gradient flow, a statement out of reach of the smooth finite-dimensional results of \cite{Zanini,AgoRos17}. 
Its genuinely infinite-dimensional, spatially extended counterpart ---the double-obstacle energy of Section \ref{ss:9.2} ---exhibits analogous
mechanisms in a richer PDE setting, involving the sliding \emph{droplets}
and \emph{kinks} introduced and discussed there.
\end{example}

\begin{figure}[ht]
\centering
\begin{tikzpicture}[xscale=1.5,yscale=1.6,baseline=(current bounding box.center)]
  \draw[->,gray] (-1.55,0) -- (1.7,0) node[below,black] {$u$};
  \draw[->,gray] (0,-0.12) -- (0,0.86) node[left,black] {\small $F$};
  \draw[very thick] plot[domain=-1:1,samples=60,variable=\u] ({\u},{0.5*(1-\u*\u)});
  \draw[very thick] (-1,0) -- (-1,0.62);
  \draw[very thick] (1,0) -- (1,0.62);
  \draw[densely dotted] (-1.5,0.62) -- (-1,0.62);
  \draw[densely dotted] (1,0.62) -- (1.5,0.62);
  \node[above] at (-1.25,0.6) {\small $+\infty$};
  \node[above] at (1.25,0.6) {\small $+\infty$};
  \node[circle,fill,inner sep=1.2pt] at (-1,0) {};
  \node[circle,fill,inner sep=1.2pt] at (1,0) {};
  \node[circle,draw=red,inner sep=1.1pt] at (0,0.5) {};
  \node[below] at (-1,-0.02) {\small $-1$};
  \node[below] at (1,-0.02) {\small $1$};
\end{tikzpicture}
\hfil
\begin{tikzpicture}[xscale=1.5,yscale=1.15,baseline=(current bounding box.center)]
  \draw[->,gray] (-2.0,0) -- (2.05,0) node[below,black] {$\ell$};
  \draw[->,gray] (0,-1.55) -- (0,1.7) node[left,black] {$u$};
  \draw[very thick] (-1,1) -- (1.8,1);
  \draw[very thick] (-1.8,-1) -- (1,-1);
  \draw[thick,densely dashed] (-1,1) -- (1,-1);
  \node[circle,fill,inner sep=1.2pt] at (1,-1) {};
  \node[circle,fill,inner sep=1.2pt] at (-1,1) {};
  \draw[->,thick] (1,-0.9) -- (1,0.9);
  \draw[->,thick] (-1,0.9) -- (-1,-0.9);
  \node[circle,draw,inner sep=1.1pt] at (1,1) {};
  \node[circle,draw,inner sep=1.1pt] at (-1,-1) {};
  \node[below right] at (1,-0.03) {\small $1$};
  \node[above left] at (-1,0.03) {\small $-1$};
  \node[above] at (1.4,1) {\small $u=1$};
  \node[below] at (-1.4,-1) {\small $u=-1$};
  \node[above right] at (-0,0) {\small $u=-\ell$};
\end{tikzpicture}
{\caption{\emph{Left}: the nonsmooth double-well potential $F(u)=\tfrac12(1-u^2)+\rmI_{[-1,1]}(u)$ (the frozen energy at $\ell=0$, shifted by the immaterial constant $\tfrac12$ so that $F\geq0$, with $F=+\infty$ off $[-1,1]$) --- an arch pinned between the obstacle walls at $u=\pm1$, so that the two \emph{phases} are the constrained minima at $u=\pm1$ (dots, at height $0$), while the interior maximizer sits at $u=0$. \emph{Right}: the resulting \emph{non-ideal relay} of Example \ref{cex:ex-relay} in the $(\ell,u)$-plane. Solid: the two stable phases $u=\pm1$; dashed: the maximizer $u=-\ell$. The limit evolution rests on a phase until the loading reaches a threshold, then it switches: it jumps up from $-1$ to $+1$ at $\ell=+1$ and back down at $\ell=-1$ (arrows), tracing the rectangular hysteresis loop. The two distinct thresholds $\pm1$ --- as opposed to the single threshold of a memoryless \emph{ideal} relay --- make it ``non-ideal''.}\label{cex:fig-relay}}
\end{figure}

 \paragraph{\bf Outlook: mechanisms for the emergence of rate-independent hysteresis.}
In general, rate-independent hysteresis emerges from viscous gradient dynamics through 
several 
distinct singular-limit mechanisms
which may be arranged according to the scale at which the nonconvexity acts. At the \emph{macroscopic} scale --- the setting of the present paper --- a single (possibly infinite-dimensional) nonconvex energy is driven slowly, and hysteresis originates in the metastability of its equilibria and in the jumps triggered when a branch of critical points is annihilated; the small parameter is the viscosity, and the limit is a vanishing-viscosity limit in the sense of \cite{EfeMie06RILS,MRS2016}, of which our results are an infinite-dimensional and nonsmooth instance. At an intermediate, \emph{mesoscopic} scale, the same snap-through (fold) mechanism acts on the many wells of a chain of bistable elements: in \cite{MieTru12} the joint discrete-to-continuum and vanishing-viscosity limit of a viscoelastic chain of bistable springs is shown to yield continuum rate-independent plasticity. At the \emph{microscopic} scale, hysteresis arises instead from the homogenization of an energy carrying fast spatial oscillations (``wiggly energies'', in the kinetics of Abeyaratne--Chu--James \cite{AbeChuJames96}): there the dissipation is of \emph{dry-friction} type and originates in the pinning of the state in the corrugated landscape, the small parameter being the oscillation wavelength, and the rigorous averaging analysis is due to \cite{Menon02}. While all three routes yield, in the limit, a rate-independent evolution, the underlying mechanisms --- macroscale metastability, mesoscale snap-through, microscale pinning --- and the corresponding mathematical techniques are genuinely different; the present paper concerns the first approach.

\paragraph{\bf From one-dimensional models to PDEs.}
The double-obstacle system of Section \ref{ss:9.2} also supports the faithfulness
of the two guiding examples as blueprints of the behavior of
infinite-dimensional gradient flows: its branch of metastable states
terminates at an explicitly computable threshold, where 
it switches to the stable phase
through a dissipative jump --- the fold
mechanism of Example \ref{cex:ex-fold} --- while its sliding `kinks'
and `droplets' form continua of critical points, which are crossed at
zero cost --- the plateau mechanism of Example \ref{cex:ex-plateau}.
For smooth energies, whose critical set is in general beyond the reach
of an explicit classification, the verification of our conditions
relies instead on the genericity results of Section \ref{s:8}: a
showcase statement, proved at the end of Section \ref{ss:9.1}, is the following quasistatic limit
of the Allen--Cahn equation.
\begin{theorem}[Quasistatic limit of the Allen--Cahn equation]
\label{cex:th-AC-short}
Let $\Omega\subset\R^d$, $d\leq3$, be a bounded connected open set, $\rmW(u)=\frac14(u^2-1)^2$, and $\ell\in\rmC^3([0,T];L^2(\Omega))$. Then every neighborhood of $0$ in $L^2(\Omega)$ contains a residual set of perturbations $y$ such that, for every initial datum $u_0 \in H^1_0(\Omega)$, the solutions of
\[
\eps\,\partial_t u_\eps = \Delta u_\eps - \rmW'(u_\eps)+\ell(t)-y \ \text{ in } \Omega, \qquad u_\eps=0 \ \text{ on } \partial\Omega, \qquad u_\eps(0)=u_0,
\]
admit a subsequence converging pointwise in time, strongly in $L^2(\Omega)$, to a \BSolution\ of the limit problem, namely a curve solving the stationary Allen--Cahn equation $-\Delta u + \rmW'(u)=\ell(t)-y$ at all but countably many times, whose dissipation measure is purely atomic and quantified, at each jump, by the energy-dissipation cost.
\end{theorem}

\paragraph{\bf Structural character  of the assumptions.}
All the conditions on the critical set entering
our main results (condition $\mathrm{(C)}$, the countability of
$\nototaldisc$, and the rectifiability of $\crit$) are designed to
be verifiable \emph{at a structural level}, i.e.\ without any detailed
knowledge of the bifurcation diagram of the energy, which is usually
much harder for the functionals of interest in applications: in the
smooth setting,  our assumptions indeed follow from the criteria and the genericity
results summarized above.
This robustness is, in our view, one of the main assets of the
variational approach developed here: it makes the theory applicable to
concrete infinite-dimensional systems, such as those discussed in
Section \ref{s:8}.

\paragraph{\bf Perspectives: slow--fast coupled systems.}
The analysis of this paper concerns systems whose state is described by a single variable undergoing viscous relaxation. A further motivation for the generality pursued here stems from \emph{coupled} systems, in which the energy
depends, in addition to time,
on a pair $(u,\chi)$: an ``elastic'' variable $u$, adjusting to equilibrium instantaneously, coupled with an internal variable $\chi$ (a phase field, or a damage, or plastic variable) relaxing on its own viscous time scale, i.e., in the simplest scenario,
\begin{equation}
\label{cex:coupled}
\rmD_u\calE(t,u(t),\chi(t))=0,
\qquad
\eps\, \chi'(t)+ \rmD_\chi \calE\big(t,u(t),\chi(t)\big)\ni 0
\qquad\text{in } (0,T).
\end{equation}
At least formally, minimizing out the variable $u$ leads to the reduced functional
$\calF_t(\chi):=\min_u \calE(t,u,\chi)$, which drives a singularly perturbed gradient flow of the form \eqref{eq:GFmodel-intro} for $\chi$ alone (we refer to \cite{RossiSavare06} for a study of
gradient flows driven precisely by reduced functionals of this
type). One should not expect, however, the reduction to be exact: the
first relation in \eqref{cex:coupled} only constrains $u(t)$ to be a
\emph{critical point} of $\calE(t,\cdot,\chi(t))$, whereas the
marginal functional $\calF_t$ is built on \emph{minimizers}; along
branches of non-minimal critical points the coupled system
\eqref{cex:coupled} and the gradient flow of $\calF$ evolve
differently, and their limits as $\eps\down0$ differ in general.
Now, the reduced functional $\calF$ is in general nonsmooth, and nonconvex, even when $\calE$ is smooth;  above all, no robust structural property of 
the
critical set of $\calE$
---isolatedness or non-degeneracy of the critical points, uniform separation of the connected components---can be expected to survive the reduction procedure, which makes the passage to the limit $\eps\down0$, in general examples,  a delicate task. This is precisely the situation that the measure-theoretic conditions of this paper, being generic and blind to the internal structure of the critical set, are designed to address. The program of deriving along these lines the jump dynamics of quasistatic phase-field systems---cf.\ \cite{MiRo23} for the vanishing-viscosity analysis of multi-rate systems featuring viscous dissipation in both variables, and \cite{ArtinaCagnettiFornasierSolombrino,AlForKleSca25} for constructions of critical-point evolutions motivated by cohesive fracture---goes beyond the scope of the present work, and will be addressed in the future.

\paragraph{\bf Plan of the paper}
In \underline{Section \ref{s:context}} we place our results in
their mathematical context: we review the related literature, and we
present in an informal way the variational approach at the core of the
paper, the structure of the proofs of our main theorems, and the role
of the transversality conditions.\nc\
\par
In \underline{Section \ref{s:2}} we will
first of all recall some preliminary notions, and results, concerning multivalued mappings; some of the proofs are postponed to a final Appendix. 
We will  then
dwell on  our standing assumptions $\mathrm{(E)}$+$\mathrm{(C)}$   on the energy $\calE$. On the one hand, we will explore the consequences of 
(E), in particular highlighting 
 the chain rule.  On the other hand, 
in Lemma \ref{le:criterion}
we will show
that the countable $\mathscr H^1$-rectifiability of the critical
 set $\crit$ does guarantee the validity of condition $\mathrm{(C)}$,
 as well as the countability of the set $\nototaldisc$ of
 topologically singular times.\nc
 \par
 Hence, in \underline{Section \ref{ss:2.2-added}} we will introduce the concepts of admissible transition curve 
 and the energy-dissipation cost 
 at the core of the notions of \emph{Dissipative Viscosity} and \emph{Balanced Viscosity} solutions, thereafter defined. We will then state our
 main results, Theorem \ref{mainth:1} and \ref{mainth:2}.  
\par
\underline{Section \ref{s:energy_diss_cost}} will  entirely revolve around the properties of the energy-dissipation cost. Its central result will be  Theorem
\ref{le:main}: on the one hand, it  guarantees that  the infimum in the definition of cost is indeed attained; on the other hand, it provides a compactness criterion which will 
play a crucial role in the construction of limiting graphs for the vanishing-viscosity analysis of  \eqref{eq:GFmodel-intro}. 
\par
With these tools at hand,  in \underline{Section \ref{s:4}} we will carry out the proof of Theorem \ref{mainth:1}.
The proof of Theorem \ref{mainth:2} will be  developed  in \underline{Section \ref{s:5}}.   
\par
In \underline{Section \ref{sec:6}} we will revert to the realm of differentiable energies. We first show that condition $\mathrm{(C)}$ is guaranteed by a bound on the dimension of the kernel of $\rmD^2\calE_t$ at critical points (Theorem \ref{thm:dim-kernel-ass1}). We will  then introduce the transversality conditions and show that, if they  are satisfied,
 the critical set $\crit$ is countably $\mathscr H^1$-rectifiable. We  will also investigate further consequences of the transversality conditions. 
\par
In \underline{Section \ref{s:8}} we will address the
 generic character of condition $\mathrm{(C)}$, of the
 countability of $\nototaldisc$, and of the rectifiability of
 $\crit$\nc:
in fact,
 we will show
 that the transversality conditions, which imply them, are generic, see Theorem \ref{thm:weak-perturb}. 
  \par
    Eventually, in Sec.\ \ref{ss:9.2} we will closely inspect the evolution by critical points of a double obstacle energy and discuss the applicability
  of our results in that setup. 

\section{Mathematical context, main results, and structure of the proofs}
\label{s:context}

Gradient flows in Hilbert spaces have been intensively studied since the late '60s: starting from the pioneering works 
\cite{Komura67, Crandall-Pazy69, Brezis71, Crandall-Liggett71},   existence and approximation of solutions, in the case
of \emph{convex} driving energy functionals, 
 have been well established in the realm of the theory of maximal monotone operators, cf.\ \cite{Brezis73}. 
 A new impulse to the theory was given in the '80s by the seminal contribution by \textsc{E.\ De Giorgi} and coworkers   \cite{DeGiorgi-Marino-Tosques80,Marino-Saccon-Tosques89, DG93}, who advanced the theory of \emph{Minimizing Movements} and \emph{Curves of Maximal Slope}, setting  up the basis for the study of gradient flows in metric spaces,  \cite{Ambrosio95, AGS08}. 
 \par
 In what follows, we 
 will stay with  gradient flows in a  (separable)  Hilbert space $\Hilbert$
 (with norm $\|\cdot \|$ and inner product
$\langle \cdot,\cdot \rangle$),
  but focus on the case in which the  driving energy functional
  \[
  \begin{gathered}
   \text{$\mathcal{E} : [0,T]\ti \Hilbert\to (-\infty,+\infty]  $ is 
  time-dependent,}
  \\
   \text{with 
   $u\mapsto \ene tu$ 
   (possibly) \emph{nonconvex} and \emph{nonsmooth}}
   \end{gathered}
   \]
 Accordingly, in 
  the gradient flow equation
 \begin{equation}
\label{g-flow-intro}
 u'(t) + 
 \partial \ene t{u(t)} 
 \ni 0 \quad \text{in } \Hilbert \quad \foraa\, t \in (0,T),
\end{equation}
we will consider  the \emph{Fr\'echet subdifferential} of
$\calE$ w.r.t.\ the variable $u$, i.e., the multivalued operator $\partial \calE \colon [0,T]{\times}\Hilbert  \rightrightarrows \Hilbert$ 
defined at every $(t,u) $ in
the domain of $\cE$ by
\begin{equation}
\label{def-fr-subdif} \xi \in \partial
\ene tu
\ \ \text{if and only
if} \ \ 
\ene tv - \ene tu 
\geq \langle \xi, v-u \rangle + o(\|
v-u\|) \quad\text{as } v \to u,
\end{equation}
and by $\partial
\ene tu
= \emptyset $ if
$\ene tu=+\infty$.
We refer the reader to \cite{Kruger_Frechet}
for a survey paper on the 
Fr\'echet subdifferential.
In  \cite{RossiSavare06} existence and approximation results for (the Cauchy problem associated with) \eqref{g-flow-intro} 
were proved, in the autonomous case (cf.\ \cite{MRS2013} for the extension to time-dependent energies), in the mainstream of the  
     \emph{variational approach} to metric gradient flows 
 \cite{Ambrosio95, AGS08}. 
\par
 In this paper, also for the  \emph{singular perturbation} problem
 \begin{equation}
\label{e:sing-perturb}
\eps u'(t) + \partial \ene t{u(t)} \ni 0 \quad \text{in } \Hilbert  \quad \foraa\, t \in (0,T),  \qquad \text{as } \eps\downarrow 0
\end{equation}
we will  develop a 
\emph{variational approach}, inspired by the theory of \emph{\BSolutions} to rate-independent systems  \cite{MRS2016}, in order  to  study  in which sense
limits of 
 solutions  to \eqref{e:sing-perturb} can be interpreted as   solutions of the limit problem
 \begin{equation}
\label{limit-problem}
 \partial  \ene t{u(t)} \ni 0 \quad \text{in } \Hilbert \quad \foraa\, t \in (0,T),
 \end{equation}
 i.e.\ as  curves of critical points for the energy $\calE$.
 \par
 This variational method notwithstanding, the gradient-flow nature
of \eqref{e:sing-perturb} sets the present problem apart from the
rate-independent theory that inspires it. In the energetic
\cite{MieThe04RIHM, MieRouBOOK} and balanced-viscosity \cite{MRS2016}
solutions of \emph{rate-independent} systems, the energy--dissipation
balance features a rate-independent, positively $1$-homogeneous
dissipation, which \emph{remains} in the limit:
it provides an a priori bound on the total variation in time of the
solutions, and, for energetic solutions, drives the state onto
\emph{global} minimizers. Here, instead, \eqref{e:sing-perturb} is a
\emph{purely viscous} gradient flow:
no rate-independent dissipation is present, the viscous dissipation itself vanishes as $\eps\down0$, and \emph{no} a priori bounded-variation control in time is available; the associated estimates degenerate, which is the very source of the analytical difficulties discussed next.
\par
 Before illustrating our results, let us  hint at the main analytical difficulties attached to the asymptotic analysis of \eqref{e:sing-perturb} as $\eps \down 0$.
 \paragraph{\bf Preliminary considerations.}
 The results from  \cite{MRS2013} guarantee
that
  for every fixed $\eps>0$ and
   for every $u_0 \in \Hilbert$  there exists at least a solution
  $u_\eps\in H^1 (0,T;\Hilbert)$ to the gradient flow \eqref{e:sing-perturb}, fulfilling
  the Cauchy condition $u_\eps(0)=u_0$. Testing \eqref{e:sing-perturb} by $u_\eps'$, integrating in time, and exploiting the chain rule for
  $\mathcal{E}$, which is one of the key assumptions in   \cite{MRS2013},   
  it is immediate to check that
  $u_\eps$ complies with the \emph{energy identity}
\begin{equation}
\label{enid-intro}
\int_s^t \eps \|u_\eps'(r)\|^2 \dd r  + \ene t{u_\eps(t)} = \ene s{u_\eps(s)}+\int_s^t \partial_t
\ene r{u_\eps(r)}   \dd  r \quad \text{for all $ 0 \leq s \leq t \leq T$},
\end{equation}
from which
all the a priori estimates on a family $(u_\eps)_{\eps}$ of solutions are deduced. 
More specifically, under the  \emph{power control} condition 
\begin{equation}
\label{power-control-intro}
|\partial_t \ene tu|
\leq 
C_1\big( 
\ene tu +C_2\big) 
\qquad \text{ for some $C_1,\, C_2>0$,}
\end{equation}
which we will adopt  
throughout this paper,
via the Gronwall Lemma
one obtains
\begin{equation}
\label{prelim-estimates-intro}
\begin{aligned}
&
\text{\emph{(i)}}  && \text{
 the energy bound $\sup_{t\in (0,T)} 
 \ene t{u_\eps(t)} \leq C $};
 \\
 &
\text{\emph{(ii)}}  && \text{
 the  
 estimate $\int_0^T \eps \|u_\eps'(t) \|^2 \dd t\leq C'$,} 
\end{aligned}
\end{equation}
for positive constants $C, \, C'>0$ independent of $\eps>0$. 
While 
\begin{itemize}
\item[-]
\emph{(i)}, joint with some suitable \emph{coercivity} condition (typically, compactness of the energy sublevels) on $\calE$, 
  yields   that there exists a \emph{compact} set  $K\subset\Hilbert$  s.t.\ 
$u_\eps(t) \in K$ for all $t\in [0,T]$ and $\eps>0$,  
\smallskip
\item[-]
the \emph{equicontinuity} estimate provided
  by \emph{(ii)} degenerates as $\eps\down 0$. Thus, no Arzel\`a-Ascoli type result applies to deduce compactness for $(u_\eps)_\eps$. 
  \end{itemize}
 This is the  major difficulty 
  in the asymptotic analysis of \eqref{e:sing-perturb}. We emphasize that 
 it
 is not related to the present infinite-dimensional setting, but   also arises in finite dimension.
 \par
  Likewise,
  in finite and infinite dimension this obstacle can be circumvented  by convexity arguments.
  For example, if $\Hilbert = \R^d$ it is possible to show that, if
  $\mathcal{E} \in \mathrm{C}^2 ([0,T]\ti \R^d)$
  (with Fr\'echet differential w.r.t.\ $u$
  at time $t$
  denoted by $\rmD \mathcal E_t$),
   and 
   $u\mapsto\ene tu$
is uniformly convex, 
  then, starting from any  
   $u_0 \in \R^d$ with $\mathrm{D} \ene 0{u_0}=0$  and $\mathrm{D}^2  \ene 0{u_0}$ positive definite (i.e., $u_0$ is a \emph{non-degenerate} critical point of $\mathcal E_0$), 
   there exists a unique curve $u \in \mathrm{C}^1([0,T];\R^d)$
of critical points, namely
 fulfilling $\mathrm{D}
\ene t{u(t)} = 0$ for every $t\in [0,T]$,
   to which the \emph{whole} family $(u_\eps)_\eps$ converges as $\eps\down 0$, uniformly on $[0,T]$.
\par
  This observation  highlights  that   it is indeed significant to focus  on the case in which
the energy  $u \mapsto \ene tu$   is allowed to be
\emph{nonconvex}.  In this context, two problems arise:

\begin{enumerate}
\item Proving that, up to the extraction of a subsequence,
the gradient 
flow curves  
$(u_\eps)_\eps$ converge in $\Hilbert$  as $\eps \down 0$ to some limit
curve $u$ (pointwise a.e.\  in $(0,T))$;
\smallskip
\item
 Describing the evolution of $u$. Namely, one expects $u$
to be a curve of critical points, 
jumping at   times $t\in (0,T)$ such that $u(t)$ is a \emph{degenerate}  critical point  for $\calE_t$. 
In this connection, one  aims to provide a thorough  description of the energetic behavior of $u$ at jump points.
\end{enumerate}
\paragraph{\bf Results for \emph{smooth} energies in finite dimension: the approach via the \emph{transversality conditions}.}
For the singular perturbation limit
 \eqref{e:sing-perturb}, a first answer to problems (1)\&(2)   was
 provided,  in the finite-dimensional case $\Hilbert = \R^d$,  in
 \cite{Zanini}
 (whose results were later extended to second order systems in \cite{agostiniani2012}),
  under the assumption that the energy $\mathcal{E} \in \mathrm{C}^3 ([0,T]\ti \R^d)$,
  has finitely many degenerate critical points, at which the vector
  field $[0,T]\ti \R^d\ni (t,u)\mapsto\mathrm D\ene tu$ complies with
  the (full set of the) \emph{transversality conditions},
 and 
  a further technical condition.
  While postponing to Section \ref{sec:6}
 (cf., in particular, Remark \ref{rmk:full-tc}), 
more details on the 
transversality conditions, well-known in the realm of  bifurcation theory  (cf., e.g., 
 \cite{GuckHolm83}), we may mention here that, essentially, they 
 prevent degenerate  critical points from being ``too singular".
  \par
  Then, in \cite[Thm.\ 3.7]{Zanini} it was shown that, starting from a ``well-prepared'' datum $u_0$,
  there exists a
 unique piecewise $\mathrm{C}^2$-curve
$u:[0,T] \to \R^d$ with a finite jump set $\mathrm{J}= \{t_1,\, \hdots, \,
t_k\}$, such that (1) 
 $\mathrm{D} \ene t{u(t)} = 0$ with
$\mathrm{D}^2\ene t{u(t)}$ positive definite
 for all
$t \in [0,T]\setminus \rmJ$;
(2)
 at every jump point $t_i\in \mathrm{J}$,
 $u(t_i-)$   is a  degenerate critical point for
$\mathcal E_{t_i}$ and there exists a  unique curve $v\in
\mathrm{C}^2(\R;\R^d)$ such that
$  \lim_{s \to -\infty}   v(s) = u(t_i-),$   $  \lim_{s \to +\infty} v(s) =
    u(t_i+)$   and 
   \begin{equation}
   \label{heterocline}
   v'(s) + \mathrm{D} \ene{t_i}{v(s)} =0 \quad \text{for all } s\in \R;
   \end{equation}
(3)
 the \emph{whole} sequence $(u_\eps)_\eps$ converges to $u$ uniformly
 on the  compact sets of $[0,T] {\setminus} \mathrm{J}$, and suitable
 rescalings of $(u_\eps)_\eps$ converge to $v$.
\par
 This is the mechanical picture of jumps already recalled in the
Introduction: at each $t_i$ the internal viscous scale, neglected by
the singular limit, takes over and governs a fast transition between
two metastable states along the heteroclinic $v$.
\par
 In the proof of \cite{Zanini}, 
 while the unique limit curve is \emph{a priori}
constructed via the Implicit Function Theorem, 
the core argument for obtaining the heteroclines connecting the left and right limits at a jump point
hinges on the transversality conditions. 
In this connection, let us mention that, in \cite{ScillaSolombrino18-1}, an attempt was made to address the singularly perturbed, finite-dimensional, gradient flow \eqref{e:sing-perturb}, 
without resorting to the transversality conditions, but nonetheless with techniques from  bifurcation theory. 
\par
Indeed, the recourse to the (full set of the) transversality conditions
 hinders the extension of the analysis to an infinite-dimensional framework and/or to nonsmooth
energies.
A preliminary step in this direction was made in  \cite{AgoRos17}. 
\paragraph{\bf A variational approach, still  in finite dimension}
In \cite{AgoRos17}, an alternative strategy for the vanishing-viscosity analysis of the singularly perturbed gradient flow
\eqref{e:sing-perturb} was advanced, still in the finite-dimensional setup $\Hilbert = \R^d$, and for  energy functionals $\calE \in \rmC^1([0,T]\ti \R^d)$,
but unhampered by the transversality conditions. In fact, \cite{AgoRos17} combined ideas from the \emph{variational approach} to gradient flows with  nonsmooth and nonconvex  energies,
\cite{AGS08, RossiSavare06, MRS2013}, with techniques for the vanishing-viscosity analysis of  rate-independent systems,  \cite{MRS2016}.
\par
The starting point, both    in \cite{AgoRos17} and in the forthcoming analysis, 
is the key observation that the energy identity \eqref{enid-intro} indeed rewrites, 
for differentiable energies,  as 
\begin{equation}
\label{enid-intro-better}
\int_s^t \left( \frac\eps 2 \|u_\eps' (r)\|^2 {+}\frac1{2\eps}
\|\mathrm{D} \ene r{u_\eps(r)}\|^2\right) \dd  r +
\ene t{u_\eps(t)} = \ene s{u_\eps(s)}+\int_s^t
\partial_t
\ene r{u_\eps(r)}   \dd  r
\end{equation}
for all $0\leq s \leq t \leq T$.
 In addition to 
estimates \eqref{prelim-estimates-intro}, 
from \eqref{enid-intro-better}  it is possible to deduce that
\begin{equation}
\label{en-diss-bound} \int_0^T \|u_\eps'(r)\| \|\mathrm{D}
\ene r{u_\eps (r)}\| \dd r \leq C.
\end{equation}
 Thus, while no (uniform w.r.t.\ $\eps>0$) bound is available on
$\int_0^T \|u_\eps'\|\dd r$, estimate \eqref{en-diss-bound} suggests
that the limit of the \emph{energy-dissipation integrals} $\int_s^t
\|u_\eps'(r)\|\|\mathrm D\ene r{u_\eps(r)}\|\dd r$ may describe the
dissipation of energy at jumps, provided the weight $\|\mathrm
D\ene{(\cdot)}{u_\eps(\cdot)}\|$ keeps such points ``well separated''
from one another.
In fact,  the latter observation prompted the recourse, in \cite{AgoRos17}, to a specific requirement on the set 
$\crit \subset [0,T]\ti \R^d$
of  \emph{critical points}
 of $\calE$,
namely that 
 for every $t\in [0,T] $
 its fiber
 \begin{equation}
  \label{non-variat-gen}
 \critset(t) := 
\big\{ u \in \R^d : \,
\mathrm{D}\ene tu =0  
\big\}  
\text{ consists of
\emph{isolated} points}.
\end{equation}
(Equivalently, by the coercivity assumed in \cite{AgoRos17}, that 
$\critset(t)\cap\sublevel\rho$ is \emph{finite} for every $t\in[0,T]$,
$\rho>0$; it is not restrictive to confine the analysis to a sublevel,
since the curves $(u_\eps)_\eps$ all lie in one by the energy bound
\eqref{prelim-estimates-intro}(i).)
\nc
\par
Now, \eqref{non-variat-gen} is at the core of the properties of a central object for the study of the  singularly perturbed gradient flow, namely 
the energy-dissipation cost   $c(t; u(t-), u(t+) )$,   recording the energy dissipated at each jump point $t$ of the solution, when it jumps from its left
 limit   $u(t-)$   to its right limit   $u(t+)$.  
In \cite{AgoRos17}, 
 $c(t; u(t-), u(t+) )$ 
 has been defined by infimizing  the energy-dissipation integrals $  \int_0^1
\|\teta' (r)\| \|\mathrm{D} \ene t{\teta(r)}\| \dd r $ over a suitable class of curves $\teta\colon [0,1]\to \R^d$, with 
 $\teta(0)=u(t-) $ and $ \teta(1)=u(t+)$. Strongly relying on \eqref{non-variat-gen}, the authors were able to prove that
for every $t\in[0,T]$
 the $
\inf$ in the definition of cost is attained, and that optimal curves consist of 
\emph{finitely many} 
locally Lipschitz pieces connecting  (finitely many)  points in the critical set $\critset(t)$.
These properties of $c$ 
have in turn played a key role in the compactness argument leading to the (pointwise in time) convergence,
 up to a subsequence,  of the gradient flows $(u_\eps)_\eps$  to a curve $u$ of critical points for $\calE$,
fulfilling the energy-dissipation  identity
\begin{subequations}
\label{diss-visc-intro}
\begin{align}
 \label{lim-enid-intro}
\mu([s,t])+ \ene t{u(t)} = \ene s{u(s)}
+\int_s^t
\partial_t \ene r{u(r)} 
\dd  r
\end{align}
with $\mu$ a positive finite measure such that  the  set of its 
atoms, $\mathrm J$,
coincides with the (at most countable) jump set of $u$,  
and satisfying the  jump relations
\begin{align}
\label{jump-relation-intro}   \mu(\{t\}) = 
 \ene t{u(t-)} -
\ene t{u(t+)} 
= 
  c(t; u(t-), u(t+) ) 
\quad \text{for all }
t \in \mathrm{J}\,.
\end{align}
\end{subequations}
It has been argued in \cite{AgoRos17} that  relations \eqref{jump-relation-intro}
do provide a description of the behavior of the limit curves at jumps:  in fact, in particular they encompass the information that 
 every optimal jump transition (i.e.,  every curve attaining
the infimum defining $c(t;u(t-),u(t+))$) can be 
reparameterized to be  a gradient flow of the frozen energy $\calE_t$,
at fixed process time $t$. 
\par
The notion of  evolution of critical points delineated by \eqref{diss-visc-intro} goes under the name of \emph{Dissipative Viscosity} solution to \eqref{limit-problem}. In 
\cite{AgoRos17}, it was also shown that, under a suitable condition on $\calE$ akin to the  \emph{{\L}ojasiewicz inequality},
Dissipative Viscosity improve to \BSolutions, defined by \eqref{diss-visc-intro} joint with the additional property that the measure $\mu$ be only atomic.
\par
Within the above setup, 
the authors of 
\cite{ScillaSolombrino18-2} 
addressed \eqref{e:sing-perturb}
from a multiscale viewpoint, analyzing the convergence of the singularly perturbed Minimizing Movement scheme for the gradient flow as \emph{both} the viscosity parameter $\eps$ \emph{and} the time step $\tau$ vanish, with a suitable reciprocal scaling.  In the same framework, 
the vanishing-viscosity \emph{and} inertia approximation of  \eqref{limit-problem} is treated in \cite{ScillaSolombrino19}.

\subsection*{Main results}
In this paper, we aim to address the singular perturbation problem \eqref{e:sing-perturb}
in  the  \emph{infinite-dimensional}  context of a Hilbert space $\Hilbert$, and for \emph{nonsmooth} energies.
\par
As for  the Hilbertian setup, we may mention that, because of the variational character of our approach, our study of  \eqref{e:sing-perturb}
could be extended, with suitable modifications, 
 to singularly perturbed gradient flows in metric spaces. Nonetheless, we have refrained from 
working in that more general
framework in order to better highlight the core of 
our arguments, and avoid
 overburdening the analysis  with
 metric-space techniques. \EEE
  \par
  As for 
   `nonsmoothness', this  concerns the dependence $u\mapsto \ene tu$. That is why, from now we will work
  with the Fr\'echet subdifferential
   $\partial \mathcal E \colon [0,T]{\times} \Hilbert
   \rightrightarrows \Hilbert$ from \eqref{def-fr-subdif}
   and
   with the critical set $\crit$ (with fibers $\critset(t)$) defined in \eqref{eq:47}
   in 
   the Introduction.
   \nc\
     Now,
  this more general setup is still compatible with the following conditions on the driving energy:
\begin{itemize}
\item[$\mathrm{(E)}$] 
$\calE$
complies with  the power control \eqref{power-control-intro},  it  has compact sublevels,  and  
there exists $\lambda \geq 0$ such that for all $t\in [0,T]$ the mapping $u\mapsto  \cE (t,u)+\tfrac\lambda2 \|u\|^2$ is convex,
\end{itemize}
  which will be assumed throughout the paper. 
Instead,  the  discreteness condition \eqref{non-variat-gen} on  the fibers $\critset(t)$  
can in fact be bypassed.
\par
Recalling the analysis carried out in  \cite{AgoRos17}, it is to be expected that giving up  \eqref{non-variat-gen} has a major impact on the compactness argument for the curves $(u_\eps)_\eps$. Indeed, it turns out that, in
 order to tackle it, 
 it is necessary to change viewpoint, and address the compactness properties of the graphs
 \[
 \Gu_\eps = \big\{ (t,u_\eps(t)) : \, t \in [0,T]\big\},
 \]
 rather than of the curves $(u_\eps)_\eps$. 
Due to the
energy 
bound \eqref{prelim-estimates-intro} (i)
and 
 condition $\mathrm{(E)}$, the graphs $( \Gu_\eps )_\eps$ are contained in an energy sublevel  $[0,T]\ti \sublevel{\rho} $, for some $\rho$, that is compactly contained in
 $ [0,T]\ti \Hilbert$. Therefore, there exists a compact subset 
 and a vanishing subsequence $n\mapsto \eps(n)$
 such that 
 \[
 \Gu \Subset [0,T]\ti \Hilbert \text{ such that }  \Gu_{\eps(n)} \to \Gu \text{ in the Kuratowski sense, with fibers } \mathrm{U}(t) {\cap} \critset(t) \neq \emptyset \text{ for all } t \in [0,T]\,.
 \]
 We emphasize that, a priori, the sets $ \mathrm{U}(t) $ need not be  singletons. 
 The immediate  intuition  would be that   a limit \emph{curve}  
 $u$ can be obtained from 
this argument,  essentially by showing 
 that   $\mathrm{U}(t) $ is a singleton for $t$  in 
a `sufficiently big' 
set (i.e., such that its complement in $[0,T]$ is at most countable).
\par
Nonetheless, we will develop this idea by disentangling
\begin{itemize}
\item[-] the construction of a curve $v $ as a \emph{Borel measurable selection} in the critical set, such that \emph{(i)} $v$ is a curve of critical points (outside an at most countable set);
\emph{(ii)} $v$ satisfies suitable analogues of  \eqref{lim-enid-intro} and \eqref{jump-relation-intro}; we shall qualify such a curve as a \emph{Dissipative Viscosity} solution;
\item[-] the compactness argument for the gradient flows  $(u_\eps)_\eps$, i.e., the proof that, up to a subsequence, the curves $u_\eps$ actually do converge to a \DSolution. 
\end{itemize}

\par
An appropriate generalization of the energy-dissipation cost $c$ from \cite{AgoRos17}, in this paper denoted by $\mathsf{c}$,    plays a key role
\emph{both} in 
 the study of the limit set $\Gu$ (in particular, in  the `extraction' of a \DSolution\ therefrom), 
 \emph{and} for the compactness argument.
  In accord with the
present multivalued approach, for fixed $t\in [0,T]$
it is expedient to define the cost 
$\mathsf{c}_t$
not only between points in $\Hilbert$ (viewed as singletons), but also between connected components of $\critset(t)$. Denoting by $\newclass t$ the class of such sets, we again define
$\mathsf{c}_t$ by infimizing 
the energy-dissipation
integrals
\begin{equation}
\label{new-cost-intro}
\cost t{\cU_0}{\cU_1}:=
\inf\left\{\int_{0}^1 \minpartial \calE t {\teta(s)} \| \teta'(s)\|\dd
s \,:\,\teta\in
 \admis{\cU_0}{\cU_1}{t} \right\} \qquad \text{for } \cU_0,\, \cU_1 \in \newclass t,
\end{equation}
where the class of admissible transitions  $\admis{\cU_0}{\cU_1}{t} $ consists of curves $\teta \colon [0,1]\to \Hilbert $ departing from $\cU_0$, ending up in $\cU_1$, locally Lipschitz 
on the set $\{ s \in [0,1]: \, \teta(s) \notin \critset(t)\}$ and such that the energy function $[0,1]\ni s \mapsto \ene t{\teta(s)}$ is globally Lipschitz.
Here, $\argminpartial {\cE}t{\teta(s)}$
is the minimal norm element of $\partial \ene t {\teta(s)}$.
\par
The soundness of the construction in \eqref{new-cost-intro} (starting from the fact that
the $\inf$ in the above definition is attained), rests on
 the
Sard-type condition $\mathrm{(C)}$ stated in the Introduction.
 Condition $\mathrm{(C)}$
for all intents and purposes replaces the discreteness condition \eqref{non-variat-gen} in this infinite-dimensional context.
\par
Under condition $\mathrm{(C)}$,
our first main result, \underline{\textbf{Theorem \ref{mainth:1}}},
has two parts.
\par
In the first part (Claims (1)--(3) of Theorem \ref{mainth:1}),
we show that condition $\mathrm{(C)}$ \emph{alone} suffices to:
\begin{itemize}
\item[-]  prove the \emph{existence} of a \DSolution\ $v:[0,T]\to \Hilbert$ 
  with $v(0)=u_0$, constructed via a Borel selection argument
  from the limit set $\Gu$;
\item[-] as a \emph{consistency} result, show that \emph{any} pointwise limit of gradient flows,
  on whatever Borel set $D\subset [0,T]$ where convergence holds,
  can be extended to a \DSolution\ on the whole of $[0,T]$.
  In particular, if $D=[0,T]$, then the limit itself is a \DSolution.
\end{itemize}
In the second part (Claims (4) and (5) of Theorem \ref{mainth:1}),
we additionally assume that
the set of \emph{topologically singular} times
\begin{equation}
  \label{eq:nototaldisc-intro}
  \nototaldisc:=\Big\{t\in [0,T]:\critset(t)\text{
    is not
    totally disconnected}\Big\}
  \quad\mbox{is at most countable.}
\end{equation}
Under these conditions, we prove that
every sequence of gradient flows $(u_\eps)_\eps$
admits a subsequence converging \emph{pointwise everywhere} in $[0,T]$
to a \DSolution.
Thus,
the countability of $\nototaldisc$
plays a crucial role
 in the  \emph{compactness  argument} for  the sequence $(u_\eps)_\eps$.
\par
 The defining properties of \DSolutions, namely the criticality outside
the at most countable atomic set $\rmJ$ of the defect measure $\mu$,
the energy-dissipation balance, and the jump relations expressed
through the cost $\mathsf c$ from \eqref{new-cost-intro}, have been
listed in ``The results in brief'' of the Introduction, and are
detailed in Definition \ref{def:sols-notions-1} ahead.
\nc
 \par
Supplying sufficient, and possibly handier, conditions for the
validity of $\mathrm{(C)}$  and for the countability of
$\nototaldisc$
is thus a relevant issue.
In fact, we have shown that    the  property that  
\begin{equation}
\label{rectifiable-intro}
 \text{$\crit$ is countably
$\mathscr H^1$-rectifiable}
\end{equation}
(i.e., that  $\crit$ can be covered, up to a set with null Hausdorff measure, by  the images of countably many Lipschitz curves),
does indeed imply $\mathrm{(C)}$, as well as the countability of
$\nototaldisc$. 
 In turn,  the measure-theoretic condition \eqref{rectifiable-intro} has proven to be interesting on its own for a two-fold reason.
\par
First of all, 
 with our second main result, 
\underline{\textbf{Theorem \ref{mainth:2}}},  
we have shown that, if  \eqref{rectifiable-intro} holds, then  any \DSolution\  is in fact a \emph{Balanced Viscosity} solution: this means that  the measure $\mu$
 recording the dissipation of energy is purely atomic, and the
energy-dissipation balance thus reformulates as
\[
\sum_{r\in \jumpname {\cap} (s,t)} \kern-12pt \cost r{u(r-)}{u(r+)} 
+ \ene t{u(t)} = \ene s {u(s)}
+\int_s^t \partial_t \ene r{u(r)} \dd r\quad \text{for all }
s,t\in [0,T]
 \setminus 
\rmJ  
\text{ with } s \leq t\]
(with the above sum extending over a \emph{countable} set). 
The proof of Theorem \ref{mainth:2}
 hinges on  measure-theoretic 
 arguments 
 that in particular include an extension of the 
\emph{Lusin property} to  functions of bounded variation on an interval $I\subset \R$.
\par
Secondly, we have proved that   the validity  of condition \eqref{rectifiable-intro} is, in its turn, guaranteed by the  transversality conditions. More precisely,
the transversality involving the sole first-order derivatives of
the energy
suffices.
This key fact allows us to conclude the \emph{generic} character of
 \eqref{rectifiable-intro}.
 
\paragraph{\bf Revisiting the transversality conditions}
As we have mentioned, the transversality conditions have played a major role in the approach to the vanishing-viscosity analysis based on bifurcation theory.
In Section \ref{sec:6} we will in some sense bridge this viewpoint, and the variational one.
\par
In fact, we shall
revert to a functional-analytic setup in which the energy functional $\calE$  is suitably smooth, with Fr\'echet differential that is a Fredholm map of index $0$.
In this setting we establish the three results already announced
in ``The results in brief'': the sole kernel bound
$\mathrm{dim}(\NN(\rmD^2\calE_t(u))) \leq 1$ at critical points
guarantees condition $\mathrm{(C)}$
(\textbf{\underline{Theorem \ref{thm:dim-kernel-ass1}}}); the
transversality conditions --- required here only in their
\emph{first-order} part, dropping the second-order condition assumed
in \cite{Zanini}, cf.\ Remark \ref{rmk:full-tc} ahead --- guarantee
the countable $\mathscr H^1$-rectifiability \eqref{rectifiable-intro}
of the critical set $\crit$ (\textbf{\underline{Theorem \ref{prop:7.1}}}); and they can
be enforced by an explicit class of arbitrarily small perturbations
(\textbf{\underline{Theorem \ref{thm:weak-perturb}}}).

\paragraph{\bf A different construction in quotient spaces}
Let us mention that an alternative   construction of \DSolutions\  to the evolution of critical points 
\eqref{limit-problem},  in the setting of a locally compact metric space $(\mathsf{X},\gdist)$,
has been recently advanced in 
\cite{AlForKleSca25},
 assuming that $\critset(t)$ can be written as a disjoint union of well-separated compact sets, together with additional  conditions on the time regularity of the  slope.
 Therein, \DSolutions\ have been obtained as limits of 
\emph{discrete quasistatic evolutions},  constructed by a refined  time-incremental approximation procedure  (cf.\ also \cite{ArtinaCagnettiFornasierSolombrino}). The idea is 
to let
the energy evolve on a small time interval, while keeping frozen the system state, and then, in a subsequent time interval freeze the 
   energy, while letting the system transit via gradient flow,  or some discrete
approximation of it, like Minimizing Movements or backward differentiation schemes. 
The \DSolution\  obtained in the limit upon sending the time step to zero does not, however,  correspond to that obtained in 
Theorem \ref{mainth:1}: indeed, it is a curve $u\colon [0,T]\to \boldsymbol{\mathcal{X}}$, 
where $\boldsymbol{\mathcal{X}}$ is the quotient of the space $[0,T]\ti\mathsf{X}$ with respect to  the equivalence relation that identifies two points in the same connected component of 
$\critset(t)$. 
\par
 Thus, this notion of evolution is, in some sense, in the same spirit of our compactness approach, which relies on a cost recording the dissipation of energy attached to a transition between one connected component of $\critset(t)$ to the other. However, our results have a different character to those in  \cite{AlForKleSca25} and are independent on them. 
\par
Indeed, in this paper
we have focused on 
revealing conditions
that would allow us to construct  an `actual' 
curve (i.e., a curve valued in $\Hilbert$, instead of its quotient w.r.t.\ the above mentioned equivalence relation) 
of critical points in the vanishing-viscosity limit, moving from structural topological assumptions on the connected components of $\critset(t)$, to 
 measure-theoretic conditions on the behaviour of $\calE$ on $\critset(t)$. In this connection, we believe condition $\mathrm{(C)}$ and the stronger \eqref{rectifiable-intro} to be of some interest, since they allow for great generality and the results of Section \ref{sec:6} show that they are generic in an appropriate sense.

\section{Basic assumptions and preliminary results}
\label{s:2}
In this section, 
first of all we will provide some preliminaries  on multi-valued maps, in particular recalling their appropriate notions of continuity and convergence. Next, 
we will settle   the basic
coercivity/$\lambda$-convexity conditions on the energy functional
$\calE$ which are required for our analysis. 
 We will then introduce the  notions of evolutions of critical points
 that shall be obtained  by singularly perturbing the gradient flow of $\calE$, \EEE
and summarize  our main results
in the forthcoming  Theorems
\ref{mainth:1} and \ref{mainth:2}.
  \EEE
\subsection{Preliminaries on multivalued maps}
\label{subsec:preliminaries}
\paragraph{\emph{Multivalued maps: Kuratowski
  convergence.}}
Let $(\gmetr,\gdist)$   be a metric space.  We will denote by
$\mathfrak{K}(\gmetr)$
(resp.~$\mathfrak{K}_{\mathrm{c}}(\gmetr)$)
the collection of the
 nonempty 
compact
(resp.~compact and connected)
subsets of $\gmetr$.
Whenever $x\in \gmetr$ and $F,F_1,F_2\in \mathfrak{K}(\gmetr)$
we will set
\begin{equation}
  \label{eq:10}
  \gdist(x,F):=\min_{y\in F} \gdist(x,y),\quad
  \mathsf{e}(F_1,F_2):=\max_{x\in F_1}\gdist(x,F_2),\quad
  \mathsf d_{\Hausdorff}(F_1,F_2)
  :=\mathsf{e}(F_1,F_2){\lor} \mathsf{e}(F_2,F_1);
\end{equation}
$(\mathfrak K(\gmetr),
\sfd_\Hausdorff)$ is a metric space, $\mathsf d_{\Hausdorff}$
is called the Hausdorff metric.
If $\gmetr$ is complete, also $\mathfrak{K}(\gmetr)$ is complete.
The class
$\mathfrak{K}_{\mathrm{c}}(\gmetr)$ is a closed subset of $\mathfrak
K(\gmetr)$:
if a sequence  $(F_n)_n \subset \mathfrak{K}_{\mathrm{c}}(\gmetr)$
of  connected compact sets converges to some $F$
 w.r.t.\ $\mathsf d_{\Hausdorff}$,  
then also $F$ is connected \cite[Thm.\ 4.4.17]{Ambrosio-Tilli}.
\par
 We say that a collection of compact sets 
$\mathcal K\subset \mathfrak K(\gmetr)$ is \emph{tight} if there exists a
common compact set $K\subset \gmetr$ such that $F\subset K$
for every $F\in \mathcal K$.
An application of the 
Blaschke
Theorem \cite[Thm.\
4.4.15]{Ambrosio-Tilli} shows that 
\begin{equation}
  \label{eq:34}
  \mathcal K\subset \mathfrak K(\gmetr)\quad\text{is relatively
    compact}
  \quad
  \Leftrightarrow
  \quad
  \mathcal K\text{ is tight.}
\end{equation}
In what follows,  we will also make use of a second notion of
convergence between sets, due to  \textsc{K.\ Kuratowski}. Again
following \cite[Def.\ 4.4.13]{Ambrosio-Tilli}, we say that a
sequence $(F_n)_n$ of compact subsets of $\gmetr$ converges in the
sense of Kuratowski to a compact set $F $,
and write $F_n \karrow F$, if
\begin{equation}
\label{kur-conv}
{\kliminf}_{n\to \infty} F_n = {\klimsup}_{n\to\infty} F_n = F,
\end{equation}
where
\begin{equation}
\label{Lliminf}  {\kliminf}_{n\to\infty} F_n  := \{ x \in 
\gmetr\,
  : \ \limsup_{n \to \infty}  \gdist(x,F_n)  =0\}
\quad \text{and} \quad
  {\klimsup}_{n\to\infty} F_n  := \{ x \in 
\gmetr\,
  : \
\liminf_{n \to \infty}  \gdist(x,F_n)  =0\}.
\end{equation}
Observe that $\kliminf_{n\to \infty} F_n \subset \klimsup_{n\to \infty} F_n $.
Therefore,
\eqref{kur-conv} is equivalent to the converse inclusion. Hence,
$F_n \karrow F$ if and only if
\begin{enumerate}
\item for every sequence $(x_n)_n$ with $x_n \in F_n$ for every $n \in \N$, and every convergent subsequence $x_{n_k}\to x$, we have $x\in F$; 
\item if $x \in F$, then there exists a \emph{whole} sequence
$(x_n)_n$ with $x_n \in F_n$ for every $n\in \N$, such that $x_n \to
x$.
\end{enumerate}
In general, 
$\mathsf d_{\Hausdorff}(F_n,F)\to0$
implies $F_n \karrow F$. If
 the sequence $(F_n)_n$ is tight, 
then  the
converse implication holds \cite[Prop.\ 4.4.14]{Ambrosio-Tilli};
moreover for every tight sequence of compact sets $(F_n)_n$ in
$\mathfrak K(\gmetr)$
\begin{equation}
  \label{eq:35}
  \klimsup_{n\to\infty}F_n\subset F\quad
  \Leftrightarrow\quad
  \text{every Hausdorff accumulation point of $(F_n)_n$ is a subset of $F$}.
\end{equation}
 The above definition of $\klimsup$ and $\kliminf$ can be obviously extended
to maps defined in an interval of $\R$. 
\paragraph{\emph {Multivalued maps: continuity properties.}} 
\newcommand{\Tsp}{\mathsf T}
\newcommand{\Zsp}{\Xsp}
Let now
$\Tsp$ be a compact interval $[a,b]\subset \R$ 
and $(\Zsp,\sfd_{\Zsp})$ be a metric space. We set $\gmetr:=\Tsp\times \Zsp$ with the usual max-product
metric. 

A multivalued map with
compact nonempty values
$\mathrm{A}:\Tsp\rightrightarrows \Zsp$ is a map from $\Tsp$ to $\mathfrak{K}(\Zsp)$; 
its graph is
\begin{equation}
 \label{eq:6}
\mathbf A=\operatorname{Graph} \mathrm A:=\big\{(t,x):t\in\Tsp,\,
x\in\rmA (t)\big\}
\subset \gmetr=\Tsp\times\Zsp.
\end{equation}
 Conversely, given a set
$\mathbf A\subset \Tsp\times \Zsp$
we can always define its sections (or fibers)
\begin{equation}
  \label{eq:45}
  \rmA(t):=\big\{x\in \Zsp:(t,x)\in \mathbf A\big\},\quad t\in \Tsp.
\end{equation}
When all the sections $\rmA(t)$ are nonempty and compact,
they define a compact-valued map $\rmA:\Tsp\to \mathfrak K(\Zsp)$
whose graph is precisely $\mathbf A$: 
in this case we say that $\mathbf A$ is the \emph{graph of a compact-valued
  map} (or just graph when no ambiguity arises).
Notice in particular that
$\pi^\Tsp\mathbf A=\Tsp$,
where $\pi^\Tsp : \Tsp \ti \Zsp \to \Tsp$ denotes the projection
operator.
%
\par
We say that $\mathrm A$ is \emph{tight} if its image is tight in $\mathfrak
K(\Zsp)$ (or, equivalently, if $\mathbf A$ is relatively compact in
$\gmetr$). 
We say that  a tight map
$\rmA$ is upper semicontinuous if for every $t\in \Tsp$ there holds
\begin{equation}
  \label{eq:9}
  \klimsup_{s\to t}\mathrm A(s)\subset \rmA(t).
\end{equation}
This property is equivalent to the closedness (thus the compactness)
of $\mathbf A$ in $\Tsp\ti
\Zsp$.
 Conversely,
if 
  \begin{gather}
    \label{closed2usc0}
\text{$\mathbf A\subset \Tsp\ti \Zsp$ is  compact  and 
  $\pi^\Tsp(\mathbf A)=\Tsp$}
\intertext{then its section map}
\label{closed2usc}
\text{$\rmA$  defined as in \eqref{eq:45} is tight and upper semicontinuous}.
\end{gather}
%
cf.\ Lemma \ref{l:A.1} ahead.
\par
\begin{definition}
  \label{def:fiberwise}
  We say that a set $\mathbf A\subset \Tsp\times \Zsp$ is
  \emph{fiberwise connected}
  if it is connected and all its sections
  \eqref{eq:45} 
  are connected.
\end{definition}
 The next result
(whose proof is postponed
to Lemma \ref{le:connection}  in Appendix \ref{s:app-1})
collects useful properties relating
connectedness and upper semicontinuity.
\begin{lemma}
  \label{le:recap}
  \ 
  \begin{enumerate}
  \item If $\mathbf A\subset \Tsp\times\Zsp$ satisfies
    \eqref{closed2usc0}
    and its sections $\rmA(t)$ are connected, then
    for every interval $I\subset \Tsp$ we have  that 
    $\mathbf A\cap (I\times \Zsp)$ is connected (in particular, $\mathbf A$ is connected and, thus, it is fiberwise connected). 
    Equivalently, the graph of
    a tight, upper semicontinuous multivalued map
    $\rmA:\Tsp\to\mathfrak K_c(\Zsp)$ is connected.
  \item The collection of the fiberwise connected compact
    graphs is closed in $\mathfrak K(\Tsp\times \Zsp)$.
  \end{enumerate}
\end{lemma}
 A map $a: \Tsp\to \Zsp$ induces a multivalued map $\rmA: \Tsp\rightrightarrows \Zsp$ 
obtained by defining $\mathbf A$ as
the graph of $a$
in $\Tsp\ti \Zsp$ and
$\rmA(\bar x):=\{a(\bar x)\}$ for every $\bar x 
\in \Tsp$.
We can then say that $a$ is \emph{tight} if it takes values in a compact
subset of $\Zsp$.
In this case its
closure $\bar\rmA$ at  $\bar x\in \Tsp$
can be expressed by 
$\bar \rmA(\bar x):=\{a(\bar x)\}\cup \klimsup_{x\to \bar x}a(x)$, where, 
with a slight abuse of notation, 
we have  written 
\begin{equation}
\label{abuse-limsup}
 \klimsup_{x\to \bar x}a(x) \text{ in place of  } 
\klimsup_{x\to \bar x}\{a(x)\}.
\end{equation}
\EEE%
If $a$ is continuous at $\bar x$ then 
$\rmA(\bar x)= \bar \rmA(\bar x) =\{a(\bar x)\}$.
In turn, 
\begin{equation}
\label{fromGR2FUNZ}
\begin{gathered}
\text{a multivalued map $\rmA$ induces a map $a:\mathrm{D}(a)\to \Zsp$
defined in the set $\rmD(a):=\{x\in \Tsp:\#\rmA(x)=1\}$;}
\\
\text{if its graph $\mathbf A$ is compact,
then $a$ is continuous in $\rmD(a)$.}
\end{gathered}
\end{equation}
  Conversely the graph of a continuous map $a: \Tsp\to \Zsp$ is compact.
\par
Notice that
any  tight 
sequence of continuous maps $a_n:\Tsp\to
\Zsp$ admits a subsequence $k\mapsto n(k)$ and a limit
 compact
and fiberwise connected graph
$\mathbf
A\subset \Tsp\ti \Zsp$ such that the graphs $\mathbf A_n:=\{(x,a_n(x)):x\in \Tsp\}$
satisfy
\begin{equation}
  \label{eq:11}
  \mathbf A_{n(k)}\karrow \mathbf A\quad\text{as }k\to\infty.
\end{equation}
 In what follows, we will also use the fact  that, 
  if the graphs  $(\mathbf{A}_n)_n \subset \Tsp \ti \Zsp$  are
  associated with a
 tight 
  sequence of uniformly Lipschitz  functions 
 $\teta_n : \Tsp \to \Zsp$ and converge in the sense of Kuratowski to
 some $\mathbf{A} \subset \Tsp \ti \Zsp$, then the functions
 $(\teta_n)_n$ uniformly converge to some $\teta:\Tsp \to \Zsp$
 and $\mathbf{A}$ is the graph of $\teta$;
 we postpone the precise statement to  Lemma \ref{l:compactness-graphs} ahead.

\paragraph{\emph{Hausdorff measures.}}
If  $(\gmetr,\gdist)$   is a metric space, 
$k=1,2$, the $k$-dimensional 
 pre-Hausdorff measure $\mathscr
H_\delta^k(A)$ and the Hausdorff measure $\mathscr H^k(A)$ are defined by
\begin{equation}
  \label{eq:12}
  \mathscr H_\delta^k(A):=
  \frac {\omega_k}{2^k}\inf\Big\{\sum_n \operatorname{diam}(A_n)^k:
  \operatorname{diam}(A_n)\le \delta,\quad
  A\subset \cup_n A_n\Big\},\quad 
  \delta>0,
  \quad 
  \omega_1:=2,\ \omega_2:=\pi,
\end{equation}
\begin{equation}
  \label{eq:13}
  \mathscr H^k(A)=\sup_{\delta>0}\mathscr H_\delta^k(A),
\end{equation}
 c.f.\ e.g.\ \cite{Ambrosio-Tilli}. We recall that 
$\mathscr H^k$ is a regular Borel measure. When $\gmetr=\R^k$ with the usual Euclidean distance, we have $\mathscr
H^k_\delta(A)=
\mathscr H^k(A)=\mathscr L^k(A)$ for every Borel set $A\subset \R^k$.

We call a  set $A$ \emph{countably 
$\mathscr H^1$-rectifiable}
if 
\begin{equation}
\label{rectifiable}
 \text{there exists a
sequence of Lipschitz curves  } \gamma_n:[0,1]\to \gmetr
\text{ such that } \mathscr H^1\big(A{\setminus} \cup_n\gamma_n([0,1])\big)=0.
\end{equation}
(up to reparameterization, we may in fact suppose all curves to be defined on $[0,1]$). 
\subsection{Time-dependent energies and their gradient flows}
\label{ss:2.1}
We now  focus on the setting discussed in the Introduction and
consider
a separable Hilbert space $\Hilbert$, a final time $T>0$, and we still keep
the convention to use boldface capital letters to
denote subsets of $[0,T]\times \Hilbert$.
Let
 \[
   \domainenergy \text{ be a convex subset of } \Hilbert,\quad
   \bE:=[0,T]\times \domainenergy,\quad
   \cE: \bE\to \R
 \]
 be a time-dependent energy functional.
 In  what follows, we will write $\ene tu$ in place of $\cE (t,u)$. 
Throughout the paper, we shall always suppose that $\cE$ satisfies the following properties:
%
\begin{description}
\item[Coercivity]
\begin{equation}
  \label{coercivita}
  \tag{${\mathrm{E}.1}$}
  \cE \text{ has compact sublevels in $\bE$;}
\end{equation}
in particular, $\cE$ is lower semicontinuous and bounded from below 
and, by 
operating 
the 
translation
\begin{equation}
  \label{eq:1}
  \eneb tu:=\ene tu-\min \cE+1
\end{equation}
we have that 
$\eneb tu\ge 1$.
Throughout
we will denote by $\sublevel\rho,\Sublevel\rho$, $\rho>0$, the sets
\begin{align}
\label{sublevel} & \sublevel{\rho} :=
                   \big\{ u \in \domainenergy\, : \ \eneb tu\leq\rho\quad\text{for every }t\in
                   [0,T]\big\},\quad
                   \Sublevel\rho:=[0,T]\ti\sublevel\rho.
\end{align}
\item[Time dependence] 
For every $(t,u) \in \bE$ 
 the partial derivative 
with respect to time
\[
\pt tu\,:=\,\partial_t \ene tu
\]
exists and satisfies 
\begin{equation}
    \label{P_t}
\tag{${\mathrm{E}.2}$}
\begin{aligned}
  |\pt tu|\leq C_P \, \eneb tu ,\qquad
 \left(
 u_n \to u, \ \
 \sup_{n \in \N} \ene t{u_n}<\infty
\right) \ \ \Rightarrow
\ \ \lim_{n \to \infty}\pt t{u_n} = \pt tu,
    \end{aligned}
  \end{equation}
  for a suitable nonnegative constant $C_P\ge 0$ 
  and every $t\in [0,T],$ $(u_n)_n, u\in \domainenergy$.
  \item[$\lambda$-convexity]
    There exists $\lambda\geq 0$ such that for every $t\in [0,T]$ 
  \begin{equation}
  \label{lambda-convex}
  \tag{${\mathrm{E}.3}$}
  \begin{aligned}
    \text{the map}\quad 
    u \mapsto \ene tu \quad \text{is \ $(-\lambda)$-convex in $\domainenergy$}
    \end{aligned}
  \end{equation}
  (namely, the functional $u\mapsto \ene tu +\frac{\lambda}2 |u|^2$
 is convex).   
  \end{description}
\par
The following lemma gathers
a series of consequences 
 of assumptions
\eqref{coercivita}--\eqref{lambda-convex} that will play a crucial role
for 
our analysis; 
recall 
that, for every $t\in [0,T]$, $\partial\cE_t: \Hilbert \rightrightarrows \Hilbert$  denotes the Fr\'echet subdifferential of
  the functional $u\mapsto \ene tu$,  cf.\ \eqref{def-fr-subdif}. We will denote by $\mathrm{dom}(\partial\cE):=\big\{(t,u)\in
[0,T]\ti\domainenergy:
\partial\ene tu\neq\emptyset\big\}$ its proper domain. 

\begin{lemma}
\label{l:conseqs-assumpts}
Let $\calE: \bE \to \R$ fulfill \eqref{coercivita}--\eqref{lambda-convex}. Then, $\calE$ enjoys the following properties:
\begin{subequations}
\label{consequences-main-ass}
\begin{description}
\item[Uniform in time energy growth]
\begin{equation}
\label{gronw-conseq}
\mathrm e^{-C_P|t-s|}\eneb su\le 
{\eneb tu}{}\le \mathrm
e^{C_P|t-s|}\eneb su
\qquad \text{for all } u \in \domainenergy,\ s,t\in [0,T].
\end{equation}
 In particular,  whenever 
$\ene tu\le C$ 
for some $t\in [0,T]$, then $u\in \sublevel\rho$
with $\rho=\rme^{C_P T}(C+1-\min \calE)$. 
\EEE
\smallskip
\item[Characterization of the Fr\'echet subdifferential]
for every $t\in [0,T], u\in \domainenergy$ and $\xi\in \Hilbert$,
\begin{equation}
\label{l-convex-charact}
\xi \in \partial \ene tu \ \ \text{if and
only if} \ \ \ene tv-\ene tu
\geq
\langle \xi, v{-}u \rangle -
 \frac{\lambda}{2} \| v{-}u\|^2 
 \quad \text{for all } v \in  \domainenergy.
\end{equation}
\item[Closedness]
For every $(t,u)\in [0,T]\ti\domainenergy$ and for all sequences
$(t_n)_n\subset[0,T]$, $(u_n)_n\subset\domainenergy$ and
$(\xi_n)_n \subset\Hilbert$ such that
$\xi_n \in\partial \ene {t_n}{u_n}$ for every $n \in \N$,
there holds 
\begin{equation}
\label{closedness}
\begin{gathered}
\left(t_n \to t,\ \ u_n \weakto u, \ \ \xi_n \weakto \xi,
\ \ \ene{t_n}{u_n} \to \mathscr E<\infty\right)
\ \ \Longrightarrow \ \ \xi \in \partial \ene tu
\text{ and }
\mathscr E= \ene tu.
\end{gathered}
\end{equation}
\item[Minimal selection of the subdifferential]
   $\partial\ene tu$ is a closed convex set; if it is not empty, its
  element of minimal norm will be denoted by  $\argminpartial {\cE}tu
  $;  we set   
  \[
  \minpartial \calE tu :=
  \begin{cases}
  \min \{\| \xi \|\, :
 \ \xi \in \partial \ene tu\}  &   \text{if }
 (t,u)\in\mathrm{dom}(\partial\cE),
 \\
 +\infty & \text{otherwise}.
 \end{cases}
 \] 
\item[Chain rule] If  $w\in\AC([a,b];\domainenergy)$,
  $\sft\in\AC([a,b];[0,T])$
and 
$\xi:(a,b)\to \Hilbert$ is a Borel map 
with 
\begin{displaymath}
\begin{gathered}
\xi(s)
\in \partial \ene {\sft (s)}{w(s)} \text{ for $\mathscr{L}^1$-a.a.\ $s \in (a,b)$}  \quad \text{ and } \quad \int_a^b \| \xi(s)\|\,\|w'(s)\| \dd s <\infty,
\end{gathered}
\end{displaymath}
then the map $s \mapsto \ene {\sft(s)}{w(s)}$  is absolutely continuous
  in $[a,b]$ and
\begin{equation}
\label{ch-rule}
\begin{gathered}
\frac{\mathrm{d}}{\mathrm{d}s}\ene {\sft(s)}{w(s)}
=
\langle \xi(s), w'(s) \rangle + \pt {\sft(s)}{w(s)}\sft'(s) \ 
 \text{ for $\mathscr{L}^1$-a.a.\ $s \in (a,b)$.}
\end{gathered}
\end{equation}
\end{description}
\end{subequations}
\end{lemma}
\begin{proof}
Estimate \eqref{gronw-conseq}
immediately follows from the power control condition
\eqref{P_t}  via  the Gronwall Lemma.
The characterization in 
\eqref{l-convex-charact} ensues from \eqref{lambda-convex} via
 a direct calculation, cf.\ e.g.\ \cite[Rmk.\ 2.5]{MRS2013}.
 We use it to prove the closedness property \eqref{closedness}. Indeed, 
 for all sequences 
$(t_n)_n$, $(u_n)_n$, and $(\xi_n)_n$ as in \eqref{closedness},
we find for every $v\in\domainenergy$
\[
\langle \xi, v{-}u \rangle -
\frac{\lambda}{2} \| v{-}u\|^2
\leq
\liminf_{n\to\infty}
\big[\,
\ene{t_n}v{-}\ene{t_n}{u_n}
\,\big]
\leq
\limsup_{n\to\infty}\ene{t_n}v-\liminf_{n\to\infty}\ene{t_n}{u_n}
\leq
\ene tv-\ene tu,
\]
whence $\xi \in \partial \ene tu$ by  the
characterization   \eqref{l-convex-charact}.
Here, the first inequality comes from \eqref{l-convex-charact} and from
the convergences $u_n\to u$ 
(note that the weak convergence $u_n\weakto u$ improves to a strong one thanks to the bound for $(\ene{t_n}{u_n})_n$ and the coercivity 
\eqref{coercivita}), 
and $\xi_n\weakto\xi$,
and the third inequality from the upper semicontinuity
of $t\mapsto\ene tv$,
which is implied by the differentiability of the same map,
and from the lower semicontinuity of
$(t,v)\mapsto \ene tv$ on $\mathrm{dom}(\cE)$.
 With an analogous argument one can also conclude the energy convergence $\ene{t_n}{u_n} \to \ene tu$ as $n\to\infty$.  Finally, the chain-rule property \eqref{ch-rule} can be proved as in \cite[Prop.\ 2.4]{MRS2013},  cf.\  also \cite[Prop.\ A.1]{MiRo23}. 
\end{proof}
\begin{remark}
\slshape
\label{rmk:other-closedness}
The very same arguments for the proof of  \eqref{closedness}, 
 based on the characterization 
\eqref{l-convex-charact},   yield 
that a functional $\calE$ as in \eqref{coercivita}--\eqref{lambda-convex} also enjoys
the following variant of the closedness property
\begin{equation}
\label{closedness-alter}
\begin{gathered}
\left(t_n \to t,\ \ u_n \to u, \ \ \xi_n \weakto \xi \right)
\ \ \Longrightarrow \ \ \xi \in \partial \ene tu
\text{ and }
\ene{t_n}{u_n} \to \ene tu,
\end{gathered}
\end{equation}
where the energy bound requirement in \eqref{closedness} is dropped, at the price of enforcing \emph{strong} convergence of $(u_n)_n$; the energy convergence
$\ene{t_n}{u_n} \to \ene tu$
 is still guaranteed. 
\end{remark}
\par
We now  recall the following  well known existence result  for the gradient flow  \eqref{g-flow-intro}.
\begin{theorem}[Existence of solutions to the gradient flow \eqref{g-flow-intro}]
\label{thm:exist-g-flow} Let $\cE: [0,T]\ti \domainenergy \to
\R$ comply with \eqref{coercivita}--\eqref{lambda-convex}.
Then,  for every $\eps>0$ and 
every $u_{0,\eps} \in \domainenergy$ there exists a solution 
$u_\eps\in H^1(0,T;\Hilbert)$,  taking values in $\domainenergy$,
solving the
Cauchy problem
\begin{equation}
\label{Cauchy-gflow}
\eps u_\eps'(t) + \minsub t{u_\eps(t)}= 0
\quad  
\text{a.e.~in}\  (0,T),\qquad
u_\eps(0)=u_{0,\eps}.
\end{equation}
The curve $u_\eps $ fulfils the energy-dissipation  identity
\begin{equation}
\label{eqn_lemma}
\int_s^t\left(\frac{\ep}2\|u_\eps'(r)\|^2
  {+}\frac1{2\eps}{\|\minsub r{u_\eps(r)}\|^2}\right)\dd r
+\mathcal E_t(u_\eps(t))=
\mathcal E_s(u_\eps(s))
+\int_s^t\pt r{u_\eps(r)}\dd r \ \ \text{for all $0\leq s \leq t \leq T$,}
\end{equation}
 and the uniform bounds 
\begin{align}
& \label{bound-energie}
  \sup_{t \in [0,T]} \eneb t{u_\eps(t)}\le \eneb
  0{u_{0,\eps}}\,\rme^{C_P T},\quad
  \sup_{t
  \in [0,T]} |\pt  t{u_\eps(t)}| \leq C_P\eneb
  0{u_{0,\eps}}\,\rme^{C_P T} ,
\\
& \label{bound-en-diss}
\int_0^T \left(\frac{\ep}2 \mdtq u\eps r
2+\frac1{2\ep} \minpartialq {\cE}{r}{u_\eps(r)}\right) \dd r
  \leq  (1{+}C_PT\,\rme^{C_P T})\eneb 0{u_{0,\eps}}.
\end{align}
\end{theorem}
Let us only mention that, indeed, estimates \eqref{bound-energie} and 
\eqref{bound-en-diss} follow easily by 
the energy identity \eqref{eqn_lemma} with $s=0$, 
the  estimate for the power functional $\mathcal{P}$ in \eqref{P_t},  and the  Gronwall Lemma.
\begin{remark}
\slshape
\label{rmk:on-exist-gflow}
We refer, e.g., to
 \cite[Thm.\ 4.4]{MRS2013} for a more general existence result for the Cauchy problem for  the gradient flow \eqref{g-flow-intro}
(cf.\ also the existence results in 
\cite{RossiSavare06} for the autonomous case). Indeed, in \cite{MRS2013}
the $(-\lambda)$-convexity assumption \eqref{lambda-convex} was directly 
replaced by conditions \eqref{closedness}--\eqref{ch-rule} for
the Fr\'echet subdifferential of the energy functional $\cE$.
\end{remark}
\subsection{Critical points}
\label{ss:2.2-RADDED}	
As mentioned in the Introduction, in addition to conditions 
\eqref{coercivita}--\eqref{lambda-convex}, typical of the variational approach to (generalized) gradient flows and rate-independent systems, our
analysis of the singular perturbation problem \eqref{e:sing-perturb} will rely on suitable assumptions on the critical points of the energy functional $\calE_t$. 

\begin{definition}[Critical points]
\label{not:critical}
For every $t\in [0,T]$ we will denote by $\critset(t)\subset \domainenergy$ the set
of critical points of $\ene t\cdot$, i.e.~those points
\begin{equation}
\label{def-critical}
 \text{$x\in \Hilbert$ such that $\partial \calE_t(x) \ni 0$, or,
   equivalently, $\minpartial \calE tx=0$};
\end{equation}
we will set 
\[
\crit:=\big\{(t,x)\in [0,T]\ti \domainenergy:x\in
\critset(t)\big\}.
\]
We will also use the notation 
\[
\necrit: = ([0,T]\ti \domainenergy)\setminus \crit
\]
 and $\necritset t $ for the section of $\necrit$ at the process time $t\in [0,T]$. 
\par
Whenever $x\in \critset(t)$, 
the connected component of $\critset(t)$
containing $x$
will
be denoted by $\comp tx$. For a connected subset $V\subset \critset(t)$ we shall use the notation $\comp tV$ for the connected component of 
$\critset (t)$ containing $V$.  
\end{definition}
%

Indeed, throughout the paper we 
 will also treat 
the time-dependent set of critical
points of $\cE$ as a multifunction $\critset: [0,T] \rightrightarrows
\Hilbert$. Observe that 
the closedness property \eqref{closedness-alter} ensures that 
the sets $\crit \cap \Sublevel\rho$ and $\critset(t) \cap \sublevel\rho$ are
compact.   

\par
The following result settles some 
estimates for 
$\calE$ on the critical set $\crit$ that will be crucial for our analysis.
In its statement, we omit to explicitly invoke our standing assumptions \eqref{coercivita}--\eqref{lambda-convex}.
\begin{lemma}
\label{l:E-Lip}
There holds
\begin{equation}
  \label{eq:16}
  \big|\ene tv-\ene tu\big|\le \frac \lambda 2 \|v-u\|^2\quad\text{for
    every }u,v\in \critset(t) \text{ and  every } t \in [0,T]. 
\end{equation}
Furthermore, for every $(t,u),\, (s,v) \in \crit$ the following estimates hold with $L: = C_P \rme^{C_P T}$, $C_P$ from \eqref{P_t}, and $\widetilde\calE$ from \eqref{eq:1}:
\begin{subequations}
\label{GESTIMATES}
\begin{align}
&
\label{GEST-1}
\int_s^t \pt ru \dd r -\frac{\lambda}{2} \|v-u\|^2 \leq \ene tu - \ene sv \leq \frac{\lambda}{2} \|v-u\|^2+\int_s^t \pt rv \dd r,
\\
& 
\label{GEST-2}
-L   \eneb tu |t-s| - \frac{\lambda}{2} \|v-u\|^2 \leq \ene tu - \ene sv \leq L  \eneb sv |t-s| +  \frac{\lambda}{2} \|v-u\|^2\,.
\end{align}
\end{subequations}
In particular, 
$\calE$ is Lipschitz in   $\crit  \cap 
\Subl\rho$   for all $\rho>0$. 
\end{lemma}
\begin{proof}
Estimate \eqref{eq:16} is an immediate consequence of \eqref{l-convex-charact}. 
In order to show \eqref{GEST-1}, we observe that 
\[
\ene tu - \ene sv  = \int_s^t \pt ru \dd r + \ene su -\ene sv   \geq  \int_s^t \pt ru \dd r  - \frac{\lambda}2 \|v-u\|^2 
\]
by \eqref{eq:16}. 
Analogously, we have 
\[
\ene tu - \ene sv = \ene tu - \ene tv + \int_s^t \pt rv \dd r \leq \frac{\lambda}2 \|v-u\|^2 
+ \int_s^t \pt rv \dd r\,,
\]
and  \eqref{GEST-1} ensues.
\par
As for   \eqref{GEST-2}, we notice that 
\[
\ene tu - \ene sv = \ene tu -\ene su +\ene su -\ene sv 
\stackrel{(1)}{\geq} \eneb tu \left(1{-} \rme^{C_P|t-s|} \right) - \frac{\lambda}{2} \|v-u\|^2
\stackrel{(2)}{\geq}  - L |t-s| \eneb tu - \frac{\lambda}{2} \|v-u\|^2 
 \]
 where for  (1) we have used that 
 \[
 \ene tu -\ene su  = \int_s^t \pt ru \dd r \geq -\int_s^t C_P \eneb r u \dd r \geq - \int_s^t C_P  \rme^{C_P|t-r|} \eneb tu \dd r 
 \]
by \eqref{P_t} and   \eqref{gronw-conseq}, while the term 
$\ene su -\ene sv $ has been estimated by \eqref{eq:16}. 
 Next, (2) is due to  the estimate
 $|\rme^{C_P|t-s|} -1| \leq L |t-s|$ for all $s,t \in [0,T]$. 
 Writing $\ene tu - \ene sv$ as $ \ene tu -\ene tv +\ene tv -\ene sv$   and arguing 
 as in the above lines 
 gives the upper estimate $\leq$  in \eqref{GEST-2}.  \end{proof}

Our main results on the singularly perturbed gradient flow will hinge
on the following
 property
of
 $\crit$.
\newcommand{\ass}[1]{\eqref{ass1}}
\begin{definition}[Fibered Sard property]
  \label{ass:crit}
  Let $\crit$ be defined as
  the set of critical points of
  $\Ene$
  according to Definition \ref{not:critical}. We will
  say that $\Ene$ satisfies
  a \emph{fibered Sard property} on $\crit$
  if
   \begin{equation}
  \label{ass1}
  \mathscr L^1\big(\ene t{\critset(t)}\big)=0
  \quad\text{for every $t\in
    [0,T]$.}
  \tag{C}
\end{equation}
\end{definition}
\begin{remark}[Clean critical set]
  \label{rem:E-constant}
  \slshape
  We say that $\mathcal E$ has a \emph{clean} critical set if
  \begin{equation}
  \label{eq:16bis}
  \ene t\cdot \text{ is constant on every connected component of
  }\critset (t).
\end{equation}
 By \eqref{coercivita}, such a property in particular implies that
\begin{equation}
  \label{eq:31}
  \text{for every $t\in [0,T]$
    every connected component of $\critset (t)$ is compact.}
\end{equation}
Notice that \eqref{eq:16bis} is surely satisfied if 
\ass1\ holds,
since $\ene t{\cdot}$ is continuous on $\critset(t)$ by
\eqref{eq:16}.
\end{remark}
Recall that a set is totally disconnected if  
its connected components are singletons.
\begin{definition}[Topologically singular times]
  We call $\totaldisc$ the subset of times $t\in [0,T]$
  for which $\critset(t)$ is totally disconnected:
  \begin{equation}
    \label{eq:46}
    \totaldisc:=\Big\{t\in [0,T]: \comp tx=\{x\}
    \quad\text{for every }x\in \critset(t)\Big\}.
  \end{equation}
  The complement of $\totaldisc$ is the 
  set $\nototaldisc\subset [0,T]$ of topologically singular 
  times, defined by
    \begin{equation}
      \label{eq:3}
      \nototaldisc:=\Big\{t\in [0,T]:\critset(t)\text{
        is not
       totally disconnected}\Big\}.
   \end{equation}
\end{definition}
If
\ass1\ holds, then by \eqref{eq:16bis} the definition of $\nototaldisc$ can be somehow localized to
the sublevels $\sublevel\rho$, i.e.
\begin{equation}
  \label{eq:3bis}
        \nototaldisc=\bigcup_{\rho>0}\Big\{t\in [0,T]:\critset(t)\cap\sublevel\rho\text{
        is not
       totally disconnected}\Big\}.
\end{equation}
It is easy to check that if
\begin{equation}
  \label{eq:2pre}
  \text{for every $t\in [0,T]$}\quad \critset(t)\quad\text{is at most countable}
\end{equation}
then  \ass1~holds
and  $\nototaldisc=\emptyset$.
In fact, 
$\ene t{\critset(t)}$ is at most countable and therefore $\mathscr
L^1$-negligible and
every countable set in a metric space is totally disconnected. We remark that
\eqref{eq:2pre} surely holds if 
\begin{equation}
  \label{eq:2}
  \text{for every $t\in [0,T]$ and $\rho>0$ the set }\critset(t)\cap\sublevel\rho\quad\text{is discrete}.
\end{equation}
Another simple case ensuring the validity of \eqref{ass1} occurs when 
$\cE$ has a clean critical set 
and every section $\critset(t)$
has countably many connected component.

The next, crucial, lemma establishes further sufficient conditions under which \eqref{ass1} holds and, also, the set $\nototaldisc$ is at most countable.

\begin{lemma}
\label{le:criterion}
  Let $\Honezero$ be the set of times
  for which $\critset(t)$ is $\mathscr H^1$-negligible.
  \begin{enumerate}[label=\rm \arabic*.]
 \item 
   $\Honezero\subset \totaldisc$, so that 
   the set of topologically singular times $\nototaldisc$ is contained in
   the set $\noHonezero$ of $\mathscr H^1$-singular times,
   defined by
  \begin{equation}
    \label{eq:32}
    \noHonezero:=\Big\{t\in [0,T]:\mathscr H^1(\critset(t))>0\Big\}.
  \end{equation}
\item If
  \begin{equation}
    \label{ass1'}
    \mathscr
    H^2\big(\critset(t)\big)=0
    \qquad
    \text{for every $t\in [0,T]$},
    \tag{$\mathrm C_1$}
  \end{equation}
  then condition {\em \ass1} is satisfied.
  \item 
In particular, if
\begin{equation}
\label{eq:15pre}
\text{$\mathscr H^1\res \crit$ is $\sigma$-finite,}
\end{equation}
e.g.~if $\crit$ is  countably 
$\mathscr H^1$-rectifiable 
in the sense of  \eqref{rectifiable},
then \emph{\ass1} holds
and, moreover,  $\nototaldisc,\ \noHonezero$ are at most countable.
\end{enumerate}
\end{lemma}
\begin{proof}
  \upshape
   As for Claim 1,
observe that,
 whenever $\mathscr
      H^1(\critset(t))=0$, 
      then every connected component of $\critset(t)$   (which has
      finite
      $\mathscr H^1$-measure
      and thus  is connected by injective
      rectifiable paths, cf.\ \cite[Thm.\ 4.4.7]{Ambrosio-Tilli})
      should contain only one point.  Thus
      \EEE $\Honezero \subset \totaldisc$.

      Claim 2:  Suppose that $\mathscr H^2(\critset(t))=0$. 
     Then, for every $\eps>0$ we may find 
      a collection of sets $B_n\subset\critset(t)$ with 
      diameter less than $\eps$ such that $\critset(t)\subset \cup_n B_n$ and
      $\sum_n \operatorname{diam}(B_n)^2\le \eps $. 
      Thus by  estimate \eqref{eq:16} the sets $B_n':=\ene t{B_n}$ satisfy $\operatorname{diam}(B_n')\le  \lambda \eps^2 $
      and induce a covering of   $\ene t{\critset(t)}$ \EEE with 
      \begin{displaymath}
        \mathscr H^1_{\lambda \eps^2}(\ene t{\critset(t)})\le \lambda \eps.\end{displaymath}
      Passing to the limit as $\eps\down0$ we get
       $\mathscr L^1(\ene t{\critset(t)})=\mathscr H^1(\ene t{\critset(t)})=0$, which yields
       \ass1. 
       \par
       Claim 3.
      If \eqref{eq:15pre} holds,
      $\crit$ can be covered by
      a countable family of 
      Borel sets $(\Gamma_n)_n$ with finite $\mathscr H^1$-measure,
and therefore null $\mathscr H^2$-measure,
      so that \ass{1} is obviously satisfied by the previous claim.
      
      By the
      $\sigma$-additivity of $\mathscr H^1$,
      the set $\noHonezero$ defined by \eqref{eq:32}
      is countable, since for every $n \in \N$ the set of
      times 
      $R_n:=\{t\in [0,T]:
      \mathscr H^1\big(\{t\}\ti \critset(t)\cap\Gamma_n\big)>0\}$ is at most countable.
\end{proof}
 With our following result we show that the energy is constant on Lipschitz-connected components of $\crit$, and use this to provide a sufficient condition for \emph{\ass1}. 
\begin{lemma}
\label{le:lip-connected}
%
If for every $t\in[0,T]$ and every $\rho>0$ the set $\critset(t)\cap\sublevel\rho$ is covered by \emph{countably many} subsets of $\critset(t)$ that are Lipschitz-path-connected (i.e., any two of their points are joined by a Lipschitz curve valued in $\critset(t)$), then $\ene t{\critset(t)}$ is countable for every $t$, and condition \emph{\ass1} holds.
\end{lemma}
\begin{proof}  Fix $t\in [0,T]$ and let $(O_i)_{i\in I}$ be a countable covering of $\critset(t)\cap\sublevel\rho$ consisting of Lipschitz path-connected sets. Pick $x,y\in O_i$ and let $\teta$ be a Lipschitz curve connecting them. 
By the chain rule \eqref{ch-rule} (which applies thanks to the $\lambda$-convexity of the energy), $s\mapsto\ene t{\teta(s)}$ is absolutely continuous with $\frac{\dd}{\dd s}\ene t{\teta(s)}=\langle\xi,\teta'(s)\rangle$ for every $\xi\in\partial\ene t{\teta(s)}$, for a.a.\ $s$; since $\teta(s)\in\critset(t)$ for every $s$, we may choose $\xi=0$, which gives  that $\ene tx = \ene ty$. Thus, $\calE$ is constant on $O_i$. The claim follows. 
\end{proof}
\begin{remark}
\slshape
\label{rmk:suff-subl}
 Since we can express $\crit$ as the countable union
$\crit=\cup_{n\in \N} \crit \cap\Sublevel n$,
it is equivalent \
to
impose \ass 1 (and, analogously, \eqref{ass1'})
for $\crit \cap \Sublevel\rho$
for any $\rho>0$.
Likewise, \eqref{eq:15pre} is equivalent to 
\begin{equation}
\label{intersection-sublevels}
  \text{$\mathscr H^1\res 
  \crit \cap \Sublevel \rho $ is 
  $\sigma$-finite 
  for any $\rho>0$,}
\end{equation}
  and $\crit$ is  countably 
$\mathscr H^1$-rectifiable if and only if $\crit\cap \Sublevel \rho$ is  countably 
$\mathscr H^1$-rectifiable for every $\rho>0$.
\end{remark}

\section{Solution concepts and main results}
\label{ss:2.2-added}
In what follows we will confine the discussion to   the case of energies with a clean critical
set,
according to Remark \ref{rem:E-constant}, so that
connected components of $\critset(t)$ are compact. 
Our solution notions for 
\eqref{limit-problem}
hinge on the following 
definitions of \emph{admissible transition} and of
\emph{energy-dissipation cost}. As we mentioned earlier,  our approach to 
the asymptotic analysis of \eqref{e:sing-perturb} will be based on  Kuratowski compactness arguments for \emph{graphs} 
  and, accordingly, 
  we will obtain as
vanishing-viscosity
 limit a multivalued graph, rather than a curve.  That is  why, 
it is expedient  to introduce the concepts of transitions and cost 
 between two \emph{compact connected sets}
$\cU_0,\cU_1$ of $\Hilbert$, rather than between two points.  

In fact, we shall restrict to a special class of compact connected
sets: 
for  fixed $t\in [0,T]$, we shall either consider the connected components of $\critset (t) $, or  the points of
$\necritset t = \Hilbert \setminus \critset(t)$.
For this class of sets, we will use the following notation (where the letter $\mathfrak{E}$ emphasizes the relation  of these sets to the energy functional $\calE$): 
\begin{equation}
\label{new-class}
\newclass t: = \{ U \in \mathfrak{K}_{\mathrm{c}}(\Hilbert)\, : \, U \text{ is a connected component of } \critset (t) \text{ or } U = \{x\} \text{ for } x \in \necritset t \}\,.
\end{equation}
\begin{definition}[Lifting of maps]
  \label{def:lifting}
  Let us suppose that $\mathcal E$ has a clean critical set, according
  to \eqref{eq:16bis}.
  Given a map $u:[0,T]\to \Hilbert$ we define its
  lifting $\rmU:[0,T]\to \mathfrak K_c(\Hilbert)$ by
  \begin{equation}
    \label{eq:lifting}
    \rmU(t):=\
    \begin{cases}
      \comp t{u(t)}&\text{if }u(t)\in \critset (t),\\
      \{u(t)\}&\text{otherwise},
    \end{cases}
  \end{equation}
  recalling that  $ \comp t{u(t)}$ denotes the connected component of
  $\critset (t) $ containing $u(t)$.

   We also in troduce the \emph{single valued} set
  \begin{equation}
    \label{eq:48}
    \SV u=\SV \rmU=\Big\{t\in [0,T]:\rmU(t)=\{u(t)\}\Big\}.
  \end{equation}
\end{definition}
Notice that $\rmU(t)\in \newclass t$ for every $t\in [0,T]$.

We now introduce the concept of admissible transition \emph{curve} between  two sets in the  class $\newclass t$.  
Nonetheless, later on, with Lemma \ref{le:graph} and Definition \ref{def:graph-trans}, we will provide a \emph{graph} characterization of such curves, in  tune with our subsequent compactness arguments. 
\begin{maindefinition}[Admissible transitions]
  \label{main-def-1}
    Let us suppose that $\mathcal E$ has a clean critical set, according
  to \eqref{eq:16bis}.
Let $t \in [0,T]$ be fixed and
let $\cU_0,\, \cU_1 \in \newclass t $.
An {\em admissible transition} between $\cU_0,\, \cU_1$
is a map $\teta: [0,1]\to \domainenergy$ 
such that 
$\teta(i)\in U_i$, $i=0,1$, and 
\begin{enumerate}[label=\alph*), nolistsep]
\item 
  the  set 
  \begin{equation}
  \label{necri-teta}
  \necri t {\teta}:= \{ s\in [0,1]\, : \teta(s) \in \necritset t\}
 = 
  \{s\in
  [0,1]:\minpartial \calE t{\teta(s)}>0\}
  \end{equation}
 is relatively open, $\teta$ is locally Lipschitz in   $\necri t {\teta}$,  and there exists a constant
  $L\ge 0$ such that
  \begin{equation}
    \label{eq:20}
    \minpartial \calE t {\teta(s)} \,\|
  \teta'(s)\|\le L\quad\text{for $\mathscr L^1$-a.e.~$s\in \necri t {\teta}$;}
  \end{equation}
  %
\item the function $s\mapsto \ene t{\teta(s)}$ is Lipschitz in $[0,1]$;
\item for every $r\in [0,1]\setminus \necri t {\teta}$ 
  we have 
  $\klimsup\limits_{s\to r}\teta(s)\subset \comp t{\teta(r)}$ (where $ \comp t{\teta(r)}$ denotes the connected component of the critical set $\critset(t)$ containing $\teta(r)$). 
\end{enumerate}
We will denote by $ \admis{\cU_0}{\cU_1}{t} $ the class of
admissible transitions connecting $\cU_0$ and $\cU_1$.
An admissible transition $\teta\in \admis{\cU_0}{\cU_1}{t} $ 
is \emph{sharp} if $\necri t {\teta}$ has full measure in $[0,1]$ (in particular it is dense).
%
\end{maindefinition}
The following chain-rule estimate
 along admissible transitions will play a crucial role for our analysis.
\begin{lemma}
\label{rmk:admiss-curves}
Let $t \in [0,T]$ be fixed and
let $\cU_0,\, \cU_1 \in \newclass t$. 
 Then, for every admissible transition 
$\teta \in  \admis{\cU_0}{\cU_1}{t}$ there holds
\begin{equation}
  \label{eq:17}
  \big|\ene t{\teta(s_1)}-\ene t{\teta(s_0)}\big|
  \le 
  \int_{\necri t {\teta}\cap (s_0,s_1)} \minpartial \calE t {\teta(r)} \| \teta'(r)\|\dd
r \qquad \text{for all } 0 \leq s_0 \leq s_1 \leq 1. 
\end{equation}
\end{lemma}
\begin{proof}
Let $s_0 \leq s_1 \in [0,1]$ be fixed.
We observe
that  the image set $\ene t{\teta([s_0,s_1])}$ decomposes as 
$\ene t{\teta([s_0,s_1])} = \ene t{\teta\big([s_0,s_1]{\setminus}\necri t {\teta}\big)} \cup  \ene t{\necri t {\teta}}$;
using that  $\teta\big([s_0,s_1]\setminus
\necri t {\teta}\big)\subset \critset(t)$, we conclude with  \ass1  that 
\begin{equation}
\label{negligible}
\mathscr{L}^1(\ene t{\teta([0,1])} {\setminus} \ene t{\necri t {\teta}} ) \leq  \mathscr{L}^1(\ene t{\critset(t)}) =0 \,.
\end{equation}
Therefore,  we have
\[
\begin{aligned}
  \big|\ene t{\teta(s_1)}-\ene t{\teta(s_0)}\big|
\stackrel{(1)}{=}
  \left|\int_{s_0}^{s_1} \frac\d{\d r}\ene t{\teta(r)}\dd r \right|  & \stackrel{(2)}{=}
  \left|\int_{\necri t {\teta} \cap (s_0,s_1)} \frac\d{\d r}\ene t{\teta(r)}\dd  r \right| \\ &  \stackrel{(3)}{\leq}   
  \int_{\necri t {\teta}\cap (s_0,s_1)} \minpartial \calE t {\teta(r)} \| \teta'(r)\|\dd r \,.
  \end{aligned}
\]
Here, 
 (1) is  due to the Lipschitz continuity of $r\in [0,1]\mapsto \ene t{\teta(r)}$, while   (2)  relies on \eqref{negligible}
 and on the property that for every absolutely continuous function $g: [0,1]\to \R$ and for every $\mathscr{L}^1$-negligible set $A \subset \R$ one has $g' \equiv 0$
 a.e.\ in $g^{-1}(A)$. Finally, estimate
 (3) follows from the chain-rule estimate \eqref{ch-rule}, taking into account the local Lipschitz continuity of $\teta$ on $\necri t {\teta}$.  
\end{proof}

\begin{maindefinition}[Energy-Dissipation cost]
For every $t \in [0,T]$ the \emph{energy-dissipation cost} between
two sets $\cU_0,\, \cU_1 \in \newclass t$
is defined by
\begin{equation}
\label{costo-simpler}
\cost t{\cU_0}{\cU_1}:=
\inf\left\{\int_{\necri t {\teta}} \minpartial \calE t {\teta(s)} \| \teta'(s)\|\dd
s \,:\,\teta\in
 \admis{\cU_0}{\cU_1}{t} \right\}
\end{equation}
with the usual convention 
$\cost t{\cU_0}{\cU_1}:=+\infty$ if $\admis{\cU_0}{\cU_1}{t} $ is empty.
\end{maindefinition}
%
Such cost enjoys properties that one expects, such as symmetry (cf.\
Theorem \ref{prop:cost}).

\subsection{Dissipative Viscosity solutions}
We are now  in a  position to give the definitions of \emph{Dissipative Viscosity}  solution
to  the evolution of critical points \eqref{limit-problem}.
%
%
\begin{maindefinition}[{\DSolution}]
\label{def:sols-notions-1}
Let us assume 
\eqref{coercivita}--\eqref{lambda-convex}
and let us suppose that $\mathcal E$ has a clean critical set, according
  to \eqref{eq:16bis}. \\
A
 Borel map $u:[0,T]\to \domainenergy$ 
with lifting $\rmU$ (according to definition \ref{def:lifting}) 
is a  
\emph{Dissipative Viscosity} 
solution to \eqref{limit-problem}  
if it 
satisfies the following
properties:
\begin{enumerate}
\begin{subequations}
\item The map
$t \mapsto \sfe(t):=\ene t{u(t)} $ has 
bounded variation 
in $[0,T]$ and 
it induces   a nonnegative finite Borel measure $\mu$ on $[0,T]$ 
satisfying
\begin{equation}
\label{to-save-eneq}
\mu([s,t])  +\ene t{u(t)}=\ene s{u(s)}+\int_s^t \pt r{u(r)} \dd r\quad \text{for all }
s,t\in [0,T]\setminus 
\mathrm{J}
\text{ with }  s \leq t,
\end{equation}
where 
$\mathrm J:=\big\{t\in
[0,T]:\mu(\{t\})>0\big\}$ 
is the (countable) atomic set of $\mu$;  $\mathrm{J}$ coincides 
with the jump set $\jump \sfe$ of $e$.
\item
  For every $t\in [0,T]\setminus \rmJ$
  we have
  %
  \begin{equation}
    \label{eq:8}
    u(t) \in   \critset(t),\qquad
    \klimsup_{s\to t}u(s)\subset \rmU(t).
  \end{equation}
\item
For every 
$t\in\rmJ$ 
there exists a unique pair  of connected
  components $\ulim(t-), \ulim(t+)$ of $\critset(t)$ 
  such that
\begin{gather}
\label{to-save-2}
\begin{aligned}
\klimsup_{s\uparrow t 
} u(s)&\subset \ulim(t-),  &\quad  
\sfe(t-)=\lim_{s\up t}\ene s{u(s)}
&=
\ene t{\rmU(t-)},
\\
\klimsup_{s\downarrow t
} u(s)&\subset \ulim(t+),
&\quad
\sfe(t+)=\lim_{s\down t}\ene s{u(s)}
&=
\ene t{\rmU(t+)},
\end{aligned}
\\
\label{to-save-4}
 0 < \cost t {\ulim(t-)}{\ulim(t+)}   = 
  \cost t{\ulim(t-)}{\rmU(t)}+
  \cost t{\rmU(t)}{\ulim(t+)} = e(t-)-e(t+)  =\mu (\{ t\}).
\end{gather}
\end{subequations}
\end{enumerate}
\end{maindefinition}
\begin{remark}
\label{energy-bound}
\slshape
From the energy-dissipation identity \eqref{to-save-eneq}, via the power estimate \eqref{P_t} and the Gronwall Lemma, we easily deduce that any \DSolution\  
$u$ satisfies
$u(t) \in  \rmU(t)\subset \ \sublevel \rho$ for all $t\in [0,T]$ and some $\rho>0$. 
\end{remark}
It is useful to point out a simple regularity property of
\DSolutions.

\begin{proposition}[Improved continuity of 
\DSolutions]
  \label{prop:improved-C}
  Every \DSolution\ $u$ has left and right limits
  at every time $t\in \totaldisc$
  (with obvious modification in case $t\in \{0,T\}$).

 Moreover,  $\totaldisc$ is a subset of the single valued set $\SV u$ and 
  at every $t\in \SV u\setminus \rmJ$ the curve 
  $u$ is continuous at $t$ (in particular $u$ is continuous at
  $\totaldisc\setminus \rmJ$).
\end{proposition}
\begin{proof}
  If $t\in \totaldisc\setminus \rmJ$,
  then $\rmU(t)$ is a singleton and therefore $\rmU(t)=\{u(t)\}$
 and $t\in \SV u$.
  When $t\in \SV u$, continuity of $u$ \nc
  follows by \eqref{eq:8}.

  If $t\in \totaldisc\cap \rmJ$, then $\rmU(t-)$
  and $\rmU(t+)$ are singletons as well; by \eqref{to-save-2}
  they respectively coincide with the left and right limits $u(t-)$,
  $u(t+)$ of $u$.
\end{proof}
The next result shows that the variational cost $\costname{t}$
\eqref{costo-simpler}
and the jump conditions \eqref{to-save-4} capture a crucial property
of \DSolutions:
a jump may leave $u(t_0-)$ only if $u(t_0-)$ is not a strict local minimizer of $\ene{t_0}\cdot$.
\begin{theorem}[Local minimality forbids jumps]
\label{rmk:localmin-nojump}
Let $u$ be a \DSolution\ and let $t_0\in (0,T]$ be such that
$\rmU(t_0-)$ contains a point
$\bar u$ which is a \emph{strict local minimizer} of the frozen energy
$\ene{t_0}\cdot$.
Then $\rmU(t_0-)=\rmU(t_0)=\rmU(t_0+)=\{\bar u\}$ and
$t_0\in \SV u\setminus \rmJ$. In particular, $u(t_0)=\bar u$ and $u$ is continuous at $t_0$.
\end{theorem}
\begin{proof}
  Being a local minimizer, $\bar u\in\critset(t_0)$.
Since
$\ulim(t_0-)$ is a connected component of $\critset(t_0)$ and the critical
set is clean, $\ene {t_0}\cdot$ is constant on $\ulim(t_0-)$. Hence, by the
strict local minimality of $\bar u$, there exists $r_0>0$ such that
$B_r(\bar u)\cap\ulim(t_0-)=\{\bar u\}$
for every $r\in(0,r_0)$.
Again by the connectedness of $\ulim(t_0-)$, we thus conclude that
$\ulim(t_0-)=\{\bar u\}$.

  If $t_0\in\rmJ$, 
  the jump condition
  \eqref{to-save-4} yields
  $\cost{t_0}{\ulim(t_0-)}{\ulim(t_0+)}=e(t_0-)-e(t_0+)=\mu(\{t_0\})>0$; since this cost equals
  the energy drop $\ene{t_0}{\ulim(t_0-)}-\ene{t_0}{\ulim(t_0+)}$, by Theorem \ref{prop:cost}
  ahead, see (1)
  and (6), there is a nonconstant optimal transition
  $\teta\in\admis{\ulim(t_0-)}{\ulim(t_0+)}{t_0}$ along which $s\mapsto\ene{t_0}{\teta(s)}$ is
  nonincreasing. As $\ulim(t_0-)=\{\bar u\}$, condition (c) of Definition \ref{main-def-1} gives
  $\klimsup_{s\to0}\teta(s)\subset\comp{t_0}{\bar u}=\{\bar u\}$, i.e.\ $\teta(s)\to\bar u$ as
  $s\downarrow0$; since $\teta$ is nonconstant and $\bar u$ is a strict local minimizer, we have
  $\ene{t_0}{\teta(s)}>\ene{t_0}{\bar u}=\ene{t_0}{\teta(0)}$ for some $s>0$, contradicting the
  monotonicity. Hence $t_0\in\SV u\setminus \rmJ$. Continuity follows by
  Proposition \ref{prop:improved-C}.
\end{proof}
\nc

\EEE We can now state our first main result,
 whose proof will be carried out in Section 
\ref{s:4}: 
if the energy satisfies the fibered Sard property \ass1
and the initial datum $u_0$ belongs to $\domainenergy$,
we are able to prove
the existence of a \DSolution\
 $u$
 with $u(0)=u_0$.
Moreover, any pointwise limit of a sequence of singularly perturbed gradient flows along a
vanishing sequence $(\eps_n)_n$ is a
\DSolution. Finally, 
any sequence of the singularly perturbed gradient flow curves 
has a subsequence pointwise converging at every $t\in \totaldisc$. If in addition 
the complement $\nototaldisc$ is countable or negligible, the limit induces
a \DSolution. 
\begin{maintheorem}
\label{mainth:1}
Assume \eqref{coercivita}--\eqref{lambda-convex}  and  
{\em \ass1}, and let $u_0\in \domainenergy$,
 let $(u_\eps)_\eps$ be a family of   solutions to \eqref{Cauchy-gflow}, 
satisfying
\begin{equation}
  \label{energy-convergence-initial}
  u_{\eps}(0) \to
    u_0,\quad  \ene 0{u_{\eps}(0)} \to \ene 0{u_0},
\end{equation}
and let $\mathbf U_
\eps:=\{(t,u_\eps(t)):t\in [0,T]\}$  be their graphs in $[0,T]\times
\domainenergy$. \nc 
Then,
\begin{enumerate}
\item
For every infinitesimal sequence $(\eps_n)_n$ there exists a further
subsequence $(\eps_n')_n$, a limit graph $\mathbf V\subset [0,T]\times
\domainenergy$, and a
  \DSolution\ $u$, such that
  $\mathbf U_{\eps_n'}\karrow \mathbf V$ and  $u(t)\in \mathrm V(t)\subset
  \rmU(t)$ for every $t\in [0,T]$.
  In particular, $u_{\eps_n'}(t)
  \to u(t)$ for every $t\in \SV u$.
\item
  There exist a \DSolution\ $u:[0,T]\to \domainenergy$, with
  $u(0)=u_0$,
   and a decreasing infinitesimal sequence $(\eps_n)_n$ such that
  the corresponding sequence of gradient flows $(u_{\eps_n})_n$ satisfies
  \begin{align}
    \label{eq:52}
    \lim_{n\to\infty}\gdist(u_{\eps_n}(t),\rmU(t))=0\quad\text{for
      every }t\in [0,T].
  \end{align}
  In particular, $u_{\eps_n}(t)
  \to u(t)$ for every $t\in \SV u$.
\item
  If $(\eps_n)_n$ is a decreasing infinitesimal sequence,
    $D\subset [0,T]$ is a Borel set 
    and
    the sequence of gradient flows
    $(u_{\eps_n})_n$ satisfies  the well-preparedness conditions \eqref{energy-convergence-initial}, as well as,  
  \begin{align}
    \label{eq:convergence}
    u_{\eps_n}(t)\to u(t)\text{ for every }t\in D \text{ as } n \to\infty,
  \end{align}
  then there exists a \DSolution\  $\tilde u:[0,T]\to \domainenergy$
  such that $\tilde u(t)=u(t)$ for every $t\in D$.
  In particular, if $D=[0,T]$, then $u$ is a \DSolution.
\item
      If $(\eps_n)_n$ is a decreasing vanishing sequence, $D\subset
      [0,T]$ is a subset such that $D\setminus \totaldisc$ is countable, 
      then every sequence of gradient flows $(u_{\eps_n})_n$
      satisfying \eqref{energy-convergence-initial}
      has a subsequence
      pointwise converging at every $t\in D$.
      
  \item    In particular, if the set of topologically singular times $\nototaldisc$ is countable
   (resp.~negligible), 
      every sequence $(u_{\eps_n})_n$
      has a subsequence pointwise converging everywhere 
      (resp.~almost everywhere)
       to a
      \DSolution.
\end{enumerate}
\end{maintheorem}
\EEE%
\begin{remark}
\label{rmk:salto-istantaneo} \slshape
For the definition of \DSolution, as well as for our existence results, Thm.\ \ref{mainth:1} and Thm.\ \ref{mainth:2} below,
it is sufficient to start from an initial datum $u_0$ with finite
energy.
 If, however, $u_0\not\in \critset(0)$, 
an instantaneous jump occurs at $t=0$.
\end{remark} 

\par
Our next results shows that the notion of \DSolution\ is rich enough to
select branches of isolated local minimizers in $\crit$.
\begin{proposition}[Stable branches]
  \label{prop:stable-branch}
  Let $[a,b]\subset [0,T]$, $r>0$, and let $v:[a,b]\to \domainenergy$ be 
  a continuous path satisfying
  \begin{equation}
    \label{eq:54}
    v(t) \text{ is a strict local minimizer of $\calE_t$,}\quad
    \critset(t)\cap B_r(v(t))=\{v(t)\}
    \quad\text{for every
      $t\in [a,b]$}.
  \end{equation}
  If $u_{\eps_n}(a)\to v(a)$ as $n\to\infty$ then
  $u_{\eps_n}\to v$ uniformly in $[a,b]$. In particular, there
  exists a \DSolution\ $u$ such that $u(t)=v(t)$ for every $t\in [a,b]$.  
\end{proposition}
\begin{proof}
  Applying Theorem \ref{mainth:1}, up to extracting a suitable (not relabeled) subsequence, we can
  suppose that
  $\mathbf U_{\eps_n}\karrow \mathbf V$ and we find a \DSolution\ $u$
  such that
  $u(t)\in \rmV(t)\subset \rmU(t)$.
  
  \noindent\underline{Claim:}
  \emph{$v(t)\in \rmU(t)$ for every $t\in [a,b]$.}
  Indeed, 
  let $D:=\{t\in [a,b]: v(t)\in \rmU(t)\}$. The set 
  $D$ contains $a$ and it is closed (since $v$ is continuous).
  Let $D_0=[a,a']$ be the (closed)
  connected component of $D$ containing $a$ (possibly $a'=a$).
  By Theorem \ref{rmk:localmin-nojump}
  $v(t)=u(t)$ and $\rmU(t)=\{v(t)\}$ for every $t\in D_0$, in
  particular for $t=a'$. In addition $a'\not\in \rmJ$.
  If $a'<b$, by the continuity of $v$ and \eqref{eq:8}
  we can find $a''\in (a',b]$ such that
  $|v(a')-v(t)|<r/2$ and 
  $|v(a')-u(t)|<r/2$ in $[a',a'']$.
  It follows that $u(t)\in B_r(v(t))$ so that
  $u(t)=v(t)$ by \eqref{eq:54}, a contradiction.

  \noindent\underline{Conclusion.}
  By Theorem \ref{rmk:localmin-nojump} $\rmU(t)=\{v(t)\}$ for every
  $t\in [a,b]$ so that the graphs of $u_{\eps_n}$ converge to
  the graph of $v$ in $[a,b]\times \domainenergy$
  in the sense of Kuratovski. The last statement follows by
  Claim (3) of Theorem \ref{mainth:1}. 
\end{proof}
\nc%
\subsection{Balanced Viscosity solutions}
If  we only consider the jump contribution to the measure $\mu$, concentrated on the jump set  of the energy $\jump\sfe $, from \eqref{to-save-eneq} we deduce that 
any \DSolution\ 
satisfies
\begin{equation}
\label{to-save-enineq-BV}
\sum_{r\in \jump\sfe \cap (s,t)} \kern-12pt\cost
r{\ulim(r-)}{\ulim(r+)}
+ \ene t{u(t)} \le \ene s {u(s)} + \int_s^t \pt r{u(r)} \dd r
\quad \text{for all } s,t\in [0,T]\setminus \jump\sfe,\text{ with } s\le t.
\end{equation}
Now, our  second notion of evolution of critical points, namely the concept of 
 \emph{Balanced Viscosity} solution, 
brings into the picture  the additional information that the measure $\mu$  is purely
atomic
and thus it is concentrated on its jump set $\jumpname = \jump \sfe$. 
Therefore, for \BSolutions\ \eqref{to-save-enineq-BV} holds as an equality. 
\begin{maindefinition}[\BSolution]
 \label{def:sols-notions-2}
Assume 
\eqref{coercivita}--\eqref{lambda-convex}
 and that $\mathcal E$ has a clean critical set. 
We call \emph{Balanced Viscosity  solution} to \eqref{limit-problem}  
a \DSolution\  $u$
with the additional property that  \EEE
the measure $\mu$ of \eqref{to-save-eneq} is
concentrated on $\jump \sfe$. 
In this
case there holds
\begin{equation}
\label{eq:19}
\sum_{r\in \jump\sfe \cap (s,t)} \kern-12pt\cost r{\ulim(r-)}{\ulim(r+)}   
+ \ene t{u(t)} = \ene s {u(s)} + \int_s^t \pt r{u(r)} \dd r
\quad \text{for all } s,t\in [0,T]\setminus \jump\sfe,\ s\le t.
\end{equation}
\end{maindefinition}
A first justification of Definition
\ref{def:sols-notions-2}
is provided by
the following result, 
which shows that
every \DSolution\ with bounded variation is also a \BSolution.
Its proof will be a consequence
of the discussion of Section
\ref{s:5}, see Corollary
\ref{cor:null-Cantor}.

\begin{proposition}\label{prop:BV}
  Let $u:[0,T]\to\domainenergy$ be a \DSolution.
  If $u$ has bounded variation then it is a \BSolution.
\end{proposition}
\par With our second main result we show that, if $\crit$ is countably
$\mathscr H^1$-rectifiable
(which implies the $\sigma$-finiteness of 
$\mathscr H^1\res \crit$  and, a fortiori,  \ass1\ and the fact that
$\nototaldisc$ is at most countable),
all \DSolutions\  are indeed \BSolutions. 
\begin{maintheorem}
\label{mainth:2}
In addition to  \eqref{coercivita}--\eqref{lambda-convex},  
assume that
the following 
condition is satisfied:
\begin{equation}
\label{eq:15}
\text{$\crit$ is countably $\mathscr H^1$-rectifiable}
\end{equation}
in the sense of \eqref{rectifiable}. Then
\begin{enumerate}
\item
  Every sequence $(u_{\eps_n})_n$ of
  gradient flows satisfying
  \eqref{energy-convergence-initial}
  for a vanishing $\eps_n\downarrow0$
has subsequence pointwise converging to a
\DSolution.
\item 
Every \DSolution\  $u$
 to \eqref{limit-problem} 
 is also a \BSolution\ and its
 set of discontinuity  $\disc u$ is at most countable. 
 \item 
   In particular, for every $u_0 \in \domainenergy$ there exists a
   Balanced Viscosity
   solution $u$ emanating from $u_0$. \EEE
\end{enumerate}
\end{maintheorem}
 Section \ref{s:5} will be devoted to the proof of the above Theorem.

\EEE

\section{Preliminary results on the energy-dissipation cost(s)}
\label{s:energy_diss_cost}
Throughout this section and the remainder of the paper,  we will
always suppose that
\[
\calE \text{ and $\crit$ comply with 
  \eqref{coercivita}--\eqref{lambda-convex} and with
 \ass1\ (in particular $\cE$ has a clean critical set),}
\]
therefore we will omit to invoke these conditions in the forthcoming
results. 
 \par
In what follows, 
 the symbols $c,\,C, \, C',\hdots$ will be used
to denote a generic positive constant depending on given data, and possibly
varying from line to line.  Furthermore, 
for a given interval $(a,b)\subset \R$, 
we will often simply write  {\em for a.a.\ $ t \in (a,b)$} in place of  {\em for $\mathscr{L}^1$-a.a.\ $t\in (a,b)$}. 
\subsection{Admissible transitions: from curves to graphs}
\label{subsec:admissible}
  With our first result,
  Lemma \ref{le:graph} below,  we will provide a graph characterization of the
admissible transition curves introduced in Definition \ref{main-def-1}. For this, we shall
   rely  on the relations between maps and their (extended) 
graphs  that we expounded at the beginning of Section \ref{ss:2.1}. Indeed,
 in what follows the \emph{graph} viewpoint will be more emphasized,
 in accordance with our approach to the asymptotic analysis for the singularly perturbed gradient flows. Thus, 
   admissible transitions
in the sense of Definition  \ref{main-def-1} will have to be regarded as \emph{selections} of `graph transitions'.  In this connection, observe that the properties of 
the 
function $\mathsf e$
from \eqref{eq:29} below
(not to be confused with the mapping $[0,T]\ni t \mapsto \sfe(t) := \ene t{u(t)}$),
 are \emph{independent} of such selections.
\begin{lemma}[Graph characterization of admissible transitions]
  \label{le:graph}
  Let $t\in [0,T]$ and $\teta\in \admis{U_0}{U_1}t$, for
  $U_0,U_1 \in \newclass t$. 
 Then, there exists $\rho>0$ 
  and a compact fiberwise connected set
  $\bfTheta\subset [0,1]\ti \sublevel\rho$ such that 
  \begin{equation}
    \label{eq:23}
    \Theta(s)=\{\teta(s)\}\quad\text{if }s\in \necri t {\teta},\quad
    \teta(s)\in\Theta(s)\subset \critset(t)\quad\text{if }s\not\in \necri t {\teta},
  \end{equation}
  with $\necri t {\teta}$ from \eqref{necri-teta}. 
  \par
  Conversely, let $\bfTheta$ be a
  \begin{equation}
  \label{basic-properties-graph}
  \text{ compact and fiberwise connected subset of $[0,1]\ti
  \sublevel \rho$ with  projection   $\pi_{[0,1]}(\bfTheta)=[0,1]$,}
  \end{equation}
   and let 
  $
  \necri t{\bfTheta}:= 
  \big\{s\in [0,1]: \Theta(s)\cap \necritset t\neq
  \emptyset\big\}$.   Suppose that 
   $\bfTheta$ enjoys the following property: 
  \begin{equation}
    \label{eq:22}
    s\in \necri t{\bfTheta} 
    \quad\Longrightarrow\quad
    \Theta(s)\text{ is a singleton, which we denote by } \{\teta(s)\}.
  \end{equation}
Then,
  \begin{enumerate}
   \item
    there is a unique map 
    $\mathsf e:[0,1]\to\R$,
    defined by
    \begin{equation}
    \label{eq:29}
    \mathsf e(s)
    :=\ene tx\quad\text{for every }(s,x)\in \bfTheta,
  \end{equation}
  such that  
    $\mathsf e$ is constant
    in each connected component of $[0,1]\setminus
  \necri t {\bf\Theta}$. In particular, 
  $
    \mathsf e(s)
     = \ene t{\tilde\teta(s)}$
  for any selection $[0,1] \ni s \mapsto \tilde{\teta}(s) \in \Theta(s)$;
    \item $  \necri t{\bfTheta}$ is relatively open in $[0,1]$;
     \item
   the map $\teta: \necri t{\bfTheta}\to \Hilbert$ is continuous and  satisfies
  \begin{equation}
  \label{newly-added}
    \klimsup_{s\to r} \teta(s) \subset  \critset(t,\Theta(r)) \quad \text{for all } r \in [0,1]\setminus \necri t {\bfTheta}.
  \end{equation}
\end{enumerate}
  Suppose, in addition, that the function
  $\teta$  from \eqref{eq:22} is locally Lipschitz in $\necri t {\bfTheta}$ and  that
  \begin{equation}
    \label{eq:28}
    \int_{\necri t {\bfTheta}}\minpartial\calE t{\teta(s)}\, \|\teta'(s)\|\,\d s<\infty;
  \end{equation}
then,
    \begin{enumerate} \setcounter{enumi}{3}
    \item
 the function 
    $\mathsf e$ from
 \eqref{eq:29} is continuous. 
  \end{enumerate}
\end{lemma}
Clearly, the function $\teta$ from \eqref{eq:22} can be extended from $\necri t{\bfTheta}$ to the whole of $[0,1]$ by simply choosing $\teta(s)$ as an element of $\Theta(s)$ whenever $s\notin \necri t{\bfTheta}$. By construction,
$\necri t {\teta} = \necri t{\bfTheta}$. 
 \begin{proof}
 In order to prove the first part of the  statement,    we  associate  with $\teta$ the multivalued map
  $\Theta: [0,1] \rightrightarrows H$ defined for all $s\in [0,1]$ by 
\begin{equation}
    \label{eq:21}
    \Theta(s):=
    \begin{cases}
      \{\teta(s)\}&\text{if }s\in \necri t {\teta},\\
      \comp t{\teta(s)}&\text{otherwise,}
    \end{cases}   
    \qquad
     \text{and set }     \bfTheta:=\operatorname{Graph}\Theta\subset [0,1]\ti\Hilbert\,.
      \end{equation}
  Clearly, \eqref{eq:23}  holds. 
  We observe that the  mapping $\Theta$ is upper semicontinuous: it is indeed (single-valued and) Lipschitz continuous  on $\necri t {\teta}$, and it fulfills
  $  \klimsup_{s\to\bar s}\Theta(s)\subset \Theta(\bar s)$ for all $\bar s\in [0,1]\setminus \necri t {\teta}$ in view of property (c) in Definition \ref{main-def-1}. Hence,
  as observed in Sec.\ \ref{subsec:preliminaries}, 
   $\bfTheta$ is a closed subset of $[0,1] \ti \sublevel\rho$, and it is therefore compact. Since  its sections $\Theta(s)$ are connected for all $s\in [0,1]$ again by construction, $\bfTheta$ is fiberwise connected thanks to Lemma \ref{le:connection} ahead. 
  \par
 Conversely,  consider   a compact and fiberwise connected 
  $\bfTheta \subset [0,1]\ti
  \sublevel \rho$
  as in the statement, in particular fulfilling property \eqref{eq:22}. 
   \begin{itemize}
 \item[$\vartriangleright \, (1)$]
  Observe that the function 
    $\mathsf e$ from \eqref{eq:29}  is well defined. Indeed,  if $  s\in \necri t{\bfTheta}$, then $\Theta(s)$ is a singleton; otherwise, since $\Theta(s)$ is connected, 
  it is contained in a connected component of $\critset (t)$, and then it suffices to recall that $\ene t{\cdot}$ is constant on the 
  connected components of $\critset (t)$, cf.\ \eqref{eq:16bis}. Furthermore,
 for any connected component 
  $I'$ of $[0,1]\setminus\necri t{\bfTheta}$
 we have that 
 $\cup_{s\in I'} \Theta(s) = {\bfTheta} \cap (I' \ti H)$ is connected,   since $\bfTheta$ is fiberwise connected.  Therefore,  $\cup_{s\in I'} \Theta(s)$ is contained in a single connected component of 
 $\critset (t)$ and, again, we find that 
    $\mathsf e$ is
 constant on 
  $I'$. 
\item[$\vartriangleright \, (2)$] To show that $  \necri t{\bfTheta}$ is relatively open in $[0,1]$, 
suppose by contradiction that  there exist $\bar s \in  \necri t{\bfTheta}$  and a sequence $(\bar s_n)_n  \subset [0,1] \setminus  \necri t{\bfTheta}$ with $\bar s_n \to \bar s $ as $n\to\infty$. By the upper semicontinuity of $\Theta$ we have that $\klimsup_{n\to\infty}\Theta(\bar{s}_n) \subset 
  \Theta (\bar s) = \{ \teta(\bar s)\}$, and, therefore,  there exists  a sequence $( \teta_{\bar{s}_n} )_{n\in \N}$ such that  $\teta_{\bar{s}_n} \in \Theta (\bar{s}_n)$
  for all $n\in \N$  and, up to a (not relabeled) subsequence, $\teta_{\bar{s}_n} \to \teta(\bar{s})$ as $n\to\infty$. Then, by the closedness property \eqref{closedness} and \eqref{eq:22} we infer that 
  $  \minpartial \calE t{\teta(\bar s)} \leq \liminf_{n\to\infty}  \minpartial \calE t{\teta_{\bar{s}_n}} =0 $, against the fact that $\bar s\in  \necri t{\bfTheta}$.
  \item[$\vartriangleright \, (3)$] 
  Since $\bfTheta$
 is compact, the function $\teta$ from \eqref{eq:22} is continuous on $D(\teta) =  \necri t{\bfTheta}$ thanks to \eqref{fromGR2FUNZ}. 
 Property \eqref{newly-added} follows from 
   \[
   \klimsup_{s\to r} \teta(s) \subset  \klimsup_{s\to r} \Theta(s) \stackrel{(1)}{\subset} \Theta(r) \stackrel{(2)}{\subset} \critset(t,\Theta(r)) \quad \text{for all } r \in [0,1]\setminus \necri t{\bfTheta} ,
  \]
  where (1) is due to  the fact that the 
  multivalued mapping $\Theta$ is upper semicontinuous,
  and (2) to the fact that $\Theta(r)$ is 
  a connected subset of  $\critset(t)$. 
 \item[$\vartriangleright \, (4)$]
 Since
   $ \necri t{\bfTheta} $ is open,  we can write it as the disjoint union of countably many open intervals $G_k = (a_k,b_k)$, $k\in \N$. 
   Since $\teta$ is locally Lipschitz in $\necri t{\bfTheta} $ and \eqref{eq:28} holds, it follows from the chain rule \eqref{ch-rule}
   that
   the function
    $s\mapsto \ene t{\teta(s)}$ is locally absolutely continuous on each of the intervals $G_k$
   and that,  for every
   Borel selection $\xi(s)\in\argminpartial {\cE}t{\teta(s)}$
defined 
for a.a.\  $ s \in G_k$,
there holds   \begin{equation}
   \label{chain-rule-Gk}
   \frac{\dd}{\dd s} \ene t{\teta(s)} = \langle \xi(s), \teta'(s) \rangle \qquad \foraa\, s \in G_k.
   \end{equation}
    In particular, 
   the map 
   $\mathsf e$ is
   continuous
 at every $s\in \necri t{\bfTheta} $.
 \par
 We now have to show that 
   $\mathsf e$ is
continuous at every $s_0 \in [0,1] \setminus \necri t{\bfTheta} $. 
In fact, we shall confine the discussion to the proof of the left continuity of $\mathsf{e}$ at 
$s_0$, since the argument for showing the right continuity is completely analogous. Therefore, we will prove
that $\lim_{s\uparrow s_0} \mathsf{e}(s) = \mathsf{e}(s_0)$ distinguishing the  three  following cases:
\begin{itemize}
\item[(A)] 
    $\exists\, \eta>0 \text{ such that } [s_0-\eta, s_0) \subset \necri t{\bfTheta};$
\item [(B)]
$\forall\, \eta>0,\, \ [s_0-\eta, s_0] \cap \necri t{\bfTheta}  \neq \emptyset \text{ and } [s_0-\eta, s_0] \cap ([0,1]{\setminus} \necri t{\bfTheta} ) \neq \emptyset$;
\item [(C)]
$\exists\, \eta>0 \text{ such that } [s_0-\eta, s_0] \subset [0,1] \setminus \necri t{\bfTheta}$.
\end{itemize}
 \begin{itemize}
    \item[\textbf{Case (A):}] 
    \begin{equation}
    \label{caseA}
    \exists\, \eta>0 \text{ such that } [s_0-\eta, s_0) \subset \necri t{\bfTheta}\,.
    \end{equation}
   First of all, observe that the limit $\lim_{s\uparrow s_0} \ene t{\teta(s)}$ does exist.
    Indeed, it is sufficient to integrate \eqref{chain-rule-Gk} on the interval $[s_0-\eta, s]$ for any $s \in [s_0-\eta, s_0)$, 
    and observe that $\lim_{s\uparrow s_0} \int_{s_0-\eta}^{s}\langle \xi(r), \teta'(r) \rangle \dd r $ exists, thanks to \eqref{eq:28}. 
     To prove that 
    \begin{equation}
    \label{2show-here}
\lim_{s\uparrow s_0} \mathsf{e}(s) =\lim_{s\uparrow s_0} \ene t{\teta(s)} =\mathsf e(s_0) =
   \ene t x 
    \qquad \text{for all } x \in \Theta(s_0),
    \end{equation}
    we further distinguish two cases:
    \begin{itemize}
        \item[\textbf{Case (A1):}]
        $\int_{s_0-\eta}^{s_0} \| \teta'(r)\| \dd r \doteq J<\infty$. 
         The limit
         $\bar \teta:=
         \lim_{s\uparrow s_0}
         \teta(s)$ exists 
         and belongs to $\Theta(s_0)$.
        The chain rule \eqref{ch-rule} applies, yielding that $s\mapsto \ene t{{\teta}(s)}$ is continuous
        in $(s_0-\eta,s_0)$ and 
        \[
        \mathsf e(s_0)= \ene t{\bar \teta}=
        \lim_{s\uparrow s_0} \ene t{\teta(s)} 
=   \lim_{s\uparrow s_0} \mathsf e(s) \,.
        \] 
        
        \item[\textbf{Case (A2):}] 
                        $\int_{s_0-\eta}^{s_0} \| \teta'(r)\| \dd r=\infty$. We reparameterize $\teta$ by its arclength
                        and we end up with a Lipschitz continuous curve $\tilde\teta$, now defined on $[0,\infty)$. From  \eqref{eq:28} we infer that 
                        $\int_0^\infty  \minpartial \calE t{\tilde\teta(z)} \dd z <\infty $, so that there exists a sequence $(z_n)_n$ such that 
                        \begin{equation}
                        \label{characts-zn}
                        z_n\to \infty
                         \quad \text{and}  \quad  \minpartial \calE t{\tilde\teta(z_n)} \to 0 \  \text{  as $n\to\infty$;}
                         \end{equation}
                          let $s_n: = \sigma(z_n)$, so that $s_n \uparrow s_0$.  Since $(\tilde{\teta}(z_n) = \teta(s_n) )_n \subset \sublevel\rho$, up to a subsequence we find that $\tilde{\teta}(z_n)  \to \tilde\teta \in \klimsup_{n\to\infty} \Theta(s_n) \subset \Theta(s_0)$.  By the second of \eqref{characts-zn}, the closedness property \eqref{closedness} applies, yielding that 
                          $
                          \lim_{n\to\infty} \ene t{\teta(s_n)} =       \lim_{n\to\infty} \ene t{\tilde\teta(z_n)} = 
                          \ene t{\tilde\teta} =\mathsf e(s_0)
           $
           where the last equality is  due to \eqref{eq:29}. Then, \eqref{2show-here}  ensues. 
        \end{itemize}
            \item[\textbf{Case (B):}] 
                \begin{equation} 
    \label{caseB}
            \begin{aligned}
            \forall\, \eta>0,\, \ [s_0-\eta, s_0] \cap \necri t{\bfTheta}  \neq \emptyset \text{ and } [s_0-\eta, s_0] \cap ([0,1]{\setminus} \necri t{\bfTheta} ) \neq \emptyset\,.
             \end{aligned}
            \end{equation}
           We will prove that 
            \begin{equation}
            \label{2showCaseB}
            \text{for all }  (s_n)_n \subset [0,1] \text{ with } s_n\uparrow s_0 \text{ we have that } \lim_{n\to\infty} \mathsf e(s_n) = \mathsf e(s_0)
            = \ene tx
\text{ for all } x \in \Theta(s_0).  
            \end{equation}
            It is sufficient
            to prove  property \eqref{2showCaseB} in the case of a sequence $(s_n)_n \subset [0,1]\setminus \necri t{\bfTheta} $ and of a sequence $(s_n)_n \subset \necri t{\bfTheta} $. 
       \par
            Let us first of all examine the case in which  $(s_n)_n \subset [0,1]\setminus \necri t{\bfTheta} $, so that $ \minpartial \calE t{\teta_{s_n}} =0$ for all
            $\teta_{s_n} \in \Theta(s_n)$ and for all 
             $n\in \N$. 
            By the compactness arguments in the previous lines, there exist
         $\bar\teta \in \Theta(s_0)$ and a (not relabeled)
             sequence
             $(\teta_{s_n})_n$, with 
              $\teta_{s_n}  \in \Theta(s_n)$ for all $n\in \N$, such that $\teta_{s_n}\to \bar\teta$ as $n\to\infty$.  Since $ \minpartial \calE t{\teta_{s_n}} =0$, the closedness property \eqref{closedness} yields that 
             $\ene t{\teta_{s_n}}  \to 
             \ene t{\bar\teta} =\mathsf e(s_0)
             $ as $n\to\infty$.
             \par
             Let us now pick a sequence $(s_n)_n \subset \necri t{\bfTheta} $ with $s_n\uparrow s_0$, and let $  b_{k_n} \in [0,1]\setminus \necri t{\bfTheta} $ be the right-end point of the 
             connected component $(a_{k_n}, b_{k_n})$ of $\necri t{\bfTheta} $ that contains $s_n$. As we have just seen, there exists $\teta_{b_{k_n}} \in \Theta(b_{k_n})$ such that   
          \begin{equation}
          \label{added-label-sec3}
           \mathsf e(b_{k_n}) = \ene t{\teta_{b_{k_n}}}  \to \mathsf e(s_0) \quad \text{ as $n\to\infty$.}
           \end{equation}
        In turn, 
  by the discussion in \textbf{Case   A}, 
 we find that        
 \[
\left|\ene t{\teta_{b_{k_n}}}  \EEE  {-} 
           \ene t{\teta(s_n)}
           \right|\leq
        \int_{s_n}^{b_{k_n}}
           \minpartial \calE t{\teta(s)} \|\teta'(s)\| \dd s  \longrightarrow 0 \text{ as } n \to\infty,
             \]
         where the last convergence follows from the fact that  $\sum_{k\in \N} \int_{a_k}^{b_k}  \minpartial \calE t{\teta(s)} \|\teta'(s)\| \dd s<\infty$ by 
             \eqref{eq:28}.
       Combining this with \eqref{added-label-sec3},
       we then conclude that 
             $\mathsf e(s_n)=
           \ene t{\teta(s_n)}
             \to \mathsf e(s_0)$,
             i.e.\ \eqref{2showCaseB}.
    \item[\textbf{Case (C):}] 
    \begin{equation}
    \label{caseC}
    \exists\, \eta>0 \text{ such that } [s_0-\eta, s_0] \subset [0,1] \setminus \necri t{\bfTheta}\,.
    \end{equation}
    Let 
    $
    s_0^-: = \min\{ s<s_0\, : [s,s_0] \subset [0,1] \setminus \necri t{\bfTheta}  \}\,.
    $ 
    Since $\mathsf e$ is constant on each connected component of $[0,1]\setminus \necri t{\bfTheta} $, we conclude that $\mathsf e(s_0^-) = \mathsf e(s_0)$. In turn, the point $s_0$  either satisfies \eqref{caseA} or \eqref{caseB}. Therefore, 
    $\lim_{s\uparrow s_0} \mathsf e(s) = \mathsf e(s_0^-) =\mathsf e(s_0)$ as desired. 
    \end{itemize}
\end{itemize}
This concludes the proof of claim (4) of the statement and, ultimately, of Lemma \ref{le:graph}. 
\end{proof}
\par
 Motivated by Lemma \ref{le:graph} we give the following definition. 
\begin{definition}[Admissible selections in a graph]
\label{def:adm-sel}
Let $\bfTheta\subset [0,1] \ti \sublevel\rho$ fulfill  \eqref{basic-properties-graph}, and let
  $
  \necri t{\bfTheta}$
  fulfill \eqref{eq:22}. 
   We call a selection $[0,1]\ni s \mapsto \teta(s) \in \Theta(s)$ \emph{admissible} if
  $\teta|_{\necri t{\bfTheta}} $  is locally Lipschitz, enjoys the continuity property \eqref{newly-added},  and satisfies estimate \eqref{eq:28}. We denote by 
  $\AS t{\bfTheta}$  the class of admissible selections of $\bfTheta$. 
\end{definition}
 A reparameterization   technique shows that, for any admissible selection, estimate \eqref{eq:28} improves to a pointwise estimate (cf.\ \eqref{normalization} below),
yielding the Lipschitz continuity of the function  $\mathsf{e}$. 
\begin{lemma}
\label{le:reparam}
Let $\bfTheta$ be as in Definition \ref{def:adm-sel}. With any $\teta \in \AS t\bfTheta$ we may associate 
an  increasing   Lipschitz continuous function $\sfs : [0,1]\to [0,1]$ such that the
function $\tilde\teta: [0,1] \to \Hilbert$ defined by $\tilde\teta(r): = \teta(\sfs(r))$ fulfills
$\tilde\teta \in \AS t\bfTheta$, the pointwise estimate 
\begin{equation}
\label{normalization}
 \sfs'(r) +  \minpartial \calE t {\tilde\teta(r)} \| \tilde\teta'(r)\| = 1+ \int_{\necri t{\bfTheta}} \minpartial \calE t {\teta(s)} \| \teta'(s)\|\dd
s  
 \end{equation}
and the identity
\begin{equation}
\label{invariance-integral}
\int_{\necri t{\bfTheta}} \minpartial \calE t {\teta(s)} \| \teta'(s)\|\dd s = \int_{\necri t\bfTheta} \minpartial \calE t {\tilde\teta(r)} \| \tilde\teta'(r)\|\dd r\,.
\end{equation} 
In particular, 
the function $\mathsf e$ 
from \eqref{eq:29} is Lipschitz continuous.
\end{lemma}
\begin{proof}
It is sufficient to introduce the arclength function
\begin{equation}
\label{reparameterization-4-normalization}
 \sfr:[0,1]\to [0,1], \qquad \sfr(s): = \frac{s+\int_{(0,s) \cap \necri t{\bfTheta}}  \minpartial \calE t {\teta(\tau)} \| \teta'(\tau)\|\dd\tau }{1+ \int_{\necri t{\bfTheta}} \minpartial \calE t {\teta(\tau)} \| \teta'(\tau)\|\dd \tau}.
 \end{equation}
 Clearly, $\sfr$ is increasing and surjective. We then
  set  $ \sfs = \sfr^{-1}: [0,1]\to [0,1]$. 
 A direct calculation shows \eqref{normalization} and \eqref{invariance-integral}; clearly, $\tilde{\teta}|_{\necri t{\bfTheta}} $ is still locally Lipschitz, 
 and fulfills  $ \klimsup_{r\to \bar{r}} \tilde\teta(r) \subset  \critset(t,\Theta(\bar r))$ for all  $\bar{r} \in [0,1]\setminus \necri t {\bfTheta}$ (so that, ultimately, $\tilde\teta \in \AS t\bfTheta$). 
  For this, it is sufficient to 
 observe that  $\bar{r} \in [0,1]\setminus \necri t {\bfTheta}$ if and only if $\sfs(\bar{r}) \in  [0,1]\setminus \necri t {\bfTheta}$, so that 
 $ \klimsup_{r\to \bar{r}} \tilde\teta(r) = \klimsup_{s \to \sfs(\bar r)} \teta(s) \subset  
  \critset(t,\teta(\sfs (\bar r))) = \critset (t,\tilde \teta(\bar r))$. 
 \par
 Finally, since  $\mathsf{e}(r) = \ene t{\tilde\teta(r)}$ for every $r\in [0,1]$, combining  the pointwise estimate from \eqref{normalization} and the chain-rule estimate \eqref{eq:17}  we conclude that 
$\mathsf{e}$ is Lipschitz continuous, too. 
\end{proof}

\par
We are now in a position to give the 
 `graph counterpart' to Definition \ref{main-def-1}.   Essentially, an admissible graph transition between two sets $U_0,\, U_1 \in \newclass t$ 
 (see \eqref{new-class})
 is a graph $\bfTheta$,  `connecting' $U_0$ to $U_1$,  
 fulfilling \eqref{basic-properties-graph} and  \eqref{eq:22},
  and such that   $\AS t\bfTheta \neq \emptyset $. 
  A fortiori, Lemma \ref{le:reparam}
guarantees that, up to a reparameterization, for  any element $\teta \in \AS t\bfTheta $ the quantity $\minpartial \calE t {\teta} \| \teta'\|$
 is uniformly bounded. 
We fix these properties in the following definition. \EEE
\begin{definition}[Graph transitions]
  \label{def:graph-trans}
  Let $t\in [0,T]$
  and $U_0,U_1 \in \newclass t $.
  An admissible \emph{graph} transition
  between $U_0$ and $U_1$ is 
  a compact and fiberwise connected set  $ \bfTheta \subset  [0,1]  \ti
  \sublevel \rho$ fulfilling  \eqref{basic-properties-graph} and  \eqref{eq:22},
  such that
  \begin{enumerate}[label=\alph*), nolistsep]
  \item  $\Theta(0) \subset U_0 $
  and $\Theta(1) \subset U_1 $;
    \item 
    $\bfTheta \cap \necrit$ 
      coincides with the graph of an admissible selection $\teta \in \AS t\bfTheta$;
    \item There exists a constant $L>0$ such that 
      $\minpartial\calE t{\vartheta(s)}\,\|\vartheta'(s)\|\le L$ for
      $\mathscr L^1$-a.e.~$s\in \necri t{\bfTheta}$.
    \item 
    The map $\mathsf e: [0,1] \to \R$
    defined by \eqref{eq:29}  is
 Lipschitz.
    \end{enumerate}
    We denote by 
 $\AG t{U_0}{U_1} $  the set of admissible graph transitions between $U_0$ and $U_1$.
\end{definition}
In fact, in view of the chain-rule estimate from Lemma \ref{rmk:admiss-curves}, in the above definition we could claim that 
the map 
$\mathsf e$
has Lipschitz constant (less or equal than) $L$. 
\subsection{A crucial compactness result}
\label{ss:3.2-added}
The main result of this section, Theorem \ref{le:main} ahead, relies on the key estimates provided by the following result where, in \eqref{eq:24-bis}, we
will consider the distance between a point and a set induced by the norm on $\Hilbert$. 
\begin{lemma}
\label{le:crucial-estimates}
  Let $\rho,\, \nu,\, C>0$ and  $K$,  a compact subset of $
  \necrit[\rho]: = 
  \Sublevel\rho\setminus \crit$, be given. Let $\gamma = (\sft,\teta): (\alpha,\beta) \to \necrit[\rho] $ be a locally Lipschitz curve
   with nondecreasing
  first component $\sft$,
   and 
  satisfying 
  \begin{subequations}
  \label{eq:24-hyp}
   \begin{align}
    \label{eq:24}
    &
    \sft'(s)
     +\minpartial \calE{\sft(s)}{\teta(s)}\,\|\teta'(s)\|\le C \ 
    \text{a.e.~in } (\alpha,\beta), \quad 
    (\sft(s_0),\teta(s_0))\in K \text{ for some }s_0\in (\alpha,\beta),
    \\
    & 
    \label{eq:24-bis}
    \sup_{s\in (\alpha, s_0)} d(\gamma(s), K) \geq \nu \qquad \left(\text{respectively, }  \sup_{s\in (s_0,\beta)} d(\gamma(s), K) \geq \nu \right).
  \end{align}
  \end{subequations}
  Then, there exist $\delta,\, \eta,\, \eta'>0$, only depending on  $\rho,\, \nu,\, C>0$, and  $K$ (and not on $s_0, \, \alpha,\, \beta,\, \gamma$),  such that 
    \begin{subequations}
  \label{eq:24-thesis}
  \begin{align}
  & 
  \begin{aligned}
    \label{eq:24-thesis-1}
    &
  \alpha< s_0-\delta \text{ and } \minpartial \calE {\sft(s)}{\teta(s)} \geq \eta \text{ for every } s \in [s_0-\delta, s_0]  
  \\
  &
  \qquad 
   \left(\text{resp., } \beta>s_0+\delta \text{ and }  \minpartial \calE {\sft(s)}{\teta(s)} \geq \eta \text{ for every } s \in [s_0,s_0+\delta]\right),
   \end{aligned}
  \\
  & 
    \begin{aligned}
    \label{eq:24-thesis-2}
    &
  \int_{\alpha}^{s_0} \left(\sft'(s){+}\minpartial \calE {\sft(s)}{\teta(s)}\|\teta'(s)\| \right) \dd s \geq \eta' 
  \\
  & \qquad  \left(\text{resp., }  \int_{s_0}^\beta \left(\sft'(s){+}\minpartial \calE {\sft(s)}{\teta(s)}\|\teta'(s)\| \right) \dd s \geq  \eta' \right)\,.
  \end{aligned}
 \end{align}
  \end{subequations}
  In particular, if both conditions in \eqref{eq:24-bis} hold, we have that 
  \begin{equation}
  \label{Lip-teta}
  \|\teta'(s)\|\leq \frac C\eta \qquad \text{ for almost all $s\in (s_0-\delta,s_0+\delta)$.}
  \end{equation} 
  \end{lemma}
  \begin{proof}
  Clearly, it is sufficient to discuss the case in which one of the two conditions in \eqref{eq:24-bis}
(for instance, the first one) holds. 
  By the lower semicontinuity of the functional $[0,T]\ti \subl\rho \ni (t,x) \mapsto \minpartial \calE t{x}$   (due the closedness property \eqref{closedness}) and
  by
   the compactness of $K$, it is not difficult to check that 
there exists $\mu\in (0,\nu)$ such that the compact set
\begin{equation}
\label{Qset}
Q_\mu[K]: = \big\{(t',x')\in [0,T]\ti
\sublevel\rho:|t'-t|+|x'-x|\le \mu\text{ for some }(t,x)\in
K\big\}
\end{equation}
satisfies 
\[
Q_\mu[K] \cap \crit = \emptyset, \qquad 
\text{with } m: =1  \wedge \min\{ \minpartial \calE {t}{x}  {:} \ (t,x)\in Q_\mu[K]\}   >0\,.
\]
Let $I_\mu  = \{ s \in (\alpha, s_0]\, : \ d(\gamma(s), K) <\mu \}$ and let $(\alpha',s_0]$ be the connected component of $I_\mu$ containing 
$s_0$. 
Obviously,
 we have that $\alpha<\alpha'$. Moreover, for every $s\in (\alpha',s_0]$ there holds 
$\minpartial \calE{\sft(s)}{\teta(s)}\geq m$, so that, in view of  \eqref{eq:24} we have 
\begin{equation}
\label{Lip-est-l3.5}
\|\teta'(s)\|\leq \frac Cm, \qquad  \sft'(s) +
\|\teta'(s)\|\leq \frac {2C}m \qquad \foraa\, s \in (\alpha',s_0]
\end{equation}
(the last estimate also due to the fact that $m<1$). 
Since $\|\gamma(\alpha'){-}\gamma(s_0)\| \geq d(\gamma(\alpha'), K) =\mu$, we deduce that 
$(s_0{-}\alpha') > \frac{\mu m}{4C} \doteq \delta$
(otherwise, \eqref{Lip-est-l3.5} would give $\|\gamma(\alpha'){-}\gamma(s_0)\|  \leq  \frac {2C}m \delta  = \frac\mu2$, 
a contradiction). Furthermore,
\[
\int_{\alpha'}^{s_0} \left(\sft'(s){+} \minpartial \calE{\sft(s)}{\teta(s)} \|\teta'(s)\| \right) \dd s  \geq
  m \|\gamma(s_0){-} \gamma(\alpha')\| \geq m\mu \,.
\]
All in all, \eqref{eq:24-thesis} hold with $\delta: = \frac{\mu m}{4C}$, $\eta: = m$, and $\eta':= m\mu$.
\par Finally, \eqref{Lip-teta} clearly follows from combining estimate \eqref{eq:24}
  with the fact that $\minpartial \calE {\sft(s)}{\teta(s)} \geq \eta$  for every $ s \in [s_0-\delta, s_0+\delta]  $. 
  \end{proof}
\par
The following compactness result will play a crucial role both for the study of the  properties of the cost $\mathsf{c}$, and for the proof of  Theorem \ref{mainth:1} carried out in Section 
\ref{s:4} (cf.\ in particular Lemma \ref{l:endiss-cost-c} and Proposition \ref{th:1}). In view of these applications, we shall consider a sequence of curves
$(\sft_n,\teta_n)_n$
with values in the  phase space
 $[0,T] \ti \Hilbert$; in accordance with our `graph approach', emphasis will be given to the compactness of the graphs $(\bfTheta_n)_n$ associated with the curves $(\teta_n)_n$. 
\begin{theorem}
  \label{le:main}
  Let $t\in [0,T]$
 and $U_0,U_1 \in \newclass t$. 
  Let $\teta_n:[0,1]\to \Hilbert$,
  $\sft_n:[0,1]\to [0,T]$ be two sequences of maps, and let
  $\bfTheta_n\subset [0,1]\ti\sublevel\rho$ be a
  family of fiberwise connected compact sets
  satisfying the following properties,  with  constants $L, \, L'>0$
   independent of $n\in \N$: 
  \begin{enumerate}[label=\roman*)]
  \item $\teta_n(r)\in \Theta_n(r)$ for
    every $r\in [0,1]$ and 
    $\Theta_n(r)=\{\teta_n(r)\}$ 
  for   every $r\in G_n$, with 
  \begin{equation}
  \label{setGn}
  G_n: = 
   \{ r \in [0,1]\, : \ \bfTheta_n(r) \cap \necritset {\sft_n(r)} \neq \emptyset \};
   \end{equation}
  \item the functions $\sft_n$ are Lipschitz and  nondecreasing, 
 the curves $\teta_n$ are locally Lipschitz in  the (relatively open) sets $G_n$
    and there holds
    \begin{equation}
    \label{Lipschitz-est-on-Gn}
 \sft_n'(r) +    \minpartial\calE{\sft_n(r)}{\teta_n(r)} \, \|\teta_n'(r)\|\le L \qquad \foraa\, r \in G_n;
  \end{equation}
  \item $r\mapsto \ene{\sft_n(r)}{\teta_n(r)}$ is Lipschitz in $[0,1]$ with 
  \begin{equation}
  \label{Lip-est-ENE}
  \frac{\dd}{\dd r} \ene{\sft_n(r)}{\teta_n(r)} \leq L'  \qquad \foraa\, r \in (0,1);
  \end{equation}
  \item  $\klimsup_{n\to\infty}\Theta_n(i)\subset U_i$, $i=0,1$ and 
    $\lim_{n\to\infty}\sup_{r\in [0,1]}|\sft_n(r)-t|=0$.    
  \end{enumerate}
   Furthermore, consider the `slope functions'  
  \begin{equation}
\label{slopes}
\slo_n: [0,1]\to [0,\infty], \qquad \slo_n(r): = \begin{cases}
\minpartial \calE{\sft_n(r)}{\teta_n(r)} & \text{if } r \in G_n,
\\
0 & \text{otherwise}.
\end{cases}
\end{equation}
\par
  Then, there exist a (not relabeled) subsequence of $(\Theta_n)_n$, a compact limit set
  $\bfTheta\subset [0,1]\ti \sublevel\rho$ fulfilling 
  properties
  \eqref{basic-properties-graph}
  and 
  \eqref{eq:22}
  (i.e., $ s\in \necri t{\bfTheta}$ implies that $  \Theta(s)$ is a singleton), 
   a lower semicontinuous function
   $\slo: [0,1]\to [0,\infty]$,  
   a map   $\teta:[0,1]\to \sublevel\rho$, and $L>0$
  satisfying  
  \begin{enumerate}[label=\arabic*)]
  \item the functions $(S_n )_n$  $\Gamma$-converge to $S$; 
  \item $\bfTheta_{n}\karrow \bfTheta$ as $n\to \infty$;
  \item $\bfTheta$ is an admissible graph transition between $U_0$
    and $U_1$
    in the 
    sense of Definition \ref{def:graph-trans}
    and $\teta$
    belongs to the class of admissible selections $\AS t{\bf\Theta}$
    (hence, a fortiori, 
     $\teta\in \admis{U_0}{U_1}t$); 
  \item $\teta_{n}\to \teta$ uniformly on each compact subset of
    $\necri t {\teta}=\necri t{\bfTheta}$;
  \item there holds
    \begin{equation}
      \label{eq:30}
      \liminf_{n\to\infty}\int_{G_n}\minpartial\calE{\sft_{n}(r)}{\teta_{n}(r)}
      \, \|\teta_{n}'(r)\|\,\d r\ge 
      \int_{\necri t {\teta}}\minpartial\calE{t}{\teta(r)} \,
      \|\teta'(r)\|\,\d r;
    \end{equation}
    \item $\minpartial\calE{t}{\teta(r)}    \|\teta'(r)\| \leq L$  for almost all $r\in (0,1)$; 
    \item 
    the function $\mathsf e$  from \eqref{eq:29} is Lipschitz
    in $[0,1]$
    with Lipschitz constant less or equal than $ L$,  and 
    \begin{equation}
    \label{e:precise-statement}
    \ene{\sft_{n}(\cdot)}{\teta_{n}(\cdot)}
    \to \mathsf e
    \qquad \text{uniformly in } F_0: = \overline{\{ r\in [0,1]\, : \slo(r)<\infty\}}\,.
    \end{equation}
  \end{enumerate}
\end{theorem}

\subsection{Proof of Theorem \ref{le:main}}
The \emph{proof} is split in various steps.  
\paragraph{\bf Step $0$: compactness of the `slopes'} 
Due to the closedness \eqref{closedness}
of the Fr\'echet subdifferential, 
 and to the continuity of $\sft_n$ and $\teta_n$,  the functions $S_n$ are lower semicontinuous. By compactness of $\Gamma$-convergence, 
there exist a (not relabeled) subsequence and a  lower semicontinuous function $\slo: [0,1]\to [0,\infty]$ such that 
$\Gamma-\lim_{n\to\infty} \slo_n = \slo$. We also introduce the \emph{open} set
\begin{equation}
\label{setG-proof}
G: = \{ r \in [0,1]\, : \ \slo(r)>0\}.
\end{equation}
We will use that, for every sequence $(r_n)_n\subset [0,1]$ 
\begin{equation}
\label{nice-lsc-slopes}
\left( r_n\to r,  \ \ \teta_n(r_n) \to \teta\right) \ \Longrightarrow \ S(r) \geq \minpartial \calE t\teta
\end{equation}
(also due to  the fact that
$|\sft_n(r_n){-}t |\to 0$ as $n\to\infty$ from hypothesis (iv)). 
From \eqref{nice-lsc-slopes}
it follows in particular that,  if $r\notin G$, then $\teta \in \critset (t)$. 
\medskip

\paragraph{\bf Step $1$: compactness of the graphs $(\bfTheta_n)_n$}  Since the sets $(\bfTheta_n)_n$ are 
all contained
in the compact set $[0,1]\ti \subl \rho$, 
 it follows from the
 Blaschke 
 Theorem 
 that, from the subsequence in Step $0$ we can extract a further, not relabeled subsequence
 $(\bfTheta_n)_n$ converging in the sense of Kuratowski to  a compact subset $\bfTheta \subset  [0,1]\ti \subl \rho$. 
 Thanks to 
 Lemma \ref{le:connection}
 in the Appendix,
  $\bfTheta$ is fiberwise connected. Since the projection operator 
 $\pi_{[0,1]}: \mathfrak{K}_{\mathrm{c}}(H) \to \mathfrak{K}_{\mathrm{c}}([0,1])$ is continuous, we have  that $\pi_{[0,1]}(\bfTheta) = [0,1]$. Hence, $\bfTheta$ enjoys  properties
 \eqref{basic-properties-graph}. Moreover,   since $\klimsup_{n\to\infty}\Theta_n(i)\subset U_i$, $i=0,1$, 
 by the general properties of Kuratowski convergence we can conclude  that $\Theta(i) \cap U_i \neq \emptyset$ for $i=0,1$. 

 \medskip
 
 \paragraph{\bf Step $2$: compactness of the curves $(\teta_n)_n$} We are now going to show that
 \begin{itemize}
 \item[(2.1)]
    $\necri t{\bfTheta}$ is a subset of the set $G$ from  \eqref{setG-proof},
     and for every $ s\in \necri t{\bfTheta} $
 the section
   $ \Theta(s)$  is a singleton  (cf.\ property  \eqref{eq:22});
 \item[(2.2)]
 there exists a map $\teta : \necri t{\bfTheta} \to H$ such that 
$\teta_n\to \teta$ uniformly on the compact subsets of $\necri t{\bfTheta}$;
\item[(2.3)]
there holds 
 \begin{equation}
 \label{initial-condition-graph}
 \Theta(0) \subset U_0, \qquad  \Theta(1)\subset U_1\,.
 \end{equation}
 \end{itemize}
 \par
 Preliminarily, 
  let  us consider a closed sub-interval $[\alpha,\beta] \subset G$,
  so that  
 $m: = \min_{r\in[\alpha,\beta]} \slo(r)>0$. By $\Gamma$-convergence, we have that 
 $\lim_{n\to\infty} \min_{r\in [\alpha,\beta]} \slo_n(r) = m $, so that there exists $\bar n \in \N$ such that 
 $\min_{r\in [\alpha,\beta]} \slo_n(r) \geq \frac m2$ for every $n\geq \bar n$. It follows that 
 $[\alpha,\beta] \subset G_n$  for every $n\geq \bar n$; further, in view of \eqref{Lipschitz-est-on-Gn}, the curves $(\teta_n)_{n \geq \bar n}$ are equi-Lipschitz
 on $[\alpha,\beta]$. 
 We then apply Lemma \ref{l:compactness-graphs} 
 in the Appendix
 and deduce that $\bfTheta \cap ([\alpha,\beta]\ti\Hilbert)$ coincides with the graph of a Lipschitz function $\teta$, and that (up to a further subsequence),
   $\teta_n\to \teta$ uniformly in $[\alpha,\beta]$. 
 \par
In order to check that 
$\bfTheta$
fulfills \eqref{eq:22}, let us pick  $r\in \necri t{\bfTheta}$:
 then,  there exists $\teta_r \in \Theta(r)$ such that $\minpartial\calE t{\teta_r}>0$.
 Since the set   $\crit\subset [0,T]\ti \sublevel\rho$   is closed
 by \eqref{closedness}, we can find $\eps>0$ such that the compact set
  $K: = ([r-\eps,r+\eps]{\times}\overline{B_\eps(\teta_r)}) \cap([0,T]{\times} \sublevel\rho)$   has empty intersection with $\crit$
  (we are supposing here $r\in (0,1)$; an analogous property holds, 
  with obvious modifications for $r=0$
 and $r=1$). 
For sufficiently large $n \geq \bar n$ there exists $(r_n,\Theta_{r_n}) \in \bfTheta_n \cap K$, with $(r_n,\Theta_{r_n}) \to (r,\teta_r)$ as $n\to\infty$.
Then, we can apply Lemma \ref{le:crucial-estimates}. More precisely,
  let $(\alpha_n,\beta_n) $ be the connected component of 
  $G_n \cap (0,1)$ containing $r_n$ and let $\nu>0$ fulfill
  $Q_{2\nu}[K] \cap \crit =\emptyset$ (with $Q_{2\nu}[K]$ from \eqref{Qset}).
  Let  $\gamma_n = (\sft_n,\teta_n)$. 
   If $\sup_{s\in (\alpha_n,r_n)} d(\gamma_n(s), K)<\nu$,  then $\alpha_n \in G_n$. Hence, $\alpha_n=0$. Otherwise, 
$  \sup_{s\in (\alpha_n,r_n)} d(\gamma_n(s), K)\geq \nu$. Hence, the curves $\gamma_n$
 fulfill conditions \eqref{eq:24-hyp}, and Lemma \ref{le:crucial-estimates} applies, giving that $\alpha_n< r_n-\delta$ with $\delta$ as in \eqref{eq:24-thesis-1}. With a similar argument
 we either deduce that 
$\beta_n=1$, or that $\beta_n>r_n+\delta$. In any case, we have that $[r_n-\delta, r_n+\delta] \subset G_n$ for some $\delta>0$, and there exists $\eta>0$ such that 
\begin{equation}
\label{slope-big-on-interval}
 \minpartial \calE {\sft_n(r)}{\teta_n(r)} \geq \eta \quad \text{for every } r \in [r_n-\delta, r_n+\delta]\,.
\end{equation}
Choose now $0<\delta'<\delta$ such that $|r-r_n| \leq \delta-\delta'$, so that 
$[r-\delta',r+\delta'] \subset[r_n-\delta, r_n+\delta] \subset G_n$.
 By estimate \eqref{Lip-teta}, the curves $\teta_n|_{[r-\delta', r+\delta']}$ are equi-Lipschitz, so that there exists a Lipschitz function $\teta: [r-\delta',r+\delta'] \to \sublevel\rho $ such that
 (up to a further subsequence),
  $\teta_n \to \teta$ in 
$  [r-\delta',r+\delta'], $ 
 with $\Theta(s) = \{\teta(s)\}$ for every $s\in [r-\delta',r+\delta']$. In particular, for every $r\in  \necri t{\bf\Theta}$ the section $\Theta(r)$ is reduced to a singleton, whence 
 \eqref{eq:22}. Moreover, recalling \eqref{nice-lsc-slopes} we have
 $\slo(r) \geq  \minpartial \calE t{\teta(r)}>0$ for every $r\in \necri t{\bfTheta}$, so that 
  \begin{equation}
 \label{inclusion-between-G}
 \necri t {\bfTheta} \subset G.
 \end{equation}
 In fact,   \eqref{nice-lsc-slopes}  and \eqref{slope-big-on-interval} respectively  yield that 
 $\slo(s)\geq \minpartial \calE t{\teta(s)} $ and $\slo(s) >0$ for  every $s\in[r-\delta',r+\delta']$. 

\par
In order to show the first inclusion in  
\eqref{initial-condition-graph} (the argument for the second inclusion being perfectly analogous), 
we will rely on the previously proved property that $\Theta(0) \cap U_0 \neq\emptyset$   and 
 distinguish two cases: (1)  if $U_0=\{u_0\}$ with $u_0 \in \necritset t $, then
 $0 \in \necri t {\bfTheta}$ and then, by Claim (2.1) above we conclude that $\Theta(0)$ is a singleton. Therefore, 
 $\Theta(0) = \{ u_0\}$; (2) If $U_0$ is a connected component of $\critset(t)$, combining the facts that
  $\Theta(0) \cap U_0 \neq\emptyset$, that $\Theta(0)$ is connected (as $\bfTheta$ is fiberwise connected), and that $\Theta(0)\subset \critset (t)$
  (otherwise, $0\in \necri t{\bfTheta}$ and $\Theta(0) \cap \critset(t) = \emptyset)$, we conclude that $\Theta(0)\subset U_0$.
  This shows Claim (2.3).
 \medskip

 \paragraph{\bf Step $3$: lower semicontinuity of the energy-dissipation integrals} We are going to prove that 
 \begin{itemize}
 \item[(3.1)]
the $\liminf$ estimate \eqref{eq:30} holds, and in particular $\teta  \in \AS t{\bfTheta}$;
 \item[(3.2)]
$\teta$ fulfills   $\minpartial\calE t{\vartheta(s)}\,\|\vartheta'(s)\|\le L$ for
      $\mathscr L^1$-a.e.~$s\in \necri t{\bfTheta}$.
      \end{itemize}
      \noindent 
       To  show Claim (3.1), we  preliminarily fix an arbitrary compact interval $[\alpha,\beta] \subset \necri t {\teta}$.
       Since 
$\teta_n\to \teta$ uniformly on $[\alpha,\beta]$, we have that 
  \begin{equation}
  \label{liminf-slopes}
  \liminf_{n\to\infty} \minpartial \calE{\sft_n(r)}{\teta_n(r)} \geq \minpartial \calE t{\teta(r)} \qquad \text{for all } r \in [\alpha,\beta].
  \end{equation}
  Furthermore, since $\teta_n' \weaksto \teta'$ in $L^{\infty} (\alpha,\beta;H)$,
  up to a further subsequence we have
  in particular
   that $\|\teta_n'\| \weakto\Lambda$ in $L^\infty(\alpha,\beta)$ for some positive function $\Lambda$ such that 
  $\Lambda(r) \geq \|\teta'(r)\|$ for a.a.\ $r\in (\alpha,\beta)$. We are then in a position to apply Lemma \ref{le:mrs-mjm} ahead  and conclude that
  \begin{equation}
  \label{liminf-alpha-beta}
  \liminf_{n\to\infty} \int_\alpha^\beta  \minpartial \calE{\sft_n(r)}{\teta_n(r)} \|\teta_n'(r)\| \dd r \geq \int_\alpha^\beta  \minpartial \calE t{\teta(r)} \|\teta'(r)\| \dd r\,.
  \end{equation}
 In order to deduce  \eqref{eq:30}, we now observe that the open set $\necri t {\teta}$ can be written as the union of countably many intervals $I_k$, and that
each interval  $I_k$ is given by the increasing union of compact intervals $([\alpha_m^k,\beta_m^k])_{m\in \N}$. 
 Since the family $([\alpha_m^k,\beta_m^k])_{m\in \N}$
 is increasing, it is not difficult to realize that, for $n$ sufficiently big, all the intervals $[\alpha_m^k,\beta_m^k]$ are contained in
  $G_n$. From \eqref{liminf-alpha-beta} we find that 
 \[
 \begin{aligned}
\sum_{k=1}^j \int_{\alpha_m^k}^{\beta_m^k}  \minpartial \calE t{\teta(r)} \|\teta'(r)\| \dd r &  \leq   \liminf_{n\to\infty}
\sum_{k=1}^j    \int_{\alpha_m^k}^{\beta_m^k}   \minpartial \calE{\sft_n(r)}{\teta_n(r)} \|\teta_n'(r)\| \dd r
\\  & \leq \liminf_{n\to\infty} \int_{G_n}  \minpartial \calE{\sft_n(r)}{\teta_n(r)} \|\teta_n'(r)\| \dd r
\end{aligned}
 \]
 for every $m\in \N$ and $j\geq 1$. Passing to the limit, first as $m\uparrow \infty$ and then as $j\uparrow \infty$, we conclude \eqref{eq:30}.
\par
  Combining \eqref{liminf-alpha-beta}  with \eqref{Lipschitz-est-on-Gn} we obtain
  \[
   \int_\alpha^\beta  \minpartial \calE t{\teta(r)} \|\teta'(r)\| \dd r \leq L (\beta{-}
 \alpha)  \quad \text{for any compact interval } [\alpha,\beta] \subset  \necri t {\teta}\,.
  \]
 Hence, we deduce   that at any $r_0 \in \necri t {\teta}$
  there holds, for $h>0$ sufficiently small,
  \[
  \int_{r_0-\tfrac{h}2}^{r_0+\tfrac{h}2}  \minpartial \calE t{\teta(r)} \|\teta'(r)\|  \dd r \leq  L h\,.
  \]
  Then, letting $h\down 0$ we conclude that $\minpartial \calE t{\teta(r_0)} \|\teta'(r_0)\| \leq L$  at any Lebesgue point $r_0 $ of the map $\necri t {\teta} \ni r\mapsto   \minpartial \calE t{\teta(r)} \|\teta'(r)\| $, whence Claim (3.2). 
         \medskip

     \paragraph{\bf Step $4$:   conclusion of the proof} It remains to show that 
     \begin{itemize}
      \item[(4.1)]
$\mathsf e_n(\cdot): = \ene{\sft_{n}(\cdot)}{\teta_{n}(\cdot)}\to \mathsf e$, with $\mathsf e$ 
      from \eqref{eq:29},  uniformly in the set   $F_0$ from \eqref{e:precise-statement};
 \item[(4.2)]
the function
$\mathsf e$ is
$L$-Lipschitz.
 \end{itemize}
 In what follows, we will  use the fact that
 \begin{equation}
 \label{extensive-info}
\left(  (t_n,\teta_n) \to (t,\teta)  \text{ and } \minpartial \calE{t_n}{\teta_n} \leq C \right) \  \ \Longrightarrow \ \ \ene{t_n}{\teta_n}\to \ene t \teta
 \end{equation}
 by  the closedness property \eqref{closedness}.
 First of all, observe that,
by \eqref{Lip-est-ENE} and the Ascoli-Arzel\`a Theorem, 
  up to a (not relabeled) subsequence the functions
$(\mathsf e_n)_n$ uniformly
  converge as $n\to\infty$ 
to some 
function $\tilde{\mathsf e} \in \mathrm{Lip}([0,1])$.
\par
We will now show that 
\begin{equation}
\label{claim-1-energy}
\begin{gathered}
\text{for every } [\alpha,\beta] \subset G \text{ the function }   
\mathsf e\text{ is Lipschitz on }
[\alpha,\beta]
\\
\text{ and coincides with }
\tilde{\mathsf e}\text{ on }
F_{[\alpha,\beta]} : = \overline{\{ r \in [\alpha,\beta]\, : \ \slo(r)<\infty\}}\,.
\end{gathered}
\end{equation}
Now, as  shown in Step $2$, for any $[\alpha,\beta]\subset G$ we have that $\teta_n \to \teta$ uniformly in $[\alpha,\beta]$.
Recalling that $[\alpha,\beta] \subset G_n$ for $n$ sufficiently big,  again by Lemma \ref{le:mrs-mjm} 
ahead
we have   that 
\[
\int_\alpha^\beta \slo(r) \|\teta'(r)\| \dd r \leq 
\liminf_{n\to\infty} \int_\alpha^\beta \minpartial \calE{\sft_n(r)}{\teta_n(r)} \|\teta_n'(r)\| \dd r \leq L\,.
\]
 Moreover,  by the definition of $\slo$
for every $r\in (\alpha, \beta)$ there exists a sequence $r_n\to r$ such that 
$\minpartial \calE{\sft_n(r_n)}{\teta_n(r_n)}\to \slo(r)$. 
If $\slo(r)<\infty$, we deduce that 
\[
\tilde{\mathsf e}(r) = \lim_{n\to\infty} \mathsf e_n(r_n)
=\lim_{n\to\infty} \ene {\sft_n(r_n)}{\teta_n(r_n)} \stackrel{(*)}{=} \ene t{\teta(r)}
= \mathsf e(r),
\]
where 
$(*)$
follows from
 \eqref{extensive-info}. We have thus proved that
 $\tilde{\mathsf e}(r) = \mathsf e(r)$ 
 on the set $\{ r \in [\alpha,\beta]\, : \ \slo(r)<\infty\}$. 
 Now,
since the map $r\mapsto \ene t{\teta(r)}$ is absolutely continuous on $[\alpha,\beta]$ by the chain rule \eqref{ch-rule}
(which holds, since $\int_\alpha^\beta  \minpartial \calE t{\teta(r)}\, \|\teta'(r)\| \dd r \leq  \int_\alpha^\beta \slo(r) \|\teta'(r)\| \dd r \leq L$),
 we also deduce that
$\tilde{\mathsf e}(r) = \mathsf e(r) = \ene t {\teta(r)}$
 on the set $F_{[\alpha,\beta]}$.  Now, 
 if $(r_0,r_1)$ is a connected component of $(\alpha,\beta)\setminus F_{[\alpha,\beta]}$,
 we have that $\slo(r)\equiv\infty$ in $(r_0,r_1)$  
 and $\teta$ is constant on 
 $(r_0,r_1)$, since $\int_\alpha^\beta \slo(r) \|\teta'(r)\| \dd r<\infty$
 and therefore $\teta'(r)=0$
 a.e.
 We have thus shown that $\mathsf{e}$ is Lipschitz in $[\alpha,\beta]$, in fact with Lipschitz constant less or equal than $L$. 
 \par
 Let us now consider a point $r \in [0,1]\setminus G$. 
 Then, there exists a sequence $(r_n,\theta_n)$, with 
 $(\sft_n(r_n),\theta_n) \in \bfTheta_n$ such that 
 $\minpartial \calE{\sft_n(r_n)}{\theta_n} \to 0$; up to a subsequence, 
 $(r_n,\theta_n)\to (r,\theta) \in \bfTheta$. By \eqref{closedness}, 
 $\minpartial \calE t\theta =0$, hence
 $\theta\in \critset (t)$ and $\ene t{\theta'} = \ene t {\theta} 
 = \mathsf e(r) 
 $ for every $\theta'\in \Theta(r)$. On the other hand,
 \[
\tilde{\mathsf e}(r) = \lim_{n\to\infty} \mathsf e(r_n)
= \lim_{n\to\infty} \ene{\sft_n(r_n)}{\theta_n} \stackrel{(1)}{=} \ene t\theta 
= \mathsf e(r). 
 \]  
 with (1) again due to \eqref{extensive-info}. 
 We thus deduce 
 that $\mathsf e$ coincides with the $L$-Lipschitz function $\tilde{\mathsf e}$
 on the set $F_0$
 from \eqref{e:precise-statement}.   On each connected component 
 $(r_0,r_1)$ of $[0,1]\setminus F_0$ the function 
 $\mathsf e$ is
 constant, and continuous on  the closure $[r_0,r_1]$, so that  
 $\mathsf e$ is
 $L$-Lipschitz on the whole $[0,1]$.
 Hence, claim (4.2) is proven.
 \par
 This concludes the proof of Theorem \ref{le:main}.
 \fin
\medskip
\par
 A proof of the following lemma,
used in the above proof,
 can be found, e.g., in \cite[Lemma 4.3]{MRS-MJM}. 
 \begin{lemma} 
 \label{le:mrs-mjm}
 Let $I$ be a measurable subset  of $\R$ and
 for every $n\in \N$
  let  $h_n,\, h,\, m_n,\, m: I \to [0,\infty]$ be measurable functions satisfying 
 \[
 \liminf_{n\to\infty} h_n(x) \geq h(x) \qquad \foraa\, x \in I, \qquad m_n \weakto m \text{ in } L^1(I).
 \]
 Then,
 \begin{equation}
 \label{liminf-hm}
 \liminf_{n\to\infty}\int_I h_n(x)m_n(x) \dd x \geq \int_I h(x)m(x) \dd x\,
 \end{equation}
 \end{lemma}
 \EEE
 \subsection{Properties of the cost}
 First of all,
let us reformulate the definition of the energy-dissipation cost from \eqref{costo-simpler}  in terms of admissible   graph transitions:    we have that 
\begin{equation}
\label{cost-in-terms-graphs}
\cost t{\cU_0}{\cU_1}:=
 \inf\left\{\int_{\necri t {\bfTheta}} \minpartial \calE t {\teta(s)} \| \teta'(s)\|\dd
s \,:\,\teta\in \AS t{\bfTheta}, \, \bfTheta \in \AG{t}{U_0}{U_1}\right\}  \quad \text{if $ \AG{t}{U_0}{U_1} \neq \emptyset$}
\end{equation}
(and that $\cost t{\cU_0}{\cU_1} = +\infty$ otherwise). 
Indeed, \eqref{cost-in-terms-graphs} follows from observing that for every fixed
admissible graph transition
  $\bfTheta \in \AG{t}{U_0}{U_1}$, thanks to Lemma 
\ref{le:reparam}
any selection 
$\teta\in \AS t{\bfTheta}$ is an admissible transition between $U_0$ and $U_1$ in the sense of Definition  \ref{main-def-1} and, conversely, by Lemma \ref{le:graph} with any 
$\teta \in  \admis{U_0}{U_1}t$ we can associate a graph $\bfTheta \in  \AG{t}{U_0}{U_1}$ via \eqref{eq:23}. 

  As a first consequence of Theorem \ref{le:main} we have the following result.
 \begin{theorem}
\label{prop:cost}
 Let $t\in [0,T]$ be fixed. Then, for any $\cU_0,\, \cU_1 \in \newclass t$  
we have:
\begin{enumerate}
\item if  $\cost t{\cU_0}{\cU_1}<\infty$, then 
there exists an optimal  graph transition  $\bfTheta$ for 
$\cost t{\cU_0}{\cU_1}$   and  thus an optimal transition curve $\teta \in \admis {\cU_0}{\cU_1}t$;
\item there holds
 \begin{equation}
 \label{ch-rule-estesa}
 \left| \ene t{\cU_1}{-} \ene t{\cU_0} \right| \leq \cost t {\cU_0}{\cU_1} \qquad \text{for every } \cU_0,\, \cU_1 \in \newclass t;
 \end{equation}
 \item  if 
 $\cost t{\cU_0}{\cU_1} =0$, then $\cU_0=\cU_1$;
 \item
 $\cost t{\cU_0}{\cU_1} =\cost t{\cU_1}{\cU_0} $;
 \item
  for every $\cU_2 \in \newclass t$ 
the triangle inequality holds:
\begin{equation}
\label{triangle-ineq}
\cost t{\cU_0}{\cU_1} \leq \cost t{\cU_0}{\cU_2} + \cost t{\cU_2}{\cU_1};
\end{equation}
On the other hand, along an optimal transition $\bfTheta$
for every $0=s_0<s_1<\cdots<s_N=1$ we have
\begin{equation}
  \label{eq:50}
  \cost t{\cU_0}{\cU_1} =\sum_{n=1}^N  \cost t{\Theta(s_{n-1})}{\Theta(s_n)}.
\end{equation}
\item
  If $\ene t{\cU_0}{-} \ene t{\cU_1} = \cost t {\cU_0}{\cU_1} $,
then, along every optimal graph transition $\Theta$ and every
admissible selection $\teta\in \AS t{\bfTheta}$,
the map $s\mapsto \ene t{\teta(s)}=
  \ene t{\Theta(s)}$ is nonincreasing and
  \begin{equation}
    \label{eq:49}
    \ene t{\teta(s')}{-} \ene t{\teta(s'')} = 
    \cost t
    {\Theta(s')}{\Theta(s'')}
    \quad
    \text{for every }0\le s'<s''\le 1.
  \end{equation}
\item 
$\costname {}$ is lower semicontinuous
w.r.t.\ convergence of the  process times  and 
 w.r.t.\ Kuratowski convergence; more precisely, 
for all
$(t_k)_{k} \subset [0,T]$ and for all 
  $(\cU_0^k)_k, \, \cU_0, \, (\cU_1^k)_k, \, \cU_1 \in
 \newclass t \cap \sublevel \rho$ for some $\rho>0$    we have 
 \begin{equation}
\label{lsc-cost}
\left(t_k\to t, \ \ 
\klimsup_{k\to\infty} \cU_0^k \subset \cU_0, \ \  \klimsup_{k\to\infty}\cU_1^k \subset \cU_1 \right) \ \  \Longrightarrow \  \ \liminf_{k\to\infty} \cost {t_k}{\cU_0^k}{\cU_1^k} \geq  \cost t{\cU_0}{\cU_1}.
\end{equation}
\end{enumerate}
 \end{theorem}
 \begin{proof}
 \par
 \noindent
 $\vartriangleright \, (1)$:  If $ \cost t{\cU_0}{\cU_1}<\infty$, 
 there exist  
 $(\bfTheta_n)_n \subset \AG t{\cU_0}{\cU_1}$ and 
 $\teta_n \in  \AS t{\bfTheta_n} $ with
\[
\lim_{n\to\infty} \int_{\necri t{\bfTheta_n}} \minpartial \calE t {\teta_n(r)} \| \teta_n'(r)\|\dd
r  =\cost t{\cU_0}{\cU_1}.
\] Now, since  either $\cU_0 $ is a point, or a connected component of 
$ \critset(t)$, by \eqref{eq:16bis} $\ene t{\cdot}$ assumes a constant value, $\ene t{\cU_0} $, on  $\cU_0$.
 It follows from the chain-rule estimate \eqref{eq:17} that 
 \[
 \ene t{\teta_n(r)} \leq \ene t{\cU_0} + \int_{\necri t{\bfTheta_n}} \minpartial \calE t {\teta_n(r)} \| \teta_n'(r)\|\dd
r  \leq C \quad \text{for all } r\in [0,1] \text{ and all } n \in \N.
 \]
 Therefore, 
$ \teta_n ([0,1]) \subset \sublevel\rho $ for some $ \rho>0.$
 We now reparameterize the curves $(\teta_n)_n $ as in Lemma \ref{le:reparam} and end up with a sequence  
 $(\sfs_n,\tilde\teta_n)_n$ fulfilling the normalization property  \eqref{normalization} and identity  \eqref{invariance-integral}. Hence, the
sequences $(\bfTheta_n)_n$ and  $(\sft_n,\tilde\teta_n)_n$ with $\sft_n(t) \equiv t$ comply with the conditions of Thm.\ \ref{le:main}. In particular, it follows from 
the chain-rule estimate \eqref{eq:17},
combined with \eqref{normalization}, 
that the functions $s\mapsto \ene t{\tilde\teta_n(s)}$ are uniformly Lipschitz continuous. 
Therefore, there exist a (not relabeled) subsequence, a limit graph transition $\bfTheta \in \AG t{\cU_0}{\cU_1}$
and  a limit curve $\tilde\teta\in   \AS t{\bfTheta}$ such that $\tilde\teta_n \to \tilde\teta$
uniformly on the compact subsets of $\necri t {\bfTheta}$
and 
\[
\cost t{\cU_0}{\cU_1}  = \lim_{n\to\infty} \int_{G[\tilde\teta_n]} \minpartial \calE t {\tilde\teta_n(s)} \| \tilde\teta_n'(s)\|\dd
s \geq \int_{G[\tilde\teta]} \minpartial \calE t {\tilde\teta(s)} \| \tilde\teta'(s)\|\dd
s, 
\]
 which gives the assertion. 
 \par
 \noindent
 $\vartriangleright \, (2)$:  
 It is sufficient to apply the chain-rule estimate \eqref{eq:17} along any optimal transition curve $\teta\in \admis {\cU_0}{\cU_1}t$ observing that,
in the case where the 
  sets  $\cU_i $ are contained in a connected component of $\critset(t)$ for $i\in \{0,1\}$, then $\ene t{\cdot}$ is constant on $\cU_i $
by \eqref{eq:16bis}, and  denoting by  $\ene t{U_i}$, $i \in \{0,1\}$, such a constant value. 
 \par
 \noindent
 $\vartriangleright \, (3)$:   
 Suppose  that  $\cost t{\cU_0}{\cU_1} =0$: we will show that  any optimal graph transition $\bfTheta$ that
 is non-trivial (namely,
$\bfTheta$ does not coincide with $[0,1]\ti\{x\}$ for any $x\in\sublevel \rho$) is contained in 
$[0,1]\ti\critset(t)$.
 Hence, suppose that $\bfTheta$ is non-trivial and,  by contradiction,
 that there exists $(r,\theta) \in \bfTheta$ 
with $\theta \in  \necritset t$. 
 Therefore, $r$ belongs to the open set $ \necri t{\bfTheta}$, and thus there exists $\eta>0$ such that $(r-\eta, r+\eta) \subset  \necri t{\bfTheta}$, so that $\Theta(s) = \{ \teta(s)\}$ for every $s\in (r-\eta, r+\eta)  $.
 Again by the chain-rule estimate \eqref{eq:17}  we immediately see that
 the image set 
  $\teta((r-\eta, r+\eta)) $ lies in $ \sublevel \rho$
 for some $\rho>0$.  With the same argument as in the proof of  Lemma \ref{le:crucial-estimates},
we find that there exists $\mu>0$ such that the compact set 
$K:= Q_\mu[(r,\theta)]: = \{(r',\theta')\in [0,1]\ti
\sublevel\rho:|r'-r|+|\theta'-\theta|\le \mu\} $  is contained in $[0,1]{\times}  (\Hilbert\setminus \critset(t))
\,.$ Up to taking a smaller $\mu$, we can suppose that $\teta(r-\eta) \notin K $. 
With the same technique as in the proof of Lemma \ref{le:reparam}, we reparameterize the curve $\teta|_{(r-\eta, r+\eta)}$ and obtain  $(\sfs,\tilde\teta): (\alpha,\beta) \to (r-\eta,r+\eta) \ti \sublevel\rho$, 
defined on some interval
$(\alpha,\beta)$,  such that the pair $(\sft,\tilde\teta)$ with $\sft(s)\equiv t$  satisfies  the conditions \eqref{eq:24-hyp} of Lemma \ref{le:crucial-estimates}, with
$s_0 \in (\alpha,\beta)$ such that $\tilde\teta(s_0) = \teta (r) \in K$.   Then,
from  \eqref{eq:24-thesis-2} we infer there exists $c'>0$ such that 
\[
0<c'\leq \int_{\alpha}^{s_0} \left(\sft'(s){+}\minpartial \calE {\sft(s)}{\tilde\teta(s)}\|\tilde\teta'(s)\| \right) \dd s = \int_{r-\eta}^r \minpartial \calE t{\teta(\rho)}\|\teta'(\rho)\| \dd \rho \leq \cost t{\cU_0}{\cU_1} ,
\]
contradicting the fact that $\cost t{\cU_0}{\cU_1} =0$. We have thus proven by contradiction that 
$\bfTheta \subset [0,1]\ti\critset (t)$.
Therefore, $\cU_0$ and $\cU_1$ are connected components of $\critset(t)$. Since $\bfTheta $ is connected and  
$\Theta(i)\subset\cU_i$
for $i\in \{0,1\}$, we ultimately deduce that 
$\cU_0=\cU_1$. 
   \par
 \noindent
 $\vartriangleright \, (4)$: With any optimal  graph transition
  $\bfTheta \in \AG t {\cU_0}{\cU_1}$
 we associate  a graph $\bfXi \subset [0,1]\ti \Hilbert$
defined by 
$\Xi(s) : = \Theta(1{-}s)$. Then  $\bfXi \in  \AG t {\cU_1}{\cU_0}$
 and for every $\teta \in \AS t{\bfTheta}$ and $\xi \in \AS t {\bfXi}$ there holds 
 \begin{align*}
\int_{\necri t{\bfXi}}
\minpartial \calE {t}{\xi(r)} \|\xi'(r)\| \dd r\ = \int_{\necri t {\bfTheta}}
\minpartial \calE {t}{\teta(r)} \|\teta'(r)\| \dd r\,.
\end{align*}
With this argument we conclude that  $\cost{t}{u_1}{u_0}\leq\cost{t}{u_0}{u_1}$.
Interchanging the role of $u_0$ and $u_1$ gives that $\cost{t}{u_0}{u_1} = \cost{t}{u_1}{u_0}$. 
   \par
 \noindent
 $\vartriangleright \, (5)$:
 To avoid overburdening the discussion, here we shall argue in terms of transition \emph{curves} rather than graphs, the argument for graphs being analogous. 
 Let 
 $\teta_{0,2} \in 
 \admis {\cU_0}{\cU_2}t$ and $
 \teta_{2,1} \in \admis {\cU_2}{\cU_1}t$ 
  be two optimal  transition curves   for $\cost t{\cU_0}{\cU_2} $ and
$\cost t{\cU_2}{\cU_1} $. 
Set
 \[
 \teta_{0,1}(s):= \begin{cases}
 \tilde{\teta}_{0,2}(s): = \teta_{0,2}(2s) & \text{for } s \in \left[ 0,\frac12 \right],\\
 \tilde{\teta}_{2,1}(s): = \teta_{2,1}(2s-1) & \text{for } s \in \left(\frac12,1 \right].
 \end{cases}
 \]
It is immediate to check that
$ \teta_{0,1}
\in \admis{\cU_0}{\cU_1}{t}$, so that 
we obtain
\[
\begin{aligned}
\cost t{\cU_0}{\cU_1}  \leq \int_{\necri t {\teta}} \minpartial \cE t
{\teta_{0,1}(s)} \,\| \teta_{0,1}' (s)\| \dd s &  =   2 \int_{G[\tilde\teta_{0,2}]}
\minpartial \cE t {\teta_{0,2}(2s)} \,\mdt {\teta}{0,2} {2s} \dd s
\\ & \quad  +  2 \int_{G[\tilde\teta_{2,1}]}  \minpartial \cE t {\teta_{2,1}(2s{-}1)}
\,\mdt {\teta}{2,1} {2s{-}1} \dd s\\ &  = \cost t{\cU_0}{\cU_2} +
\cost t{\cU_2}{\cU_1},
\end{aligned}
\]
and conclude \eqref{triangle-ineq}.

Formula \eqref{eq:50} is also immediate since the (rescaled) restriction of an
optimal transition $\bfTheta$ to each subinterval $(s_{n-1}, s_n)$ is
clearly optimal. 
    \par
 \noindent
 $\vartriangleright \, (6)$:
 we argue by contradiction, assuming that $\ene t{\teta(s')}-\ene
 t{\teta(s'')}<\cost t{\Theta(s')}{\Theta(s'')}$
 for some $0\le s'<s''\le 1$.
Combining \eqref{ch-rule-estesa} and the triangle inequality
\eqref{triangle-ineq} we get
\begin{align*}
  \cost t{\cU_0}{\cU_1}&=\ene t{U_0}-\ene t{U_1}
  =\ene t{U_0}-\ene t{\Theta(s')}+
  \ene t{\Theta(s')}-\ene t{\Theta(s'')}+
                         \ene t{\Theta(s'')}-\ene t{U_1}\\
  &<\cost t{U_0}{\Theta(s')}+
    \cost t{\Theta(s')}{\Theta(s'')}+
    \cost t{\Theta(s'')}{U_1}
\end{align*}
violating \eqref{eq:50}.
    \par
 \noindent
 $\vartriangleright \, (7)$: Clearly, it is meaningful to prove 
 \eqref{lsc-cost} in the case in which 
$\liminf_{k\to\infty} \cost {t_k}{\cU_0^k}{\cU_1^k} <\infty$. Then, up to a subsequence we have 
$\sup_k \cost {t_k}{\cU_0^k}{\cU_1^k} \leq C $. Now, for every $k \in \N$ we find an optimal transition graph
$\bfTheta_k \in  \AG {t_k}{\cU_0^k}{\cU_1^k} $ such that for every 
$\teta_k \in \AS {t_k}{\bfTheta_k}$ there holds
\[
\int_{\necri {t_k}{\bfTheta_k}} \minpartial \calE {t_k} {\teta_k(r)} \| \teta_k'(r)\|\dd
r   = \cost {t_k}{\cU_0^k}{\cU_1^k} \leq C\,.
\]
It follows from  the chain-rule estimate \eqref{eq:17} that $\teta_k([0,1]) \subset \sublevel{\rho'}$ (and, a fortiori, $\bfTheta_k \subset \sublevel{\rho'}$), for 
some $\rho'>0$. 
We reparameterize 
 $(\teta_k)_k$ 
 as in Lemma \ref{le:reparam} 
and end up with a sequence
$(\sft_k,\tilde\teta_k)_k$, with 
$\tilde\teta_k$ fulfilling \eqref{normalization} and 
\eqref{invariance-integral},
and 
$\sft_k(t) \equiv t_k$. Again,  Thm.\ \ref{le:main} applies to the sequences $(\bfTheta_k)_k$ and 
$(\sft_k,\tilde\teta_k)_k$: in particular, observe that 
$\klimsup_{k\to\infty} \tilde\teta_k(i) = \klimsup_{k\to\infty} \teta_k(i)\subset \klimsup_{k\to\infty} U_k^i \subset U_i$ for $i=0,1$. 
 Hence
there exist a (not relabeled) subsequence, a limit transition graph $\bfTheta$, and   limit  selection $\tilde\teta\in   \AS t{\bfTheta}$ such that 
\[
\cost t{\cU_0}{\cU_1}  \leq \int_{\necri t{\bfTheta}} \minpartial \calE t {\tilde\teta(r)} \| \tilde\teta'(r)\|\dd
r  \leq \liminf_{k\to\infty} \int_{\necri {t_k}{\bfTheta_k}}  \minpartial \calE t {\teta_k(r)} \| \teta_k'(r)\|\dd
r  =   \liminf_{k\to\infty}  \cost {t_k}{\cU_0^k}{\cU_1^k},
\]
which gives the assertion.
\end{proof}
 We conclude this section with a further
lower semicontinuity result
which 
relates $\costname{}$ with the limit of the
 energy-dissipation integrals, evaluated   along 
 curves $(v_n)_n \subset \AC ([0,T];\Hilbert)$ (a fortiori, 
 the solutions of the gradient flows \eqref{e:sing-perturb}),
  and computed on intervals $[\alpha_n,\beta_n]$ shrinking  to the fixed process time 
  $t\in[0,T]$.
\begin{lemma}
\label{l:endiss-cost-c}
For every $t\in [0,T]$,
$\cU_0,\, \cU_1 \in \newclass t$, and $\rho>0$,
let us consider a family of curves $v_n\in \AC([\alpha_n,\beta_n];\subl\rho)$ 
such that 
\begin{equation}
\label{properties-of-curves}
\alpha_n,\, \beta_n \to t, \quad
\klimsup_{n\to\infty} v_n(\alpha_n) \subset \cU_0, 
\quad \klimsup_{n\to\infty} v_n(\beta_n) \subset \cU_1\,.
\end{equation}
Then,
\begin{equation}
\label{desired-LSC}
 \liminf_{n \to \infty}
 \int_{\alpha_n}^{\beta_n}
 \minpartial \calE {r}{v_n(r)} \|v_n'(r)\| \dd r \geq \cost t {\cU_0}{\cU_1}\,.
\end{equation}
\end{lemma}
\begin{proof}
 Mimicking the argument from the proof of Lemma \ref{le:reparam}, we  introduce the  arclength functions
\begin{equation}
\label{new-arclength}
 \sfs_n:[\alpha_n,\beta_n]\to [0,1], \qquad \sfs_n(r): = \frac{r-\alpha_n+\int_{(\alpha_n,r)}  \minpartial \calE \tau {v_n(\tau)} \| v_n'(\tau)\|\dd\tau }{\beta_n-\alpha_n+ \int_{\alpha_n}^{\beta_n} \minpartial \calE \tau {v_n(\tau)} \| v_n'(\tau)\|\dd\tau}.
 \end{equation}
We  set 
$
 \sfr_n: = \sfs_n^{-1}: [0,1]\to [\alpha_n,\beta_n] $ and accordingly 
 rescale the curves $v_n$, defining
$ \mathsf{v}_n(s): = v_n(\sfr_n(s)) $ for all  $s \in [0,1]$. 
In this case, we have 
\[
 \sfr_n'(s) +  \minpartial \calE  {\sfr_n(s)}{\sfv_n(s)} \| \sfv_n'(s)\| = \beta_n -\alpha_n+ \int_{\alpha_n}^{\beta_n} \minpartial \calE t {v_n(r)} \| v_n'(r)\|\dd
r  \leq C,
\]
and
\[
 \int_{\alpha_n}^{\beta_n}
 \minpartial \calE {r}{v_n(r)} \|v_n'(r)\| \dd r =  \int_{0}^{1}
 \minpartial \calE {\sfr_n(s)}{\sfv_n(s)} 
\|\sfv'_n(s)\|
 \dd s\,.
\]
Hence, the curves  $(\sfr_n,\sfv_n)_n$
fulfill 
 the conditions of Thm.\ \ref{le:main}. Thus, we conclude that there exist a (not relabeled) subsequence and a limit curve $\sfv \in   \admis{\cU_0}{\cU_1}{t}$ such that  the curves
 $\sfv_n$ converge to $\sfv$ uniformly on the compact subsets of
$\necri t {\sfv}$, 
and
\[
 \liminf_{n \to \infty}
 \int_{\alpha_n}^{\beta_n}
 \minpartial \calE {r}{v_n(r)} \|v_n'(r)\| \dd r  = \liminf_{n\to\infty}  \int_0^1 \minpartial \calE {\sfr_n(s)}{\sfv_n(s)}
 \|\sfv'_n(s)\|
 \dd s
\geq \int_{\necri t{\sfv}} \minpartial \calE t {\sfv (s)} \| \sfv'(s)\|\dd
s, 
\]
 which gives \eqref{desired-LSC}.
\end{proof}

 \section{Proof of Theorem \ref{mainth:1}}
 \label{s:4}
 In this section we will present the complete proof of Theorem
\ref{mainth:1}.
We organize our presentation in various steps, trying to highlight
the main arguments and the partial results of independent interest.
 \subsection{A priori compactness estimates
   for the energy and for the graphs of the solutions}
 \label{subsec:apriori}
 Our first step is 
 to
 analyze the asymptotic behavior of  
a family of solutions $(u_{\eps_n})_n$
to the Cauchy problem 
\eqref{Cauchy-gflow} for a vanishing sequence $(\eps_n)_n$,
 assuming
 that the initial data $(u_{0,\eps_n})_n$
 fulfill \eqref{energy-convergence-initial}
 (notice that we can always choose $u_{\eps_n}(0)=u_0$ for every $n\in
 \N$).

\par
Combining \eqref{energy-convergence-initial} with estimate \eqref{bound-energie} from Theorem \ref{thm:exist-g-flow} we infer that 
$\sup_{t\in [0,T]}  \eneb t{\unn(t)} <\infty
$,
hence
\begin{equation}
\label{in-un-compattone}
\exists\, \rho>0 \ \   \forall\, n \in \N, \  \forall\, t \in [0,T] \, : \qquad \unn(t) \in \subl\rho.
\end{equation}
Recall that the set $ \subl\rho$ is compact
by assumption \eqref{coercivita}
on the energy.
Before exploring the outcome of \eqref{in-un-compattone}, 
let us fix the other consequences of the estimates stemming from
 Theorem \ref{thm:exist-g-flow} 
 in the following result. In particular, we are concerned with the compactness properties of the sequence of measures
 \begin{equation}
\label{not-mu_n}
\mu_{\eps_n}:= \left(
\frac{\ep_n}2 \mdtq u{\eps_n}{\cdot}2
+\frac1{2\eps_n} \minpartialq \cE {(\cdot)}{u_{\eps_n}(\cdot )}
\right)\, \mathscr{L}^1
\end{equation}
associated with the curves $(\unn)_n$, and with the pointwise limit $\limen:[0,T] \to \R$ of the sequence $(\ene t {u_{\eps_n}(t)})_n$.  We will denote
by $\limen (t_-)$ and $\limen (t_+)$ the left
and right limits of $\limen$ in $t$,  for every $t\in[0,T]$ (with the convention $\limen (0_-) : = 
\limen(0)$ and  $\limen (T_+) : = 
\limen(T)$), and by
$\frac{\mathrm{d}}{\mathrm{d}t}  \limen$
its distributional derivative.
 \begin{lemma}
\label{prop:1compactness}
There exist
 a (not relabeled)  subsequence $(\eps_n)_n$,
 a  
 nonnegative finite Borel measure
 $\mu $ on $[0,T]$,
a Borel set $B\subset [0,T]$ of full Lebesgue measure, 
 and functions $\limen\in
  \BV ([0,T])$ and $\limp\in L^\infty (0,T)$ such that
  the following convergences hold
 as $n
\to\infty$                                
\begin{align}
& \label{1converg}
\mu_{\eps_n} \weaksto \mu \quad &&\text{in }\mathrm{M}_+(0,T),
\\
& \label{2converg}
\lim_{n \to \infty} \ene t {u_{\eps_n}(t)}\,=\,
\limen(t) \quad &&\text{for all  } t \in [0,T],
\\
& \label{bis-2converg}
\pt t{u_{\eps_n}(t)} \weaksto \limp
     \quad&&\text{in $L^\infty (0,T)$,}\\
  &     \label{set-B}
    \minpartial \cE
    {t}{u_{{\ep}_n}(t) } \to 0\quad&&
                                       \text{for every }t\in B.
\end{align}
Moreover,
\begin{equation}\label{identita-energia-limen}
\mu([s,t])+\limen(t)=\limen(s)
+\int_s^t \limp(r)\dd r
\qquad\mbox{for every }t\in[0,T],\ s\in[0,t].
\end{equation}
Furthermore,
the
following identities hold
\begin{align}
&\label{left_right_limits}
 \limen (t-)-\limen (t+)  =\mu(\{t\}),\\
  &\label{identity-between-measures}
     \frac
    {\mathrm{d}}
    {\mathrm d t}
    \EEE \limen + \mu= \limp
    \mathscr{L}^1\,
    \quad\text{in the sense of distributions in }[0,T].
\end{align}
Finally, the set where the measure $\mu$ is atomic
\begin{equation}
\label{enhanced-prop-3_var} \jumpname= \{t\in[0,T]\,:\,\mu (\{t\})>0 \}
\text{ consists of at most countably many points}.
\end{equation}
\end{lemma}
\begin{proof}
It follows from estimate \eqref{bound-en-diss} in  Theorem \ref{thm:exist-g-flow}
that the measures $(\mu_{\eps_n})_n$ have uniformly
bounded
 total mass,
therefore \eqref{1converg} follows. As for
\eqref{2converg}, we observe that, by the energy-dissipation identity \eqref{eqn_lemma}, the maps
$t\mapsto \mathscr{F}_n(t):=  \ene t{u_{\eps_n}(t)}-\int_0^t \pt  s
{u_{\eps_n}(s)} \dd s $ are nonincreasing on $[0,T]$. Therefore, by
Helly's Compactness Theorem
there exists $\mathscr{F} \in \BV([0,T])$ such that, up to a
subsequence, $\mathscr{F}_n(t) \to \mathscr{F}(t)$ for all $t \in
[0,T]$. On the other hand, estimate \eqref{bound-energie} also yields
\eqref{bis-2converg}, up to a subsequence. Therefore,
\eqref{2converg} follows with $ \limen(t):= \mathscr{F}(t)+\int_0^t
\limp(s)\dd s. $

 Finally, 
\eqref{bound-en-diss} yields
\[
\lim_{n \to\infty}\int_0^T
\minpartialq\cE{r} {u_{{\ep}_n}(r)} \dd r=0,
\]
hence we can extract a further subsequence and find a Borel set $B$ of
full measure
such that
the integrand converges pointwise to $0$ for every $t\in B$, thus
obtaining
\eqref{set-B}.

Passing to the limit as $n \to \infty$ in the energy identity
\eqref{eqn_lemma} and taking into account convergences
\eqref{1converg}--\eqref{bis-2converg}, we obtain
\eqref{identita-energia-limen}. In particular, identity
\eqref{identita-energia-limen} implies that
\[
\limen (t-\sigma)-\limen (t+\sigma) +
\int_{t-\sigma}^{t+\sigma}\limp(s)\dd s
=
\mu ([t-\sigma,t+\sigma]),
\]
for every $t\in(0,T)$ and $\sigma>0$ arbitrarily small. Observe that,
since $\limen \in \BV ([0,T])$, the left and right limits   $\limen
(t-)$ and $\limen (t+)$  exist for every $t \in [0,T]$. Therefore,
taking the limit  as $\sigma \down 0$ in the above identity gives
(\ref{left_right_limits}). Identity
\eqref{identity-between-measures} trivially follows from
(\ref{identita-energia-limen}).

Finally, let us denote by $(\frac{\mathrm{d}}{\mathrm{d}t} \limen)_{\mathrm{jump}}$ the
jump part of the measure $\frac{\mathrm{d}}{\mathrm{d}t}  \limen$: it follows from
\eqref{identity-between-measures} that
\begin{equation}
\label{interesting-observation}
\mathrm{supp}\left(\Big(\frac{\mathrm{d}}{\mathrm{d}t} 
\limen\Big)_{\mathrm{jump}}\right) = \jumpname.
\end{equation}
Then, \eqref{enhanced-prop-3_var} follows from recalling that
$\limen \in \BV([0,T])  $ has countably many jump  points.
\end{proof}
We will now investigate the consequences of  \eqref{in-un-compattone} by adopting the `graph' viewpoint 
of Section \ref{s:energy_diss_cost}. Namely, we will address the compactness properties of the 
graphs of the curves $(\unn)_n$, namely the compact 
sets
\begin{equation}
\label{notation-gran}
\Gunn := \big\{ (t,\unn(t))\, : \ t \in [0,T]\big\}.
\end{equation}
Our first result 
 concerns the Kuratowski convergence of (a subsequence) of the graphs $(\Gunn)_n$ to a
compact and connected set $\Gv$, whose fibers $\vlim(t)$ at fixed
process time $t\in[0,T]$ will  provide a ``generalized'' limit for the
gradient flows $(u_{\eps_n})_n$. 
\begin{lemma}
\label{l:compactness-1}
Up to extracting a further (not relabeled) subsequence, 
there exists a compact and fiberwise connected set $\Gv \subset [0,T] \ti
\subl\rho$, with $\subl\rho$ from \eqref{in-un-compattone},
such that
\begin{equation}
\label{Hausd-Kur-conv}
\Gunn \karrow \Gv \ \text{as } n \to \infty.
\end{equation}
The multivalued map
 with values in $\mathfrak K_c(\Hilbert)$ 
\begin{equation}
\label{mathscrU}
\vlim: [0,T] \rightrightarrows\Hilbert,  \qquad 
\vlim(t) := \{ v \in \Hilbert\, : \ (t,v) \in \Gv\}\quad   \text{ for all } t\in [0,T],
\end{equation}
is upper semicontinuous.
\end{lemma}
\begin{proof}
It follows from \eqref{in-un-compattone} that  the graphs $\Gunn$ are contained in the compact set  $ [0,T]\ti \subl\rho$ for all $n\in \N$. Therefore, by the Blaschke Theorem the sets $(\Gunn)_n$ converge in (the Hausdorff and, a fortiori in) the Kuratowski sense  to a compact set $\Gv$, which is also
fiberwise connected thanks to Lemma \ref{le:connection} ahead. In view of \eqref{closed2usc}, the multivalued map $\Gvf$ from \eqref{mathscrU} is upper semicontinuous.  
\end{proof}
\subsection{Limit properties of the energy and of the graphs}
\label{subsec:limit}
We will now provide a thorough description of
the properties of the sets $(\vlim(t))_{t\in (0,T)}$ by resorting to
the  energy-dissipation cost.
Our following result relates the measure $\mu$
 from 
  Lemma
\ref{prop:1compactness} 
with the cost $\costname{}{}$.
\begin{lemma}
  \label{lemma:convergence-to-critical}
   Let $(\eps_n)_n$ be a subsequence along which
  the conclusions of Lemma
  \ref{prop:1compactness} and \ref{l:compactness-1} hold.\\
Let $t\in [0,T]$, $H\in \N_+$, 
 and 
let $ (t_h^n)_n$, $h=0,\cdots,H$ 
be sequences in 
$[0,T]$  satisfying 
$0\leq t_0^n \leq t_1^n 
\leq \cdots\leq t_H^n\leq T$ for all $n\in \N$ and, as $n\to\infty$,
\begin{equation}
\label{hyp-lemma-conv2} t_h^n \to t, 
\qquad
\klimsup_{n\to\infty} u_{\eps_n}(t_h^n) \subset \cU_h
\quad\text{for
some } \cU_h
\in 
\newclass t,\ h=0,\cdots,H\,.
\end{equation}
Then, we have that
\begin{equation}\label{mu_costo}
\mu(\{t\})\geq 
\sum_{h=1}^{H}\cost{t}{\cU_{h-1}}{\cU_h}.
\end{equation}
\end{lemma}
\begin{proof}
Let $ t\in[0,T]$, $(t_h^n)_n$, $U_h$
satisfy \eqref{hyp-lemma-conv2}. Observe that for every
$\sigma>0$ there holds
\begin{equation}
\label{to-be-cited-later}
\begin{aligned}
\mu ([t-\sigma, t+\sigma])
&\,\geq\,\limsup_{n \to \infty}\mu_{\eps_n}([t_0^n,t_H^n])
\\
&
\,=\,\limsup_{n \to \infty}\int_{t_0^n}^{t_H^n}
\left( \frac{\ep_n}2\|\dot u_{\eps_n}(s ) \|^2
+\frac1{2\eps_n} \minpartialq \cE{s} {u_{\eps_n}(s)} \right) \dd s
\\
&\,\geq\,\limsup_{n \to \infty} \int_{t_0^n}^{t_H^n}\minpartial \cE{s} {u_{\eps_n}(s)}
\|\dot u_{\eps_n}(s)\|\dd s
\\& 
\,\geq\,
\sum_{h=1}^H
\liminf_{n \to \infty} \int_{t_{h-1}^n}^{t_h^n}
\minpartial \cE{s} {u_{\eps_n}(s)}
\|\dot u_{\eps_n}(s)\|\dd s
\,\geq\, 
\sum_{h=1}^H\cost t{\cU_{h-1}}{\cU_h},
\end{aligned}
\end{equation}
where the first inequality is due to convergence \eqref{1converg},
the second one to 
definition \eqref{not-mu_n} 
of $\mu_{\eps_n},$ the
third one to  Young's inequality, 
the fourth one
to the fact that the $\limsup$ of a sum is larger than the sum of the $\liminf$
of the individual terms, and the last one to the previous  \eqref{desired-LSC}.
Letting $\sigma \down 0$ in \eqref{to-be-cited-later}, 
 we conclude \eqref{mu_costo}.
\end{proof}

\begin{proposition}
\label{th:1}
 The multivalued map $\vlim: [0,T] \rightrightarrows\Hilbert$
of Lemma \ref{l:compactness-1}
  enjoys the following properties:
\begin{enumerate}
\item
there holds
\begin{equation}
\label{sup-cost}
\mu(\{ t \})
\,\,\geq
\sup
 \Big\{ 
\cost{t}{\cU_0}{\cU_1}:
\cU_0,\, \cU_1 \in 
\newclass t,  \cU_0\cap \vlim(t) \neq \emptyset,\ \cU_1 \cap 
\vlim(t) \neq \emptyset\Big\} 
 \quad \text{for every } t \in [0,T];
\end{equation}
\item
$\vlim(t) \cap \critset(t) \neq \emptyset$
for every $t\in[0,T]$;
\item
for every $t \in[0,T]\setminus \jumpname$, with $\jumpname$ the set of atoms of
$\mu$,
\begin{equation}
\label{conn-component-ulim}
\text{$\vlim(t)$ is contained in a single
  connected component
  $\ulim(t):=\comp t{\vlim(t)}$
of $\critset(t)$;}
\end{equation}
\item
  For every $t \in[0,T]\setminus \jumpname$
  \begin{equation}
    \label{eq:33}
    \klimsup_{\substack{s\to t\\ s\in [0,T]\setminus\jumpname}}\ulim(s)\subset \ulim(t).
  \end{equation}
\item there holds
  \begin{equation}
\label{better}
\limen(t)=\ene tv \qquad \text{for all } t \in [0,T]\setminus J,\
v\in \ulim(t).
\end{equation}
\end{enumerate}
\end{proposition}
\begin{proof}
$\vartriangleright\, (1)$
Let $\cU_0, \, \cU_1 \in \newclass t$ 
 fulfill $\cU_0 \cap \vlim(t) \neq \emptyset$, 
$\cU_1 \cap \vlim(t)  \neq \emptyset$. Pick 
 $ u_0 \in \cU_0 \cap \vlim(t)$ and  $ u_1 \in \cU_1 \cap \vlim(t)$. 
The characterization of the Kuratowski convergence
$\Gunn \karrow \Gv$ ensures that there exist sequences
$(t_0^n, u_{\eps_n}(t_0^n))_n$ and $(t_1^n,u_{\eps_n}(t_1^n))_n$
 converging to $(t, u_0)$ and $(t, u_1)$, respectively. 
Up to changing the role of $(t_0^n)_n$ and of $(t_1^n)_n$ and
up to a subsequence,
we may suppose that $t_0^n\leq t_1^n$ for every $n$.
Applying Lemma \ref{lemma:convergence-to-critical} then gives
$\mu(\{ t\})\geq\cost{t}{\cU_0}{\cU_1}$.
By the arbitrariness of $\cU_0$ and $\cU_1$, 
we conclude that \eqref{sup-cost} holds.
\par
\noindent 
$\vartriangleright\, (2)$
Let us fix $t\in B $, with  the set $B$ from \eqref{set-B}.  By its definition, we have that 
 $\minpartial \cE {t}{u_{\eps_n}(t)} \to 0$ as $n\to\infty$. Now, thanks to 
 the energy estimate \eqref{in-un-compattone}
and to \eqref{coercivita}
 we may extract from the
sequence $(u_{\eps_n}(t))_n$ a subsequence $u_{\eps_{n_k}}(t)$ converging to  some
$u\in \Hilbert$. By definition of Kuratowski convergence, we
conclude that $(t,u)\in\Gv$, hence $u \in \vlim(t)$. The closedness
property \eqref{closedness}, with $t_{n_k} \equiv t$ for every $k\in \N$, guarantees
that $u\in \critset(t)$ as well.
\par
 For $t\in [0,T] \setminus B$, we pick a sequence $(t_j,u_j)_j \in
\vlim(t_j)\cap \critset(t_j)$
with $(t_j)_j$ in the
dense set $B$
and $t_j \to t$.
By using the compactness of $\Gv\cap \crit$ and extracting a converging subsequence,
we deduce that $\vlim(t)\cap \critset(t)\neq \emptyset$. 
\par
\noindent 
$\vartriangleright\, (3)$
It follows from \eqref{sup-cost} that for every
$t\in [0,T] \setminus \jumpname$
 the cost 
$\cost t{\cU_0}{\cU_1} =0$ for every pair of sets
$\cU_0 ,\,  \cU_1 \in 
\newclass t $ with $  \cU_0\cap 
\vlim(t) \neq \emptyset,\ \cU_1 \cap 
\vlim(t) \neq \emptyset$.
 In particular, let us consider a point $v_0\in \critset (t) \cap
\vlim(t) $
(which is non-empty by  claim (2))
and let $U_0$ be the connected component of $\critset(t)$ containing $v_0$.
Since 
$\cost t{U_0}{\{x\}} >0$ for every $x\in \vlim(t)\setminus \critset(t)$
by 
Theorem \ref{prop:cost}\ (3),
it follows that $\vlim(t)\subset \critset (t)$.
Since $\vlim(t)$ is also connected, we conclude. 
\par
\noindent 
$\vartriangleright\, (4)$
Recalling property \eqref{eq:35}, it is sufficient to prove that
any Hausdorff limit $U$ of $\ulim(s)$ along
a sequence $(s_n)_n$ in $[0,T]\setminus \rmJ$ converging to
$t$ is contained in $\ulim(t)$.
Since $\ulim(s_n)$ are compact subsets of $\critset(s_n)$ and $\crit$
is closed,
$U$ is a compact and connected subset of $\critset(t)$, i.e., 
$U\subset \critset(t)$. Since $\ulim(s_n)\supset \vlim(s_n)$
and $\klimsup_{n\to\infty}\vlim(s_n)$ is a nonempty subset of
$\vlim(t)$,
we deduce that $U\cap \vlim(t)$ is nonempty as well. We thus  conclude that
$U\subset \ulim(t)$.
\par
\noindent 
$\vartriangleright\, (5)$
 By the previous Claim and the fact that $\calE_t$ is constant on the connected components of $\critset (t)$
due to \ass1\ (see Remark \ref{rem:E-constant}), we know that $\calE_t$ is constant on $\ulim(t)$.
We can then extract a subsequence
$k\mapsto \eps_{n(k)}$ and find a limit point $w\in \vlim(t)\subset \ulim(t)$ such that
$u_{\eps_{n(k)}}(t)\to w$.
If $t\in B\setminus \rmJ$, then \eqref{closedness} yields
$\mathscr E(t)= \ene t{w}=\ene t{\ulim(t)}$.

If $t\in [0,T]\setminus J$, we can find a sequence
$t_n\in B\setminus \rmJ$ converging to $t$,
a sequence $w_n\in \vlim(t_n)$ and a limit point $w\in \vlim(t)$
(here we are using that $\Gv$ is compact)
such that $w_n\to w$.
Since $\mathscr E$ is continuous at $t$ we know that
$\mathscr E(t_n)\to \mathscr E(t)$.
On the other hand $\partial \ene{t_n}{w_n}\ni0$ so that the closedness property
\eqref{closedness} yields
$\ene {t_n}{w_n}\to \ene tw$. We conclude that
the identity $\mathscr E(t)=\ene t{\ulim(t)}$ holds for every $t\in
[0,T]\setminus \rmJ$.
\end{proof}
We discuss now the case when $t\in \rmJ$. 
\begin{proposition}
  \label{prop:lr-limits}
  For every $t\in \rmJ\cap (0,T)$ there exists
  a unique pair of compact 
  connected components $\ulim(t\pm)$ of $\critset(t)$ such that
\begin{equation}\label{eq:36}
\klimsup_{s\uparrow t}\vlim(s)\subset \ulim(t-),\qquad
    \klimsup_{s\downarrow t}\vlim(s)\subset \ulim(t+).
  \end{equation}
  They satisfy
  \begin{equation}
    \label{eq:25}
 \limen (t-) 
=\ene t{\ulim(t-)},\qquad
  \limen (t+) 
=\ene t{\ulim(t+)},
  \end{equation}
  \begin{equation}
    \label{eq:39}
    \mu(\{t\})=\ene t{\ulim(t-)}-\ene t{\ulim(t+)}=
    \mathsf c_t(\ulim(t-),\ulim(t+)).
  \end{equation}
  Moreover, for every $u\in \vlim(t)$, letting $\ulim$ be the lifting of $u$ according to Definition \ref{def:lifting}, we have that
\begin{equation}
\label{eq:39bis}
    \mu(\{t\})=
    \mathsf c_t(\ulim(t-),\ulim(t+))
    =
    \mathsf c_t(\ulim(t-),\ulim)+
   \mathsf c_t(\ulim,\ulim(t+)).
  \end{equation}
 In the cases when $t\in \rmJ\cap\{0,T\}$ we have only the
  appropriate limit
  from the left or from the right.
\end{proposition}
\begin{proof}
 Let us prove the first claim 
  in the case of the limit from the left for $t\in
  \rmJ\cap (0,T]$, the other case follows by a similar argument.

  Let $u',u''$ be two limit points in
  $\klimsup_{s\uparrow t}\vlim(s)$,
  with liftings $U', U''\in \newclass t$ respectively.
  We can find two increasing sequences $s_k',s_k''$, $k\in \N$,
  converging to $t$
  and corresponding sequences
  $v_k'\in \vlim(s_k),\ v_k''\in \vlim(s_k'')$
  such that $v_k'\to u',\ v_k''\to u''$ as $k\to\infty$.
  Up to extracting a further subsequence, it is not restrictive to
  assume that $s_k'< s_k''$.
  By definition of $\Gv$, we can find 
  two families of sequences $n\mapsto (s_{k,n}',u_{\eps_n}(s_{k,n}'))$
  and
  $n\mapsto (s_{k,n}'',u_{\eps_n}(s_{k,n}''))$, indexed by $k\in \N$,
  converging to $(s_k',v_k')$ and $(s_k'',v_k'')$ respectively
  as $n\to\infty$, with $s_{k,n}'< s_{k,n}''$.
  As for \eqref{to-be-cited-later}
  \begin{equation}
\label{to-be-cited-later2}
\begin{aligned}
  \mu ([s_k', s_k''])
  \,\geq\,\limsup_{n \to \infty}\mu_{\eps_n} ([s_{k,n}',s_{k,n}''])
\,\geq\,\limsup_{n \to \infty} \int_{s_{k,n}'}^{s_{k,n}''}
\minpartial \cE{s} {u_{\eps_n}(s)}
\|\dot u_{\eps_n}(s)\|\dd s.
\end{aligned}
\end{equation}
Choosing $n=n(k)$ sufficiently large so that
$t_k':=s_{k,n(k)}'$, $t_k'':=s_{k,n(k)}''$,
$u_k':=u_{\eps_{n(k)}}(s_{k,n(k)}')$,
$u_k'':=u_{\eps_{n(k)}}(s_{k,n(k)}'')$
satisfy
\begin{displaymath}
  |t_k'-s_k'|+|t_k''-s_k''|+  \|v_k'-u_k'\| +\|v_k''-u_k''\|=o(1),\quad
  \int_{t_{k}'}^{t_{k}''}
\minpartial \cE{s} {u_{\eps_{n(k)}}(s)}
\|\dot u_{\eps_{n(k)}}(s)\|\dd s\le \mu ([s_k', s_k''])+o(1)
\end{displaymath}
as $k\to\infty$, we can apply the $\liminf$ estimate
\eqref{desired-LSC} to obtain
\begin{equation}
  \label{eq:38}
  0=\lim_{k\to\infty}\mu ([s_k', s_k''])+o(1)\ge \mathsf c_t(U',U'').
\end{equation}
We conclude that there is a unique connected component
$\ulim(t-)\in \newclass t$ containing all the points
of $\klimsup_{s\to t}\vlim(s)$.
Since $\ulim(t-)$ has nonempty intersection with
$\critset(t)$ (obtained by taking limit points of
sequences in $\crit$),
we conclude that $\ulim(t-)\subset \critset(t)$
and in turn inclusions
\eqref{eq:36}.

We can then pick a
sequence of approximating points 
$s_k\uparrow t$ with $s_k\in [0,T]\setminus \rmJ$
with $v_k\in \ulim(s_k)\subset \critset(s_k)$ converging
to an arbitrary point of $v\in \ulim(t-)$.
The fact that $\limen \in \BV ([0,T])$, property
\eqref{better}, the continuity
of $\calE$ on $\crit$, and the fact that $u \mapsto \ene t{u}$ is constant
on the connected components 
of $\critset(t)$
yield
\begin{equation}
\label{eq:lim_jumps}
\lim_{s\uparrow t}\mathscr E(s)
=
\lim_{k\to\infty}
\mathscr E(s_k)
=  \lim_{k\to\infty}\ene{s_k}{v_k}=
  \ene tv=
  \ene t{\ulim(t-)}
\end{equation}
which proves \eqref{eq:25}.
%
%
%
In order to prove \eqref{eq:39}, we observe that 
\[
\mu(\{t\}) \stackrel{(1)}{=} 
\limen(t-) - \limen(t+)
\stackrel{(2)}{=} \ene t{\ulim(t-)} -  \ene t{\ulim(t+)} \stackrel{(3)}{\leq} \cost t{\ulim(t-)}{\ulim(t+)}  \stackrel{(4)}{\leq} \mu(\{t\}) \qquad \text{for all } t \in \jumpname.
\]
where (1) follows from
 \eqref{left_right_limits}, 
 (2) from 
\eqref{eq:25},
(3) from  the chain-rule estimate \eqref{ch-rule-estesa}, and (4) from
 Lemma \ref{lemma:convergence-to-critical}.
All in all, we gather
\begin{equation}
\label{atom-mu}
0<\mu(\{t\}) = \cost t{\ulim(t-)}{\ulim(t+)} 
=  \ene t{\ulim(t-)} -  \ene t{\ulim(t+)}
\qquad \text{for all } t \in \jumpname.
\end{equation}

 In order to prove 
 \eqref{eq:39bis} 
 we pick a sequence
 $(r_n)_n$
 converging to $t$ 
 such that $u_{\eps_n}(r_n)
 \to u$ and we 
 argue as in the first part
 of the proof.
 We first consider 
 two sequences
 $(s_k^-)_k$ and  $(s_k^+)_k$
 and elements $v_k^-\in \vlim(s_k^-)$,
 $v_k^+\in \vlim(s_k^+)$
 such that 
 \begin{equation}
 \label{eq:boh}
 s_k^-<t<s_k^+,\quad 
 s_k^-\to t,\ s_k^+\to t,\quad 
 v_k^-\to u^-\in \ulim(t-),\quad 
 v_k^+
 \to u^+\in \ulim(t+)
 \quad \text{as }k\to\infty.
\end{equation} 
By definition of $\Gv$, we can find 
  two families of sequences $n\mapsto (s_{k,n}^-,u_{\eps_n}(s_{k,n}^-))$
  and
  $n\mapsto (s_{k,n}^+,u_{\eps_n}(s_{k,n}^+))$
  converging to $(s_k^-,v_k^-)$ and $(s_k^+,v_k^+)$ respectively
  as $n\to\infty$.
  Since the elements of the sequence 
  $(r_n)_n$ eventually belong to $(s_k^-,s_k^+)$,   
  we can then choose an increasing sequence
  $k\mapsto n(k)$ 
  so that 
  $t_k^-:=s_{k,n(k)}^-$,
  $t_k^+:=s_{k,n(k)}^+$,
  $t_k:=r_{n(k)}$,
  $u_k^-:=u_{\eps_{n(k)}}(s_{k,n(k)}^-)$,
  $u_k^+:=u_{\eps_{n(k)}}(s_{k,n(k)}^+)$
  satisfy
  $$
  |t_k^--s_k^-|+|t_k^+-s_k^+|+
  \|v_k^--u_k^-\|+
  \|v_k^+-u_k^+\|=o(1),\quad 
  t_k^-<t_k<t_k^+
  $$
  We can then apply 
  Lemma \ref{lemma:convergence-to-critical}
  to the sequences 
$(t_k^-,u_{\eps_{n(k)}}(t_k^-))$,
  $(t_k,u_{\eps_{n(k)}}(t_k))$,
  $(t_k^+,u_{\eps_{n(k)}}(t_k^+))$
  converging to
  $(t,u^-)$, $(t,u)$, and $(t,u^+)$
  respectively; 
  we obtain
  \begin{equation*}
      \mu(\{t\})\ge 
      \cost t{\ulim(t-)}{\ulim}
      +\cost t{\ulim}{\ulim(t+)}.
  \end{equation*}
  The converse inequality then follows 
  by \eqref{eq:39} and the 
  triangle inequality \eqref{triangle-ineq}.
\end{proof}
\subsection{Existence of \DSolutions: a selection argument}
\label{subsec:selection}
Combining all the results of the previous section, 
it is straightforward to
obtain
our first main characterization:

\begin{theorem}
  \label{thm:selection}
  Let $ u\nc :[0,T]\to \Hilbert$ be a Borel map satisfying $ u (0)=u_0$ 
  and
  the selection
  criterion
  \begin{equation}
    \label{eq:40}
   u\nc (t)\in \vlim(t)\quad\text{for every }t\in [0,T],\qquad
    \mathcal P_t( u (t))=\mathscr P(t)\quad\text{a.e.~in }(0,T).
  \end{equation}
  Then $ u $ is a \DSolution.
\end{theorem}
\begin{proof}
Defining the function $\sfe(t):=\ene t{u (t)}$, for $t\in[0,T]$, and arguing as in \eqref{eq:lim_jumps}, one can prove that
\[
\sfe(t-)=\ene t{\ulim(t-)},
\qquad
\sfe(t+)=\ene t{\ulim(t+)}
\qquad\mbox{for all\ }t\in\jumpname.
\]
Therefore, identities \eqref{to-save-2} and \eqref{to-save-4} in Definition \ref{def:sols-notions-1} of \DSolution\ hold true, also thanks to \eqref{eq:39}--\eqref{eq:39bis}. Moreover, we have that
$\sfe(t)=\limen(t)$ for every $t\in[0,T]\setminus\jumpname$, and $\sfe(t\pm)=\limen(t\pm)$
for every $t\in\jumpname$, thanks to \eqref{better} and \eqref{eq:25}, respectively. 
In particular, since $\limen\in\BV ([0,T])$, then $\sfe\in\BV ([0,T])$ as well, and, due to \eqref{left_right_limits}, $\jumpname=\jump\sfe$.
\end{proof}
In order to conclude the proof of
 Claims (1) and (2) 
of Theorem \ref{mainth:1}
it is then sufficient to prove that
there exists a Borel selection satisfying \eqref{eq:40}.
This is precisely the statement of our next result.
\begin{theorem}
  \label{thm:selection2}
  There exists a Borel map
  $ u :[0,T]\to \Hilbert$ satisfying $u(0)=u_0$
  and \eqref{eq:40}.
\end{theorem}
\begin{proof}
  Let $\mathscr P$ be a Borel representative of the weak$^*$ limit
  \eqref{bis-2converg}.
  It is easy to check that we can always suppose that 
  \begin{displaymath}
    P^-(t)=\liminf_{n\to\infty}\mathcal P_t(u_{\eps_n}(t))
    \le \mathscr P(t)\le
    \limsup_{n\to\infty}\mathcal
    P_t(u_{\eps_n}(t))=P^+(t)\quad\text{for every }t\in [0,T].
  \end{displaymath}
  On the other hand, using compactness of the energy sublevel $\sublevel \rho$ such that 
  $u_{\eps_n}([0,T]) \subset \sublevel \rho$ for all $n\in \N$, as well as the fact that
  all the limit points of $(u_{\eps_n}(t))_n$ belong to $\vlim(t)$,
  we clearly have
  \begin{equation}
    \label{eq:41}
    \min_{v\in \vlim(t)}\mathcal P_t(v)\le P^-(t)\le P^+(t)\le
    \max_{v\in \vlim(t)}\mathcal P_t(v).
  \end{equation}
  Since $\vlim(t)$ is connected, for every $t\in [0,T]$
  \begin{equation}
    \label{eq:42}
    \mathrm V^*(t):=
    \Big\{v\in \vlim(t):\mathcal P_t(v)=\mathscr P(t)\Big\}
    \quad\text{is compact and nonempty}.
  \end{equation}
  The corresponding set $\mathbf V^*\subset [0,T]\times \Hilbert$
  defined by
  \begin{equation}
    \label{eq:43}
    \mathbf V^*:=\Big\{(t,v)\in \Gv: \mathcal P_t(v)=\mathscr P(t)\Big\}
  \end{equation}
  is Borel. Since its sections are compact and nonempty,
  the  Arsenin--Kunugui theorem
  (cf.\ \cite[Thm.\ 6.9.6]{Bogachev07.II})
  yields the existence of a Borel 
  map $u $ 
  with the desired properties.
\end{proof}
\subsection{Convergence of Singularly perturbed gradient flows:
proof of Claims 3,4,5 of Theorem \ref{mainth:1}}
\label{subsec:end1}

Claim (3) is an immediate consequence of the fact
that the limit $u$ satisfies
\begin{equation}
  \label{eq:44}
  u(t)\in \vlim(t)\quad\text{for every }t\in D,\quad
  \mathcal P_t(u(t))=\mathscr P(t)\quad\text{a.e.~in }D.
\end{equation}
Claim (4) then follows by the fact that
$u_{\eps_n}(t)$ converges pointwise
for every $t\in [0,T]\setminus(\nototaldisc\cup\rmJ)$.
Since 
$D\cap( \nototaldisc\cup\rmJ)$
is countable, we can further extract a
subsequence
pointwise converging in $D$.

Claim (5) is an obvious consequence.

 \section{Proof of Theorem \ref{mainth:2}}
 \label{s:5}
The proof  that \DSolutions\ in fact improve to Balanced Viscosity under the rectifiability condition 
\eqref{eq:15} on $\crit$
will rely  both on     known measure-theoretic  results, reviewed below for the reader's convenience,
 and  on  new ones. 
\par
We start by recalling the  \emph{area formula} for absolutely continuous functions (cf., e.g., \cite[Thm.\ 3.4.6]{Ambrosio-Tilli}).
\begin{theorem}
Let $\varphi: I \to \R$ be absolutely continuous, and let $h:I \to \R$ be a non-negative Borel map. Then,
\begin{equation}
\label{area-formula}
\int_I h(s) |\varphi'(s)| \dd s = \int_{\R} \left( \sum_{s:\, \varphi(s) =r} h(s)\right) \dd r \,.
\end{equation}
\end{theorem}
\par
A measurable function
$f: [a,b]\to \R$ is said to fulfill the \emph{Lusin property} if it satisfies
\begin{equation}
\label{Lusin}
\mathscr{L}^1(f(\mathcal{O})) =0 \qquad \text{for every Borel set } \mathcal{O}\subset [a,b] \text{ with } 
\mathscr{L}^1(\mathcal{O}) =0.
\end{equation}
 We recall the \emph{Banach-Zaretskii Theorem}.  
\begin{theorem}
\label{thm:Lusin}
Let $f:[a,b]\to \R$ be a continuous function with bounded variation. Suppose that
$f$ satisfies the Lusin property \eqref{Lusin}. 
Then, $f\in \AC ([a,b])$.
\end{theorem}
\par
Our first preliminary result provides an extension of 
Theorem \ref{thm:Lusin} to functions that are only with bounded variation.
\begin{proposition}[Extension of the Lusin property to $\BV$ functions]
\label{prop:extLus}
Let $f \in \BV([a,b])$ satisfy \eqref{Lusin}. Then, the Cantor part of the distributional derivative of $f$ is null.
\end{proposition}
 \begin{proof}
 Let $\jump f$ denote the jump set of $f$, let 
  $\nu = \frac{\dd}{\dd t} f$ be  its distributional derivative with jump part $j: = \nu\res\jump f $, and let $g: [a,b]\to \R$ have distributional derivative equal to $j$. For instance, we may take
  \[
  g(t): = \sum_{\tau \in \jump f} \left( f(\tau{+}){-} f(\tau{-})\right) H(t-\tau) 
  \] 
  with $H$ the Heaviside function. Clearly, the function
  $\tilde{f}: = f-g$ is continuous on $[a,b]$, with distributional derivative
  $
  \frac{\dd}{\dd t} \tilde{f} = \nu-j
 $. In particular,  the Cantor part of $\nu$
 coincides with the Cantor part of $\frac{\dd}{\dd t} \tilde{f}$. In order to show that the latter vanishes, in view of Theorem \ref{thm:Lusin} it is sufficient to prove that $\tilde f$
 complies with  \eqref{Lusin}. 
 With this aim, let us introduce the image set
 \[
\imsetG: =  
 g([a,b]) .
 \]
 We observe that $\imsetG$ is countable and that 
 $\tilde{f}(t) \in f(t)-\imsetG$ 
 for all $t\in [a,b]$. 
 Therefore, for every $\mathscr{L}^1$-negligible set
$\mathcal{O}\subset[a,b]$ 
 there holds
 \[
 \tilde{f}(\mathcal{O}) 
 \subset
 f(\mathcal{O})-\imsetG = \cup_{r\in \imsetG} \left( f(\mathcal{O})-r\right)\,.
 \]
 By assumption, $f$ satisfies \eqref{Lusin}, and thus 
 $ f(\mathcal{O})$
 is $\mathscr{L}^1$-negligible. Since $\imsetG$ is countable, we conclude that 
 $\tilde{f}(\mathcal{O})$ 
 is
 $\mathscr{L}^1$-negligible, as well. 
 \end{proof}
 \begin{remark}
 \slshape
 \label{rmk:Lusin}
 For later use we point out that, in practice, 
 for any $f\in \BV([a,b])$
 it is sufficient to check \eqref{Lusin} for any $\mathscr{L}^1$-negligible 
subset of $[a,b]\setminus \jump f$,
 since $\jump f$ is countable and then $f(\jump f) $ is countable
 (hence $\mathscr{L}^1$-negligible), as well.
 \end{remark}
 \par
 Let us now fix an energy sublevel $\subl\rho$. Recall that, by Lemma \ref{l:E-Lip}, the functional  $\calE$ is Lipschitz on
 $\crit  \cap \Subl\rho$.  Hence, there exists
 \begin{equation}
 \label{hat-calE}
  \text{a Lipschitz function 
 $\widehat\calE: [0,T]\ti \Hilbert \to \R$ such that  }  \widehat\calE|_{\crit  {\cap} 
\ti \subl\rho} = \calE.
 \end{equation}
 We omit to explicitly denote the dependence of $\widehat\calE$ on $\rho>0$. 
Our next result provides a suitable chain-rule formula for $\widehat{\calE}$ along Lipschitz curves. 
 In what follows,
 for a curve $\gamma=(\gamma_0,\gamma_1)$ with values in $[0,T]\ti \Hilbert$, 
  with slight abuse of notation we will write $\widehat{\calE}(\gamma(s))$ in place of
  $\widehat{\calE}_{\gamma_0(s)}(\gamma_1(s))$. 
 \begin{lemma}
 \label{l:chain-rule-critset}
 Let $\gamma = (\gamma_0,\gamma_1): [a,b]\to [0,T]\ti \Hilbert$ be a Lipschitz curve, and let 
 \begin{equation}
 \label{setD}
 D:= \{s \in [a,b]\, : \ \gamma(s) \in  \crit  \cap 
\ti \subl\rho \}\,.
 \end{equation}
 Then, there holds
 \begin{equation}
 \label{ch-rule-gamma}
 \frac{\dd}{\dd s} \widehat{\calE}(\gamma(s))  = \pt{\gamma_0(s)}{\gamma_1(s)} \gamma_0'(s) \qquad \foraa\, s \in D. 
 \end{equation}
 \end{lemma}
 \begin{proof}
 Since the subset $D_{\mathrm{iso}}\subset D$ consisting of the isolated points of $D$ is at most countable,
 it is not restrictive to prove \eqref{ch-rule-gamma} for $\mathscr{L}^1$-a.a.\ $s\in D_{\mathrm{acc}} := D\setminus D_{\mathrm{iso}}$. Since $\widehat{\calE} \circ\gamma$ 
 is Lipschitz and $\gamma_0$ is Lipschitz, setting
 \[
 \widetilde{D}_{\mathrm{acc}}: = \{ s \in D_{\mathrm{acc}} \, : \  \widehat{\calE} \circ\gamma \text{ and } \gamma_0 \text{ are differentiable at } s \},
 \]
 formula \eqref{ch-rule-gamma} follows if we prove that 
 \begin{equation}
 \label{2show4ch-rule-g}
 l:= \lim_{r\to s, \ r \in  D_{\mathrm{acc}}}
 \frac{ \widehat{\calE}(\gamma(r)) -  \widehat{\calE}(\gamma(s))}{r-s} 
 =  \pt{\gamma_0(s)}{\gamma_1(s)} \gamma_0'(s) \qquad \text{for all } s \in  \widetilde{D}_{\mathrm{acc}}
 \end{equation}
 (observe that, in the calculation of the above limit we can restrict to points
 $r\to s$ such that 
  $r\in D_{\mathrm{acc}} $, since we already know that $l$ exists). 
 Now, it follows from estimate \eqref{GEST-1} that 
 \begin{equation}
 \label{from-gest1}
 \begin{aligned}
 &
 \begin{aligned}
 \widehat{\calE}(\gamma(r)) -  \widehat{\calE}(\gamma(s)) &  \leq\int_{\gamma_0(s)}^{\gamma_0(r)} \pt{\tau}{\gamma_1(s)} \dd \tau + 
 \frac{\lambda}{2} \|\gamma_1(r){-}\gamma_1(s) \|^2 
 \\
 & 
\leq  \pt{\gamma_0(s)}{\gamma_1(s)} (\gamma_0(r){-}\gamma_0(s)) + \mathrm{o}(\|\gamma_0(r)-\gamma_0(s)\|) + \frac{\lambda}{2}\mathrm{Lip}(\gamma)^2|r-s|^2 \quad \text{as } r \to s,
\end{aligned}
\\
&
\begin{aligned}
  \widehat{\calE}(\gamma(r)) -  \widehat{\calE}(\gamma(s)) &  \geq \int_{\gamma_0(s)}^{\gamma_0(r)} \pt{\tau}{\gamma_1(r)} \dd \tau -
 \frac{\lambda}{2} \|\gamma_1(r){-}\gamma_1(s) \|^2 
 \\
 & 
 \geq  \pt{\gamma_0(r)}{\gamma_1(r)} (\gamma_0(r){-}\gamma_0(s)) + \mathrm{o}(\|\gamma_0(r)-\gamma_0(s)\|) - \frac{\lambda}{2}\mathrm{Lip}(\gamma)^2|r-s|^2
   \quad \text{as } r \to s,
   \end{aligned}
      \end{aligned}
 \end{equation}
 with $\mathrm{Lip}(\gamma)$ the Lipschitz constant of $\gamma$. Let us now assume that $s$ is a left-accumulation point for $\widetilde{D}_{\mathrm{acc}}$, namely that $(s,s+\delta) \cap \widetilde{D}_{\mathrm{acc}} \neq\emptyset$ for every $\delta>0$. Thus,  we can consider  \eqref{from-gest1} for
 $r\down s$ and  divide it by
  $(r-s)$. Passing to the limit as $r\down s$  in both estimates in
  \eqref{from-gest1}
   we obtain
  \[
  l\leq \pt{\gamma_0(s)}{\gamma_1(s)} \gamma_0'(s) \quad\text{and}\quad  l\geq \pt{\gamma_0(s)}{\gamma_1(s)} \gamma_0'(s)\,.
  \]
With the very same arguments we obtain \eqref{2show4ch-rule-g} also if $s$ is a right accumulation point for 
 $\widetilde{D}_{\mathrm{acc}}$. This concludes the proof.
 \end{proof}
 \par
 Let us now consider a \DSolution\  $u:[0,T]\to \Hilbert$ to \eqref{limit-problem}. Recall that the function 
$e: [0,T]\to\R$, $e(t): = \ene t{u(t)}$ is in $\BV([0,T])$. 
As a corollary of Lemma \ref{l:chain-rule-critset}, we have the following result.
\begin{lemma}
\label{cor-l-ch-rule-gamma}
Let $\gamma = (\gamma_0,\gamma_1): [a,b]\to [0,T]\ti \Hilbert$ be a Lipschitz curve, and let 
$\mathcal{O}\subset [0,T]\setminus \jump{e}$
be a $\mathscr{L}^1$-negligible set.
Let us define
\begin{equation}
\label{sets-Agamma+Bgamma}
A[\gamma]: = \{s \in D\, : \ \gamma_0(s)
\in \mathcal{O}
\}, \qquad B[\gamma]: = (\calE{\circ}\gamma)(A[\gamma])= \{ \ene{\gamma_0(s)}{\gamma_1(s)}\, : \ s \in D, \ \gamma_0(s)
\in \mathcal{O}
\}
\end{equation}
with $D$ from \eqref{setD}.
Then, $B[\gamma]$ is $\mathscr{L}^1$-negligible.
\end{lemma}
\begin{proof}
Observe that 
\begin{equation}
\label{coincidence-energies}
(\calE{\circ}\gamma)(s)= (\widehat{\calE}{\circ}\gamma)(s)\qquad \text{ for every $s\in A[\gamma]$,}
\end{equation}
 since, by definition  of the set $D$ (cf.\ \eqref{setD}), 
$\gamma(s) \in \crit \cap \ti\Subl\rho$ for all $s\in A[\gamma]$ and $\widehat{\cE}$ coincides with $\cE$ on $ \crit \cap \ti\Subl\rho$. Now, since
$\widehat{\calE}{\circ}\gamma$ is a Lipschitz function, 
by the area formula \eqref{area-formula}
we have
\begin{equation}
\label{area4energy}
\int_a^b h(s) |(\widehat{\calE}{\circ}\gamma)'(s)| \dd s = \int_{\R} \Big( \sum_{s\,:\  \widehat{\calE}(\gamma(s)) =r}  h(s) \Big)\dd r \,.
\end{equation}
Then, we  find
\[
\begin{aligned}
\mathscr{L}^1(B[\gamma]) = \mathscr{L}^1((\calE{\circ}\gamma)(A[\gamma])) 
&
\stackrel{(1)}{=}  \mathscr{L}^1((\widehat{\calE}{\circ}\gamma)(A[\gamma])) 
\\ & 
\leq \int_{\R}\left(  \sum_{s\, : \ \widehat{\calE}(\gamma(s))=r}  \chi_{A[\gamma]}(s) \right) \dd r
\\
&
\stackrel{(2)}{=}    \int_a^b   \chi_{A[\gamma]}(s) |(\widehat{\calE}{\circ}\gamma)'(s)| \dd s  \\ &  = \int_{A[\gamma]}  |(\widehat{\calE}{\circ}\gamma)'(s)|  \dd s 
\\
& 
\stackrel{(3)}{=} \int_{A[\gamma]} |\pt{\gamma_0(s)}{\gamma_1(s)}| |\gamma_0'(s)| \dd s \stackrel{(4)}{=} 0, 
\end{aligned}
\]
where (1) derives from \eqref{coincidence-energies}, 
(2)  follows from choosing
 $h(s) =  \chi_{A[\gamma]}(s)$ in \eqref{area4energy},  (3) is due to the chain-rule formula \eqref{ch-rule-gamma}, and  (4) is due to the fact that $\gamma_0'(s) =0$ for almost all $s\in A[\gamma]$: indeed,
applying the area formula \eqref{area-formula} to $\gamma_0|_{A[\gamma]}$ gives
$\int_{A[\gamma]}|\gamma_0'(s)|\dd s = \int_{\mathcal{O}} \#\{s\in A[\gamma]:\gamma_0(s)=r\}\dd r =0$,
since $\gamma_0(A[\gamma])\subset \mathcal{O}$ and 
 the set $\mathcal{O}$
is $\mathscr{L}^1$-negligible. This concludes the proof.
\end{proof}
\par
In the proof of the  next result we eventually resort to the rectifiability requirement \eqref{eq:15}.  
\begin{corollary}
\label{cor:null-Cantor}
Under the assumptions of Theorem \ref{mainth:2} or of Proposition 
\ref{prop:BV}, let $u:[0,T]\to \Hilbert$  be a  Dissipative  Viscosity solution to \eqref{limit-problem}. 
 Then, the Cantor part of the distributional derivative of the function $t\mapsto e(t) = \ene t{u(t)}$ is null. 
\end{corollary}
  \begin{proof}
  We consider first the case of
  Theorem \ref{mainth:2}.   
  We shall prove the assertion by applying Proposition \ref{prop:extLus} to the function
  $e$. Taking into account Remark \ref{rmk:Lusin} 
(in fact, we may remove from $[0,T]$ also the \emph{countable} set $\nototaldisc$ from \eqref{eq:3}), 
  it is sufficient to show that
  \begin{equation}
  \label{null-2-prove}
\text{  for every $\mathscr{L}^1$-negligible Borel set 
$\mathcal{O}\subset[0,T]\setminus  ( \jump e{\cup} \nototaldisc) $ 
there holds
$\mathscr{L}^1(e(\mathcal{O}))=0$.
}
\end{equation}
 In fact, recall that, in the present context, for  our \DSolution\ 
there holds $\disc u\subset  \jump e{\cup} \nototaldisc $, 
cf.\ Proposition \ref{prop:improved-C}.  Thus, by
 restricting to sets $\mathcal{O}\subset[0,T]\setminus ( \jump e{\cup}\nototaldisc ) $, we will ensure that  $\mathcal{O}\subset[0,T]\setminus \disc u$, which we shall use below.

 Now, since $\crit$
   is countably  
$\mathscr H^1$-rectifiable
   by definition  (cf.\ \eqref{rectifiable}),
   there exist countably many Lipschitz curves $\gamma_k = (\gamma_{k,0},\gamma_{k,1}) :[0,1]\to [0,T]\ti \Hilbert $  such that 
  \begin{equation}
  \label{rectifiability}
\crit = \bigcup_{k=1}^\infty \gamma_k([0,1])  \cup C_0 \qquad \text{with } \mathscr{H}^1(C_0)=0.
  \end{equation}
 In accord with  the notation from Lemma \ref{cor-l-ch-rule-gamma}, we set 
  \[
A[\gamma_k]: = \{s \in [0,1]\, : \ \gamma_{k,0}(s)
\in \mathcal{O},
\, \ \gamma_k(s) \in \crit  \cap\Sublevel\rho\}, \qquad B[\gamma_k]: = (\widehat\calE{\circ}\gamma_k)(A[\gamma_k])\,.
  \]
  Due to \eqref{rectifiability} and  also observing that, for all 
$t\in \mathcal{O}  \subset[0,T]\setminus  ( \jump e{\cup} \nototaldisc )  \subset [0,T]{\setminus}\disc u$, we have 
  $(t,u(t)) \in \crit \cap \subl\rho$
  for some  $\rho>0$ (cf.\ Remark \ref{energy-bound}),
  we have
  \[
  \{ \ene t{u(t)}\, : \ 
t \in \mathcal{O}
\} \subset  \bigcup_{k=1}^\infty B[\gamma_k]  \cup \calE(C_0 {\cap} \Subl\rho).
  \]
  On the one hand, since $C_0 \cap \Subl\rho \subset  \crit \cap \Subl\rho$  and the restriction of $\calE$ to $ \crit \cap\Subl\rho$ coincides with the   Lipschitz function $\widehat\calE$ (with Lipschitz constant $\mathrm{Lip}_{\widehat\calE}$),  we have  
\[
\mathscr{L}^1( \calE(C_0 {\cap} \Subl\rho)) \leq \mathrm{Lip}_{\widehat\calE} (\mathscr{H}^1(C_0)) =0\,.
\]
  On the other hand, 
  by Lemma \ref{cor-l-ch-rule-gamma} we have that $\mathscr{L}^1(B[\gamma_k]) =0$ for every $k\in \N$. Then, 
  $\mathscr{L}^1(\{ \ene t{u(t)}\, : \ t \in \mathcal{O} \} ) =0$,
  whence \eqref{null-2-prove}. 
  This concludes the proof.
  The case when $u$ is a function
  of bounded variation 
  follows by the same argument, by observing
  that the graph of $u$ is contained
  in the graph of a Lipschitz curve.
  \end{proof} 
  
  We are now in a position to carry out the
\paragraph{\textbf{Proof of Theorem \ref{mainth:2}}}
  Indeed, it is sufficient to show that, for any \DSolution\  $u:[0,T]\to \Hilbert$, the function $t\mapsto e(t)= \ene t{u(t)}$ satisfies 
    \begin{equation}
  \label{AC-derivative}
  e'(t) = \pt t {u(t)} \qquad \text{for } \mathscr{L}^1\text{-a.a. } t \in (0,T).
  \end{equation}
  This yields that the absolutely continuous part of the distributional derivative  of the function $e$ is given by $ \pt t {u(t)}\mathscr{L}^1$. Since its Cantor part is null 
 by Corollary \ref{cor:null-Cantor},  we conclude that the measure $\mu$ from 
 \eqref{to-save-eneq} is purely atomic, with
 $\mu(\{t\}) = \cost t {\ulim(t-)}{\ulim(t+)}$ for all $t\in \jump e$ by \eqref{to-save-4}. 
 \par
 In order to prove \eqref{AC-derivative}, let us define
 \[
 I: = \{ t \in [0,T]{\setminus}\jump e\, : \ \exists\, \lim_{t'\to t} \frac{e(t')-e(t)}{t'-t}  = \lim_{t'\to t} \frac{\ene{t'}{u(t')} - \ene t{u(t)}}{t'-t}  = \ell_t\in \R \}\,.
 \]
 Clearly, we have $\mathscr{L}^1([0,T]{\setminus}I) =0$.
Let us set
\[
D_k : = \{ s \in [0,1]\, : \gamma_k(s) \in \crit \cap \subl\rho, \ \gamma_{k,0}(s) \in I, \ u(\gamma_{k,0}(s)) = \gamma_{k,1}(s) \}\,,
\]
where we use the same notation as in Corollary \ref{cor:null-Cantor}.
We observe that 
\[
\bigcup_{k=1}^\infty \gamma_{k,0}(D_k) \subset I \subset \bigcup_{k=1}^\infty \gamma_{k,0}(D_k)  \cup \pi_{[0,T]}(C_0)
\]
(recall that $\pi_{[0,T]}$ denotes the projection on $[0,T]$),
 where the second inclusion follows from the fact
 that $I \subset \pi_{[0,T]}(\crit)$ (since for $t\in I\subset [0,T]\setminus \jump e$ the solution satisfies $u(t)\in \critset(t)$, hence $(t,u(t))\in \crit$) and from
\eqref{rectifiability}. 
 Since $\mathscr{H}^1(C_0) =0$, we have that 
 $\mathscr{L}^1( \pi_{[0,T]}(C_0)) =0$, so that 
 \begin{equation}
 \label{neglig-1}
 \mathscr{L}^1 \left(I {\setminus}  \bigcup_{k=1}^\infty \gamma_{k,0}(D_k) \right)=0.
 \end{equation}
 Setting
 \[
 D_k': = \{s \in D_k\, : \ \gamma_{k,0} \text{ is differentiable at } s \},
 \]
 we have that 
 $\mathscr{L}^1(D_k{\setminus}D_k') =0$. Since $\gamma_{k,0}$ is Lipschitz, we get 
  \[
 \mathscr{L}^1(\gamma_{k,0}(D_k){\setminus}\gamma_{k,0}(D_k')) =0 \qquad \text{for all } k \in \N\,.
 \]
 Combining this with \eqref{neglig-1}, we conclude that 
  \begin{equation}
 \label{neglig-2}
 \mathscr{L}^1 \left(I {\setminus}  \bigcup_{k=1}^\infty \gamma_{k,0}(D_k') \right)=0.
 \end{equation}
 \par
 Eventually, we set
 $D_k'': =  \{ s\in D_k'\, : \ \gamma_{k,0}'(s) \neq 0\}$. The area formula \eqref{area-formula}  yields
 $\mathscr{L}^1(\gamma_{k,0}(D_k'{\setminus}D_k'')) =0$.
 Hence, setting
 $
 T_k: = \gamma_{k,0}(D_k''),
 $
 from \eqref{neglig-2} we infer
   \begin{equation}
 \label{neglig-3}
 \mathscr{L}^1 \left(I {\setminus}  \bigcup_{k=1}^\infty T_k \right)=0.
 \end{equation}
 Furthermore, let $T_k'$ denote the set of points of $T_k$ with Lebesgue density $1$.
Since $\mathscr{L}^1(T_k{\setminus}T_k')=0$, we ultimately infer
   \begin{equation}
 \label{neglig-4}
  \mathscr{L}^1 \left(I {\setminus}  \bigcup_{k=1}^\infty T_k' \right)=0. 
 \end{equation}
Ultimately, it is sufficient to prove formula \eqref{AC-derivative} at every $t\in T_k'$, for every $k\in \N$.
To this end, for every $t\in T_k'$ we write
$(t,u(t)) = \gamma_k(s)$ for some $s\in D_k''$
and we may assume that $s$ is an accumulation point of $D_k''$
(the set of isolated points of $D_k''$ is at most countable, hence
its image under $\gamma_{k,0}$ is negligible).
We then select a
sequence $(s_j)_j$ in $D_k''$ converging to $s$.
Setting $t_j:=\gamma_{k,0}(s_j)$, we have $t_j\to t$ and
$t_j\neq t$ for $j$ large enough, since
$\gamma_{k,0}'(s)\neq 0$. Since $e'(t)=\ell_t$ exists, we can compute
\[
\begin{aligned}
\ell_t &=
\lim_{j\to\infty} \frac{e(t_j) - e(t)}{t_j-t}
\stackrel{(1)}{=}
\lim_{j\to\infty}  \frac{\ene{\gamma_{k,0}(s_j)}{\gamma_{k,1}(s_j)} - \ene {\gamma_{k,0}(s)}{\gamma_{k,1}(s)}}{s_j-s} \cdot \frac{s_j-s}{\gamma_{k,0}(s_j)-\gamma_{k,0}(s)}
\\
&
\stackrel{(2)}{=}  \pt{\gamma_{k,0}(s)}{\gamma_{k,1}(s)}  \gamma_{k,0}'(s) \cdot \frac1{ \gamma_{k,0}'(s) } =   \pt{\gamma_{k,0}(s)}{\gamma_{k,1}(s)} =\pt t{u(t)},
\end{aligned}
\]
 where (1) follows
 from the fact that $e(t_j) = \ene{\gamma_{k,0}(s_j)}{\gamma_{k,1}(s_j)}$
 (since $s_j\in D_k''$ gives $u(t_j)= \gamma_{k,1}(s_j)$),
 while  (2) derives from  \eqref{ch-rule-gamma} and the fact that $ \gamma_{k,0}'(s)\neq 0$ since $s\in D_k''$.\EEE 
  \par
  All in all, we have obtained \eqref{AC-derivative}.
 This proves that every \DSolution\ is a \BSolution, i.e.\ Claim (2)
  up to the countability of $\disc u$.
  To check this, we observe that
  by Proposition \ref{prop:improved-C}
  the solution $u$ is continuous at every
  $t\in \totaldisc\setminus \rmJ$, so that
  $\disc u \subset \rmJ \cup \nototaldisc$.
  Since $\rmJ$ is at most countable
  (being the atomic set of a finite measure)
  and $\nototaldisc$ is at most countable by
  Lemma \ref{le:criterion}(3),
  we conclude that $\disc u$ is at most countable.
  \par
  Claim (1) follows from
Theorem \ref{mainth:1}(5): 
  indeed, by Lemma \ref{le:criterion}(3),
  the rectifiability condition \eqref{eq:15} implies
  \ass1 and that the set $\nototaldisc$ is at most countable,
  so that
 Theorem \ref{mainth:1}(5) 
  applies.
  \par
  Claim (3) is an immediate consequence of Claims (1) and (2).
 \QED

 \section{Further results under the transversality conditions}
 \label{sec:6}
 This section revolves around the relations between our standing assumption \ass1 on the critical set $\crit$, 
 as well as the countability of $\nototaldisc$, and the transversality conditions. To explore them,
 we will need to restrict the class of admissible energies $\cE$, reverting to differentiable functionals. What is more, 
 we will suppose that the  Fr\'echet  \emph{differential} $\rmD\calE_t$ is a  \emph{Fredholm map}. In this context, we will introduce the transversality conditions at the 
 (degenerate) critical points of $\cE$. 
 \par
 With the main result of this section, Theorem \ref{prop:7.1}, we will show that the  transversality conditions guarantee the validity of the rectifiability property 
  \eqref{eq:15} for $\crit$ and thus, ultimately, of condition  \ass1 and  the countability of $\nototaldisc$.
  \par
  Finally, we will show that, under an additional transversality condition, Dissipative/Balanced 
  Viscosity
  solutions enjoy enhanced properties
  that provide a finer description of their behaviour at jumps.

\subsection{Setup for the transversality conditions}
\label{s:7.0}
In this section we   settle some notation  and precisely state  the smoothness requirements on $\cE$  underlying  the transversality conditions.  
   \begin{notation}[Notation and preliminary definitions for Fredholm maps]
   \upshape
   \label{not:7.1}
Given a Banach space $\Xsp$,
we denote by $\mathcal{L}(\Xsp;\Ysp)$ the space of linear bounded
operators from $\Xsp$ to $\Ysp$.
The adjoint of an operator  $L\in \mathcal L(\Xsp;\Ysp)$ is denoted  by  $L^*$,
with   $L^* \in \mathcal L(\Ysp^*;\Xsp^*)$.
We will use the notation
\begin{equation}
  \label{eq:4}
  \text{kernel of $L$}:\
  \NN(L)=\big\{x\in \Xsp:Lx=0\big\},\quad
  \text{range of $L$}:\
  \RR(L)=\big\{Lx:x\in \Xsp\big\}.
\end{equation}
We recall that a  subspace $M\subset \Ysp$ has finite codimension
if there exists a finite-dimensional space $ S\subset \Ysp$ such
that $M+ S= \Ysp$.
In this case $M$ is closed and
we can always choose $ S$ so that $ S\cap
M=\{0\}$; the codimension of $M$ is then defined as
$\mathrm{codim}(M):=\mathrm{dim}( S)$.

We call an operator $L\in\mathcal L(\Xsp;\Ysp)$
 a \emph{Fredholm operator} if
$
\mathrm{dim}(\NN( L)) $ and $ \mathrm{codim}(\RR(L)) $  are finite.
In particular, the range of a Fredholm operator is closed.
The \emph{index} of the operator $L$ is defined as
$
  \ind(L):=\dim(\NN(L))-\codim(\RR(L)).
$
The stability theorem  for Fredholm operators shows that the collection of all Fredholm operators is open in $\mathcal{L}(\Xsp;\Ysp)$ and
the index  is a locally constant function.

Let $\opx\subset \Xsp$ be a connected open subset of $\Xsp$.
A map $f\in \mathrm{C}^1(\opx;\Ysp)$ is a Fredholm map if for every
$x\in \opx$ the differential
 $\mathrm{D} f(x)\in \mathcal L(\Xsp,\Ysp)$ is a Fredholm operator.
 The \emph{index} of $f$ is defined as the index of $\mathrm{D} f(x)$
 for some $x\in \opx$. By the stability theorem for Fredholm operators
and the connectedness of $\opx$,
this definition is independent of $x$.
   \end{notation}

Let us now enhance
our conditions 
 on the energy functional   $\calE\colon[0,T]\ti \Hilbert \to ({-}\infty,{+}\infty]$, with  domain $\mathrm{dom}(\cE) = [0,T]\ti\domainenergy
$, preparatory to the statement of the transversality conditions. 
Hereafter, we shall require that
   \begin{equation}
   \label{spaces}
   \begin{gathered}
     \text{there exist two separable Banach spaces }
     \Xsp \subset \Hilbert, \, \Ysp \subset   \Hilbert=\Hilbert^*\subset \Xsp^*   \text{ with continuous and dense inclusions}
    \end{gathered}
   \end{equation}
   (observe that under the above conditions, the duality pairing 
   $\pairing{\Xsp^*}{\Xsp}{y}{x}$ between $x\in \Xsp$ and $y \in \Ysp$
   coincides with the $\Hilbert$-inner product
   $\langle x, y \rangle $), such that
the energy functional  $\cE:[0,T]\ti\domainenergy\to \R$,
enjoys the following properties
\begin{subequations}
\label{basic-hyps-transv}
\begin{align}
\nonumber
&
  \text{there exists a connected and  open set } \opx \subset \Xsp \text{ such that } 
    \\
&
\label{basic-hyps-transv-a}
\forall\, (t, u) \in [0,T] \ti\domainenergy
\qquad
0\in\partial\ene tu\ \Rightarrow \ u \in \opx \text{ and } \partial\ene tu =\{0\},
\\
&
\label{basic-hyps-transv-b}
 \cE \in \mathrm{C}^2 ([0,T]\ti \opx)   \quad \text{ and }\quad      \gder{}\cE tu
\in \Ysp \text{ for every } (t,u) \in [0,T]\ti \opx,
\\
&
\label{basic-hyps-transv-c}
\forall\, t \in [0,T] \quad \gder {}\cE t{\cdot} : \opx \to \Ysp  \quad \text{is a Fredholm map of index $0$.}
\end{align}
 In particular, observe that \eqref{basic-hyps-transv-b} ensures that, when restricted to $[0,T]\ti \opx$, the 
Fr\'echet
differential of the energy w.r.t.\ $u$ takes values in the space $\Ysp$, possibly smaller that $\Hilbert$. In turn, \eqref{basic-hyps-transv-a} is a regularity result for  the critical points of $\calE$. 

\par
For later convenience, it will also be crucial to require that 
for every $\rho>0$ the set 
$\crit \cap \Sublevel \rho$ 
is  compact w.r.t.\ the topology of  $[0,T]\ti \Xsp$.
Taking into account that, by \eqref{coercivita},  the sublevels $\Sublevel\rho$ are compact in  $[0,T]\ti \Hilbert$, 
it is indeed equivalent  to impose that  on 
$\crit \cap \Sublevel \rho$ the topologies of  $[0,T]\ti \Hilbert$ and of  $[0,T]\ti \Xsp$ do coincide.
In practice, we shall require
that
  \begin{equation}
 \label{basic-hyps-transv-e}
 \begin{aligned}
 &
 \forall\, (t_n, u_n)_n, \, (t,u)  \subset \crit \cap \Sublevel \rho, 
 \qquad 
  u_n \to u  \text{ in } \Hilbert
  \ 
 \
  \Rightarrow \  \ u_n \to u \text{ in } \Xsp.
  \end{aligned}
 \end{equation}

  Finally, as in \cite[Sec.\ 3.2]{AgRoSa15}   (cf.\ also
  \cite[Remark
1.1]{Saut-Temam79}),
we will also suppose that
for every $(t,u)\in [0,T] \ti \opx$
the linear operator $L$ associated with the second
order differential $\gder 2\cE tu \in\mathcal{L}(\Xsp; \Ysp)$
admits a unique continuous
extension $\tilde L\in \mathcal{L}(\Hilbert;\Hilbert)$ satisfying
\begin{equation}
\label{eq:26}
\big\{h\in \Hilbert:\tilde L h\in \Ysp\big\}\subset \Xsp.
\end{equation}
It can be verified (cf.\  \cite[Sec.\ 3.2]{AgRoSa15}) that   $\tilde L$ is selfadjoint.
\end{subequations}

\begin{theorem}
  \label{thm:dim-kernel-ass1}
  In the setting of \eqref{spaces},
  let $\cE:[0,T]\ti \Hilbert \to (-\infty,+\infty]$ comply
  with \eqref{coercivita}--\eqref{P_t},
  \eqref{basic-hyps-transv},
  and suppose that
  \begin{equation}
    \label{eq:dim-kernel-le1}
    \mathrm{dim}(\NN(\gder 2\cE tu)) \leq 1
    \quad \text{for every } (t,u) \in \crit.
  \end{equation}
  Then \ass1  holds. 
  In particular, the conclusions of Theorem \ref{mainth:1} hold.
\end{theorem}
\begin{proof}
  Fix $t\in [0,T]$ and $\rho>0$.
  Let $u\in \critset(t)\cap \sublevel\rho$ be  a non-degenerate critical point 
  (i.e., $\gder 2\cE tu$ invertible): the implicit function theorem
  ensures that $u$ is an isolated point of $\critset(t)$.
  At every degenerate point $u\in \critset(t)\cap \sublevel\rho$
  with $\mathrm{dim}(\NN(\gder 2\cE tu))=1$,
  the Lyapunov--Schmidt reduction (cf.\ \cite[Sec.~I.2]{Kielhoefer12}) shows that
  $\critset(t)$ near $u$ is the image of a closed subset of $\R$
  under a $\rmC^1$-map into $\Xsp$
  (see the proof of Theorem \ref{prop:7.1}(1) below for more details);
  being $\critset(t)$ near $u$ 
  the Lipschitz image of a subset of $\R$,
  it has zero $\mathscr H^2$-measure.
  Since $\critset(t)\cap \sublevel\rho$ is compact in $\Xsp$
  by \eqref{basic-hyps-transv-e} and \eqref{coercivita},
  it can be covered by finitely many such neighborhoods, giving
  $\mathscr H^2(\critset(t)\cap \sublevel\rho)=0$.
  Letting $\rho\uparrow +\infty$,
  we obtain \eqref{ass1'} (i.e,  $\mathscr{H}^2\big(\critset(t)\big)=0$), and in turn \ass1.
\end{proof}
\EEE


\subsection{The transversality conditions and structural properties of the critical set}
\label{s:7.1}
Let us now introduce  the
\emph{transversality} conditions,
concerning the properties of the energy $\cE$ at its
\emph{degenerate} (or \emph{singular}) critical points, i.e.
the points in the set
\begin{equation}
\label{singular-points}
\sing = \{ (t,u) \in \crit\, : \ \gder 2{\cE} tu \in \mathcal{L}(\Xsp;\Ysp)
\quad \text{is non-invertible}\}
\end{equation}
(accordingly, we will denote by $\singset(t)$
the fibre of $\sing$ at $t\in [0,T]$).
As observed in \cite[Sec.\ 3.2]{AgRoSa15}, by  \eqref{eq:26}, the kernels of
$\gder 2{\cE} tu$ and $(\gder 2{\cE} tu)^*$
can be identified.
In this setup,  the transversality conditions read as follows. 
\begin{definition}[Transversality conditions] 
\label{def:transv}
In the setting of \eqref{spaces}, let $\cE \colon [0,T] \ti E \to \R$ comply with \eqref{basic-hyps-transv}. We say that 
$\cE$ satisfies 
the transversality conditions at a point $(t_0,u_0) \in \sing$ if 
  \begin{enumerate}[label=\rm ({T}.\arabic*)]
\item
\label{assT1}
  $\mathrm{dim}(\NN (\gder 2\cE {t_0}{u_0}))=1$;
\item
\label{assT2}
  If $0\neq v\in \NN(\gder 2\cE {t_0}{u_0})$
  then $\pairing{\Xsp^*}{\Xsp}{\partial_t
  \gder{}{\cE}{t_0}{u_0}} {v}  = \langle \partial_t
  \gder{}{\cE}{t_0}{u_0}, v \rangle
  \neq 0$.
\end{enumerate}
The functional $\cE$ satisfies the
transversality conditions if
\ref{assT1}--\ref{assT2} hold at every $(t_0,u_0)\in \sing$.
    \end{definition}
 \par
 Observe that,   \eqref{eq:dim-kernel-le1} in Theorem \ref{thm:dim-kernel-ass1} is, indeed, equivalent to the requirement that
      $  \mathrm{dim}(\NN(\gder 2\cE tu)) = 1$ 
for every $ (t,u) \in \sing$, namely to the validity of the sole $\rm ({T}.1)$. 
    \par
The main result of this section states that, under the partial transversality conditions  $\rm ({T}.1)$--$\rm ({T}.2)$, for every $t\in [0,T]$
the connected components of $\critset(t)$ are either points,
or curves. What is more, we show that the critical set complies with condition \eqref{intersection-sublevels},
namely that for every $\rho>0$  the set $\crit\cap \Sublevel\rho $ (recall  Remark \ref{rmk:suff-subl}), is countably 
$\mathscr H^1$-rectifiable.
\begin{theorem}
\label{prop:7.1}
In the setting of  \eqref{spaces}, 
let $\cE:[0,T]\ti \Hilbert \to (-\infty,+\infty]$  comply  
 with \eqref{coercivita}--\eqref{P_t}, 
 \eqref{basic-hyps-transv},
and  with the
transversality
conditions.
\par
Then, for every $t\in [0,T]$, we have:
\begin{enumerate}
\item For all  $u \in \singset(t)$,
the connected component
$\pij ut$
of $\critset(t)$ containing $u$ 
\begin{itemize}
\item[-]
either reduces to a point,
\item[-]
 or
there exist
$R>0$, $\delta>0$, and a  $\rmC^1$-curve 
$\curv u :I_{\delta} \to \Xsp$
such that
\begin{equation}
\label{statement-1.7.1}
\critset(t) \cap B_R(u) = \pij ut \cap B_R(u)
=  \{ \curv u(s)\, : \ s \in I_{\delta} \}.
\end{equation}
Here, the interval $I_{\delta}$  can be either
$(-\delta,\delta)$ or $(-\delta,0]$ or $[0,\delta)$.
\end{itemize}
\item
For every  
$\rho>0$
there exists a finite number $M_\rho$ of   $\rmC^1$-curves
$\gamma_i=(\curv t_i,\curv u_i)
:[0,1]\to[0,T]\ti\opx$
such that
\begin{equation}
\label{eq:strutt_C}
\crit\cap \Sublevel\rho=\bigcup_{i=1}^{M_\rho}\gamma_i([0,1])
\qquad\mbox{and}\qquad
\gamma_i'(s)\neq0
\quad\mbox{for every}\quad s\in[0,1],\ i=1,\hdots,M_\rho.
\end{equation}

\end{enumerate}

\end{theorem}

\begin{proof}
$\vartriangleright (1)$ As for the first item of the statement, the argument developed in the proof of
\cite[Theorem 2.5]{AgRoSa15} shows that for every
$(t,u)\in \sing$ there exist
a neighborhood
$(t-\delta,t+\delta){\times}B_R(u)$
of $(t,u)$  in the $[0,T]{\times}\Xsp$-topology,   an open interval $I \ni 0$,
and a $\rmC^1$-curve
$\parcur=(\curv t,\curv u) : I \to [0,T]\ti \Xsp$
such that
\begin{equation}
\label{from-ars1-1}
\crit\cap\big((t{-}\delta,t{+}\delta){\times}B_R(u)\big)
= \parcur(I),
\quad
\parcur(0)
= (t,u),
\quad
\curv t'(0) =0,
\quad
\parcur'(s)\neq (0,0) \quad \text{for all }  s\in I.
\end{equation}
In particular, we have that
\begin{equation}
\label{eq:equiv_ins_conn}
\{t\}\ti\big(B_R(u)\cap\critset(t)\big)
=\parcur(I_t)=\{(t,\curv u (s))\, : s \in I_t\},
\end{equation}
where
$0\in I_t := \{ s\in I\, : \ \curv t(s) = t\}$.
Observe that $\parcur(I_t)$ is connected.
In particular, the first identity in \eqref{eq:equiv_ins_conn}
simplifies to
\begin{equation*}
\{t\}\ti\big(B_R(u)\cap\pij ut \big)
=\parcur(I_t).
\end{equation*}
Now, if $\pij ut$
does not reduce to a point,
taking into account that $\parcur$ is a homeomorphism we ultimately conclude that,
up to a reparameterization of $\curv u$,
$I_t$ is
either $(-\delta,\delta)$ or $(-\delta,0]$ or $[0,\delta)$, for some $\delta>0$.
Hence, \eqref{statement-1.7.1} follows.
\par
$\vartriangleright  (2)$
In order to prove the second part of the statement, first of all we 
note that the same argument employed above for the proof of item $(1)$, coupled with compactness, implies that
for any connected component $\Conn$ of $\crit\cap \Sublevel\rho$ 
there exists a  $\rmC^1$-curve $\gamma:[0,1]\to[0,T]\ti\Xsp$,  with
$\gamma'(s)\neq(0,0)$ for every $s\in[0,1]$, such that 
 $\Conn=\gamma([0,1])$.
\par
We will now show that the connected components of  $\crit \cap \Sublevel\rho$ are isolated w.r.t.\ the 
topology of the Hausdorff distance $\sfd_\Hausdorff$, cf.\ \eqref{eq:10}, 
between the subsets of  $[0,T]\ti \Xsp$.
Indeed, suppose by contradiction that there exist a connected component  $\Conn$  of $\crit\cap\Sublevel\rho$ 
and a sequence
$(\Conn_n)_n$ of connected components 
such that
\begin{equation*}
\sfd_\Hausdorff(\Conn_n,\Conn)\longrightarrow0
\qquad\mbox{as}\quad n\to\infty.
\end{equation*}
Then, for every  $(\hat t,\hat u)\in\Conn$ 
 there exists a sequence of points $(t_n,u_n)\in\Conn_n$
such that $(t_n,u_n)\to(\hat t,\hat u)$, as $n\to\infty$.
But then, again by the implicit function theorem we have that
$(t_n,u_n)\in\Conn$ for every $n$ sufficiently large, hence $\Conn_n \equiv \Conn$ for $n$ sufficiently large.
\par
Since $\crit \cap  \Sublevel\rho$ is compact w.r.t.\  topology of $ [0,T]\ti \Xsp$
thanks to the previously assumed \eqref{basic-hyps-transv-e}, we conclude that $\crit \cap  \Sublevel\rho$ only has 
a finite number of connected components, and thus \eqref{eq:strutt_C} follows.
This concludes the proof.
\end{proof}
Under the hypotheses of Theorem \ref{prop:7.1}, which in particular ensure \eqref{eq:strutt_C},  hence condition \eqref{intersection-sublevels},  condition \eqref{eq:15} is satisfied (see Remark \ref{rmk:suff-subl}). Consequently, Lemma \ref{le:criterion} yields \ass 1, as well as the countability of $\nototaldisc$. We therefore conclude that the hypotheses of Theorems \ref{mainth:1} and \ref{mainth:2} are fulfilled. Hence, we obtain the following

\begin{corollary}
\label{cor:obvious}
In the setting of  \eqref{spaces},
let $\cE:[0,T]\ti \Hilbert \to (-\infty,+\infty]$  comply
 with \eqref{coercivita}--\eqref{lambda-convex},
with  \eqref{basic-hyps-transv},
and  with the
transversality
conditions.
\par
Then,  all the conclusions of Theorem \ref{mainth:2} hold for $\cE$: every \DSolution\ is in fact a \BSolution\ with an at most countable jump set, and for every $u_0\in\domainenergy$ there exists a \BSolution\ emanating from $u_0$. In particular,  every sequence of
gradient flows $(u_{\eps_n})_n$  starting from initial data  $ (u_{0,\eps_n})_n $ as in \eqref{energy-convergence-initial}
converges pointwise  in $[0,T]$,  up to a subsequence,   to a \BSolution\ $u $ to \eqref{limit-problem}.
\end{corollary}

\nc%
We conclude this section by deriving from 
the 
  transversality conditions,  another
important structural property of the critical set $\crit$. 
\begin{proposition}
\label{prop:7.2}
Under the assumptions of Theorem  \ref{prop:7.1},
for every $\rho>0$ and
for almost every $t\in (0,T)$
all the points in $\critset(t)\cap \sublevel\rho$ are non-degenerate.
\end{proposition}

\begin{proof}
To prove the proposition, we 
 show that the set
\begin{equation*}
\{t\in[0,T]:\gder 2{\cE} t{u}
\mbox{ is singular for some }u\in\critset(t)\cap \sublevel\rho\}
\end{equation*}
has null measure.
From Theorem \ref{prop:7.1} (2) we know that
there exists a finite number $M_\rho$ of  $\rmC^1$-curves
$\gamma_i=(\curv t_i,\curv u_i)
:[0,1]\to[0,T]\ti\Xsp$
such that \eqref{eq:strutt_C} holds.
In particular, differentiating the equation
$\gder{} {\cE}{\curv t_i(s)}{\curv u_i(s)}=0$ w.r.t. $s$, we get
\begin{equation}
\label{eq:id_diff_i}
\curv t_i'(s)\partial_t \gder {}\cE{\curv t_i(s)}{\curv u_i(s)}
+
\gder 2\cE{\curv t_i(s)}{\curv u_i(s)}[\curv u_i'(s)]
=0
\qquad\mbox{for every }
s\in[0,1],\ i=1,\hdots, M_\rho.
\end{equation}
Using this identity, the fact that $\gamma_i'\neq(0,0)$ in $[0,1]$,
and the \emph{partial} transversality conditions,
it is easy to show that
\begin{equation}
\label{eq:A_i_B_i}
A_i:=
\big\{s\in[0,1]:\curv t_i'(s)=0\big\}
=
\big\{s\in[0,1]:\gder 2\cE{\curv t_i(s)}{\curv u_i(s)}
\mbox{ is singular}\big\}
=:B_i.
\end{equation}
Moreover, it is elementary to check that
\begin{equation*}
\bigcup_{i=1}^{M_\rho}\curv t_i(B_i)
=
\{t\in[0,T]:\gder 2{\cE} t{u}
\mbox{ is singular for some }u\in\critset(t)\cap \sublevel\rho\}.
\end{equation*}
The last identity, together with \eqref{eq:A_i_B_i},
implies that
\begin{equation*}
\big|\{t\in[0,T]:\gder 2{\cE} t{u}
\mbox{ is singular for some }u\in\critset(t)\cap \sublevel\rho\}\big|
\leq\sum_{i=1}^{ M_\rho}
\big|\curv t_i(A_i)\big|.
\end{equation*}
We can thus conclude the proof observing that
$|\curv t_i(A_i)|=0$ for every $i=1,\hdots, M_\rho$,
in view of Sard's Lemma
and of the regularity of the 
functions
$\curv t_i$.
\end{proof}


\subsection{Results under a further transversality condition}
\label{s:7.2}
As we have seen, the sole  transversality conditions 
from Definition \ref{def:transv}
are sufficient to ensure that the critical set $\crit$ enjoys the crucial property \eqref{intersection-sublevels}. 
\par
We now introduce a third transversality condition. 
To state it properly, we need to enhance \eqref{basic-hyps-transv-b}, by requiring that 
\[
 \cE \in \mathrm{C}^3 ([0,T]\ti \opx). 
 \]
The additional transversality condition
 at a  point $(t_0,u_0) \in \sing$ reads as follows:
 \begin{enumerate}
\item
[\rm ({T}.3)]
\label{assT3}
  If $0\neq v\in \NN(\gder 2\cE {t_0}{u_0})$, then
  $
  \gder 3\cE {t_0}{u_0}[v,v,v]
    \neq 0$.
\end{enumerate}
We will say that 
 $\cE$ satisfies the \emph{full set of transversality conditions} if 
\ref{assT1}--$\mathrm{(T.3)}$ hold at every $(t_0,u_0)\in \sing$.
\begin{remark}
\label{rmk:full-tc}
\slshape
Indeed, it is under the full set of  transversality conditions that the convergence results in \cite{Zanini} were proved.
As we will see in Section \ref{s:8}, the full set of transversality conditions  is noticeably stronger than \ref{assT1}--\ref{assT2}:  both of  \ref{assT1}--\ref{assT2} and 
\ref{assT1}--
$\mathrm{(T.3)}$
 are \emph{generic}, but the genericity of \ref{assT1}--$\mathrm{(T.3)}$ is weaker than that of 
 \ref{assT1}--\ref{assT2}.  
\end{remark}

First of all, under the additional $\mathrm{(T.3)}$  
we  have the following result,
proved in \cite{AgRoSa15},
ensuring that the critical set $\critset(t)$ is discrete \emph{at every point} $t\in [0,T]$. 
\begin{proposition}
\label{prop:7.3}
In the setting of  \eqref{spaces},
let $\cE:[0,T]\ti \Hilbert \to (-\infty,+\infty]$  comply
 with \eqref{coercivita}--\eqref{lambda-convex},
 \eqref{basic-hyps-transv},
and  with the full set of
transversality
conditions.

Then, in addition to all the conclusions of Corollary~\ref{cor:obvious},
for every $\rho>0$ and every $t\in[0,T]$ the set
$\critset(t)\cap \sublevel\rho$
consists of finitely many points.

\end{proposition}

\begin{proof}
It was proved in \cite[Thm.\ 2.5, Cor.\ 3.6]{AgRoSa15} that
\emph{for every $t\in [0,T]$} the fiber
 $\critset(t)$ consists of isolated points in $\Xsp$. Then, \eqref{basic-hyps-transv-e} ensures that the critical points in  $\critset(t) \cap \sublevel\rho$ are also isolated with respect to the topology of
 $\Hilbert$; since $\critset(t) \cap \sublevel\rho$  is  a compact subset of $\Hilbert$, the thesis follows. 
\end{proof}
\par
With Theorem \ref{th:counter-2} below,
 we deduce further properties of  Dissipative Viscosity (which are a fortiori Balanced Viscosity, as Theorem 
\ref{mainth:2} is in force)
  solutions: in particular,
we provide an enhanced description of their jump set.
\begin{theorem}
\label{th:counter-2}
In the setting of  \eqref{spaces},
let $\cE:[0,T]\ti \Hilbert \to (-\infty,+\infty]$  comply
 with \eqref{coercivita}--\eqref{lambda-convex},
 \eqref{basic-hyps-transv},
and  with the full set of
transversality
conditions. Let $u$ be a \emph{\DSolution} to  \eqref{limit-problem}, with jump set $\jumpname$.

Then, by Corollary~\ref{cor:obvious}, $u$ is in fact a Balanced
Viscosity solution, and moreover:
\begin{enumerate}
\item
if  $  u(t-)  \in \singset(t)$ at some $t\in (0,T]$, then
$t\in \jumpname$;
\item $\jumpname$ consists of finitely many points.
\end{enumerate}
\end{theorem}

\begin{remark}
\slshape
\label{rmk:added-V}
Note that the converse of property (1) is not true in general.
Indeed, when  $ u(t-) $   is a non-degenerate critical point for $\cE_t$
one cannot exclude that the system jumps at time $t$.
For example, if $ u(t-) $ is a saddle point, then the system may
jump along the direction associated with a negative eigenvalue of
 $\gder 2{\cE} t{u(t-)}$.
\end{remark}

\begin{proof}
$\vartriangleright (1)$
Let $t \in (0,T]$  be such that $u(t-)$ is a degenerate critical point.
A perusal of the proof of
\cite[Theorem 2.5]{AgRoSa15}
shows that there exist
a neighborhood   $(t-\delta,t+\delta) \ti B_{R}(u(t-))$ 
of $(t,u(t-))$,
an open interval $I \ni 0$, and a $\rmC^2$-curve 
 (as we are now supposing  $\cE \in \mathrm{C}^3 ([0,T]\ti \opx)$) 
$\parcur : I \to [0,T]\ti \Xsp$ such that \eqref{from-ars1-1} holds, i.e.\
$\crit \cap \left( (t-\delta,t+\delta) \ti B_{R}(u(t-)) \right)$ is parameterized by
$\parcur  = (\curv t,\curv u)$,
with
$(\curv t(0),\curv u(0))= (t,u(t-))$, and
$\curv t'(0)=0$.
Moreover, the validity  of the transversality condition
$\mathrm{(T.3)}$ ensures that the curve has
the additional property
$\curv t{''}(0) \neq0$.
In particular, the function $\curv t$ has a  strict local extremum at $s=0$.
Therefore,  for every $\tau \in (t,t+\delta)$  the functional 
 $\cE_{\tau}(\cdot)$ possesses   no critical points   in the ball $ B_{R}
(u(t-))$,
i.e.,
\begin{equation}
\label{no-crit-points}
\critset(\tau ) \cap B_{R}(u(t-)) = \emptyset
\qquad\mbox{for every}\quad
\tau \in (t,t+\delta).
\end{equation}
Therefore, $t$ is a jump point for $u$: otherwise, for all
$(t_n)_n$ with 
 $t_n \downarrow t$
the critical points $(u(t_n))_n$ would converge to
$u(t-) = u(t+)$,
against
\eqref{no-crit-points}.

$\vartriangleright (2)$
In order to prove that $\jumpname$ consists of finitely many points,
we show that $\jumpname$ consists of isolated points.
Let $t\in \jumpname$ and suppose by contradiction that
there exists a sequence $(t_n)_n\subset \jumpname$ such that $t_n\to t$.
Up to subsequences, we can always reduce to either of the cases
$t_n \uparrow t$ or $t_n \downarrow t$.
Let us focus on the first case, since the second one can be
treated similarly.
From the definition of $u_-(t_n)$, $u_+(t_n)$, and of $u(t-)$,
and from the fact that $t_n$ approaches $t$ from the left
it is easy to deduce that
\begin{equation}
\label{eq:cvg_lim_sx}
u_-(t_n)\rightarrow u(t-)\quad\mbox{and}\quad
u_+(t_n)\rightarrow u(t-),
\qquad\mbox{as}\quad n\to\infty.
\end{equation}
 We now consider two subcases:

\begin{description}

\item[Case $1$]

$u(t-)$ is a non-degenerate critical point for $\cE_t(\cdot)$, i.e.\ $\NN (\gder 2{\cE} t{u(t-)}) = \{ 0\}$.
By the implicit function theorem, there exist a neighbourhood
$ (t{-}\delta, t{+}\delta) \ti U$ of $(t,u(t-))$
and a $\rmC^2$-curve
$\curv u:(t-\delta, t+\delta)   \to U$ such that
$
\crit\cap ( (t{-}\delta, t{+}\delta) {\ti} U)
= \{ (\tau,\curv u(\tau))\, : \ \tau \in (t-\delta, t+\delta) \} $  and
$\curv u(t)= u(t-)$.
But from \eqref{eq:cvg_lim_sx} we have that, for $n$ sufficiently big,
$(t_n,  u_-(t_n)) , \, (t_n,  u_+(t_n))  \in\crit\cap  ( (t{-}\delta, t{+}\delta) {\ti} U)$,
which contradicts the fact that $\crit\cap  ( (t{-}\delta, t{+}\delta) {\ti} U)$ is a graph.

\item[Case $2$]

$u(t-)$ is a degenerate critical point for $\cE_t$.
In this case, the same arguments as in the proof of (1)
yield that there exists a $\rmC^2$-curve
$\gamma=(\curv t,\curv u) : I \to [0,T]\ti \Xsp$
parameterizing $\crit \cap ((t-\delta,t+\delta)\ti U)$, with
$(t-\delta,t+\delta)\ti U$ a neighborhood of  $(t,u(t-))$,
such that  $(\curv t(0),\curv u(0))=(t,u(t-))$, 
$\curv t'(0)=0$ and
$0\neq\curv u'(0)\in\NN (\gder 2{\cE} t{u(t-)})$,
and
\begin{equation*}
\langle
\partial_t \gder {}{\cE}{t}{u(t-)}, \curv u'(0)
\rangle\neq0,
\qquad\qquad
\curv t''(0)=
-\frac{\gder 3\cE {t}{u(t-)}[\curv u'(0)]^3}
{\langle  \partial_t \gder {}\cE{t}{u(t-)}, \curv u'(0) \rangle}
\neq 0.
\end{equation*}
To proceed, we need two claims.

\smallskip
\emph{\underline{Claim (A)}}.
The function $\mathfrak{e}(s):= \ene {\curv t(s)}{\curv u(s)}$
has an  inflection point at $s=0$.

\smallskip
\noindent To see this, note first that the function
$\mathfrak{e}(s)$ fulfills
 \[
 \mathfrak{e}'(0)=
\curv t'(0) \partial_t  \ene {t}{ u(t-)}
+\langle \gder {}{\cE}{t}{u(t-)}, \curv u'(0) \rangle
=0,
 \]
since $\gder {}{\cE}{\curv t(s)}{\curv u(s)}=0$ for every $s\in I$,
and, neglecting null terms, we find
\[
\mathfrak{e}{''}(0) =
\curv t{''}(0) \partial_t  \ene {t}{u(t-)}
+  \curv t'(0)^2 \partial_t^2  \ene {t}{u(t-)}+
\curv t'(0)
\langle  \partial_t \gder {}\cE{t}{u(t-)}, \curv u'(0) \rangle
=0,
\]
as we may suppose that
$\partial_t  \ene {t}{u(t-)}=0$,
up to adding to $\cE$ a time-dependent function.
Another straightforward computation then gives
\[
\mathfrak{e}{'''}(0) =
2  \curv t{''}(0)
\langle
\partial_t \gder {}{\cE}{t}{u(t-)}, \curv u'(0) \rangle
\neq 0.
\]
Thus, the claim is proved.
In particular,
from now on we may suppose without loss of generality that
$\mathfrak{e}$ is strictly increasing in a neighborhood of $0$.
Hence, in an interval
$(-\nu,\nu)\subseteq I$ there holds 
\begin{equation}
\label{eq:flesso_e}
\mathfrak e(s) < \mathfrak e(0) < \mathfrak e(s')
\qquad\text{for every}\quad
-\nu< s < 0 < s' < \nu.
\end{equation}
Note that, from the behavior of the function $\curv t$ at $0$,
we have that, up to a smaller $\delta$,
$\curv t((-\nu,\nu))=(t-\delta, t]$, and that
every $\hat t\in(t-\delta, t)$ can be written as
\begin{equation}
\label{eq:hat_t}
\hat t=\curv t(\hat s)=\curv t(\hat s'),
\qquad\mbox{for some}\quad
\hat s\in(-\nu,0),\
\hat s'\in(0,\nu).
\end{equation}

\smallskip
\emph{\underline{Claim (B)}}. There exists a left-neighborhood
of $t$ where the limit solution $u$ jumps at most once.

\smallskip
\noindent Indeed, let $\hat t$ be as in \eqref{eq:hat_t}
and suppose it belongs to the jump set $\jumpname$, so that
\begin{equation}
\label{eq:two_possib}
(u(\hat t-),u(\hat t+))=(\curv u(\hat s),\curv u(\hat s'))
\quad\mbox{or}\quad
(u(\hat t-),u(\hat t+))=(\curv u(\hat s'),\curv u(\hat s)).
\end{equation}
Note that the first possibility can be excluded, otherwise
from \eqref{eq:flesso_e} one would get
\begin{equation*}
\ene {\hat t}{u(\hat t-)}=\mathfrak e(\hat s)
<
\mathfrak e(\hat s')=\ene {\hat t}{u(\hat t+)},
\end{equation*}
which contradicts the fact that
$\ene {\hat t}{u(\hat t-)}-\ene {\hat t}{u(\hat t+)}=
\cost {\hat t}{u(\hat t-)}{u(\hat t+)}\geq0$.
Therefore, the second identity in \eqref{eq:two_possib}
holds true. Now, suppose by contradiction that
$\jumpname \cap(\hat t, t)\neq\emptyset$, and in particular that
$\tilde t\in(\hat t, t)$ is the
first jump time following $\hat t$. Clearly, the point
$(\tilde t,u(\tilde t-))$
lies on the same branch of solutions to
$\gder {}\cE{(\cdot)}{\cdot}=0$ as $(\hat t,u(\hat t+))$,
that is
$(\tilde t,u(\tilde t-))=(\curv t(\tilde s\,),\curv u(\tilde s\,))$,
for some $\hat s<\tilde s<0$. Then, it must be
$(\tilde t,u(\tilde t+))=(\curv t(\tilde s'),\curv u(\tilde s'))$,
for some $0<\tilde s'<\hat s'$, and in turn, again from
\eqref{eq:flesso_e}, we get a contradiction.
This proves the claim.

\smallskip
The latter claim, together with the convergences in
\eqref{eq:cvg_lim_sx}, gives the desired contradiction to the
fact that the sequence of jump points $t_n$ converge s
(from the left) to $t$.
\end{description}
This concludes the proof.
\end{proof}


\section{Applications}
\label{s:8}
This section revolves around the applicability of our main results, Theorems \ref{mainth:1} \& 
 \ref{mainth:2}. In particular, we will focus on the 
 verifiability of the  conditions of  Theorem \ref{mainth:1}, i.e.,   \ass 1 and the countability of $\nototaldisc$, 
 as well as on the validity of  the rectifiability condition \eqref{eq:15} on the critical set in Theorem  \ref{mainth:2}.
 As we have demonstrated in Section \ref{sec:6}, the transversality conditions  \ref{assT1}--\ref{assT2} 
 do imply \eqref{eq:15}  and, a fortiori, 
 \ass 1 and the countability of $\nototaldisc$.
 That is why, it is 
  relevant to directly  address the verifiability of  \ref{assT1}--\ref{assT2}.
\subsection{Genericity  of the transversality conditions} 
\label{ss:9.1}
 \par
  In what follows, we will show that \ref{assT1}--\ref{assT2} are \emph{generic}. Namely, 
 along the lines of
\cite{AgRoSa15}  we will prove that there is  a ``reasonably big'' set of perturbations by which it is possible
 to modify a given functional $\calE$,  in order  to obtain  a new one
  that satisfies the  transversality conditions.
This is basically the content of Theorem \ref{thm:weak-perturb} below. 
\par
In order to state it,  we need to recall some preliminary notions:   
a set  in a topological space is said to be \emph{nowhere dense} if the interior of its closure is empty, which is equivalent to the fact that it  is contained in the complement of an open and dense set.
 A set $B$ is said to be \emph{meagre} if it is contained in the countable union of nowhere dense sets. Finally, we call the complement of a meagre set \emph{residual}: in other words, a residual set contains the intersection of a countable collection of open and dense sets.
\par
We are now in a position to give the following.
\begin{theorem}
\label{thm:weak-perturb}
In the setting of \eqref{spaces}, let $\calE \in \rmC^3([0,T]\times\Xsp)$
satisfy    \eqref{basic-hyps-transv}.
\par
Then, every open neighborhood of $V$ of the origin in $\Ysp$ contains a \emph{residual} set
$V_r$ such that for every $y \in V_r$ the functionals
\begin{equation}
\label{lo-perturbed}
\calE^{y}_t(u) : =\ene tu + \pairing{\Xsp^*}{\Xsp}{y}{u}
\end{equation}
satisfy the transversality conditions.
 Moreover, if $\calE$ satisfies the standing conditions $\mathrm{(E)}$, so does $\calE^y$ for every $y\in\Ysp$ (the perturbation being affine), and Theorem \ref{prop:7.1} then gives \eqref{eq:15} for $\calE^y$; hence, for every $y\in V_r$, \emph{all the conclusions of Theorem \ref{mainth:2} hold for $\calE^y$}. 
\end{theorem}
\begin{proof}
It was shown in \cite[Lemma 2.4]{AgRoSa15} that a functional $\calE$ fulfilling  
 \eqref{basic-hyps-transv}
complies with the partial transversality conditions at its singular critical points if and only if $0$ is a regular value for the field $f= \rmD \calE : [0,T] \ti  \Xsp \to \Ysp$, meaning that,
 at every $(t_0,u_0)$ with $f(t_0,u_0)=0$, the mapping  $\rmD f(t_0,u_0)$ is onto.
Let us now consider the augmented field
\begin{equation}
\label{augmented-field-s:7}
F: [0,T] \ti  \Xsp  \ti \Ysp \to \Ysp \ \text{ defined by } \ F(t,u,y):= f(t,u) +y. 
\end{equation}
We will then prove that the set
\[
\mathfrak{Y} := \{ y \in \Ysp\, : \ 0 \text{ is a regular value of the map } (t,u)\mapsto F(t,u,y) \}  \text{ is a residual set}.
\] To show this, we apply \cite[Thm.\ 1.1, Rmk.\ A.1]{Saut-Temam79}, which requires to show that $0$ is a regular value of the map $F$ in \eqref{augmented-field-s:7}. This is clear, since
at any $(t_0,u_0,y_0) \in F^{-1} (0)$
the Fr\'echet differential of $F$ is the operator given by
$\rmD F (t_0,u_0,y_0)[\tilde{t},\tilde{u},\tilde{y}]:= \tilde{t} \partial_t f(t_0,u_0) + \rmD_u f(t_0,u_0)[\tilde u] +\tilde{y}$.
\end{proof}
\par
Let us now address the 
\emph{full} set of transversality conditions, i.e.,  \ref{assT1}--$\mathrm{(T.3)}$. It turns out that  conditions  \ref{assT1}--$\mathrm{(T.3)}$ as well are generic, albeit in a weaker
sense than \ref{assT1}--\ref{assT2}. More precisely, the class of perturbations that turn a given functional $\calE$ into one complying with 
 \ref{assT1}--$\mathrm{(T.3)}$  is of higher order than that for 
 \ref{assT1}--\ref{assT2}. In fact, our next result, 
 Theorem \ref{thm:strong-perturb}, which was proved in \cite[Thm.\ 3.2, Cor.\ 3.7]{AgRoSa15}, shows that  the 
  perturbations leading  to \ref{assT1}--$\mathrm{(T.3)}$    
   additionally feature a symmetric bilinear form on $\Xsp \ti \Xsp$ of the type
\[
\mathcal{K}(u,v) = \sum_{j=1}^m \pairing{\Xsp^*}{\Xsp}{\omega_j}{u} \pairing{\Xsp^*}{\Xsp}{\omega_j}{v};
\]
we denote by
$\mathfrak{N}_{\mathrm{sym}}$ the closure in $\mathcal{L}(\Xsp \ti \Xsp;\R)$ of the set of  all such  symmetric forms.
We will continue to denote by $\mathcal{K}$ the elements in $\mathfrak{N}_{\mathrm{sym}}$.
For a further technical reason that we do not detail here, in Thm.\  \ref{thm:strong-perturb}
we have to require $\calE \in \rmC^4 ([0,T]\ti \Xsp)$ .
\begin{theorem}[\cite{AgRoSa15}]
\label{thm:strong-perturb}
In the setting of \eqref{spaces}, let $\calE$ satisfy \eqref{basic-hyps-transv}, with $\calE \in \rmC^4 ([0,T]\ti \Xsp)$,  and \eqref{basic-hyps-transv-c}.
Then, every open neighborhood of $U$ of the origin in $\Ysp \ti \mathfrak{N}_{\mathrm{sym}}$ contains a \emph{residual} set
$U_r$ such that for every $(y, \mathcal{K})  \in U_r$ the functionals
\begin{equation}
\label{lo-perturbed-2}
\calE^{y,\mathscr{K}}_t(u) : =\ene tu + \pairing{\Xsp^*}{\Xsp}{y}{u} + \frac12 \mathcal{K}(u,u)
\end{equation}
satisfy the full transversality conditions.
\end{theorem}
 \begin{remark}[The finite-dimensional case]
\slshape
\label{better-in-finite-dimension}
In the finite-dimensional case  $\Xsp = \Ysp =\R^d$, in the statements of Thms.\ \ref{thm:weak-perturb} \& \ref{thm:strong-perturb} it is possible to replace the words
``residual set'' by ``set of full (Lebesgue) measure''. The lines of the proofs are the same (cf.\ also \cite[Rmk.\ 3.4]{AgRoSa15}), but in finite dimension it is possible to
apply Sard's Theorem, instead of the result from \cite{Saut-Temam79}. It is this that  leads to the stronger property that the set of admissible perturbations has full measure.
\end{remark}

\subsection{Examples}
\label{ss:true-9.2}

As a first, direct illustration, let us verify the finite-dimensional Corollary \ref{cex:cor-findim} anticipated in the Introduction.
\begin{proof}[Proof of Corollary \ref{cex:cor-findim}]
(1) By the Morse--Sard theorem \cite{Sard42}, for every $t\in[0,T]$
the set of critical values of $\calE_t \in \rmC^d(\R^d)$ is $\mathscr
L^1$-negligible: this is precisely condition $\mathrm{(C)}$.

(2)  In this setting, conditions \eqref{spaces} and
\eqref{basic-hyps-transv} are obviously satisfied, with
$\Tsp=\Ysp=\Hilbert=\R^d$ and $\opx=\R^d$.
We can thus apply
Theorem \ref{thm:weak-perturb} together with Remark
\ref{better-in-finite-dimension}.
\nc
\end{proof}

We conclude this section by illustrating Thms.\ \ref{thm:weak-perturb} and \ref{thm:strong-perturb} with the following infinite-dimensional example. We refer to the discussion in \cite[Ex.\ 3.8]{AgRoSa15} for all details.
\nc

\begin{example}
\label{ex:revisited}
\upshape
Let $\Omega$ be a bounded connected open set in $\R^d$,
$d\leq 3$. Consider the following triple $(\Xsp,\Hilbert,\Ysp)$:
\[
\Xsp = H^2(\Omega) \cap H_0^1 (\Omega), \qquad
\Ysp =\Hilbert = L^2(\Omega),
\]
 and the energy functional $\calE\colon [0,T]\ti \Hilbert \to ({-}\infty,{+}\infty]$ 
\begin{equation}
\label{explicit-energy}
\ene tu:=
\begin{cases}
\int_\Omega \left( \frac12 |\nabla u(x)|^2 {+}\rmW(u(x))
{-} \ell(t,x) u(x) \right)\dd x & \text{if } u \in H_0^1(\Omega),
\\
+\infty & \text{otherwise}.
\end{cases}
\end{equation}
Here, $\rmW\colon \R \to \R$ is  the  double-well potential $\mathcal{W}(u) = (u^2-1)^2/4$, while  we have 
$\ell \in\rmC^3([0,T]; L^2(\Omega))$. That is why, 
 $\cE\in\rmC^3([0,T]\ti\Xsp)$.
The Fr\'echet differential of $u\mapsto \ene tu$ is the vector field
\[
f: = \rmD\cE \colon [0,T]\ti\Xsp \to\Ysp, \quad
f(t,u) =Au + \rmW'(u) - \ell(t),
\]
with $A \colon \Xsp \to \Ysp$ the Laplacian operator with homogeneous Dirichlet boundary conditions.
Note that $A$ is Fredholm with index $0$.
It follows from \cite[Thm.\ 5.26, p.\ 238]{Kato} that also
$f(\cdot, t)$ is a Fredholm map with index $0$
for every $t \in[0,T]$
and it is easy to check that the other conditions \eqref{basic-hyps-transv} are satisfied.
\par
Let us now illustrate explicitly the perturbations from Thm.\ \ref{thm:weak-perturb}: it 
  leads to a perturbed energy of the form
\[
\calE_t^y(u):=
\int_\Omega \left( \frac12 |\nabla u|^2 +\mathcal{W}(u) {+}y u 
{-} \ell(t) u \right)\dd x \qquad  \text{with } y \in L^2(\Omega),
\]
whereas an example of perturbation of the type in  Thm.\ \ref{thm:strong-perturb} leads to an energy
\[
\calE_t^{y,z}(u):=
\int_\Omega \left( \frac12 |\nabla u|^2  {+}\mathcal{W}(u) {+} \frac12 z u^2  {+}y u  
{-} \ell(t) u \right)\dd x \qquad \text{with } z \in \mathrm{C}^0 (\overline\Omega), \ y \in L^2(\Omega). 
\]
\end{example}
\nc
\begin{proof}[Proof of Theorem \ref{cex:th-AC-short}]
Conditions \eqref{basic-hyps-transv} for $\calE$ were checked in
Example \ref{ex:revisited}, and the standing conditions $\mathrm{(E)}$
follow from $\rmW''\geq -1$ and the compactness of the embedding
$H^1_0(\Omega)\hookrightarrow L^2(\Omega)$.
By Theorem \ref{thm:weak-perturb}, for $y$ in a residual set all the conclusions of Theorem \ref{mainth:2} hold for $\calE^y$. \nc
\end{proof}
\nc

\section{A double obstacle problem in an interval}
\label{ss:9.2}
Our last example concerns the nonsmooth
double-obstacle energy in a 1D interval $(-L,L)$,
We emphasize that its purely scalar counterpart is the relay hysteresis of Example
\ref{cex:ex-relay} in the Introduction, obtained by collapsing
the space $(-L,L)$ to a point
and relaxing the homogenous Dirichlet boundary conditions.
After setting up  the example,  we will show how the three levels of our
hierarchy  of results,   namely existence and consistency, compactness, and the
Balanced Viscosity structure,
are each realized, or fail to be, depending on the length $L$ of the underlying domain.

 Let $\Hilbert=L^2(-L,L)$ and consider
\begin{equation}
\label{cex:do-energy}
\EK :=\Big\{u \in H^1_0(-L,L):\ |u|\leq 1 \Big\},\qquad \ene tu:=
\begin{cases}
\displaystyle\int_{-L}^L \Big(\frac12 |u'|^2-\frac12 u^2 -\ell(t)\, u \Big)\dd x
& \text{if } u \in \EK ,
\\
+\infty &\text{otherwise},
\end{cases}
\end{equation}
with a spatially constant and piecewise strictly monotone loading $\ell \in \rmC^1([0,T])$: this is the Dirichlet energy with the (nonsmooth,
nonconvex) double-obstacle potential
$\rmW(u)=\rmI_{[-1,1]}(u)-\frac{u^2}2$
\cite{BloweyElliott91,ChenElliott94}, whose gradient flow is a 
double-obstacle parabolic problem.
Throughout this example we assume the standing non-degeneracy condition
\begin{equation}
\label{cex:do-nondeg}
2L \notin \pi\N,
\end{equation}
excluding the 
domain lengths 
at which the classification below
 switches to a different regime.
The standing conditions $\mathrm{(E)}$ hold trivially:   for each $L$, energy  sublevels
are bounded in $H^1_0(-L, L)$, hence compact in $L^2(-L,L)$;  the mapping $u\mapsto \ene tu
+\frac12\|u\|^2$ is convex ($\lambda$-convexity holds with $\lambda=1$);
$|\partial_t \ene tu| =|\dot\ell(t)\int_{-L}^L u| \leq 2L\,|\dot\ell(t)|$
on $\EK $.
The elements of the fiber $\critset(t)$, i.e.\ the critical points of $\ene t\cdot$, are the solutions of the stationary variational inequality
\begin{equation}
\label{cex:do-VI}
u \in \EK : \qquad \int_{-L}^L u'(w{-}u)'\dd x \geq \int_{-L}^L (u+\ell(t))(w{-}u)\dd x \quad \text{for all } w \in \EK
\end{equation}
(this can be obtained by applying the characterization \eqref{l-convex-charact} of the
subdifferential to
$v=(1-s)u+sw, s\in[0,1], w\in\EK$,
then dividing by $s>0$ and letting $s\downarrow0$). 
Thus, critical points 
 enjoy $W^{2,\infty}$-regularity by standard obstacle-problem
theory 
(see, e.g., \cite{Kin_Sta}).
Denoting by $q\in \{-1,1\}$ the two obstacles,  inequality \eqref{cex:do-VI} corresponds to
\begin{equation}\label{eq:residual-signs}
  r_u:=-u''-u-\ell,\qquad 
 \begin{cases}
  r_u=0&\text{on }\{|u|<1\},\\
  q\, r_u\leq0&\text{on }\{u=q\}
 \end{cases}
\end{equation}
where, here and below, $\ell$ is a short-hand for $\ell(t)$. 
\paragraph{\bfseries Rigidity and classification of the critical points}
On every open interval $I\subset (-L,L)$ where $|u|<1$ 
(a
\emph{noncoincidence set}),
\eqref{eq:residual-signs} reduces to the linear equation
\begin{equation}\label{eq:free-equation}
  -u''-u=\ell,\quad
  \text{so that $u=-\ell+R\cos(\cdot-\varphi)$ on $I$}.
\end{equation}
There is only one smooth solution satisfying the boundary conditions
and the constraint  $|u|<1$, namely
\begin{equation}\label{eq:unique-free-profile}
 u^{\mathrm f}_\ell(x)
 =\ell\left(\frac{\cos x}{\cos L}-1\right),\quad
 |\ell|<\ell_0(L)=\min_{0\le x\le L}\left|\frac{\cos L}{\cos
     x-\cos L}\right|=
 \begin{cases}
  \dfrac{\cos L}{1-\cos L},&0<L<\dfrac\pi2,\\[3mm]
  \dfrac{|\cos L|}{1+|\cos L|},&L>\dfrac\pi2,
 \end{cases}
\end{equation}
where the superscript ``$\mathrm f$'' stands for ``free'', since this
profile does not touch either obstacle.
For $L>\pi/2$, at $|\ell|=\ell_0(L)$ the free profile grazes an
obstacle and disappears; for $\pi/2<L<\pi$ this threshold is precisely
the fold value $\ell^*(L)$ of regime (ii) ahead, at which the free
profile merges with the terminating single-phase branch.
In order to determine the other solutions touching the obstacles, we
observe that, 
since $u \in \rmC^1$, at every contact point where
$u=q\in \{-1,1\}$ we should have $u'=0$: hence the arc adjacent to a
plateau $\{u= q\}$ is the \emph{rigid} profile
\begin{equation}
  \label{eq:rigid}
  -\ell+(\ell+q)\cos(\cdot-x_c)=
  q[1+(q\ell+1)(\cos(\cdot-x_c)-1)],
  \quad
  \text{$x_c$ being the contact point.}
\end{equation}
Consequently, every critical point is either  the \emph{smooth}
solution $u^{\mathrm f}_\ell$ 
from \eqref{eq:unique-free-profile},
or a $\rmC^1$ concatenation of plateaus lying on
the obstacles $\{u=\pm1\}$ with rigid arcs joining them, of the following four types
 ($q\in\{-1,1\}$ denoting the obstacle attained by the arc):
\begin{enumerate}
\item \emph{boundary arcs}, connecting the Dirichlet endpoint
  $u=0$ to the obstacle $q$: imposing the value $0$ on the profile
  \eqref{eq:rigid} shows that their length is
  \begin{equation}\label{eq:a-alpha}
   \alpha_q(\ell):=\arccos\left(\frac{q\ell}{1+q\ell}\right),
   \quad \text{defined whenever}\quad q\ell\geq-\frac12,
  \end{equation}
  with $\alpha_q(\ell)\geq\pi/2$ precisely when $q\ell\leq0$ 
  (see Figure \ref{cex:fig-stable});
\item boundary arcs that \emph{over-swing}: 
a boundary arc need not connect the obstacle $u=q$ to the Dirichlet boundary value $u=0$ along the shortest portion of the rigid profile. Starting from a contact point with the obstacle, the profile may instead cross the level $u=0$ at an interior point, continue through a droplet-like excursion, and then return to $u=0$ on the descending branch, with this second zero occurring exactly at the boundary point $x=\pm L$. Such arcs exist only when
$-\frac12 \le q\ell \le 0$ and their length is $
2\pi-\alpha_q(\ell)
$.
Unlike interior droplets and kinks, these arcs cannot be freely translated, since their second zero is required to occur exactly at the Dirichlet boundary $x=\pm L$. Hence they give rise only to finitely many additional stationary configurations, rather than continuous families of critical points.
\item interior arcs connecting two plateaus at $q$ (``droplets'')
  of the form \eqref{eq:rigid} if $-1\le q\ell\le 0;$
  they have length exactly $2\pi$, and their position is a \emph{free
    parameter}: they can be translated at no energetic cost;
  they can appear only if $2L>3\pi$, since for negative values of
  $q\ell$ each of the two arcs reaching the boundary has length at least $\pi/2$.
\item interior arcs connecting the two obstacles (``kinks''), still
  of the form \eqref{eq:rigid}: they
  exist \emph{only} at the times where $\ell(t)=0$, have length
  exactly $\pi$, and can likewise slide freely.
  They can appear only if $L>\pi$.
\end{enumerate}

The simplest situation,
 illustrated in Figure \ref{cex:fig-stable},
is given by solutions with just two boundary
arcs: setting $L_q(\ell):=L-\alpha_q(\ell)$ and assuming that $L_q(\ell)\ge0 $, we can
define the symmetric stationary profile
\begin{equation}\label{eq:profile}
 u_{q,\ell}(x)=
 \begin{cases}
  -\ell+(\ell+q)\cos\bigl(|x|-L_q(\ell)\bigr),
      &L_q(\ell)\le |x|\le L,\\[2mm]
      q,&|x|\le  L_q(\ell).
 \end{cases}
\end{equation}
The two boundary arcs have length $\alpha_q(\ell)$ and the contact
plateau is
\begin{equation}
\label{eq:contact-set}
 C_{q,\ell}
 =[-L_q(\ell),L_q(\ell)].
\end{equation}
 On its plateau,
\begin{equation}\label{eq:plateau-residual}
 r_{u_{q,\ell}}=-q-\ell=-q(1+q\ell).
\end{equation}
\begin{figure}[ht]
\centering
\begin{tikzpicture}[xscale=0.72,yscale=0.85,baseline=(current bounding box.center)]
  \draw[->,gray] (-1.3,0) -- (4.6,0) node[below,black] {$x$};
  \draw[densely dotted] (-1.3,1) -- (4.6,1);
  \draw[densely dotted] (-1.3,-1) -- (4.6,-1);
  \node[left] at (-1.3,1) {\small $1$};
  \node[left] at (-1.3,-1) {\small $-1$};
  \draw[very thick] (-1.1,-1) -- (0,-1);
  \draw[very thick] plot[domain=0:3.14159,samples=60,variable=\x] ({\x},{-cos(\x r)});
  \draw[very thick] (3.14159,1) -- (4.4,1);
  \draw[<->] (0,-1.55) -- (3.14159,-1.55);
  \node[below] at (1.5708,-1.55) {\small $\pi$};
  \node[above] at (1.5708,1.55) {\small \emph{kink} (only at $\ell=0$)};
\end{tikzpicture}
\hfil
\begin{tikzpicture}[xscale=0.4,yscale=0.85,baseline=(current bounding box.center)]
  \draw[->,gray] (-1.3,0) -- (7.9,0) node[below,black] {$x$};
  \draw[densely dotted] (-1.3,1) -- (7.9,1);
  \draw[densely dotted] (-1.3,-1) -- (7.9,-1);
  \node[left] at (-1.3,1) {\small $1$};
  \node[left] at (-1.3,-1) {\small $-1$};
  \draw[very thick] (-1.1,-1) -- (0,-1);
  \draw[very thick] plot[domain=0:6.2832,samples=100,variable=\x] ({\x},{-0.3-0.7*cos(\x r)});
  \draw[very thick] (6.2832,-1) -- (7.7,-1);
  \draw[<->] (0,-1.55) -- (6.2832,-1.55);
  \node[below] at (3.1416,-1.55) {\small $2\pi$};
  \node[above] at (3.1416,1.55) {\small \emph{droplet} (any $\ell\in(0,1)$)};
\end{tikzpicture}
\caption{The two types of freely sliding interior arcs, both solving $-u''-u=\ell$ between two contacts with the obstacles $u=\pm1$ (dotted). Left: a \emph{kink}, of length $\pi$, joining the two obstacles -- it exists only at $\ell=0$, since for $\ell\neq0$ the arc adjacent to $u=-1$ never reaches $u=1$. Right: a \emph{droplet}, of length $2\pi$, leaving and returning to the \emph{same} obstacle $u=-1$: a bump of the phase favored by the loading (here $\ell=0.3$, peak height $1-2\ell$) nucleating within the metastable background. Both can be translated along $x$ at no energetic cost, as long as they fit in the domain without colliding with another arc or the boundary.}
\label{cex:fig-droplet-kink}
\end{figure}

\begin{figure}[ht]
\centering
\begin{tikzpicture}[xscale=0.95,yscale=1.05,baseline=(current bounding box.center)]
  \draw[->,gray] (-4.3,0) -- (4.6,0) node[below,black] {$x$};
  \draw[densely dotted] (-4.3,1) -- (4.3,1);
  \draw[densely dotted] (-4.3,-1) -- (4.3,-1);
  \node[left] at (-4.3,1) {\small $1$};
  \node[left] at (-4.3,-1) {\small $-1$};
  \draw (-4,-0.07) -- (-4,0.07); \node[above left,inner sep=1pt] at (-4,0) {\small $-L$};
  \draw (4,-0.07) -- (4,0.07); \node[above right,inner sep=1pt] at (4,0) {\small $L$};
  \draw[very thick] (-1.9863,-1) -- (1.9863,-1);
  \draw[very thick] plot[domain=1.9863:4,samples=50,variable=\x] ({\x},{-0.3-0.7*cos((\x-1.9863) r)});
  \draw[very thick] plot[domain=-4:-1.9863,samples=50,variable=\x] ({\x},{-0.3-0.7*cos((-\x-1.9863) r)});
  \draw[very thick] (-2.6621,1) -- (2.6621,1);
  \draw[very thick] plot[domain=2.6621:4,samples=50,variable=\x] ({\x},{-0.3+1.3*cos((\x-2.6621) r)});
  \draw[very thick] plot[domain=-4:-2.6621,samples=50,variable=\x] ({\x},{-0.3+1.3*cos((-\x-2.6621) r)});
  \node[below] at (0,-1.06) {\small $u_{-1,\ell}$};
  \node[above] at (0,1.06) {\small $u_{+1,\ell}$};
  \draw[<->] (1.9863,-1.5) -- (4,-1.5);
  \node[below] at (3.0,-1.5) {\small $\alpha_{-1}(\ell)>\tfrac\pi2$};
  \draw[<->] (2.6621,1.5) -- (4,1.5);
  \node[above] at (3.33,1.5) {\small $\alpha_{+1}(\ell)<\tfrac\pi2$};
\end{tikzpicture}
{
\caption{The two \emph{stable} single-phase profiles $u_{\pm1,\ell}$ of
\eqref{eq:profile}, drawn for $\ell=0.3$: a plateau on one obstacle,
joined to the Dirichlet endpoints by boundary arcs of length
$\alpha_q(\ell)$ --- shorter than $\pi/2$ on the obstacle favored by
the loading ($q\ell>0$), between $\pi/2$ and $\pi$ on the disfavored
obstacle ($q\ell\le0$). The disfavored arcs reach the length $\pi$ exactly
at $q\ell=-\tfrac12$, where the branch folds.}
\label{cex:fig-stable}}
\end{figure}

Consequently, for every $t$, the fiber $\critset(t)$ is the union of \emph{finitely many} families of critical points, each parametrized by the ordered positions of its droplets/kinks. Whenever at least one such freely sliding structure is present, the corresponding family forms a one-parameter continuum of critical points; in the absence of droplets and kinks, the family reduces to a single critical point;
each family is Lipschitz-path-connected in $\critset(t)$ and, by Lemma \ref{le:lip-connected} 
, the energy is constant along it. In particular, the image set
\begin{equation}
\label{cex:do-C}
\ene t{\critset(t)} \ \text{is finite for every } t \in [0,T]:
\qquad \text{condition } \mathrm{(C)} \text{ holds for every }
L>0,\quad
2L\not\in \pi\N.
\end{equation}

In fact, the energy of every critical point is explicitly computable.
Multiplying $r_u=-u''-u-\ell$ by $u$ and integrating over $(-L,L)$ ---
the boundary term vanishes by the Dirichlet condition --- we obtain
$\int|u'|^2\dd x=\int r_u\,u\dd x+\int(u^2+\ell u)\dd x$; since, by
\eqref{eq:residual-signs} and \eqref{eq:plateau-residual}, $r_u$
vanishes off the contact set and $r_u\,u=-(1+q\ell)$ on a plateau lying
on the obstacle $q$, we conclude that
\begin{equation}
\label{cex:do-energy-formula}
\forall\, u \in \critset(t)\,: \qquad 
\ene tu=-\frac12(1+q\ell)\,\big|C_u\big|-\frac\ell2\int_{-L}^{L}u\dd x
\qquad \text{with } C_u:=\big\{x:|u(x)|=1\big\},
\end{equation}
where $q$ is the obstacle carrying the plateaus of $u$: for $\ell\neq0$
they all lie on the same obstacle, since joining the two would require
a kink; when $\ell=0$, or when $C_u=\emptyset$, the value of $q$ is
immaterial. The coefficient $1+q\ell$ is precisely the strength of the
plateau multiplier \eqref{eq:plateau-residual}. At $\ell=0$ the energy
reduces to minus one half of the length of the contact set, the free
arcs contributing nothing. Moreover, \eqref{cex:do-energy-formula}
makes two facts evident: inserting a droplet in a plateau raises the
energy by exactly $\pi(1+q\ell)^2$ (the length $|C_u|$ decreases by $2\pi$ and
$\int u$ changes by $-2\pi(\ell+q)$), and translating droplets or
kinks neither  changes $|C_u|$ nor $\int u$, so the energy is constant
along each family, as used in \eqref{cex:do-C}.

\paragraph{\bfseries Stable and unstable critical points}
The classification just obtained has a sharp stability counterpart,
which singles out the branches that the limit evolution can actually
follow, and rests on an elementary remark: a critical point $u$
containing a free arc of length $d>\pi$ is \emph{not} a local
minimizer. Indeed, one has  $|u|<1$ along the open arc, so every perturbation
$w$ supported in its interior is admissible (in both directions) for
small amplitudes; the first variation vanishes there by
\eqref{eq:free-equation} and, the energy being quadratic, one has the energy variation 
\begin{equation}
\label{cex:do-second-var}
\ene t{u+w}-\ene tu=\frac12\int_{-L}^{L}\big(|w'|^2-|w|^2\big)\dd x,
\end{equation}
which is negative when $w$ is a sine profile supported on a subinterval
of length in $(\pi,d)$. Hence:
\begin{itemize}
\item[-] every critical point other than $u_{\pm1,\ell}$ contains a
  free arc of length $\geq\pi$ --- the free profile spans $2L>\pi$ in
  the nonconvex regimes, droplets have length $2\pi$, over-swung
  boundary arcs  have length $2\pi-\alpha_q(\ell)>\pi$, and kinks have length  exactly $\pi$, the
  borderline case, where sliding is already an energy-neutral direction
  --- and is therefore \emph{never} a strict local minimizer;
\item[-] the single-phase profile $u_{q,\ell}$, whose free arcs have
  length $\alpha_q(\ell)<\pi$, is a strict local minimizer with
  \emph{quadratic growth} whenever $q\ell>-\tfrac12$, uniformly so on
  $q\ell\geq-\tfrac12+\eta$ (Appendix \ref{s:app-do})
  .
\end{itemize}
The \emph{stable part} of  the critical set $\crit$ thus consists exactly of the two
single-phase branches $\ell\mapsto u_{\pm1,\ell}$. Each branch
terminates at a \emph{fold}: for $L>\pi$ at $q\ell=-\tfrac12$, where
the boundary arcs reach the length $\pi$ and minimality and existence
are lost together; for $\pi/2<L<\pi$ already at $q\ell=-\ell^*(L)$
with $\ell^*(L)<\tfrac12$ as in regime (ii) below, where the plateau
shrinks to a point and the branch merges with the free profile.

\paragraph{\bfseries Four regimes}
The size of the domain now discriminates the applicability of the finer parts of the theory (throughout, recall the standing non-degeneracy assumption \eqref{cex:do-nondeg}; 
 justifications   are in Appendix \ref{s:app-do}).
\begin{itemize}
\item[\emph{(i)}] $2L<\pi$: the energy is  strictly  \emph{convex} (Appendix \ref{s:app-do}): the critical point is unique and the evolution is trivial (no jumps).
\item[\emph{(ii)}] $\pi<2L<2\pi$: the energy is nonconvex but neither
  droplets
   nor kinks fit in the domain (of length $2L$). Hence the fibers are \emph{finite} for \emph{every} $t\in[0,T]$,
   the set of topologically singular times 
    $\nototaldisc$ is $\emptyset$, and $\crit$ is the union of finitely many Lipschitz branches (graphs over time intervals of the explicitly parametrized configurations): $\crit$ is countably $\mathscr H^1$-rectifiable and \textbf{the full theory applies} (namely, Theorems \ref{mainth:1} \emph{and} \ref{mainth:2}).
    Moreover the evolution is nontrivial: by the reflection symmetry of the problem, the branch of solutions with a plateau at $-1$ (the metastable phase) is centered at $x=0$, and it exists only for
    \[
    \ell(t)\leq \ell^*(L):=\frac{|{\cos L}|}{1+|{\cos L}|}\in \Big(0,\frac12\Big)
    \]
    (Appendix \ref{s:app-do}), at which threshold the plateau shrinks
    to a point and the branch \emph{terminates}: a limit evolution
    starting in the metastable phase must jump, at $t^*=\inf\{t:
    \ell(t)>\ell^*(L)\}$, to the branch with a plateau at $+1$, with a
    positive, explicitly computable, dissipation cost --- a completely
    explicit \emph{nucleation-with-hysteresis} scenario, the nonsmooth
    PDE analogue of the fold mechanism of Example \ref{cex:ex-fold}.
    A loading sweeping back and forth beyond $\pm\ell^*(L)$ thus
    produces a symmetric hysteresis loop with jumps at $\ell=\pm\ell^*(L)$,
    entirely analogous to (but simpler than) the one described below for
    regime (iv)
    (Appendix \ref{s:app-do}).
  \item[\emph{(iii)}] $2\pi<2L<3\pi$:
    droplets still do not fit, and the only continua in the fibers
    are the sliding \emph{kinks}, at the finitely many times where
    $\ell(t)=0$ (recall that $\ell$ is piecewise strictly monotone):
    these times are precisely the elements of $\nototaldisc$, which is
    thus a finite set.
    $\crit$ is still countably $\mathscr H^1$-rectifiable (the kink
    continua contribute finitely many Lipschitz curves in the time
    slices $\{\ell(t)=0\}$) and \textbf{the full theory applies} (Theorems \ref{mainth:1} \emph{and} \ref{mainth:2}).
     Since $L>\pi$, the stable branches now persist up to
    $|\ell|=\tfrac12$: the evolution displays  the same hysteresis loop as in
    regime (iv) below, with jumps at $\ell=\pm\tfrac12$.
\item[\emph{(iv)}] $2L>{3\pi}$ \emph{(sharpness of the assumptions)}:
in this regime, the domain is sufficiently long to accommodate a droplet
of length $2\pi$, together with the two boundary arcs connecting the plateau to the Dirichlet endpoints. More precisely, let
$q\in\{-1,1\}$ denote the obstacle disfavored by the loading
($q\ell<0$). A configuration containing a single droplet leaves room for
the droplet to slide whenever $2\alpha_q(\ell)+2\pi<2L$,
or, equivalently,
$
\alpha_q(\ell)<L-\pi$.
Indeed, the strict inequality leaves a positive amount of plateau length
that can be distributed on either side of the droplet. Its position can
therefore vary over a nondegenerate interval, giving rise to a
one-parameter continuum of critical points in the fiber $\critset(t)$.
The above condition holds for every
$
0<|\ell|<\frac12,
$
when $L> 2\pi,$ and for every $
0<|\ell|<\ell^*(L),$
when
$
\frac{3\pi}{2}<L<2\pi.
$
Consequently, whenever $\ell(t)$ belongs to such
window, the fiber $\critset(t)$ contains a sliding-droplet family and is
therefore not totally disconnected. Since the window is open
and nonempty and $\ell$ is piecewise strictly monotone, the set $\nototaldisc$ is \emph{uncountable}.
Therefore, $\crit$ is not is \emph{not} countably $\mathscr H^1$-rectifiable, by Lemma~\ref{le:criterion}(3).
In turn, condition $\mathrm{(C)}$ still holds by \eqref{cex:do-C}.
Hence the existence and consistency claims  (1)--(3)
of Theorem
\ref{mainth:1} apply, whereas the compactness claims (4)--(5) and
Theorem \ref{mainth:2} are not applicable \emph{as stated}: this
natural PDE separates the hierarchy of our conditions, showing that
$\mathrm{(C)}$, the countability of $\nototaldisc$, and
\eqref{rectifiable-intro} are genuinely independent requirements.
The evolution, however, preserves the same structure, as we describe next.

\end{itemize}

\paragraph{\bfseries A full hysteresis loop ($2L>3\pi$)}
We conclude by following one evolution through a complete loading cycle
in the most interesting regime (iv); as noted above, regimes
(ii)--(iii) produce the same picture with simpler justifications. Let
$2L>3\pi$ and let the (piecewise strictly monotone) loading sweep from
$\ell(0)=\ell_-<-\tfrac12$ up to $\ell(\tfrac T2)=\ell_+>\tfrac12$ and
back to $\ell(T)=\ell_-$.

\emph{The initial critical state is forced.} 
Since $\ell(0)<-\frac12$, the classification above yields
$
\critset(0)=\{u_{-1,\ell(0)}\}$.
Indeed, the $+1$ single-phase profile requires
$\ell\geq-\frac12$, droplets and over-swung boundary arcs require
$-\frac12\leq q\ell\leq0$, kinks require $\ell=0$, and the free
profile requires $|\ell|<\ell_0(L)<\frac12$. 
Consequently, if the limiting initial datum $u_0$ is required to belong to $\critset(0)$, then necessarily $u_0=u_{-1,\ell(0)}$.  
By the symmetry $(\ell,u)\mapsto(-\ell,-u)$, the evolution eventually
returns to the $-1$ phase on the decreasing part of the loading cycle.

\emph{Persistence on the lower branch up to the fold.} 
On every interval where
$\ell(t)\leq\tfrac12-\eta$ the profile $u_{-1,\ell(t)}$ is a strict
local minimizer of $\ene t\cdot$ and is \emph{uniformly separated} from
the rest of $\critset(t)$ (see Lemma~\ref{lem:do-uniform-isolation} in Appendix \ref{s:app-do}): by Proposition
\ref{prop:stable-branch} the well-prepared vanishing-viscosity solutions
track it, and letting $\eta\to0$ the tracking holds up to the fold. The
limit evolution thus rides the lower branch across the \emph{whole}
droplet window: the sliding continua are present in every fiber it
crosses but, lying at a positive distance from the branch, are never
visited.

\emph{The jump.} At the time $t^*$ with $\ell(t^*)=\tfrac12$ the
boundary arcs of $u_{-1,\ell}$ reach the length $\pi$, strict minimality is lost, and the
branch ceases to exist altogether (the fold of the stable part
described above). The solution must jump, and the
critical state left for $\ell(t)>1/2$ 
is the $+1$ phase: writing $u^\mp:=u_{\mp1,1/2}$, the jump dissipates
the explicitly computable cost
\begin{equation}\label{cex:do-loop-cost}
\mathsf c(L)=\cost{t^*}{u^-}{u^+}=\ene{t^*}{u^-}-\ene{t^*}{u^+}
\,,
\end{equation}
which \eqref{cex:do-energy-formula} renders completely explicit;
it grows linearly with $L$, since the jump converts the whole metastable
plateau into the stable phase --- an \emph{extensive} energy release. On
the way back, the branch $u_{+1,\ell}$ folds at the time $t^{**}$ with
$\ell(t^{**})=-\tfrac12$ and the solution jumps back, at the same cost.
Altogether
\begin{equation}\label{cex:do-loop-data}
\rmJ=\disc u=\{t^*,t^{**}\},\qquad
\mu=\mathsf c(L)\,\big(\delta_{t^*}+\delta_{t^{**}}\big),
\end{equation}
and the total dissipation is $\mu([0,T])=2\mathsf c(L)$ (Figure \ref{cex:fig-hysteresis}).

\begin{figure}[ht]
\centering
\begin{tikzpicture}[xscale=1,yscale=1]
  \fill[gray!14]
    (-3.000,-2.057) -- (-2.700,-2.054) -- (-2.400,-2.051) -- (-2.100,-2.048)
    -- (-1.800,-2.044) -- (-1.500,-2.040) -- (-1.200,-2.036) -- (-0.900,-2.031)
    -- (-0.600,-2.026) -- (-0.300,-2.021) -- (0.000,-2.014) -- (0.300,-2.007)
    -- (0.600,-1.999) -- (0.900,-1.990) -- (1.200,-1.979) -- (1.500,-1.966)
    -- (1.800,-1.950) -- (2.100,-1.930) -- (2.400,-1.902) -- (2.700,-1.858)
    -- (3.000,-1.704) -- (3.000,2.057) -- (2.700,2.054) -- (2.400,2.051)
    -- (2.100,2.048) -- (1.800,2.044) -- (1.500,2.040) -- (1.200,2.036)
    -- (0.900,2.031) -- (0.600,2.026) -- (0.300,2.021) -- (0.000,2.014)
    -- (-0.300,2.007) -- (-0.600,1.999) -- (-0.900,1.990) -- (-1.200,1.979)
    -- (-1.500,1.966) -- (-1.800,1.950) -- (-2.100,1.930) -- (-2.400,1.902)
    -- (-2.700,1.858) -- (-3.000,1.704) -- cycle;

  \draw[->,gray] (-4.4,0) -- (4.6,0) node[below,black] {$\ell$};
  \draw[->,gray] (0,-2.75) -- (0,2.95) node[above,black] {\small $u$};

  \draw (3,-0.07) -- (3,0.07);
  \node[below right,inner sep=1.5pt] at (3,0) {\small $\tfrac12$};

  \draw (-3,-0.07) -- (-3,0.07);
  \node[above left,inner sep=1.5pt] at (-3,0) {\small $-\tfrac12$};

  \draw[densely dashed] (-2.485,0.779) -- (2.485,-0.779);
  \node[left] at (-2.485,0.779) {\small $u^{\mathrm f}_\ell$};

  \draw[densely dashed] (0,-1.217) -- (0,1.217);
  \node[right,inner sep=2pt] at (0.02,1.45) {\small kinks};

  \begin{scope}
    \clip (0,-2.5) rectangle (3,0);
    \draw[
      line width=9pt,
      gray!45,
      opacity=.55,
      line cap=butt,
      line join=round
    ]
    plot coordinates {
      (0.000,-1.040)
      (0.120,-1.057) (0.326,-1.086)
      (0.531,-1.114) (0.737,-1.141)
      (0.943,-1.168) (1.149,-1.194)
      (1.354,-1.219) (1.560,-1.243)
      (1.766,-1.265) (1.971,-1.285)
      (2.177,-1.303) (2.383,-1.317)
      (2.589,-1.324) (2.794,-1.317)
      (3.000,-1.217)
    };
  \end{scope}
  \begin{scope}
    \clip (-3,0) rectangle (0,2.5);
    \draw[
      line width=9pt,
      gray!45,
      opacity=.55,
      line cap=butt,
      line join=round
    ]
    plot coordinates {
      (0.000,1.040)
      (-0.120,1.057) (-0.326,1.086)
      (-0.531,1.114) (-0.737,1.141)
      (-0.943,1.168) (-1.149,1.194)
      (-1.354,1.219) (-1.560,1.243)
      (-1.766,1.265) (-1.971,1.285)
      (-2.177,1.303) (-2.383,1.317)
      (-2.589,1.324) (-2.794,1.317)
      (-3.000,1.217)
    };
  \end{scope}

  \node[below] at (2.05,-1.38) {\small droplets};
  \node[above] at (-2.05,1.38) {\small droplets};

  \draw[very thick]
  plot coordinates {
    (-3.900,-2.065) (-3.612,-2.063)
    (-3.325,-2.060) (-3.038,-2.058)
    (-2.750,-2.055) (-2.463,-2.052)
    (-2.175,-2.049) (-1.888,-2.045)
    (-1.600,-2.042) (-1.313,-2.038)
    (-1.025,-2.033) (-0.738,-2.029)
    (-0.450,-2.023) (-0.163,-2.018)
    (0.125,-2.011) (0.412,-2.004)
    (0.700,-1.996) (0.987,-1.987)
    (1.275,-1.976) (1.562,-1.963)
    (1.850,-1.947) (2.137,-1.927)
    (2.425,-1.899) (2.712,-1.856)
    (3.000,-1.704)
  };

  \draw[very thick]
  plot coordinates {
    (-3.000,1.704) (-2.712,1.856)
    (-2.425,1.899) (-2.138,1.927)
    (-1.850,1.947) (-1.562,1.963)
    (-1.275,1.976) (-0.988,1.987)
    (-0.700,1.996) (-0.413,2.004)
    (-0.125,2.011) (0.163,2.018)
    (0.450,2.023) (0.737,2.029)
    (1.025,2.033) (1.312,2.038)
    (1.600,2.042) (1.887,2.045)
    (2.175,2.049) (2.462,2.052)
    (2.750,2.055) (3.037,2.058)
    (3.325,2.060) (3.612,2.063)
    (3.900,2.065)
  };

  \draw[very thick,->] (-1.025,-2.033) -- (-0.900,-2.031);
  \draw[very thick,->] (1.025,2.033) -- (0.900,2.031);

  \draw[very thick,->] (3.000,-1.704) -- (3.000,2.00);
  \draw[very thick,->] (-3.000,1.704) -- (-3.000,-2.00);

  \node[circle,fill,inner sep=1.2pt] at (3.000,-1.704) {};
  \node[circle,fill,inner sep=1.2pt] at (-3.000,1.704) {};

  \node[below] at (-2.0,-2.06) {\small $u_{-1,\ell}$};
  \node[above] at (2.0,2.06) {\small $u_{+1,\ell}$};

\end{tikzpicture}

\caption{A schematic picture of the loading cycle in regime (iv),
for $L>2\pi$ (the vertical direction symbolically represents the state
$u$). The limit evolution rides
the stable single-phase branches $u_{\pm1,\ell}$ (solid; the folds at
$q\ell=-\tfrac12$ are marked by dots) and jumps between them at
$\ell=\pm\tfrac12$ (vertical arrows), tracing a hysteresis loop.
The unstable free profile $u^{\mathrm f}_\ell$ and the sliding
kinks at $\ell=0$ (dashed) and the sliding-droplet families (bands: each fiber contains a continuum of critical points) 
are never visited by the evolution.
Away from the folds, the stable single-phase branches are uniformly
isolated from the remaining critical points; this
metric separation is not represented to scale.}
\label{cex:fig-hysteresis}
\end{figure}



 \appendix
 
 \section{Auxiliary results}
 \label{s:app-1} 
 The following results settle the properties of multivalued maps stated in Section \ref{subsec:preliminaries}.
 Recall that  \RRR $\Tsp$ is a compact interval $[a,b] \subset \R$, \EEE  and $(\Ysp,\sfd_{\Ysp})$ is a metric space.
 $\pi_{\Tsp}: \Tsp \ti \Ysp \to \Tsp$ denotes the projection onto
 $\Tsp$. 
 \begin{lemma}
 \label{l:A.1}
 Let $\mathbf{A} \subset \Tsp\ti \Ysp$ be
  compact  with $\pi_{\Tsp}(\mathbf{A}) =\Tsp$, and define the multivalued map
 \[
 \rmA: \Tsp \rightrightarrows \Ysp, \qquad  \rmA(t): = \{ x \in \Ysp\, : \ (t,x) \in \mathbf{A}\}.
 \]
 Then, $\rmA$ is upper semicontinuous.
 \end{lemma}
 \begin{proof}
 It is sufficient to show  that for every accumulation
point $\bar x\in \Tsp$  and every sequence $(x_n)_n$ with $d_{\Tsp}(x_n, x) \to 0$, every $v\in   \klimsup_{n\to\infty}\mathrm A(x_n)$ is contained in 
$ \rmA(\bar x)$. Now, by the definition of Kuratowski $\limsup$ we have that $\liminf_{n\to \infty} d_{\Ysp}(v,\mathrm A(x_n))=0$ so that, 
up to the extraction of a (not relabeled) subsequence there exists 
a sequence $(y_n)_n$ with $y_n \in \rmA(x_n)$ for every $n\in \N$ and 
$d_{\Tsp}(v,y_n) \to 0$ as $n\to\infty$. 
Since $\mathbf{A}$ is closed, we then have that $(\bar{x},v) \in \mathbf{A}$, i.e.\ $ v \in \rmA(\bar x)$ as desired.
\end{proof}
\par
Our next result relates the compactness properties of graphs to the compactness properties of the associated functions. From now on, we will assume that 
\begin{equation}
\label{both-spaces-compact}
\text{the metric space } 
(\Ysp,d_{\Ysp}) \text{ is compact.} 
\end{equation}
We also recall that a sequence of functions $\teta_n : \Omega \to \Ysp$, with $\Omega \subset \Tsp$ an open subset, 
is \emph{equicontinuous} at a point $x_0 \in \Omega $ if 
\begin{equation}
\label{equicontinuous}
\forall\, \eps>0 \ \exists\, \delta>0\, : \ \ \forall\, x,\, y \in B_\delta(x_0) \quad  
d_{\Ysp}(\teta_n(x),\teta_n(y))<\eps.
\end{equation}
\begin{lemma}
\label{l:compactness-graphs}
Suppose \eqref{both-spaces-compact}.  Let $(\teta_n)_n \subset  \mathrm{C}^0(\Tsp;\Ysp)$ and let $\mathbf{A}_n : = \mathrm{Graph}(\teta_n)$.  
\begin{enumerate}
\item The sequence $\teta_n\to \teta$ uniformly in $\Tsp$ for some $\teta \in  \mathrm{C}^0(\Tsp;\Ysp)$ if and only if 
the sequence 
$(\mathbf{A}_n)_n$ converges in the sense of Kuratowski  to $ \mathbf{A} :=  \mathrm{Graph}(\teta)$.
\item If $\mathbf{A}_n \karrow \mathbf{A}$ for some closed $ \mathbf{A} \subset \Tsp \ti \Ysp$, 
then 
there exists $\teta \in  \mathrm{C}^0(\Tsp;\Ysp)$ such that  $\teta_n\to \teta$ uniformly in $\Tsp$ and $\mathbf{A} =  \mathrm{Graph}(\teta)$.
\item If $\mathbf{A}_n \karrow \mathbf{A}$  and the sequence $(\teta_n)_n$ is equicontinuous in an open subset $\Omega \subset \Tsp$, then there exists a function 
$\teta:\Omega \to \Ysp $ such that
\begin{equation}
\label{2show-le:compactness}
\rmA(x) = \{ \teta(x)\} \quad \text{for every } x \in \Omega, \text{ and } \teta_n\to \teta \text{ uniformly on the compact subsets of } \Omega.
\end{equation}
\end{enumerate}
\end{lemma}
\begin{proof}
\textbf{Item (1):} Suppose that  $\teta_n\to\teta$ uniformly. Pick a sequence 
$(x_n,y_n)_n $ with 
 $(x_n,y_n) \in \mathbf{A}_n$ for every $n\in \N$. Then, any limit point $(x,y)$ of a converging subsequence 
$(x_{n_k},y_{n_k})_k $  satisfies $y_{n_k} = \teta_{n_k}(x_{n_k}) \to \teta(x)$
since $\teta_n \to \teta$ uniformly, 
 so that $y=\teta(x)$. This shows that $\mathbf{A}_n\karrow \mathrm{Graph}(\teta)$. 

\par In order to prove the converse implication, observe that, since $\Tsp$ is compact, the sequence $(\teta_n)_n \subset \mathrm{C}^0(\Tsp;\Ysp)$ is equicontinuous in $\Tsp$.
 By the Ascoli-Arzel\`a theorem it is relatively compact with respect to uniform convergence, hence any subsequence $(\teta_{n_k})_k$ admits a (not relabeled) subsequence uniformly converging to some continuous function. We will show that the limit function  does not depend on the subsequence by proving that the (whole) sequence $(\teta_n)_n$ pointwise converges to  the function $\teta$ such that $\mathbf{A} = \mathrm{Graph}(\teta)$. With this aim,
 let us fix $\eps>0$ and $x\in \Tsp$. 
 From  $\mathbf{A}_n\karrow  \mathrm{Graph}(\teta)$ we gather that there exist $(x_n, \teta_n(x_n) ) \in \mathbf{A}_n$ with $x_n\to x$ and $\teta_n(x_n)\to\teta(x)$. In particular, there exists $\bar{n}_1  \in \N$ such that 
 $d_{\Ysp}(\teta_n(x_n),\teta(x)) \leq \tfrac\eps2$ for every $n\geq \bar{n}_1$. On the other hand, by equicontinuity 
 of $(\teta_n)_n$
  there exists
 $\delta>0$ such that for every $y \in B_\delta(x)$ and every $n\in \N$  there holds  $d_{\Ysp}(\teta_n(y),\teta_n(x)) \leq \tfrac\eps2$  and then there exists $\bar{n}_2 \in \N$ such that $x_n \in B_\delta(x)$ for every $n\geq \bar{n}_2. $ All in all, for every $n\geq \bar{n}:= \max\{\bar{n}_1, \bar{n}_2\}$ we have
 \[
 d_{\Ysp}(\teta_n(x),\teta(x)) \leq   d_{\Ysp}(\teta_n(x_n),\teta_n(x))+d_{\Ysp}(\teta_n(x),\teta(x)) \leq \eps,
 \]
 which shows 
that $\teta_n(x)\to\teta(x)$ for every $x\in \Tsp$ and thus concludes the proof.
\par
\textbf{Item (2):}  Since the sequence $(\teta_n)_n \subset \mathrm{C}^0(\Tsp;\Ysp)$ is equicontinuous, 
by the Ascoli-Arzel\`a theorem it is relatively compact with respect to uniform convergence.
Any subsequence  $(\teta_{n_k})_k$ admits a (not relabeled) subsequence converging 
 uniformly to some $\teta \in  \mathrm{C}^0(\Tsp;\Ysp)$. By Item (1), $\mathbf{A}_{n_k} \karrow \mathrm{Graph}(\teta) = \mathbf{A}$. Hence, the limiting function is uniquely determined, and the \emph{whole} sequence $(\teta_n)_n$ converges to $\teta$. 
 \par
 \textbf{Item (3):} It is sufficient to apply Item (2) to every compact  subset $K\Subset \Omega$.   
\end{proof}

\par We now 
show that Kuratowski
 (or, equivalently in this context, Hausdorff) 
 convergence preserves \emph{fiberwise connectedness}.
 
 \begin{lemma}
  \label{le:connection}
  Suppose
  that
  $\mathbf
  {A}\subset \Tsp\ti \Ysp$
  is a compact set with $\pi_{\Tsp}(\mathbf A)=\Tsp$.
  \begin{enumerate}[label=\arabic*.]
  \item If every section $\rmA(t)$, $t\in \Tsp$, is connected then
    $\mathbf A\cap (I\times \Ysp)$ is 
       connected for every interval
       $I\subset \Tsp$. In particular, $\mathbf A$ is connected and, thus, fiberwise connected.
     \item If $\mathbf A\cap(I\times \Ysp)$ is connected for every
       open interval $I$, then $\mathbf A$ is fiberwise connected.
 \item If $\mathbf A$ is the Hausdorff limit of a sequence of 
   fiberwise connected compact graphs $(\mathbf A_n)_n$
   in $\mathfrak K(\Tsp\times \Ysp)$
   (in particular, of graphs of
   continuous maps $a_n:\Tsp\to \Ysp$),
   then $\mathbf A$ is fiberwise connected.
  \end{enumerate}
\end{lemma}
\begin{proof}
  $\vartriangleright$ 
  1. Let us first assume $I=[a,b]$ compact.
  
  We argue by contradiction and we suppose that $\mathbf A_I
  :=\mathbf A\cap (I\times \Ysp)$ is
  disconnected. We can find  two disjoint, nonempty, closed subsets
  $\mathbf F_1,\mathbf F_2$ of $\mathbf A_I$ such that
  $\mathbf A_I=\mathbf F_1\cup\mathbf F_2$.
  Since $\mathbf A_I$ is compact, the sets  $\mathbf F_i$ are compact.
  Hence there exist two disjoint open sets $\mathbf G_i$, $i=1,2$,
  such that $\mathbf F_i\subset \mathbf G_i$.
  The sets $G_i:=\pi_{\Tsp}(\mathbf G_i)$ are open, nonempty, and cover
  $I$ since $\pi_{\Tsp}(\mathbf A_I)=I$.
  Since $I$ is connected, we must have $G_1\cap G_2\neq\emptyset$.
  Pick a point $t\in G_1\cap G_2$.
  It follows that
  $\mathbf G_i\cap(\{t\}\times \rmA(t))\neq \emptyset$,
  a contradiction since the sets
  $\mathbf G_i$ are open,
  disjoint, and their union contains the connected set
  $\{t\}\times \rmA(t)$.
  
  In the general case, we can represent $I$ as an increasing union
  $\bigcup_{n\in \N}I_n$ of
  compact intervals $(I_n)_n$, so that $\mathbf A\cap (I\times \Ysp)$
  is the increasing union of the connected sets
  $(\mathbf A\cap (I_n\times \Ysp))_n$ and therefore it is
  connected as well.

  \par
  \par $\vartriangleright$  2.
  The argument is similar: we argue by contradiction
  and
  we suppose that there is a point $t\in \Tsp$
  such that $\rmA(t)$ is not connected. By the inner approximation argument in the 
  above lines,
  we may also assume that $\mathbf A_I$ is connected for every
  relatively open interval in $\Tsp$.

  We can decompose $\rmA(t)=F_1\cup F_2$ as the union of two nonempty, closed
  (thus compact), and disjoint 
  sets of $\Ysp$.
  We can find
  two disjoint open sets $G_i\subset \Ysp$
  such that $F_i\subset G_i$.

  Since $\mathbf A$ is compact the section map $\rmA$ is upper
  semicontinuous;
  since $G_1\cup G_2\supset \rmA(t)$, there exists 
  an interval $I$ relatively open in $\Tsp$ and containing $t$
  such that $\rmA(s)\subset G_1\cup G_2$ for every $s\in I$,
  i.e.~$\mathbf A\cap (I\times \Ysp)\subset I\times (G_1\cup G_2)$.
  But $\mathbf A\cap (I\times \Ysp)$ is connected
  and it has nonempty intersection with each open set
  $I\times G_i$, a contradiction.

  \par $\vartriangleright$  3.
  As in the previous claim, 
  we suppose that there is a point $t\in \Tsp$
  such that $\rmA(t)$ is not connected.
  We can find
  two disjoint open sets $G_i\subset \Ysp$, two elements $x_i\in \rmA(t)$,
  and a relatively open interval $I\ni t$ of $\Tsp$
  such that
  $\rmA(t)\cap G_i\ni x_i$ and
  $\mathbf A\cap (I\times \Ysp)\subset I\times (G_1\cup G_2)$.
  
  Let $J\ni t$ be a relatively open interval whose closure is contained in $I$;
  since 
  $\mathbf{A}$  is the Hausdorff (and thus Kuratowski) limit of the sets
  $(\mathbf{A}_n)_n$,
  we know that the sets
  $\mathbf B_n:=\mathbf A_n\cap (J\ti \Ysp)$ are contained in 
  $I\times (G_1\cup G_2)$ for $n$ sufficiently big.
  On the other hand, since
  $(t,x_i)\in \mathbf A$,
  by the definition of Kuratovski limit we know that $\mathbf B_n\cap
  (J\ti G_i)$
  is non-empty for sufficiently large $n$: this is impossible, since
  $\mathbf B_n$ is connected thanks to Claim 1.
\par
\end{proof}
 \par

\section{The double-obstacle energy: rigidity and classification}
\label{s:app-do}
This appendix collects the computations supporting  the classification of
Section  \ref{ss:9.2}
(Setup, Rigidity and classification, Three regimes), for the energy \eqref{cex:do-energy} on $(-L,L)$.

\paragraph{\bfseries Convexity threshold (regime (i))}
The first Dirichlet eigenvalue of $-\dd^2/\dd x^2$ on $(-L,L)$ is $(\pi/(2L))^2$, with (even) eigenfunction $\cos(\pi x/(2L))$; hence, for every $u\in H^1_0(-L,L)$,
\[
\Big(\frac{\pi}{2L}\Big)^2\int_{-L}^L u^2\,\dd x \ \leq\ \int_{-L}^L |u'|^2\,\dd x,
\]
with equality approached along the first eigenfunction. Consequently
\[
\int_{-L}^L\Big(\frac12|u'|^2-\frac12u^2\Big)\dd x \ \geq\ \frac12\Big(\Big(\frac{\pi}{2L}\Big)^2-1\Big)\int_{-L}^Lu^2\,\dd x,
\]
and the right-hand side is a (strictly) positive multiple of $\int u^2$ if and only if $2L<\pi$, i.e.\ $L<\pi/2$: in this case $\ene t\cdot$ is (strictly) convex on $\EK $, hence has a unique minimizer, which is the only critical point.

\paragraph{\bfseries The threshold $\ell^*(L)$ (regime (ii))}
Let $\pi/2<L<\pi$. The single-plateau branch $u_{-1,\ell}$ of \eqref{eq:profile} exists as long as its two boundary arcs fit inside the domain, that is $L_{-1}(\ell)=L-\alpha_{-1}(\ell)\geq0$, where
\[
\alpha_{-1}(\ell)=\arccos\Big(\frac{-\ell}{1-\ell}\Big)\qquad\big(\ell<\tfrac12\big)
\]
is the boundary-arc length \eqref{eq:a-alpha} for $q=-1$: a formula valid for \emph{every} $\ell<\tfrac12$, negative loadings included --- where the $-1$ phase is the favored one and $\alpha_{-1}(\ell)<\pi/2$ --- and increasing from $\alpha_{-1}(0)=\pi/2$ to $\pi$ as $\ell\to\tfrac12^-$. The plateau $C_{-1,\ell}=[-L_{-1}(\ell),L_{-1}(\ell)]$ shrinks to a point exactly when $L_{-1}(\ell)=0$, i.e.\ $\alpha_{-1}(\ell)=L$, i.e.\ $\cos L=-\ell/(1-\ell)$; since $\cos L\in(-1,0)$ for $\pi/2<L<\pi$, this has the unique solution
\begin{equation}
\label{cex:do-l-star}
\ell^*(L)\ =\ \frac{-\cos L}{1-\cos L}\ =\ \frac{|\cos L|}{1+|\cos L|}\ \in\ \Big(0,\tfrac12\Big),
\end{equation}
increasing from $0^+$ (as $L\to(\pi/2)^+$) to $\tfrac12^-$ (as $L\to\pi^-$). This is exactly the plateau-extinction fold of Proposition \ref{prop:strict}: at $\ell=\ell^*(L)$ the existence condition $L_{-1}(\ell)>0$ of \eqref{cex:do-exist} fails, the branch loses its plateau and merges with the free profile \eqref{eq:unique-free-profile}. By the symmetry $(\ell,u)\mapsto(-\ell,-u)$ the $+1$-branch folds symmetrically at $-\ell^*(L)$.
\par
\paragraph{\bfseries Stability and cost} The computation above produces this one-parameter branch of critical points and the value of $\ell^*(L)$; its quadratic-growth minimality while the plateau persists, the folds terminating it, and the dissipation cost of the ensuing jumps, for general $L$, are established  at the end of  this appendix.

\paragraph{\bfseries Stability of the critical points}
The classification of Section \ref{ss:9.2} can be sharpened from stationarity to genuine \emph{quadratic-growth} local minimality
 (see Prop.\ \ref{prop:strict}  for a precise statement) 
 --- the property a slowly driven gradient flow actually feels, and the one that rules out premature jumps. Throughout this subsection we write $\calE_\ell$ for the energy \eqref{cex:do-energy} with frozen loading value $\ell$, so that $\ene t\cdot=\calE_{\ell(t)}$, and we consider the single-phase profiles $u_{q,\ell}$ of \eqref{eq:profile}, assuming that they exist with a nontrivial plateau:
\begin{equation}
\label{cex:do-exist}
q\ell>-\tfrac12
\qquad\text{and}\qquad
L_q(\ell)=L-\alpha_q(\ell)>0 .
\end{equation}
Let $v\in\EK$ and $w:=v-u_{q,\ell}$. Since $\calE_\ell$ is quadratic, its Taylor expansion at $u_{q,\ell}$ has no remainder; evaluating the first variation via \eqref{eq:residual-signs} and observing that feasibility imposes $qw\leq0$ on the plateau $C_{q,\ell}$, where the residual is \eqref{eq:plateau-residual}, we obtain the \emph{exact} identity
\begin{equation}
\label{cex:do-exact-id}
\calE_\ell(v)-\calE_\ell(u_{q,\ell})
=(1+q\ell)\int_{C_{q,\ell}}|w|\dd x+\frac12\,Q(w),
\qquad
Q(w):=\int_{-L}^{L}\big(|w'|^2-|w|^2\big)\dd x,
\end{equation}
valid for \emph{every} $v\in\EK$. It contains the whole stability mechanism. The first term, carrying the strictly positive multiplier $1+q\ell>\tfrac12$, penalizes at first order \emph{any} excursion off the obstacle; the quadratic form $Q$, in turn, is coercive on the perturbations that \emph{vanish} on the plateau, since these reduce on each boundary arc to an element of $H^1_0$ of an interval of length $\alpha_q(\ell)<\pi$, so that Poincar\'e's inequality (sharp constant $(\alpha_q(\ell)/\pi)^2$, as in the convexity computation opening this appendix) gives
\begin{equation}
\label{cex:do-cone-coercive}
Q(w)\ \geq\ \Big(1-\big(\alpha_q(\ell)/\pi\big)^2\Big)\int_{-L}^{L}|w'|^2\dd x
\qquad\text{whenever } w=0\ \text{a.e.\ on } C_{q,\ell}.
\end{equation}
These two effects --- a first-order penalty off the obstacle, coercivity along the plateau --- already \emph{suggest} local minimality, but they cannot simply be superposed: a feasible perturbation is constrained on the plateau only by the \emph{one-sided} inequality $qw\leq0$, not by $w=0$, so the split into a linear ``critical cone'' plus a transverse direction is only heuristic and the two regimes must be made to interact. We do so directly, through a blow-up that couples the obstacle term of \eqref{cex:do-exact-id} with the coercivity \eqref{cex:do-cone-coercive}, and we measure minimality in the $L^2$ norm --- the metric of the gradient flow, and the one the tracking argument below actually uses.

\begin{proposition}[Strict local stability with quadratic coercivity estimate]
\label{prop:strict}
Assume \eqref{cex:do-exist}. Then $u_{q,\ell}$ is a strict local minimizer of $\calE_\ell$ in $\EK$, with \emph{quadratic coercivity}: there exist constants $\kappa,\rho>0$ such that
\begin{equation}
\label{cex:do-stability}
\calE_\ell(v)\ \geq\ \calE_\ell(u_{q,\ell})+\kappa\,\|v-u_{q,\ell}\|_{L^2}^2
\qquad\text{for every } v\in\EK \text{ with } \|v-u_{q,\ell}\|_{L^2}\leq\rho .
\end{equation}
\end{proposition}

\begin{proof}
Fix $\ell$ and suppose, by contradiction, that no pair
$(\kappa,\rho)$ works:
then there
are 
$v_n=u_{q,\ell_n}+w_n\in\EK$ with
\begin{equation}
\sigma_n:=\|w_n\|_{L^2}>0,\quad 
  \sigma_n\downarrow 0,
\quad\text{and}\qquad
(1+q\ell)\int_{C_{q,\ell}}|w_n|\dd
x+\frac12\,Q(w_n)=\calE_{\ell}(v_n)-\calE_{\ell}(u_{q,\ell})\leq\tfrac1n\,\sigma_n^2.\label{eq:55}
\end{equation}
Dropping the nonnegative obstacle term of \eqref{eq:55} and rearranging,
\[
\tfrac12\!\int_{-L}^{L}\!|w_n'|^2\dd x
\ \leq\ \big[\calE_{\ell}(v_n)-\calE_{\ell_n}(u_{q,\ell})\big]+\tfrac12\!\int_{-L}^{L}\!|w_n|^2\dd x
\ \leq\ \big(\tfrac1n+\tfrac12\big)\sigma_n^2,
\]
so $\|w_n\|_{H^1}^2\leq(2+o(1))\,\sigma_n^2$.
Hence $z_n:=w_n/\sigma_n$ has $\|z_n\|_{L^2}=1$ and is bounded in
$H^1$: up to a further subsequence, $z_n\rightharpoonup z$ in $H^1$,
hence $z_n\to z$ uniformly and in $L^2$, so $\|z\|_{L^2}=1$ and
$z\neq0$.

Using
$Q(w_n)\geq-\sigma_n^2$ and $1+q\ell\geq\tfrac12$ in
\eqref{cex:do-exact-id},
the same energy bound gives $$\int_{C_{q,\ell}}|w_n|\dd
x\leq3\sigma_n^2,
\qquad \int_{C_{q,\ell}}|z_n|\dd x\leq3\sigma_n\to0.$$
With the uniform convergence of $z_n$ this forces $z=0$ a.e.\ on
$C_{q,\ell}$.

Finally, dividing the energy inequality in \eqref{eq:55} by
$\sigma_n^2$ and discarding the obstacle term
gives
$$\limsup_nQ(z_n)\leq0;$$
since $\|z_n\|_{L^2}^2=1\to\|z\|_{L^2}^2$ and the gradient term is weakly lower semicontinuous, $Q(z)\leq\liminf_nQ(z_n)\leq0$. But $z$ vanishes on $C_{q,\ell}$, hence lies in $H^1_0$ of the two boundary arcs of length $\alpha_q(\ell)<\pi$, so \eqref{cex:do-cone-coercive} gives $Q(z)\geq\big(1-(\alpha_q(\ell)/\pi)^2\big)\int|z'|^2\dd x>0$ (as $z\neq0$ in $H^1_0$ of those arcs): a contradiction.
\end{proof}

\par
\paragraph{\bfseries The folds.} Both conditions in \eqref{cex:do-exist} are sharp,
and their failures are exactly the folds described in Section
\ref{ss:9.2}. For $\pi/2<L<\pi$ the branch ends at $\ell=\ell^*(L)$ with
$L_q(\ell)\downarrow0$: the plateau shrinks to a point and
$u_{q,\ell}$ merges with the free profile $u^{\mathrm f}_\ell$ of
\eqref{eq:unique-free-profile}, cf.\ \eqref{cex:do-l-star}; beyond,
neither survives. For $L>\pi$ the branch instead reaches
$q\ell=-\tfrac12$, where the boundary arcs attain the length
$\alpha_q(\ell)=\pi$, the coercivity \eqref{cex:do-cone-coercive}
degenerates, and the branch ceases to exist beyond (the angle
\eqref{eq:a-alpha} is no longer defined).

\begin{lemma}[Uniform isolation of the single-phase branches]
\label{lem:do-uniform-isolation}
Let $q\in\{-1,1\}$ and let $\eta>0$. Let $I\subset[0,T]$ be a
compact interval such that
$q\ell(t)\geq-\frac12+\eta$ for every $t\in I$.
Then there exists $\varrho=\varrho(\eta)>0$ such that
\[
\critset(t)\cap
B_{\varrho}\bigl(u_{q,\ell(t)}\bigr)
=
\{u_{q,\ell(t)}\}
\qquad\text{for every }t\in I.
\]
\end{lemma}

\begin{proof}
It is enough to consider the branch corresponding to $q=-1$, since the argument for the $+1$ phase is completely analogous. By the classification of the critical points, the opposite single-phase profile, the sliding droplets, and the kinks remain at a positive distance from the branch $u_{-1,\ell}$. The only critical configurations that may approach this branch are the over-swung boundary arcs and, in regime (ii), the free profile $u^{\mathrm f}_\ell$. These configurations merge with $u_{-1,\ell}$ only at the corresponding fold. Since the loading values under consideration range in a compact set bounded away from the fold, their distance from $u_{-1,\ell}$ admits a strictly positive uniform lower bound. This yields the claimed radius $\varrho>0$.
\end{proof}

\paragraph{\bfseries Tracking of well-prepared data.} The evolution scenario of
Section \ref{ss:9.2} is an instance of the general Proposition
\ref{prop:stable-branch}. Fix $\eta>0$ and restrict to a sub-interval 
of the increasing part of the loading cycle,
where $\ell(t)\leq\tfrac12-\eta$ (resp.\ $\ell(t)\leq\ell^*(L)-\eta$ in
regime (ii)). There the branch $v(t):=u_{-1,\ell(t)}$ is a \emph{strict
local minimizer} of $\calE_{\ell(t)}$ 
by
Proposition \ref{prop:strict})
and, by Lemma~\ref{lem:do-uniform-isolation}, 
is \emph{uniformly separated} from the rest of the critical set.
Proposition
\ref{prop:stable-branch} then gives $u_{\eps_n}\to v$ uniformly on the
sub-interval, with a \DSolution\ equal to $v$ there; letting
$\eta\to0$, the limit evolution coincides with $u_{-1,\ell(t)}$ up to
the fold time $t^*$, and never rests on the sliding-droplet continua,
which lie off the trajectory. \emph{Only the pointwise strict
minimality of the branch and the uniformity of its separation enter
here}.

At $t^*$ the branch loses minimality and ceases to exist (the fold
above), and the outer limit switches to the other single-phase branch,
the value at the single time $t^*$ being immaterial.


\bibliographystyle{siam}
\bibliography{ricky_lit.bib}

 \end{document}